\documentclass[11pt,reqno,a4paper]{amsart}
\usepackage[utf8]{inputenc}
\usepackage{mathtools,amsmath,amssymb}
\usepackage[T1]{fontenc}
\usepackage{lmodern}
\usepackage{microtype}
\usepackage{xspace}
\usepackage{bibentry}
\usepackage{easybmat}
\usepackage{mathdots}

\usepackage{cite}

\usepackage{xr-hyper}
\usepackage[
pdftex,
bookmarks=true,
colorlinks=true,
debug=true,
pdfnewwindow=true]{hyperref}
\hypersetup{bookmarksdepth = 3}

\usepackage[usenames,dvipsnames,svgnames,table]{xcolor}
\definecolor{todocolor}{HTML}{D7E1E5}
\usepackage[textsize=tiny,
            backgroundcolor=todocolor,
            bordercolor=todocolor,
            linecolor=todocolor,
            textwidth=3.9cm]{todonotes}

\usepackage{comment}
\usepackage{rotating}

\usepackage{subcaption}
\usepackage{bibentry}
\usepackage{tikz}
\usepackage{dynkin-diagrams}
\usepackage[margin=2.4cm]{geometry}

\usepackage{tikz}
\usetikzlibrary{decorations.pathreplacing}
\numberwithin{equation}{section}

\usepackage{amsmath}
\usepackage{amsthm}

\makeatletter
\let\@wraptoccontribs\wraptoccontribs
\makeatother

\theoremstyle{plain}
\newtheorem{Thm}{Theorem}[section]
\newtheorem{Cor}[Thm]{Corollary}
\newtheorem{Lem}[Thm]{Lemma}
\newtheorem{Prop}[Thm]{Proposition}

\newtheorem*{Thm-non}{Theorem}

\theoremstyle{definition}

\theoremstyle{remark}

\newtheorem{Rem}[Thm]{Remark}

\newcommand{\id}{{\ID }}

\newcommand{\ol}{\overline}
\newcommand\iso{\,\vphantom{j^{X^2}}\smash{\overset{\sim}{\vphantom{\rule{0pt}{0.20em}}\smash{\longrightarrow}}}\,}

\DeclareMathOperator{\diag}{diag}

\newmuskip\pFqmuskip
\newcommand*\pFq[6][8]{%
    \begingroup
    \pFqmuskip=#1mu\relax
    \mathcode`\,=\string"8000
    \begingroup\lccode`\~=`\,
    \lowercase{\endgroup\let~}\pFqcomma
    {}_{#2}F_{#3}{\left(\genfrac..{0pt}{}{#4}{#5};#6\right)}%
    \endgroup}
\newcommand{\pFqcomma}{\mskip\pFqmuskip}

\def\be{\begin{eqnarray}}
\def\ee{\end{eqnarray}}

\newcommand{\sfrac}[2]{{\textstyle\frac{#1}{#2}}}

\newcommand{\ID}{\mathrm{I}}

\newcommand{\rtt}{\mathrm{rtt}}
\newcommand{\End}{\mathrm{End}}

\newcommand{\opp}{\mathrm{op}}

\newcommand{\fg}{\mathfrak{g}}
\newcommand{\fh}{\mathfrak{h}}

\newcommand{\gl}{\mathfrak{gl}}
\newcommand{\ssl}{\mathfrak{sl}}

\newcommand{\BZ}{\mathbb{Z}}
\newcommand{\BC}{\mathbb{C}}

\makeatletter
\g@addto@macro\bfseries{\boldmath}
\makeatother

\makeatletter
\newcases{rrcases}{\quad}{%
  \hfil$\m@th\displaystyle{##}$}{$\m@th\displaystyle{##}$\hfil}{\lbrace}{.}
\makeatother

\newcommand{\fosp}{\mathfrak{osp}}

\newcommand{\sfc}{{\mathsf{c}}}
\newcommand{\sfC}{{\mathsf{C}}}

\newcommand{\DrJ}{\mathrm{DJ}}
\newcommand{\uqV}{U_{q}(\fosp(V))}
\newcommand{\UdqV}{U_{q}(\widehat{\fosp}(V))}
\newcommand{\UqV}{U'_{q}(\widehat{\fosp}(V))}

\newcommand{\sse}{\mathsf{e}}
\newcommand{\ssf}{\mathsf{f}}
\newcommand{\ssh}{\mathsf{h}}

\newcommand{\wtd}{\widetilde}

\newcommand{\sF}{\mathsf{F}}
\newcommand{\sU}{\mathsf{U}}
\newcommand{\sW}{\mathsf{W}}
\newcommand{\sL}{\mathsf{L}}
\newcommand{\chw}{\mathrm{tw}}
\newcommand{\rl}{\mathrm{l}}
\newcommand{\hgt}{\mathrm{ht}}
\newcommand{\sq}{\mathsf{q}}

\newcommand{\sfk}{{\mathsf{k}}}
\newcommand{\sft}{{\mathsf{t}}}

\begin{document}
\title[Orthosymplectic $R$-matrices]{\large{\textbf{Orthosymplectic $R$-matrices}}}

\author{Kyungtak Hong}
\address{K.H.: Purdue University, Department of Mathematics, West Lafayette, IN 47907, USA}
\email{hong420@purdue.edu}

\author{Alexander Tsymbaliuk}
\address{A.T.: Purdue University, Department of Mathematics, West Lafayette, IN 47907, USA}
\email{sashikts@gmail.com}

\begin{abstract}
We present a formula for trigonometric orthosymplectic $R$-matrices associated with any parity sequence, and establish
their factorization into the ordered product of $q$-exponents parametrized by positive roots in the corresponding reduced
root systems. The latter is crucially based on the construction of orthogonal bases of the positive subalgebra through
$q$-bracketings and combinatorics of dominant Lyndon words, as developed in~\cite{chw}.
We further evaluate the affine orthosymplectic $R$-matrices, establishing their intertwining property as well as matching
them with those obtained through the Yang-Baxterization technique of~\cite{gwx}. This reproduces the celebrated formulas
of~\cite{jim} for classical BCD types and the formula of~\cite{mdgl} for the standard parity sequence.
\end{abstract}

\maketitle
\tableofcontents


\section{Introduction}
\label{sec:intro}


\subsection{Summary}
\label{ssec:summary}
\

For classical Lie algebras $\fg$, the quantum groups $U^{\rtt}_q(\fg)$ first implicitly appeared in
the work of Faddeev's school on the \emph{quantum inverse scattering method}, see e.g.~\cite{frt}.
In this \emph{RLL realization}, the algebra generators are encoded by two square matrices $L^\pm$
subject to the famous \emph{RLL-relations}
\begin{equation*}
  RL^\pm_1L^\pm_2=L^\pm_2L^\pm_1R \,, \qquad RL^+_1L^-_2=L^-_2L^+_1R
\end{equation*}
(and some additional relations to kill the center), where $R$ is a solution of the \emph{Yang-Baxter equation}
\begin{equation}\label{eq:YBE-intro}
  R_{12}R_{13}R_{23} = R_{23}R_{13}R_{12} \,.
\end{equation}
This is a natural analogue of the matrix realization of classical Lie algebras, and it manifestly exhibits
the Hopf algebra structure, with the coproduct $\Delta$, antipode $S$, and counit $\epsilon$ given by
\begin{equation*}
  \Delta(L^\pm)=L^\pm\otimes L^\pm \,, \qquad   S(L^\pm)=(L^\pm)^{-1} \,, \qquad \epsilon(L^\pm)=\ID \,.
\end{equation*}
The uniform definition of quantum groups $U^{\DrJ}_q(\fg)$ for any Kac-Moody Lie algebra $\fg$ was provided
independently by Drinfeld~\cite{d0} and Jimbo~\cite{j0}, and is usually referred to as the
\emph{Drinfeld-Jimbo realization}. In this presentation, the generators $e_i,f_i,k^{\pm 1}_i=q^{\pm h_i}$
are labeled by simple roots $\alpha_i$ of $\fg$, while the Hopf algebra structure is given formally by the
assignment on the generators. In $A$-type, the corresponding isomorphism $U^{\rtt}_q(\gl_n)\simeq U^{\DrJ}_q(\gl_n)$,
and subsequently its $\ssl_n$-counterpart, were constructed in~\cite[\S2]{df} by considering
the Gauss decomposition of the generator matrices $L^\pm$.

The next important class of Kac-Moody Lie algebras is the so-called affine Lie algebras $\widehat{\fg}$,
which admit a similar Chevalley-Serre type presentation associated with extended Dynkin diagrams. It is
well-known that they are central extensions of the corresponding loop algebra $L\fg=\fg\otimes \BC[t,t^{-1}]$:
\begin{equation*}
  0\to \BC\cdot c \to \widehat{\fg} \to L\fg \to 0 \,.
\end{equation*}
The aforementioned construction of~\cite{frt} was extended to the loop setup of $L\fg$ in~\cite{frt2} by
crucially replacing the $R$-matrices satisfying~\eqref{eq:YBE-intro} with parameter-dependent $R$-matrices
$R(z)$ satisfying
\begin{equation}\label{eq:YBE-trig-intro}
  R_{12}(z)R_{13}(zw)R_{23}(w) = R_{23}(w)R_{13}(zw)R_{12}(z) \,,
\end{equation}
the so-called \emph{Yang-Baxter equation (with a spectral parameter)}. The generators of these algebras
$U^\rtt_q(L\fg)$ are now encoded by two square matrices $L^\pm(z)$ subject to analogous \emph{RLL-relations}
\begin{equation*}
  R(z/w)L^\pm_1(z)L^\pm_2(w)=L^\pm_2(w)L^\pm_1(z)R(z/w) \,, \qquad
  R(z/w)L^+_1(z)L^-_2(w)=L^-_2(w)L^+_1(z)R(z/w) \,.
\end{equation*}
Finally, this was generalized to $\widehat{\fg}$ in~\cite{rs}, thus producing $U^\rtt_q(\widehat{\fg})$
by incorporating the central charge. For classical $\fg$, this construction is an exact affine analogue
of the construction from~\cite{frt}.

There is yet another realization~\cite{d} of quantum affine groups $U_q(\widehat{\fg})$, which is usually
referred to as the \emph{new Drinfeld realization} (a.k.a.\ \emph{current realization}). The isomorphism
$U_q(\widehat{\fg})\simeq U^{\DrJ}_q(\widehat{\fg})$ was stated in~\cite{d} without a proof, while the complete
details appeared a decade later in the  work of Beck and Damiani. In $A$-type, the corresponding isomorphism
$U^{\rtt}_q(\widehat{\gl}_n)\simeq U_q(\widehat{\gl}_n)$, and subsequently its $\widehat{\ssl}_n$-counterpart,
were first constructed in~\cite{df} by considering the Gauss decomposition of the generator matrices $L^\pm(z)$,
similarly to the finite type. For affinizations of other classical Lie algebras such isomorphisms were first
discovered in~\cite{hm} and were revised more recently in~\cite{jlm2,jlm3}.

The above results also admit \emph{rational} counterparts, with quantum loop/affine groups replaced by the Yangians
$Y^J_\hbar(\fg)$ (in the \emph{$J$-realization}), first introduced in~\cite{d0}. The representation theory of these
algebras has been developed using their alternative \emph{(new) Drinfeld realization} $Y_\hbar(\fg)$
proposed in~\cite{d}, though it should be noted that the Hopf algebra structure is much more involved in this presentation.
One can also adapt~\cite{frt,frt2} to define $Y^\rtt_\hbar(\fg)$ in the \emph{RTT-realization}. In this presentation,
the algebra generators are encoded by a square matrix $T(u)$ subject to a single \emph{RTT-relation}
(together with an extra central reduction)
\begin{equation*}
  R(u-v)T_1(u)T_2(v) = T_2(v)T_1(u)R(u-v)
\end{equation*}
where the $R$-matrix $R(u)$ is again a solution of the \emph{Yang-Baxter equation}
\begin{equation*}
   R_{12}(u)R_{13}(u+v)R_{23}(v) = R_{23}(v)R_{13}(u+v)R_{12}(u) \,.
\end{equation*}
For classical series, the relevant $R$-matrices $R(u)$ are the Yang's matrix in type $A$, and the
Zamolodchikov-Zamolodchikov's matrix in types $BCD$. The Hopf algebra structure on such $Y^\rtt_\hbar(\fg)$
is especially simple, with the coproduct $\Delta$, antipode $S$, and counit $\epsilon$ given explicitly by
\begin{equation*}
  \Delta(T(u))=T(u)\otimes T(u) \,, \qquad S(T(u))=T^{-1}(u) \,, \qquad \epsilon(T(u))=\ID \,.
\end{equation*}
These features make the RTT-realization to be well-suited both for the representation theory as well as the
study of corresponding integrable systems. An explicit isomorphism $Y^\rtt_\hbar(\fg) \simeq Y_\hbar(\fg)$ is
again constructed using the Gauss decomposition of the generator matrix $T(u)$. For $A$-type this was carried
out in~\cite{bk}, for $BCD$-types it was carried a decade later in~\cite{jlm1}, while a less explicit isomorphism
in general types was established in~\cite{w}. Finally, we note that the RTT realization of the (antidominantly)
shifted Yangians $Y_\mu(\fg)$ from~\cite{bfn} was recently obtained in~\cite{fpt,ft0} for classical $\fg$,
thus significantly simplifying some of their basic structures as well as producing integrable systems on the
corresponding quantized Coulomb branches of $3d$ $\mathcal{N}=4$ quiver gauge theories.

The theory of quantum groups and Yangians associated with Lie superalgebras is still far from a full development.
While the Drinfeld-Jimbo realization of quantum finite and affine supergroups was proposed two decades ago in~\cite{y0,y},
there is no uniform (new) Drinfeld realization of such algebras in affine types, besides for $A$-type. A novel feature
of Lie superalgebras is that they admit several non-isomorphic Dynkin diagrams. The isomorphism of the Lie superalgebras
corresponding to different Dynkin diagrams of the same finite/affine type was established in~\cite[Appendix]{lss}.
Upgrading the latter to the isomorphism of quantum finite/affine superalgebras associated with different Dynkin diagrams
is a highly non-trivial technical task that constitutes one of the major results of~\cite{y}. The renewed interest in
quantum supergroups over the last decade is often motivated by intriguing predictions in string theory. In particular,
the recent work~\cite{xz} established a certain duality between $U_q(\fosp(2m+1|2n))$ and $U_{-q}(\fosp(2n+1|2m))$
generalizing a conjecture of~\cite{mw}.

Likewise, there is no $J$- or new Drinfeld realizations of superYangians. The cases studied mostly up to date involve
rather the RTT realization. The general linear RTT Yangians $Y^\rtt_\hbar(\gl(n|m))$ and the orthosymplectic RTT Yangians
$Y^\rtt_\hbar(\fosp(N|2m))$ first appeared in~\cite{n} and~\cite{aacfr}, respectively, using the super-analogues of
the Yang's and Zamolodchikov-Zamolodchikov's rational $R$-matrices. In the above classical types, the underlying $R$-matrices
posses natural symmetries, which yield isomorphisms of $Y^\rtt_\hbar(\fg)$ associated with different Dynkin diagrams.
In the recent work~\cite{ft}, the new Drinfeld realization of orthosymplectic superYangian $Y_\hbar(\fosp(V))$
was finally obtained for any Dynkin diagram, generalizing the one of~\cite{m} for the case of the distinguished Dynkin diagram
(we note that the orthosymplectic type simultaneously resembles all three classical types $B,C,D$).
The key idea of~\cite{ft} was to derive all the defining relations (including the higher order Serre relations)
from the RTT-relations by performing the Gauss decomposition of the generating matrix, thus generalizing the derivation of the new
Drinfeld realization of $Y^\rtt_\hbar(\gl(n|m))$ from~\cite{peng,ts}.

In this note, we evaluate finite and affine $R$-matrices associated with the orthosymplectic Lie algebras and any Dynkin diagram.
In analogy with the aforementioned orthosymplectic superYangian case, in the sequel note~\cite{ht} we shall derive
the new Drinfeld realization of the orthosymplectic quantum affine algebras, thus generalizing the $BCD$-types of~\cite{jlm2,jlm3}.
While for the distinguished Dynkin diagram, the corresponding $R$-matrices were presented almost 20 years ago in~\cite{mdgl}
(also cf.~\cite{gm}), we hope that the present work also adds more in understanding the origin of these formulas.
Our presentation follows closely~\cite{mt}, the recent joint work of I.~Martin and the second author.
To this end, we first evaluate the finite orthosymplectic $R$-matrices, also presenting their factorization
into ``local $q$-exponents''.
While the latter resembles the classical factorization of finite $R$-matrices for (non-super) quantum groups
$U^{\DrJ}_q(\fg)$ (see~\cite{kr}), our proof is quite different as it relies on the combinatorial shuffle approach.
We then apply the Yang-Baxterization technique of~\cite{gwx} to get several potential
candidates for the affine $R(z)$, and finally choose the correct one that intertwines the tensor product
of two evaluation modules.
We note that a combination of the present work and~\cite{mt,mt2} allows to evaluate finite and affine $R$-matrices
for two-parameter orthosymplectic quantum groups and subsequently derive their new Drinfeld presentation.


\subsection{Outline}
\label{ssec:outline}
\

The structure of the present paper is the following:

\smallskip
\noindent
$\bullet$
In Section~\ref{sec:orthosymplectic}, we recall the basic conventions on superalgebras as well as definitions
of orthosymplectic Lie superalgebras $\fosp(V)$ and their Drinfeld-Jimbo type quantizations $\uqV$.

\smallskip
\noindent
$\bullet$
In Section~\ref{sec:column_repr}, we explicitly construct the first fundamental representation $V$ of $\uqV$, see
Proposition~\ref{prop:fin-repn}. We further present three highest weight vectors $w_1,w_2,w_3$ in $V\otimes V$, see
Proposition~\ref{prop:highest-weight-vec} (and Subsection~\ref{ssec:generating-BCD}),
such that $w_1$, $w_2$ and yet another vector $\wtd{w}_3$ or $\hat{w}_3$
(or actually $w_3$ unless $n=m$) generate the entire tensor square $V\otimes V$ under the $\uqV$-action.

\smallskip
\noindent
$\bullet$
In Section~\ref{sec:R-matrices}, we evaluate the universal intertwiner $\hat{R}_{VV}$ from Proposition~\ref{prop:universal-R}
on the first fundamental $\uqV$-representation from Proposition~\ref{prop:fin-repn}, see Theorem~\ref{thm:R_osp_finite}.
This generalizes the formula of~\cite{mdgl} for the standard parity sequence. Our proof is quite different though, as we
directly verify the intertwining property, see~Proposition~\ref{prop:intertwiner}
(and Subsection~\ref{sec:intertwiner-proof} for its proof), and match the action on the generating three vectors from
Proposition~\ref{prop:highest-weight-vec}, see Propositions~\ref{prop:eig-calc-1},~\ref{prop:eig-calc-R-hat}.

\smallskip
\noindent
$\bullet$
In Section~\ref{sec:factorization-R}, we present an alternative proof of the formula for $R_{VV}$ from
Theorem~\ref{thm:R_osp_finite} by factorizing it into the ``local'' operators parametrized by the positive roots of
the reduced root system $\bar{\Phi}$. To do so, we use a combinatorial construction of orthogonal PBW bases of the
positive subalgebra of $U_q(\fosp(V))$ developed in~\cite{chw}, see Proposition~\ref{prop:CHW-main-theorem}. While
the setup of~\cite{chw} slightly differs from ours (in that they use a different pairing, coproduct, and a twisted product),
we relate the two explicitly, thus obtaining dual bases of the positive and negative subalgebras of $U_q(\fosp(V))$
with respect to the bialgebra pairing~\eqref{eq:Hopf-parity}, see Theorem~\ref{thm:PBW-general}.
This implies the factorization formula of Theorem~\ref{cor:Theta-factorization}, cf.~Remark~\ref{rem:q-exp}, thus
providing a conceptual origin of~\eqref{eq:R_0} from Theorem~\ref{thm:R_osp_finite}.

\smallskip
\noindent
$\bullet$
In Section~\ref{sec:affine-R}, we extend the first fundamental $\uqV$-module from Proposition~\ref{prop:fin-repn}
to evaluation modules over the orthosymplectic quantum affine supergroup $\UdqV$ and its reduced version $\UqV$
in Propositions~\ref{prop:affine-repn}--\ref{prop:non-reduced affine module}. The main result of this Section is
Theorem~\ref{thm:R_osp_affine} which evaluates the universal intertwiner of $\UdqV$ on the tensor product of two
such representations, generalizing the orthogonal and symplectic types due to~\cite{jim}. According to~\cite{jim},
this produces a solution of the Yang-Baxter relation with a spectral parameter, cf.~\eqref{eq:qYB-affine}. While
the proof of Theorem~\ref{thm:R_osp_affine} is straightforward (see Subsection~\ref{ssec:main thm proof}), we derived
the formula from its finite counterpart (Theorem~\ref{thm:R_osp_finite}) through the \emph{Yang-Baxterization}
technique of~\cite{gwx}, cf.\ Proposition~\ref{prop:Yang-Baxterization}.

\smallskip
\noindent
$\bullet$
In Appendix~\ref{sec:app_super-A-type}, we present a similar treatment for $A$-type quantum finite and affine supergroups, and
derive the corresponding finite and affine $R$-matrices in Theorems~\ref{thm:R_gl_finite} and~\ref{thm:R_gl_affine}, respectively.
In Subsection~\ref{ssec:factorization-A}, we also present a factorization formula for the corresponding finite $R$-matrix,
which seems to be missing in the literature. We also emphasize that the affine $R$-matrix can be obtained from the finite one
through the Yang-Baxterization of~\cite{gwx}, up to a prefactor.

\smallskip
\noindent
$\bullet$
In Appendix~\ref{sec:app_generating}, we present a direct tedious verification of the ``generating'' property
of $V\otimes V$ by the two highest weight vectors in type $A$ and three vectors in orthosymplectic type
(which can be chosen to be the highest weight vectors unless $n=m$),
cf.\ Propositions~\ref{prop:highest-weight-vec-gl}(b) and~\ref{prop:highest-weight-vec}(b,c).
Our analysis emphasizes an importance of the special case $n=m$, when $V\otimes V$ is not semisimple.


\subsection{Acknowledgement}\label{ssec:acknowl}
\

K.~H.\ is grateful to A.~Uruno for discussions on the orthosymplectic Lie superalgebras; to J.~Kwon for
discussions on the decomposition of the tensor square. A.~T.\ is grateful to M.~Finkelberg for stimulating
discussions on orthosymplectic quantum groups; to L.~Bezerra, V.~Futorny, I.~Kashuba for a correspondence
on~\cite{bfk}; to R.~Frassek and I.~Martin for collaboration on~\cite{ft} and \cite{mt,mt2}.
We are also grateful to the anonymous referees for very useful suggestions.

Both authors were partially supported by the NSF Grant DMS-$2302661$.


\section{Orthosymplectic Lie superalgebras and quantum groups}\label{sec:orthosymplectic}


In this Section, we recall the definition of orthosymplectic Lie superalgebras and associated quantum groups,
following~\cite{cw}.

\subsection{Setup and notations}\label{ssec:setup}
\

Fix non-negative integers $m,n$ so that $n$ is even, and set $N = m+n$ as well as $s = \big\lfloor \frac{N}{2} \big\rfloor$.
We shall assume that $N>2$. We equip the index set $\mathbb{I} = \{1,2,\ldots,N\}$ with an involution $'$ via:
\begin{equation*}
  i\mapsto i' \coloneqq N+1-i \,.
\end{equation*}
Consider a superspace $V=\BC^{m|n}$ with a homogeneous $\BC$-basis $\{v_{1}, v_{2}, \ldots, v_{N}\}$ such that each $v_i$
is either \emph{even} (that is $v_i\in V_{\bar{0}}$) or \emph{odd} (that is $v_i\in V_{\bar{1}}$), the dimensions are
$\dim(V_{\bar{0}})=m, \dim(V_{\bar{1}})=n$, and the vectors $v_i,v_{i'}$ have the same parity for any $i\in \mathbb{I}$
(in particular, $v_{s+1}$ is even for odd $N$). The latter condition means that
\begin{equation}\label{eq:parity-sym}
  \ol{i}=\ol{i'} \,,
\end{equation}
where for $i\in \mathbb{I}$, we define its $\BZ_2$-parity $\ol{i}\in \BZ_2$ via:
\begin{equation*}
  \ol{i} = |v_i| =
  \begin{cases}
    \bar{0} & \text{if } v_i \text{ is even} \\
    \bar{1} & \text{if } v_i \text{ is odd}
  \end{cases} \,.
\end{equation*}
The choice of $\BZ_{2}$-parity of the basis vectors can be encoded by the \emph{parity sequence}
\begin{equation*}
  \gamma_{V} \coloneqq (|v_{1}|, \ldots, |v_{s}|) = (\ol{1}, \ldots, \ol{s}) \in \{\bar{0}, \bar{1}\}^{s}.
\end{equation*}
We also choose a sequence $\vartheta_{V} \coloneqq (\vartheta_{1},\vartheta_{2},\ldots,\vartheta_{N})$ of $\pm 1$'s satisfying
\begin{equation}
\label{eq:theta}
  \vartheta_{i} =
  \begin{cases}
    1 & \text{if\, } \ol{i} = \bar{0}\\
    -\vartheta_{i'} & \text{if\, } \ol{i} = \bar{1}
  \end{cases}
\end{equation}
(we do not assume that $\vartheta_i=1$ for $i\leq s$). Under the conventions $(-1)^{\bar{0}}=1, (-1)^{\bar{1}}=-1$, we get
\begin{equation*}
  \vartheta_{i}^{2} = 1 \qquad \mathrm{and} \qquad \vartheta_{i}\vartheta_{i'} = (-1)^{\ol{i}}
  \qquad \mathrm{for\ any} \qquad i \in \mathbb{I} \,.
\end{equation*}

Any superalgebra $A = A_{\bar{0}} \oplus A_{\bar{1}}$ can be equipped with a natural \emph{Lie superalgebra} structure via:
\begin{equation}
\label{eq:super commutator}
  [x,x'] = \mathrm{ad}_x(x') \coloneqq xx'-(-1)^{|x| \cdot |x'|}x'x
\end{equation}
defined for homogeneous $x,x' \in A$ with $|x|,|x'|$ denoting their $\BZ_2$-grading, and extended linearly.
Given two superspaces $A = A_{\bar{0}} \oplus A_{\bar{1}}$ and $B = B_{\bar{0}} \oplus B_{\bar{1}}$, their
tensor product $A \otimes B$ is also a superspace with
  $(A \otimes B)_{\bar{0}} = (A_{\bar{0}} \otimes B_{\bar{0}}) \oplus (A_{\bar{1}} \otimes B_{\bar{1}})$
and
  $(A \otimes B)_{\bar{1}} = (A_{\bar{0}} \otimes B_{\bar{1}}) \oplus (A_{\bar{1}} \otimes B_{\bar{0}})$.
Furthermore, if $A$ and $B$ are superalgebras, then $A\otimes B$ is made into a superalgebra, called the
\emph{graded tensor product} of the superalgebras $A$ and $B$, via the following multiplication:
\begin{equation}\label{eq:graded tensor product}
  (x \otimes y)(x' \otimes y') = (-1)^{|y| \cdot |x'|} (xx') \otimes (yy')
  \quad \mathrm{for\ any\ homogeneous\ }\ x,x' \in A,\, y,y' \in B \,.
\end{equation}
We will use only the graded tensor products of superalgebras, unless stated otherwise.


\subsection{Orthosymplectic Lie superalgebras}\label{ssec:classical}
\

Consider the set $\gl(V)$ of all linear endomorphisms of $V$. Using the basis $\{v_{1}, v_{2}, \ldots, v_{N}\}$ of $V$,
we can identify $\gl(V)$ with the set of all $N \times N$ matrices. It can be made into a Lie superalgebra, called the
\emph{general linear Lie superalgebra}, by defining the $\BZ_2$-grading
\begin{equation*}
  |E_{ij}| \coloneqq \ol{i}+\ol{j}
\end{equation*}
and consequently with the Lie superbracket given by (cf.~\eqref{eq:super commutator})
\begin{equation*}
  [E_{ij}, E_{ab}] = \delta_{ja} E_{ib} - \delta_{bi}(-1)^{(\ol{i}+\ol{j})(\ol{a}+\ol{b})} E_{aj} \,,
\end{equation*}
where we use the basis $\{E_{ij}\}_{i,j=1}^N$ of elementary matrices with $1$ at the $(i,j)$-entry and $0$ elsewhere.

Consider a bilinear form $B_{G} \colon V \times V \to \BC$ defined by the anti-diagonal matrix (cf.~\eqref{eq:theta})
\begin{equation*}
  G = \sum_{i=1}^{N} \vartheta_{i} E_{i'i} \,.
\end{equation*}
The \emph{orthosymplectic Lie superalgebra} $\fosp(V)$ associated with the bilinear form $B_{G}$ is the Lie subalgebra
of $\gl(V)$ consisting of all matrices $X = \sum_{i,j} x_{ij} E_{ij} \in \gl(V)$ preserving $B_{G}$, i.e.\ satisfying
\begin{equation*}
  X^{\mathrm{st}} G + G X = 0
\end{equation*}
where the \emph{supertransposition} of $X$ is defined via
\begin{equation}\label{eq:supertranspose}
  X^{\mathrm{st}} \coloneqq \sum_{i,j=1}^{N} (-1)^{\ol{j}(\ol{i}+\ol{j})} x_{ij} E_{ji} \,.
\end{equation}
Thus, $\fosp(V)$ is spanned by the elements
\begin{equation}\label{eq:X-elements}
  X_{ij} = E_{ij} - (-1)^{\ol{i}(\ol{i}+\ol{j})} \vartheta_{i}\vartheta_{j} E_{j'i'}
  \qquad  \forall\ 1\leq i,j\leq N \,.
\end{equation}
We note that $X_{j'i'} = - (-1)^{\ol{i}(\ol{i}+\ol{j})} \vartheta_{i}\vartheta_{j} \cdot X_{ij}$.
The following elements form a basis of $\fosp(V)$:
\begin{equation*}
  \big\{ X_{ij} \,\big|\, i+j < N+1 \big\} \cup
  \big\{ X_{ii'} \,\big|\, \ol{i} = \bar{1},\, 1\leq i\leq s \big\} \,.
\end{equation*}

We choose the Cartan subalgebra $\fh$ of $\fosp(V)$ to consist of all diagonal matrices. Thus, $\fh$ has a basis
$\{X_{ii}\}_{i=1}^s$. Consider the linear functionals $\{\varepsilon_{r}\}_{r=1}^N$ on $\gl(V)$ defined by
$\varepsilon_{r}(\sum_{i,j} x_{ij} E_{ij})=x_{rr}$.
Their restrictions to $\fh$ satisfy
\begin{equation}\label{eq:degree-symmetry}
  \varepsilon_{i}|_{\fh} = -\varepsilon_{i'}|_{\fh} \quad \mathrm{for\ any}\ i \,, \qquad
  \varepsilon_{s+1}|_{\fh} = 0 \quad \textrm{for odd } N \,.
\end{equation}
Therefore, $\{\varepsilon_{i}\}_{i=1}^s$ give rise to a basis of $\fh^{*}$ that is dual to the basis $\{X_{ii}\}_{i=1}^s$
of $\fh$. The computation $[X_{ii},X_{ab}] = (\varepsilon_{a} - \varepsilon_{b})(X_{ii})X_{ab}$ shows that $X_{ab}$ is
a root vector corresponding to the root $\varepsilon_{a} - \varepsilon_{b}$. Hence, we get the \emph{root space decomposition}
$\fosp(V) = \fh \oplus \bigoplus_{\alpha \in \Phi} \fosp(V)_{\alpha}$ with the root system
\begin{equation}\label{eq:root-sys}
  \Phi = \big\{ \varepsilon_{a} - \varepsilon_{b} \,\big|\, a+b < N+1,\, a \neq b \big\} \cup
         \big\{ 2\varepsilon_{a} \,\big|\, \ol{a} = \bar{1} \big\} \,.
\end{equation}
We further have a decomposition $\Phi = \Phi_{\bar{0}} \cup \Phi_{\bar{1}}$ into \emph{even} and \emph{odd} roots.
We also choose the following polarization of $\Phi$:
\begin{equation}\label{eq:polarization}
\begin{split}
  & \Phi^{+} = \big\{ \varepsilon_{a} - \varepsilon_{b} \,\big|\, a<b<a' \big\} \cup
               \big\{ 2\varepsilon_{a} \,\big|\, \ol{a} = \bar{1},\, a \leq s \big\} \,, \\
  & \Phi^{-} = \big\{ \varepsilon_{a} - \varepsilon_{b} \,\big|\, b<a<b' \big\} \cup
               \big\{ 2\varepsilon_{a} \,\big|\, \ol{a} = \bar{1},\, a' \leq s \big\} \,.
\end{split}
\end{equation}
Let $\bar{\Phi}=\bar{\Phi}_{\bar{0}} \cup \bar{\Phi}_{\bar{1}}$ be the \emph{reduced} root system of $\fosp(V)$, that is:
\begin{equation}\label{eq:reduced_roots}
  \bar{\Phi}=\left\{ \gamma\in \Phi \,|\, \sfrac{1}{2}\gamma \notin \Phi \right\} \,,\quad
  \bar{\Phi}_{\bar{0}}=\bar{\Phi}\cap \Phi_{\bar{0}} \,,\quad \bar{\Phi}_{\bar{1}}=\bar{\Phi}\cap \Phi_{\bar{1}} \,.
\end{equation}


\subsection{Chevalley-Serre type presentation}
\

Consider the non-degenerate \emph{supertrace} bilinear form $(\cdot,\cdot) \colon \fosp(V) \times \fosp(V) \to \BC$
defined by
\begin{equation*}
  (X,Y) = \sfrac{1}{2} \mathrm{sTr}(XY) \,.
\end{equation*}
Its restriction to the Cartan subalgebra $\fh$ of $\fosp(V)$ is non-degenerate, thus giving rise to an
identification $\fh \simeq \fh^{*}$ via $\varepsilon_{i} \leftrightarrow (-1)^{\ol{i}} X_{ii}$ and inducing
a bilinear form $(\cdot, \cdot) \colon \fh^{*} \times \fh^{*} \to \BC$ such that
\begin{equation}\label{eq:epsilon-pairing}
  (\varepsilon_{i}, \varepsilon_{j}) = \delta_{ij} (-1)^{\ol{i}}
  \qquad \textrm{for\ any} \qquad 1 \leq i,j \leq s \,.
\end{equation}
We also recall that an odd root $\alpha\in \Phi_{\bar{1}}$ is called \emph{isotropic} if $(\alpha,\alpha)=0$.

Following the choice of the polarization~\eqref{eq:polarization} of the root system~\eqref{eq:root-sys},
the simple roots and the corresponding root vectors can be written as follows:
\begin{itemize}

\item[$\bullet$]
Case 1: $m$ is odd.
\begin{equation}\label{eq:Lie-action-case1}
\begin{split}
  & \alpha_{i} = \varepsilon_{i} - \varepsilon_{i+1} \qquad  \mathrm{for}\ \ 1 \leq i \leq s \,, \\
  & \sse_{i} = X_{i,i+1} \qquad \mathrm{for}\ \ 1 \leq i \leq s \,, \\
  & \ssf_{i} = (-1)^{\ol{i}} X_{i+1,i} \qquad \mathrm{for}\ \ 1 \leq i \leq s \,, \\
  & \ssh_{i} = [\sse_{i}, \ssf_{i}] =
    (-1)^{\ol{i}} X_{ii} - (-1)^{\ol{i+1}} X_{i+1,i+1} \qquad \mathrm{for}\ \ 1 \leq i \leq s \,.
\end{split}
\end{equation}

\item[$\bullet$]
Case 2: $m$ is even and $\ol{s} = \bar{0}$.
\begin{equation}
\label{eq:Lie-action-case2}
\begin{split}
  & \alpha_{i} =
  \begin{cases}
    \varepsilon_{i} - \varepsilon_{i+1} & \mathrm{if}\ 1 \leq i < s \\
    \varepsilon_{s-1} - \varepsilon_{s+1} = \varepsilon_{s-1} + \varepsilon_{s} & \mathrm{if}\ i=s
  \end{cases} \,, \\
  & \sse_{i} =
  \begin{cases}
    X_{i,i+1} & \mathrm{if}\ 1 \leq i < s \\
    X_{s-1,s+1} & \mathrm{if}\ i=s
  \end{cases} \,, \\
  & \ssf_{i} =
  \begin{cases}
    (-1)^{\ol{i}} X_{i+1,i} & \mathrm{if}\ 1 \leq i < s \\
    (-1)^{\ol{s-1}} X_{s+1,s-1} & \mathrm{if}\ i=s
  \end{cases} \,, \\
  & \ssh_{i} = [\sse_{i}, \ssf_{i}] =
  \begin{cases}
    (-1)^{\ol{i}} X_{ii} - (-1)^{\ol{i+1}} X_{i+1,i+1} & \mathrm{if}\ 1 \leq i < s \\
    (-1)^{\ol{s-1}} X_{s-1,s-1} - (-1)^{\ol{s+1}} X_{s+1,s+1} & \mathrm{if}\ i=s
  \end{cases} \,.
\end{split}
\end{equation}

\item[$\bullet$]
Case 3: $m$ is even and $\ol{s} = \bar{1}$.
\begin{equation}\label{eq:Lie-action-case3}
\begin{split}
  & \alpha_{i} =
  \begin{cases}
    \varepsilon_{i} - \varepsilon_{i+1} & \mathrm{if}\ 1 \leq i < s \\
    2\varepsilon_{s} & \mathrm{if}\ i=s
  \end{cases} \,, \\
  & \sse_{i} =
  \begin{cases}
    X_{i,i+1} & \mathrm{if}\ 1 \leq i < s \\
    E_{s,s+1} & \mathrm{if}\ i=s
  \end{cases} \,, \\
  & \ssf_{i} =
  \begin{cases}
    (-1)^{\ol{i}} X_{i+1,i} & \mathrm{if}\ 1 \leq i < s \\
    -2E_{s+1,s} & \mathrm{if}\ i=s
  \end{cases} \,, \\
  & \ssh_{i} = [\sse_{i}, \ssf_{i}] =
  \begin{cases}
    (-1)^{\ol{i}} X_{ii} - (-1)^{\ol{i+1}} X_{i+1,i+1} & \mathrm{if}\ 1 \leq i < s \\
    -2X_{ss} & \mathrm{if}\ i=s
  \end{cases} \,.
\end{split}
\end{equation}

\end{itemize}

\medskip
Define the \emph{symmetrized Cartan matrix} $(a_{ij})_{i,j=1}^{s}$ of $\fosp(V)$ via
\begin{equation}\label{eq:sym-cartan}
  a_{ij} = (\alpha_{i},\alpha_{j}) \,.
\end{equation}
Then, the above elements $\{\sse_{i},\ssf_i,\ssh_i\}_{i=1}^s$ are easily seen to satisfy the Chevalley-type relations:
\begin{equation}\label{eq:chevalley-rel}
  [\ssh_{i}, \ssh_{j}] = 0 \,,\quad [\ssh_{i}, \sse_{j}] = a_{ij} \sse_{j} \,,\quad
  [\ssh_{i}, \ssf_{j}] = -a_{ij} \ssf_{j} \,,\quad [\sse_{i}, \ssf_{j}] = \delta_{ij} \ssh_{i} \,.
\end{equation}

In fact, the Lie superalgebra $\fosp(V)$ admits a generators-and-relations presentation, which is a special case
of~\cite[Main Theorem]{z}. Explicitly, it is generated by $\{\sse_i,\ssf_i,\ssh_i\}_{i=1}^{s}$, with the $\BZ_2$-grading
\begin{equation}\label{eq:Z2-grading}
  |\sse_i|=|\ssf_i|=
  \begin{cases}
     \bar{0} & \mbox{if } \alpha_i\in \Phi_{\bar{0}} \\
     \bar{1} & \mbox{if } \alpha_i\in \Phi_{\bar{1}}
  \end{cases} \,,\qquad
  |\ssh_i|=\bar{0} \,,
\end{equation}
while the defining relations are~\eqref{eq:chevalley-rel} together with the \emph{standard Serre relations} and the
\emph{higher order Serre relations}. As we shall not need the explicit form of the Serre relations, we rather refer
the interested reader to~\cite[\S3.2.1]{z} for the exact formulas.

\begin{Rem}
We note that our choice of the generators is a rescaled version of those used in~\cite{y,z}, as we use the symmetrized
Cartan matrix instead of the non-symmetrized one in~\eqref{eq:chevalley-rel}.
\end{Rem}


\subsection{Orthosymplectic quantum groups}\label{sec:q-orthosymplectic}
\

The \emph{orthosymplectic quantum supergroup} $\uqV$ is a natural quantization of the universal enveloping superalgebra $U(\fosp(V))$.
Explicitly, $\uqV$ is a $\BC(q^{\pm 1/2})$-superalgebra generated by $\{e_{i}, f_{i}, q^{\pm h_{i}/2}\}_{i=1}^{s}$, with the
$\BZ_{2}$-grading as in~\eqref{eq:Z2-grading}, subject to the following analogues of~\eqref{eq:chevalley-rel}:
\begin{align}
  & [q^{h_{i}/2}, q^{h_{j}/2}] = 0 \,, \qquad
    q^{\pm h_{i}/2} q^{\mp h_{i}/2} = 1 \,,
  \label{eq:q-chevalley-rel-hh}\\
  & q^{h_{i}/2} e_{j} q^{-h_{i}/2} = q^{a_{ij}/2} e_{j} \,, \quad
    q^{h_{i}/2} f_{j} q^{-h_{i}/2} = q^{-a_{ij}/2} f_{j} \,,
  \label{eq:q-chevalley-rel-he} \\
  & [e_{i}, f_{j}] = \delta_{ij} \frac{q^{h_{i}} - q^{-h_{i}}}{q - q^{-1}} \,,
  \label{eq:q-chevalley-rel-ef}
\end{align}
as well as the \emph{standard} and the \emph{higher order $q$-Serre relations}, which the interested reader may find
in~\cite[Proposition 10.4.1]{y0}, cf.~\cite[Proposition 2.7]{chw}.

\begin{Rem}\label{rem:other_conventions_denominator}
We note that our choice of the denominator $q-q^{-1}$ in~\eqref{eq:q-chevalley-rel-ef} follows conventions of~\cite{y},
and thus may differ from some other literature by a mere rescaling of $f_j$'s.
\end{Rem}

Consider a lattice $P=\bigoplus_{i=1}^s \BZ \varepsilon_i$, a root lattice $Q=\bigoplus_{i=1}^s \BZ \alpha_i$ contained
in $P$ via~\eqref{eq:Lie-action-case1}--\eqref{eq:Lie-action-case3}, and set $Q^+=\bigoplus_{i=1}^s \BZ_{\geq 0} \alpha_i$.
The algebra $\uqV$ is naturally graded by $Q$ (thus also by $P$) via:
\begin{equation}\label{eq:grading algebra}
  \deg(e_i)=\alpha_i \,,\qquad \deg(f_i)=-\alpha_i \,,\qquad \deg(q^{\pm h_i/2})=0 \qquad \forall\, 1\leq i\leq s \,.
\end{equation}
For elements $a,b\in \uqV$ homogeneous with respect to the $\BZ_2$-grading and~\eqref{eq:grading algebra}, we set:
\begin{equation}\label{eq:q-superbracket}
  [\![ a,b ]\!] = ab-(-1)^{|a| |b|} q^{(\deg(a),\deg(b))}\cdot ba \,.
\end{equation}

Moreover, the following formulas endow $\uqV$ with a Hopf superalgebra structure:
\begin{equation}\label{eq:comult}
\begin{split}
  & \Delta(e_{i}) = q^{h_{i}/2} \otimes e_{i} + e_{i} \otimes q^{-h_{i}/2} \,,\\
  & \Delta(f_{i}) = q^{h_{i}/2} \otimes f_{i} + f_{i} \otimes q^{-h_{i}/2} \,,\\
  & \Delta(q^{\pm h_{i}/2}) = q^{\pm h_{i}/2} \otimes q^{\pm h_{i}/2} \,,
\end{split}
\end{equation}
the counit
\begin{equation*}
  \epsilon(e_{i}) = 0 \,,\qquad  \epsilon(f_{i}) = 0 \,, \qquad  \epsilon(q^{\pm h_{i}/2}) = 1 \,,
\end{equation*}
and the antipode
\begin{equation*}
  S(e_{i}) = -q^{-a_{ii}/2} e_{i} \,, \qquad
  S(f_{i}) = -q^{a_{ii}/2} f_{i} \,, \qquad
  S(q^{\pm h_{i}/2}) = q^{\mp h_{i}/2} \,.
\end{equation*}
Following~\cite{jan}, we also define a superalgebra homomorphism $\Delta^J\colon \uqV\to \uqV^{\otimes 2}$ via
\begin{equation}
\label{eq:Jantzen-comult}
  \Delta^J(e_{i}) = q^{h_{i}} \otimes e_{i} + e_{i} \otimes 1 \,,\ \
  \Delta^J(f_{i}) = 1 \otimes f_{i} + f_{i} \otimes q^{-h_{i}} \,,\ \
  \Delta^J(q^{\pm h_{i}/2}) = q^{\pm h_{i}/2} \otimes q^{\pm h_{i}/2} \,.
\end{equation}

\begin{Rem}\label{rem:other_conventions}
We prefer to follow the notations from physics literature as we use $q^{\pm h_i}$ instead of the more common
generators $K^{\pm 1}_i$ and use the coproduct~\eqref{eq:comult} more often than~\eqref{eq:Jantzen-comult}.
\end{Rem}


\section{Column-vector representations}\label{sec:column_repr}


\subsection{First fundamental representations}
\

Henceforth, we shall use the following convention $q^D$ for a diagonal matrix $D = \diag(d_{1}, \ldots, d_{N})$:
\begin{equation*}
  q^{D}=q^{d_1}E_{11} + \dots + q^{d_{N}} E_{NN} \,.
\end{equation*}
We shall also identify $\End(V)$ with the set of $N \times N$ matrices using the basis $\{v_{1}, \ldots, v_{N}\}$ of $V$.

\begin{Prop}\label{prop:fin-repn}
The following defines a representation $\varrho \colon \uqV \to \End(V)$:
\begin{equation}\label{eq:generator-action}
  \varrho(e_{i}) = \sse_{i} \,, \qquad \varrho(f_{i}) = \kappa_{i}\ssf_{i} \,, \qquad
  \varrho(q^{\pm h_{i}/2}) = q^{\pm \ssh_{i}/2}  \qquad \text{for} \quad 1 \leq i \leq s \,,
\end{equation}
where $\{\sse_{i}, \ssf_{i}, \ssh_{i}\}_{i=1}^{s}$ denote the action of Chevalley-type generators of $\fosp(V)$
from~\eqref{eq:Lie-action-case1}--\eqref{eq:Lie-action-case3},~and
\begin{equation*}
  \kappa_{i} =
  \begin{cases}
    \frac{q+q^{-1}}{2} & \mathrm{if}\ m \mathrm{\ is\ even,}\ \ol{s} = \bar{1}, \mathrm{\ and}\ i=s \\
    1 & \mathrm{otherwise}
  \end{cases}.
\end{equation*}
\end{Prop}

\begin{proof}
We need to show that the operators~\eqref{eq:generator-action} satisfy the defining
relations~\eqref{eq:q-chevalley-rel-hh}--\eqref{eq:q-chevalley-rel-ef} together with the standard and the higher order
$q$-Serre relations. The relations~\eqref{eq:q-chevalley-rel-hh} are obvious as all $\ssh_i$ are diagonal in the basis
$\{v_{1}, \ldots, v_{N}\}$. For the first relation of~\eqref{eq:q-chevalley-rel-he}, we note that its left-hand side is
the conjugation of $\sse_j$ by the diagonal matrix $q^{\ssh_{i}/2}$. Hence, it suffices to compare $q^{a_{ij}/2}$ to
the ratios of the eigenvalues of $q^{\ssh_{i}/2}$ on the vectors $\sse_{j} v_{a}$ and
$v_{a}$ (assuming the former is nonzero), which follows from the second equality
of~\eqref{eq:chevalley-rel}. The second relation
of~\eqref{eq:q-chevalley-rel-he} is checked analogously. Finally, the relation \eqref{eq:q-chevalley-rel-ef} follows
from~\eqref{eq:chevalley-rel}, since $\ssh_i$ is diagonal with $\{0, \pm 1, \pm 2\}$ on diagonal
(with $\pm 2$ appearing only when $m$ is even, $\ol{s} = \bar{1}$, and $i=s$) and
\begin{equation}\label{eq:trivial-quantization}
  (q^{r} - q^{-r})/(q - q^{-1}) = r \qquad \textrm{for} \quad r \in \{0, \pm 1\} \,.
\end{equation}

To verify the $q$-Serre relations, let us equip $V$ with a natural $P$-grading (with $P$ as above) via $\deg(v_j)=\varepsilon_j$
for all $1\leq j\leq N$, where we follow~\eqref{eq:degree-symmetry} for $s<j\leq N$. Evoking the $P$-grading~\eqref{eq:grading algebra}
of $\uqV$, we see that the assignment~\eqref{eq:generator-action} preserves this $P$-grading:
\begin{equation}\label{eq:grading_compatibility}
  \deg(\varrho(x)v_j)=\deg(x)+\deg(v_j) \qquad \mathrm{for\ any} \quad
  x\in \{e_i,f_i,q^{\pm h_{i}/2}\}_{i=1}^s \,,\, 1\leq j\leq N \,.
\end{equation}
Referring to the explicit form of all $q$-Serre relations, left-hand sides of which are presented in~\cite[Definition 4.2.1]{y0},
one can easily see that all of them, besides (v), are homogeneous whose degrees are \underline{not} in the set
$\{\varepsilon_i-\varepsilon_j \,|\, 1\leq i,j\leq N\}$. Hence, they act trivially on the superspace $V$.

It remains to verify only the cubic $q$-Serre relation (cf.\ notation~\eqref{eq:q-superbracket})
\begin{equation}\label{eq:cubic-Serre}
  [\![ e_{s}, [\![e_{s-1},e_{s-2}]\!] ]\!] - [\![ e_{s-1}, [\![e_{s},e_{s-2}]\!] ]\!] = 0
\end{equation}
arising from~\cite[Definition 4.2.1(v)]{y0}, which occurs only when $\gamma_V=(\ast,\ldots,\ast,\bar{1},\bar{0})$
and $N\geq 6$ is even. Evoking the above $P$-grading, we note that the left-hand side of~\eqref{eq:cubic-Serre} has
degree $\alpha_{s-2}+\alpha_{s-1}+\alpha_s=\varepsilon_{s-2}+\varepsilon_{s-1}$, and hence acts trivially on $v_j$
for $j\notin \{(s-1)',(s-2)'\}$. It is straightforward to check that it also acts trivially on the basis vectors
$v_{(s-1)'}$ and $v_{(s-2)'}$.
\end{proof}


\subsection{Tensor square of the first fundamental representation}
\

Recall that a vector $w$ in a $\uqV$-module $W$ is called \emph{highest weight vector of weight $\mu$} if
\begin{equation*}
  e_i(w)=0  \qquad \mathrm{and} \qquad q^{h_i/2}(w)=q^{\mu(\ssh_i/2)}w \qquad \mathrm{for\ all} \quad  1\leq i\leq s \,.
\end{equation*}

\begin{Prop}\label{prop:highest-weight-vec}
(a) The following are highest weight vectors in $\uqV$-module $V\otimes V$:
\begin{equation}\label{eq:w-vectors}
\begin{split}
  & w_{1} = v_{1} \otimes v_{1} \,, \\
  & w_{2} = v_{1} \otimes v_{2} - (-1)^{\ol{1}(\ol{1}+\ol{2})} q^{(-1)^{\ol{1}}} \cdot v_{2} \otimes v_{1} \,, \\
  & w_{3} = \sum_{i=1}^{N} c_{i} \cdot v_{i} \otimes v_{i'} \,,
\end{split}
\end{equation}
where the coefficients $c_{i}$'s in $w_{3}$ are determined by $c_1=1$ and the following relations:
\begin{equation}\label{eq:w3-coeff}
\begin{split}
  & c_{a+1} = q^{(-1)^{\ol{a}}/2+(-1)^{\ol{a+1}}/2} \cdot \vartheta_{a}\vartheta_{a+1} \cdot c_{a}
    \qquad \mathrm{for}\ \ 1 \leq a <s \,, \\
  & c_{a'} = (-1)^{\ol{a}+\ol{a+1}} \cdot q^{(-1)^{\ol{a}}/2+(-1)^{\ol{a+1}}/2} \cdot \vartheta_{a}\vartheta_{a+1} \cdot c_{(a+1)'}
    \qquad \mathrm{for}\ \ 1 \leq a <s \,,
\end{split}
\end{equation}
as well as one of the following
\begin{equation}\label{eq:w3-coeff-case1}
  c_{s+1} = q^{(-1)^{\ol{s}}/2}\cdot \vartheta_{s}\vartheta_{s+1}\cdot c_{s} \,, \quad
  c_{s+2} = (-1)^{\ol{s}+\ol{s+1}}\cdot q^{(-1)^{\ol{s}}/2}\cdot \vartheta_{s}\vartheta_{s+1}\cdot c_{s+1}
  \qquad \mathrm{for\ odd}\ m \,,
\end{equation}
\begin{equation}\label{eq:w3-coeff-case2}
  c_{s+1} = q^{(-1)^{\ol{s}}/2+(-1)^{\ol{s-1}}/2} \cdot \vartheta_{s-1}\vartheta_{s+1} \cdot c_{s-1}
  \qquad \mathrm{for\ even}\ m \ \mathrm{and}\ \ol{s} = \bar{0} \,,
\end{equation}
\begin{equation}\label{eq:w3-coeff-case3}
  c_{s+1} = - q^{(-1)^{\ol{s}} \cdot 2} \cdot c_{s} = - q^{-2} \cdot c_{s}
  \qquad \mathrm{for\ even}\ m \ \mathrm{and}\ \ol{s} = \bar{1} \,.
\end{equation}

\noindent
(b) For $n\ne m$, the $\uqV$-module $V \otimes V$ is generated by these vectors $w_1, w_2, w_3$ of~\eqref{eq:w-vectors}.

\medskip
\noindent
(c) For any $n$ and $m$, the $\uqV$-module $V \otimes V$ is generated by the vectors $w_1, w_2$,
$\wtd{w}_3=v_{1}\otimes v_{1'}$, as well as by the vectors $w_1, w_2$, $\hat{w}_3=v_{1'}\otimes v_{1}$.
\end{Prop}

\begin{proof}
(a) Let us show that the vectors $w_{1}, w_{2}, w_{3}$ are indeed highest weight vectors for the action
$\varrho^{\otimes 2}$ of $\uqV$ on $V\otimes V$. First, we note that they are eigenvectors with respect
to $\{q^{h_i/2}\}_{i=1}^s$:
\begin{equation*}
  \varrho^{\otimes 2}(q^{h_i/2})w_1=q^{2\varepsilon_1(\ssh_i/2)}w_1 \,, \quad
  \varrho^{\otimes 2}(q^{h_i/2})w_2=q^{(\varepsilon_1+\varepsilon_2)(\ssh_i/2)}w_2 \,, \quad
  \varrho^{\otimes 2}(q^{h_i/2})w_3=w_3 \,.
\end{equation*}
It remains to verify that $w_1,w_2,w_3$ are annihilated by all $\varrho^{\otimes 2}(e_i)$. The equality
$\varrho^{\otimes 2}(e_i)(w_1)=0$ follows from $\varrho(e_i)(v_1)=0$. Likewise, $\varrho^{\otimes 2}(e_i)(w_2)=0$
for $i>1$ follows from $\varrho(e_i)v_1=\varrho(e_i)v_2=0$. Meanwhile, combining $\varrho(e_1)v_2=v_1$,
$\varrho(e_1)v_1=0$, $\varrho(q^{h_{1}/2})v_1=q^{(-1)^{\ol{1}}/2} v_1$, and~\eqref{eq:comult},
we also get:
\begin{align*}
  \varrho^{\otimes 2}(e_{1})w_{2}
  &= (\varrho(q^{h_{1}/2}) \otimes \varrho(e_{1}))(v_{1} \otimes v_{2}) -
     (-1)^{\ol{1}(\ol{1}+\ol{2})} q^{(-1)^{\ol{1}}} (\varrho(e_{1}) \otimes \varrho(q^{-h_{1}/2}))(v_{2} \otimes v_{1}) \\
  &= \Big((-1)^{(\ol{1}+\ol{2})\ol{1}} \cdot q^{(-1)^{\ol{1}}/2} -
           (-1)^{\ol{1}(\ol{1}+\ol{2})} q^{(-1)^{\ol{1}}} \cdot q^{-(-1)^{\ol{1}}/2} \Big)\cdot v_{1} \otimes v_{1} = 0 \,,
\end{align*}
where the sign $(-1)^{\ol{1}(\ol{1}+\ol{2})}$ in the beginning of the second line is due to the
conventions~\eqref{eq:graded tensor product}.

It remains to verify $\varrho^{\otimes 2}(e_a)w_3=0$ for all $a$. First, let us consider the case $1 \leq a <s$. Then:
\begin{align*}
  \varrho^{\otimes 2}(e_{a})w_{3} =& \,
  (\varrho(e_{a}) \otimes \varrho(q^{-h_{a}/2}))
    \big( c_{a+1}\cdot v_{a+1} \otimes v_{(a+1)'} + c_{a'}\cdot v_{a'} \otimes v_{a} \big) \\
  & + (\varrho(q^{h_{a}/2}) \otimes \varrho(e_{a}))
    \big( c_{(a+1)'}\cdot v_{(a+1)'} \otimes v_{a+1} + c_{a} \cdot v_{a} \otimes v_{a'} \big) \\ =&\,
    c_{a+1} \cdot q^{-(-1)^{\ol{a+1}}/2} \cdot v_{a} \otimes v_{(a+1)'}
    - c_{a'} \cdot q^{-(-1)^{\ol{a}}/2} \cdot (-1)^{\ol{a}(\ol{a}+\ol{a+1})} \vartheta_{a}\vartheta_{a+1} \cdot
    v_{(a+1)'} \otimes v_{a} \\
  & + (-1)^{\ol{a+1}(\ol{a}+\ol{a+1})}\cdot c_{(a+1)'} \cdot q^{(-1)^{\ol{a+1}}/2} \cdot v_{(a+1)'} \otimes v_{a} \\
  & - (-1)^{\ol{a}(\ol{a}+\ol{a+1})} \cdot c_{a} \cdot q^{(-1)^{\ol{a}}/2} \cdot (-1)^{\ol{a}(\ol{a}+\ol{a+1})}
    \vartheta_{a}\vartheta_{a+1} \cdot v_{a} \otimes v_{(a+1)'} \,,
\end{align*}
with the first signs in the last two lines due to the conventions~\eqref{eq:graded tensor product}.
Evoking both defining relations~\eqref{eq:w3-coeff}, we see that the right-hand side above vanishes,
and so $\varrho^{\otimes 2}(e_{a})w_{3}=0$ for $1\leq a<s$.

To evaluate $\varrho^{\otimes 2}(e_s)w_3$, we have to consider three separate cases:
\begin{itemize}

\item[$\bullet$]
Case 1: $m$ is odd.
In this case, we likewise have:
\begin{align*}
  \varrho^{\otimes 2}(e_{s})w_{3} =&\,
    (\varrho(e_{s}) \otimes \varrho(q^{-h_{s}/2}))
    \big( c_{s+1}\cdot v_{s+1} \otimes v_{s+1} + c_{s+2}\cdot v_{s+2} \otimes v_{s} \Big)\\
  & + (\varrho(q^{h_{s}/2}) \otimes \varrho(e_{s}))
    \big( c_{s+1}\cdot v_{s+1} \otimes v_{s+1} + c_{s}\cdot v_{s} \otimes v_{s+2} \big) \\
  =\, &
  c_{s+1} \cdot v_{s} \otimes v_{s+1}
    - c_{s+2} \cdot q^{-(-1)^{\ol{s}}/2} \cdot (-1)^{\ol{s}(\ol{s}+\ol{s+1})} \vartheta_{s}\vartheta_{s+1}\cdot
    v_{s+1} \otimes v_{s} \\
  & + (-1)^{\ol{s+1}(\ol{s}+\ol{s+1})} \cdot c_{s+1} \cdot v_{s+1} \otimes v_{s} \\
  & - (-1)^{\ol{s}(\ol{s}+\ol{s+1})} \cdot c_{s} \cdot q^{(-1)^{\ol{s}}/2} \cdot (-1)^{\ol{s}(\ol{s}+\ol{s+1})}
    \vartheta_{s}\vartheta_{s+1} \cdot v_{s} \otimes v_{s+1} \,,
\end{align*}
with the first signs in the last two lines due to the conventions~\eqref{eq:graded tensor product}.
Evoking both defining relations~\eqref{eq:w3-coeff-case1}, we see that the right-hand side above vanishes,
and so $\varrho^{\otimes 2}(e_{s})w_{3}=0$.

\item[$\bullet$]
Case 2: $m$ is even and $\ol{s} = \bar{0}$.
In this case, we obtain:
\begin{align*}
  \varrho^{\otimes 2}(e_{s})w_{3} =&\,
  (\varrho(e_{s}) \otimes \varrho(q^{-h_{s}/2}))
    \big( c_{s+1}\cdot v_{s+1} \otimes v_{s} + c_{s+2}\cdot v_{s+2} \otimes v_{s-1} \big) \\
  & + (\varrho(q^{h_{s}/2}) \otimes \varrho(e_{s}))
    \big( c_{s}\cdot v_{s} \otimes v_{s+1} + c_{s-1}\cdot v_{s-1} \otimes v_{s+2} \big) \\
  =&\,
  c_{s+1} \cdot q^{-(-1)^{\ol{s}}/2} \cdot v_{s-1} \otimes v_{s}
  - c_{s+2} \cdot q^{-(-1)^{\ol{s-1}}/2} \cdot (-1)^{\ol{s-1}(\ol{s-1}+\ol{s})}\cdot \vartheta_{s-1}\vartheta_{s+1}\cdot
    v_{s} \otimes v_{s-1} \\
  & + (-1)^{\ol{s}(\ol{s-1}+\ol{s})} \cdot c_{s} \cdot q^{(-1)^{\ol{s}}/2} \cdot v_{s} \otimes v_{s-1} \\
  & - (-1)^{\ol{s-1}(\ol{s-1}+\ol{s})} \cdot c_{s-1} \cdot q^{(-1)^{\ol{s-1}}/2} \cdot
    (-1)^{\ol{s-1}(\ol{s-1}+\ol{s})}\cdot \vartheta_{s-1}\vartheta_{s+1}\cdot v_{s-1} \otimes v_{s} \,.
\end{align*}
Evoking~(\ref{eq:w3-coeff},~\ref{eq:w3-coeff-case2}), we see that the right-hand side above vanishes,
and so $\varrho^{\otimes 2}(e_{s})w_{3}=0$.

\item[$\bullet$]
Case 3: $m$ is even and $\ol{s} = \bar{1}$.
In this case, we again get the desired vanishing by~\eqref{eq:w3-coeff-case3}:
\begin{align*}
  \varrho^{\otimes 2}(e_{s})w_{3} =&\,
  (\varrho(q^{h_{s}/2}) \otimes \varrho(e_{s})) \big( c_{s}\cdot v_{s} \otimes v_{s+1} \big) +
  (\varrho(e_{s}) \otimes \varrho(q^{-h_{s}/2})) \big( c_{s+1}\cdot v_{s+1} \otimes v_{s} \big) \\
  =&\, c_{s} \cdot q^{-1}\cdot v_{s} \otimes v_{s} +
       c_{s+1} \cdot q \cdot v_{s} \otimes v_{s} = 0 \,.
\end{align*}
\end{itemize}

(b) Part (b) is established in
Propositions~\ref{prop:decomp-B}(b,d),~\ref{prop:decomp-CD-fork}(b,e),~\ref{prop:decomp-CD-nofork}(b,e)
from Appendix~\ref{sec:app_generating}.

(c) Part (c) is established in
Propositions~\ref{prop:decomp-B}(b--d),~\ref{prop:decomp-CD-fork}(c--e),~\ref{prop:decomp-CD-nofork}(c--e)
from Appendix~\ref{sec:app_generating}.
\end{proof}

\begin{Rem}\label{rem:unique-ht.wt.vect}
Reversing the above calculations, we see that the only highest weight vectors in $\uqV$-module $V\otimes V$
of weights $2\varepsilon_1, \varepsilon_1+\varepsilon_2, 0$ are multiples of $w_1,w_2,w_3$, respectively.
\end{Rem}


\section{R-Matrices}\label{sec:R-matrices}


\subsection{Universal construction}\label{ssec:universal-R}
\

We first recall the general construction of a $\uqV$-module isomorphism $W_1 \otimes W_2 \to W_2 \otimes W_1$
arising through the universal $R$-matrix, following the treatment of~\cite[\S7]{jan} in the non-super setup.
First of all, for any superspaces $A$ and $B$, we define the \emph{graded permutation} operator $\tau=\tau_{A,B}$ via
\begin{equation}\label{eq:tau}
  \tau \colon A \otimes B \to B \otimes A  \,,\qquad
  x\otimes y \mapsto (-1)^{|x| |y|} y\otimes x \quad \mathrm{for\ homogeneous}\ x\in A, y\in B \,.
\end{equation}
We note that if both $A,B$ are superalgebras, then $\tau$ of~\eqref{eq:tau} is a superalgebra homomorphism.
In particular, we can now define opposite coproducts $\Delta^{\opp}$ and $\Delta^{J,\opp}$,
which are superalgebra homomorphisms, via:
\begin{equation}\label{eq:coproduct-opp}
  \Delta^{\opp}(x)=\tau\circ \Delta(x)  \,,\qquad \Delta^{J,\opp}(x)=\tau\circ \Delta^J(x) \qquad
  \forall\, x\in \uqV \,,
\end{equation}
cf.~(\ref{eq:comult},~\ref{eq:Jantzen-comult}). Henceforth, we shall only work with \emph{type $1$} $P$-graded
$\uqV$-modules $W$, i.e.\
\begin{equation*}
   W=\bigoplus_{\mu\in P} W[\mu] \qquad \mathrm{with} \qquad
   W[\mu]=\left\{ w\in W \,\big|\, q^{h_i/2}(w)=q^{(\alpha_i/2,\mu)}w\quad \forall\, 1\leq i\leq s \right\}.
\end{equation*}
For any such $\uqV$-modules $W_1,W_2$, we define a linear map $\wtd{f}\colon W_1 \otimes W_2 \to W_1 \otimes W_2$ via
\begin{equation*}
  \wtd{f}(w_1 \otimes w_2) = q^{-(\nu,\mu)} \cdot w_1 \otimes w_2 \qquad
  \mathrm{for\ any}\quad w_1 \in W_1[\nu] \,,\, w_2 \in W_2[\mu] \,.
\end{equation*}

Let $U^+_{q}(\fosp(V))$ and $U^-_{q}(\fosp(V))$ be the subalgebras of $\uqV$ generated by $\{e_i\}_{i=1}^s$ and
$\{f_i\}_{i=1}^s$, respectively. We also define $U^\geq_{q}(\fosp(V))$ and $U^\leq_{q}(\fosp(V))$ as subalgebras
of $\uqV$ generated by $\{e_{i},q^{\pm h_i/2}\}_{i=1}^s$ and $\{f_{i},q^{\pm h_i/2}\}_{i=1}^s$, respectively.
The basic property of all quantum (super)groups is that they are Drinfeld doubles with respect to a bialgebra pairing,
which presently relies on the
following result (generalizing~\cite[Propositions 6.12, 6.18]{jan} to the super setup):

\begin{Prop}\label{prop:pairing_finite}
There exists a unique bilinear bialgebra pairing
\begin{equation}\label{eq:Hopf-parity}
  (\cdot,\cdot)_J\colon U^\leq_{q}(\fosp(V)) \times U^\geq_{q}(\fosp(V)) \longrightarrow \BC(q^{\pm 1/4})
\end{equation}
which satisfies the following structural properties
\begin{equation}\label{eq:Hopf-properties}
  (yy',x)_J = (y \otimes y',\Delta^J(x))_J \,,\qquad
  (y, xx')_J = (\Delta^{J,\opp}(y),x \otimes x')_J
\end{equation}
for any $x,x' \in U^\geq_{q}(\fosp(V))$, $y,y'\in U^\leq_{q}(\fosp(V))$ with the pairing in the right-hand sides
defined~by
\begin{equation*}
   (x\otimes x', y\otimes y')_J = (-1)^{|x'||y|} (x,y)_J (x',y')_J \qquad
   \mathrm{for\ any}\ \BZ_2\mathrm{-homogeneous}\  x,x',y,y' \,,
\end{equation*}
and is given on the generators by (for any $1\leq i,j\leq s$):
\begin{equation}\label{eq:generators-parity}
\begin{split}
  & (f_{i},q^{\pm h_j/2})_J = 0 \,,\quad (q^{\pm h_j/2}, e_{i})_J = 0 \,,\\
  & (f_{i},e_{j})_J = \frac{\delta_{ij} (-1)^{|f_i||e_j|}}{q^{-1}-q} \,,\quad
    (q^{h_i/2},q^{h_j/2})_J = q^{-a_{ij}/4} \,.
\end{split}
\end{equation}
The pairing~\eqref{eq:Hopf-parity} is non-degenerate.
\end{Prop}

\begin{Rem}\label{rem:yamane}
The non-degeneracy of~\eqref{eq:Hopf-parity} easily follows from the non-degeneracy of its restriction to
$U^-_{q}(\fosp(V)) \times U^+_{q}(\fosp(V))$. We note that the latter is a highly non-trivial result, and
is actually the main result of~\cite{y0}, where the $q$-Serre relations were precisely derived to satisfy this property.
\end{Rem}

Recall the $P$-grading on $\uqV$, hence on all the subalgebras above, cf.~\eqref{eq:grading algebra}.
We note that the graded components $U^-_{q}(\fosp(V))_\nu$ and $U^+_{q}(\fosp(V))_\mu$ are all finite-dimensional.
Furthermore, in accordance with~(\ref{eq:Hopf-properties},~\ref{eq:generators-parity}), the pairing~\eqref{eq:Hopf-parity}
is graded:
\begin{equation*}
  (y,x)_{J}=0 \qquad \mathrm{for}\quad y\in U^-_{q}(\fosp(V))_\nu \,, x\in U^+_{q}(\fosp(V))_\mu
  \quad \mathrm{with} \quad \mu+\nu \ne 0 \,.
\end{equation*}
Let us pick dual bases $\{x_{i}^{\mu}\}$, $\{y_{i}^{\mu}\}$ of $U^+_{q}(\fosp(V))_\mu$, $U^-_{q}(\fosp(V))_{-\mu}$
with respect to~\eqref{eq:Hopf-parity}, and set
\begin{equation}\label{eq:Theta}
  \Theta = 1 + \sum_{\mu > 0}\Theta_{\mu} \qquad \text{with} \qquad
  \Theta_{\mu} = \sum_{i} x_{i}^{\mu} \otimes y_{i}^{\mu} \,.
\end{equation}

The following is proved completely analogously to~\cite[Theorem 7.3]{jan}:

\begin{Prop}\label{prop:universal-R}
For any $\uqV$-modules $W_1$ and $W_2$ as above, the map
\begin{equation}\label{eq:universal-R}
  \hat{R}_{W_1,W_2} = \tau \circ \wtd{f}^{1/2} \circ \Theta \circ \wtd{f}^{1/2} \colon
  W_1 \otimes W_2 \longrightarrow W_2 \otimes W_1
\end{equation}
is an isomorphism of $\uqV$-modules.
\end{Prop}

\begin{Rem}\label{rem:R-vs-hatR}
Completely analogously to~\cite[Theorem 7.3]{jan}, one verifies that
\begin{equation*}
  \Delta^{J,\opp}(x)\Theta\wtd{f} = \Theta\wtd{f}\Delta^{J}(x) \qquad \forall\, x\in \uqV \,.
\end{equation*}
Comparing the defining formulas of $\Delta^{J}$ from~\eqref{eq:Jantzen-comult} to
those of $\Delta$ from~\eqref{eq:comult}, it is easy to see that
\begin{equation*}
  \Delta^J(x)=\wtd{f}^{-1/2} \Delta (x) \wtd{f}^{1/2} \,,\qquad
  \Delta^{J,\opp}(x)=\wtd{f}^{-1/2} \Delta^{\opp} (x) \wtd{f}^{1/2} \qquad \forall\, x\in \uqV \,.
\end{equation*}
Combining the above two equalities, we obtain:
\begin{equation*}
  \Delta^{\opp}(x) \wtd{f}^{1/2} \Theta \wtd{f}^{1/2} = \wtd{f}^{1/2} \Theta \wtd{f}^{1/2} \Delta(x)
  \qquad \forall\, x\in \uqV \,,
\end{equation*}
which together with $\tau(x(w_1\otimes w_2))=\Delta^{\opp}(x)(\tau(w_1\otimes w_2))$ for all
$x\in \uqV, w_1\in W_1, w_2\in W_2$ establishes Proposition~\ref{prop:universal-R}.
\end{Rem}

Let
  $ R_{W_1,W_2}=\tau\circ \hat{R}_{W_1,W_2} =
    \wtd{f}^{1/2} \circ \Theta \circ \wtd{f}^{1/2} \colon W_1\otimes W_2\to W_1\otimes W_2$.
Given $\uqV$-modules $W_1,W_2,W_3$ as above, we define the following three endomorphisms of $W_1\otimes W_2\otimes W_3$:
\begin{equation*}
  R_{12}=R_{W_1,W_2}\otimes \mathrm{Id}_{W_3} \,,\quad
  R_{23}=\mathrm{Id}_{W_1}\otimes R_{W_2,W_3} \,,\quad
  R_{13}=(\mathrm{Id}\otimes \tau)R_{12}(\mathrm{Id}\otimes \tau) \,.
\end{equation*}
We likewise define linear operators $\hat{R}_{12}, \hat{R}_{23}, \hat{R}_{13}$. Then, analogously to~\cite{jan}, we have:
\begin{equation}\label{eq:qYB-two}
\begin{split}
  & R_{12} R_{13} R_{23} = R_{23} R_{13} R_{12}\colon
    W_1\otimes W_2\otimes W_3 \to W_1\otimes W_2\otimes W_3 \,, \\
  & \hat{R}_{12} \hat{R}_{23} \hat{R}_{12} = \hat{R}_{23} \hat{R}_{12} \hat{R}_{23}\colon
    W_1\otimes W_2\otimes W_3 \to W_3\otimes W_2\otimes W_1 \,.
\end{split}
\end{equation}
In particular, we obtain a whole family of solutions of the quantum Yang-Baxter equation:
\begin{equation}\label{eq:qYB}
  R_{12}R_{13}R_{23}=R_{23}R_{13}R_{12} \in \End(W\otimes W\otimes W)\,.
\end{equation}

\begin{Cor}
For any $\uqV$-module $W$ as above, $R_{WW}$ satisfies~\eqref{eq:qYB}.
\end{Cor}


\subsection{Explicit R-matrices}
\

Let
\begin{equation}\label{eq:weyl-vec}
  \rho = \frac{1}{2} \sum_{\alpha \in \Phi_{\bar{0}}^{+}} \alpha \, - \,  \frac{1}{2} \sum_{\alpha \in \Phi_{\bar{1}}^{+}} \alpha
\end{equation}
be the \emph{Weyl vector} of the root system $\Phi$, which is the graded half-sum of all positive roots.
Due to \cite[Proposition~1.33]{cw}, we have
\begin{equation}\label{eq:rho-simple-root}
  (\rho, \alpha_i) = \sfrac{1}{2}(\alpha_i,\alpha_i) \qquad \forall\, 1\leq i\leq s \,.
\end{equation}
In accordance with~\eqref{eq:tau}, we consider the \emph{graded permutation operator}
$\tau_{VV} \colon V \otimes V \to V \otimes V$ defined via
$\tau(v_{i} \otimes v_{j}) = (-1)^{\ol{i}\,\ol{j}}\, v_{j} \otimes v_{i}$ for any $1\leq i,j\leq N$,
which is explicitly given by
\begin{equation*}
  \tau_{VV} = \sum_{i,j=1}^{N} (-1)^{\ol{j}} E_{ij} \otimes E_{ji} \,.
\end{equation*}

We are now ready to state our first main result:

\begin{Thm}\label{thm:R_osp_finite}
The $\uqV$-module isomorphism $\hat{R}_{VV}\colon V \otimes V \iso V \otimes V$ from Proposition~\ref{prop:universal-R}
and its inverse $\hat{R}_{VV}^{-1}$ for the $\uqV$-module $V$ constructed in Proposition~\ref{prop:fin-repn} are given by
\begin{equation}\label{eq:key finite equality}
  \hat{R}_{VV}=\tau_{VV}\circ R_0   \qquad  \mathrm{and} \qquad  \hat{R}_{VV}^{-1}=\tau_{VV}\circ R_\infty
\end{equation}
with the following explicit operators
\begin{equation}\label{eq:R_0}
\begin{split}
  R_{0} = \ID
  & + (q^{-1/2} - q^{1/2}) \sum_{i=1}^N (-1)^{\ol{i}} E_{ii} \otimes
      \Big( q^{-(\varepsilon_{i},\varepsilon_{i})/2} E_{ii} - q^{(\varepsilon_{i},\varepsilon_{i})/2} E_{i'i'} \Big) \\
  & + (q^{-1} - q) \sum_{i<j} (-1)^{\ol{j}} E_{ij} \otimes
      \Big( E_{ji} - (-1)^{\ol{j}(\ol{i}+\ol{j})} \vartheta_{i}\vartheta_{j} q^{(\rho, \varepsilon_{i} - \varepsilon_{j})} E_{i'j'} \Big)
\end{split}
\end{equation}
and
\begin{equation}\label{eq:R_inf}
\begin{split}
  R_{\infty} = \ID
  & + (q^{1/2} - q^{-1/2}) \sum_{i=1}^N (-1)^{\ol{i}} E_{ii} \otimes
      \Big( q^{(\varepsilon_{i},\varepsilon_{i})/2} E_{ii} - q^{-(\varepsilon_{i},\varepsilon_{i})/2} E_{i'i'} \Big) \\
  & + (q - q^{-1}) \sum_{i > j} (-1)^{\ol{j}} E_{ij} \otimes
      \Big( E_{ji} - (-1)^{\ol{j}(\ol{i}+\ol{j})} \vartheta_{i}\vartheta_{j} q^{(\rho, \varepsilon_{i} - \varepsilon_{j})} E_{i'j'} \Big) \,.
\end{split}
\end{equation}
\end{Thm}

\begin{Rem}\label{rem:R-usual-coproduct}
Evoking Remark~\ref{rem:R-vs-hatR}, we can also recover the formula for the operator $R$
defined via $R=\Theta\circ \wtd{f}$ and its inverse $R^{-1}$, corresponding
to the usual coproduct $\Delta^{J}$, as follows:
\begin{multline*}
  R
  = \wtd{f}^{-1/2} \circ R_{0} \circ \wtd{f}^{1/2}
  = \ID + (q^{-1/2} - q^{1/2}) \sum_{i=1}^N (-1)^{\ol{i}} E_{ii} \otimes
      \Big( q^{-(\varepsilon_{i},\varepsilon_{i})/2} E_{ii} - q^{(\varepsilon_{i},\varepsilon_{i})/2} E_{i'i'} \Big) \\
  \quad + (q^{-1} - q) \sum_{i<j} (-1)^{\ol{j}} E_{ij} \otimes
      \Big( E_{ji} - (-1)^{\ol{j}(\ol{i}+\ol{j})} \vartheta_{i}\vartheta_{j} q^{(\rho, \varepsilon_{i} - \varepsilon_{j})}
            q^{-(\varepsilon_{i},\varepsilon_{i})/2}q^{(\varepsilon_{j},\varepsilon_{j})/2} E_{i'j'} \Big)\,,
\end{multline*}
\begin{multline*}
  R^{-1}
  = \tau \circ \wtd{f}^{-1/2} \circ R_{\infty} \circ \wtd{f}^{1/2} \circ \tau
  = \ID + (
  q^{1/2} - q^{-1/2}) \sum_{i=1}^N (-1)^{\ol{i}} E_{ii} \otimes
      \Big( q^{(\varepsilon_{i},\varepsilon_{i})/2} E_{ii} - q^{-(\varepsilon_{i},\varepsilon_{i})/2} E_{i'i'} \Big) \\
  \quad + (q - q^{-1}) \sum_{i < j} (-1)^{\ol{j}} E_{ij} \otimes
      \Big( E_{ji} - (-1)^{\ol{j}(\ol{i}+\ol{j})} \vartheta_{i}\vartheta_{j} q^{-(\rho, \varepsilon_{i} - \varepsilon_{j})}
        q^{-(\varepsilon_{i},\varepsilon_{i})/2}q^{(\varepsilon_{j},\varepsilon_{j})/2} E_{i'j'} \Big) \,.
\end{multline*}
\end{Rem}

To prove Theorem~\ref{thm:R_osp_finite}, we first establish some properties of $R_0$ and $R_\infty$ defined in~\eqref{eq:R_0}
and~\eqref{eq:R_inf}. By abuse of notation, we shall often denote $(\varrho\otimes \varrho)(\Delta(x))$ simply by
$\Delta(x)$\footnote{We note that a similar convention was already used in our Remark~\ref{rem:R-vs-hatR} above.}
or $\varrho^{\otimes 2}(x)$. We start with a straightforward result, the proof of which is postponed till the end of this Section:

\begin{Prop}\label{prop:intertwiner}
For any element $x \in \uqV$, the following equalities hold (cf.~\eqref{eq:coproduct-opp}):
\begin{equation}\label{eq:intertwiner}
  R_{0} \Delta(x) = \Delta^{\opp}(x) R_{0} \qquad \text{and} \qquad R_{\infty} \Delta(x) = \Delta^{\opp}(x) R_{\infty} \,.
\end{equation}
\end{Prop}

Next, we evaluate the action of $\tau_{VV} R_0,\tau_{VV} R_\infty$ on the generating vectors
from Proposition~\ref{prop:highest-weight-vec}.

\begin{Prop}\label{prop:eig-calc-1}
(a) The highest weight vectors $w_{1}, w_{2}, w_{3}$ are eigenvectors of $\tau_{VV} \circ R_{0}$
\begin{equation*}
  \tau_{VV} R_{0}\colon \quad
  w_1\mapsto \mu_{1}^{0}\cdot w_1 \,,\quad
  w_2\mapsto \mu_{2}^{0}\cdot w_2 \,,\quad
  w_3\mapsto \mu_{3}^{0}\cdot w_3
\end{equation*}
with the eigenvalues given explicitly by:
\begin{equation}\label{eq:mu_0}
  \mu_{1}^{0} = (-1)^{\ol{1}} q^{-(-1)^{\ol{1}}} \,,\qquad
  \mu_{2}^{0} = -(-1)^{\ol{1}} q^{(-1)^{\ol{1}}} \,,\qquad
  \mu_{3}^{0} = q^{m-n-1} \,.
\end{equation}

\noindent
(b) The action of $\tau_{VV} \circ R_{0}$ on $\wtd{w}_3=v_1\otimes v_{1'}$ is given by
$\tau_{VV}R_{0}(\wtd{w}_3)=(-1)^{\ol{1}}q^{(-1)^{\ol{1}}}\cdot \hat{w}_3$.
\end{Prop}

\begin{proof}
(a) We evaluate each eigenvalue separately.
\begin{itemize}

\item[$\bullet$ $\mu_{1}^{0}$.]
For $w_{1} = v_{1} \otimes v_{1}$, the direct computation shows that
\begin{equation*}
  R_{0}(w_{1}) =
  v_{1} \otimes v_{1} + (q^{-1/2} - q^{1/2})(-1)^{\ol{1}}q^{-(-1)^{\ol{1}}/2} v_{1} \otimes v_{1} =
  q^{-(-1)^{\ol{1}}} v_{1} \otimes v_{1} \,.
\end{equation*}
The above equality implies the desired formula
\begin{equation*}
  \tau_{VV} R_{0}(w_{1}) = (-1)^{\ol{1}}q^{-(-1)^{\ol{1}}} v_{1} \otimes v_{1} \,.
\end{equation*}

\item[$\bullet$ $\mu_{2}^{0}$.]
For $w_{2} = v_{1} \otimes v_{2} - (-1)^{\ol{1}(\ol{1}+\ol{2})} q^{(-1)^{\ol{1}}}\cdot v_{2} \otimes v_{1}$,
the direct computation shows that
\begin{align*}
  R_{0}(w_{2})
  & = v_{1} \otimes v_{2} -
   (-1)^{\ol{1}(\ol{1}+\ol{2})} q^{(-1)^{\ol{1}}}
   \Big( v_{2} \otimes v_{1} + (q^{-1} - q)(-1)^{\ol{2}} (E_{12} \otimes E_{21})(v_{2} \otimes v_{1}) \Big)\\
  & = v_{1} \otimes v_{2} - (-1)^{\ol{1}(\ol{1}+\ol{2})} q^{(-1)^{\ol{1}}} \cdot v_{2} \otimes v_{1}
    + (q - q^{-1})(-1)^{\ol{1}} q^{(-1)^{\ol{1}}} \cdot v_{1} \otimes v_{2}\\
  & = q^{(-1)^{\ol{1}} \cdot 2} \cdot v_{1} \otimes v_{2}
    - (-1)^{\ol{1}(\ol{1}+\ol{2})} q^{(-1)^{\ol{1}}} \cdot v_{2} \otimes v_{1} \,.
\end{align*}
The above equality implies the desired formula
\begin{equation*}
  \tau_{VV} R_{0}(w_{2}) = -(-1)^{\ol{1}} q^{(-1)^{\ol{1}}} \cdot v_{1} \otimes v_{2} +
  (-1)^{\ol{1}\cdot \ol{2}} q^{(-1)^{\ol{1}} \cdot 2} \cdot v_{2} \otimes v_{1} =
  -(-1)^{\ol{1}} q^{(-1)^{\ol{1}}}\cdot w_2 \,.
\end{equation*}

\item[$\bullet$ $\mu_{3}^{0}$.]
For $w_{3} = \sum_{i=1}^{N} c_{i}\cdot v_{i} \otimes v_{i'}$, we note that $\tau_{VV} R_{0}(w_{3})$ is also a linear combination of
$\{v_{i} \otimes v_{i'}\}_{i=1}^N$. Furthermore, the intertwining property of Proposition~\ref{prop:intertwiner} shows that
\begin{equation*}
  \Delta(e_{a})\big(\tau_{VV} R_{0}(w_{3})\big) = \tau_{VV} R_{0} \big(\Delta(e_{a})w_{3}\big) =
  \tau_{VV} R_{0} \big(\varrho^{\otimes 2}(e_{a})w_{3}\big) = 0 \qquad \forall\, 1\leq a\leq s \,.
\end{equation*}
Combining this with Remark~\ref{rem:unique-ht.wt.vect}, we see that this forces $\tau_{VV} R_{0}(w_{3})$ to be
a scalar multiple of $w_{3}$. Therefore, to find $\mu_{3}^{0}$ it is enough to compare the coefficients of the
term $v_{1} \otimes v_{1'}$. To this end, we note that
\begin{align*}
  R_{0}(w_{3}) =
  & \sum_{1\leq i\leq N} c_{i}\cdot v_{i} \otimes v_{i'} -
    (q^{-1/2} - q^{1/2}) \sum_{\overset{1\leq i\leq N}{i \neq i'}} (-1)^{\ol{i}} q^{(\varepsilon_{i},\varepsilon_{i})/2}c_i\cdot
    (E_{ii} \otimes E_{i'i'})(v_{i} \otimes v_{i'}) \\
  & + (q^{-1} - q) \sum_{i' < i} (-1)^{\ol{i}}c_i\cdot (E_{i'i} \otimes E_{ii'})(v_{i} \otimes v_{i'}) \\
  & - (q^{-1} - q) \sum_{j < i} (-1)^{\ol{i}\,\ol{j}} q^{(\rho, \varepsilon_{j} - \varepsilon_{i})} \vartheta_{j}\vartheta_{i} c_i \cdot
    (E_{ji} \otimes E_{j'i'})(v_{i} \otimes v_{i'}) \\ =
  & \sum_{1\leq i\leq N} c_{i}\cdot v_{i} \otimes v_{i'} -
    (q^{-1/2} - q^{1/2}) \sum_{1\leq i\leq N} (-1)^{\ol{i}} q^{(\varepsilon_{i},\varepsilon_{i})/2} \cdot
    c_{i}\cdot v_{i} \otimes v_{i'} \\
  & + (q^{-1} - q) \sum_{i' \leq s} (-1)^{\ol{i}} \cdot c_{i}\cdot v_{i'} \otimes v_{i}
    - (q^{-1} - q) \sum_{j < i} (-1)^{\ol{i}} q^{(\rho, \varepsilon_{j} - \varepsilon_{i})} \vartheta_{j}\vartheta_{i} \cdot c_{i}\cdot
    v_{j} \otimes v_{j'} \,.
\end{align*}
In particular, the coefficient of $v_{1'}\otimes v_{1}$ in $R_{0}(w_{3})$ equals
\begin{equation*}
  \Big( 1 - (q^{-1/2} - q^{1/2}) (-1)^{\ol{1}} q^{(-1)^{\ol{1}}/2} \Big)\, c_{1'} = q^{(-1)^{\ol{1}}} c_{1'} \,,
\end{equation*}
and therefore the coefficient of $v_{1} \otimes v_{1'}$ in $\tau_{VV} R_{0}(w_{3})$ is $(-1)^{\ol{1}} q^{(-1)^{\ol{1}}} c_{1'}$.
The latter implies $\mu_{3}^{0} = (-1)^{\ol{1}} q^{(-1)^{\ol{1}}} c_{1'}/c_{1}$.
As $c_1=1$, to deduce the last formula of \eqref{eq:mu_0} it suffices to show:
\begin{equation}\label{eq:c-ratios}
  c_{1'} = (-1)^{\ol{1}} q^{-(-1)^{\ol{1}}} q^{m-n-1} \,.
\end{equation}
The proof of~\eqref{eq:c-ratios} is straightforward and is based on~\eqref{eq:w3-coeff}--\eqref{eq:w3-coeff-case3}.
Indeed, multiplying
\begin{equation*}
  c_{a'}/c_{a} = (-1)^{\ol{a}+\ol{a+1}} q^{(-1)^{\ol{a}}+(-1)^{\ol{a+1}}} \cdot c_{(a+1)'}/c_{(a+1)}
  \qquad \text{for} \quad 1 \leq a < s \,,
\end{equation*}
due to~\eqref{eq:w3-coeff}, we find
\begin{equation*}
  c_{1'}/c_{1} = (-1)^{\ol{1}+\ol{s}} q^{(-1)^{\ol{1}}+(-1)^{\ol{2}}\cdot 2+\ldots+(-1)^{\ol{s-1}}\cdot 2+(-1)^{\ol{s}}} \cdot c_{s'}/c_{s} \,.
\end{equation*}
Combining this with $\sum_{i=1}^N (-1)^{\ol{i}}=m-n$ and the explicit formula
\begin{equation*}
  c_{s'}/c_{s} =
  \begin{cases}
    (-1)^{\ol{s}+\ol{s+1}} q^{(-1)^{\ol{s}}} & \mathrm{if}\ m \ \mathrm{is\ odd} \\
    (-1)^{\ol{s}} & \mathrm{if}\ m \ \mathrm{is\ even\ and}\ \ol{s} = \bar{0} \\
    (-1)^{\ol{s}} q^{(-1)^{\ol{s}}\cdot 2} & \mathrm{if}\ m \ \mathrm{is\ even\ and}\ \ol{s} = \bar{1}
\end{cases} \,,
\end{equation*}
due to~\eqref{eq:w3-coeff}--\eqref{eq:w3-coeff-case3}, we obtain the uniform formula for $c_{1'}/c_1=c_{1'}$ from~\eqref{eq:c-ratios}.

\end{itemize}

(b) For $\wtd{w}_3=v_1\otimes v_{1'}$, the direct computation shows that
\begin{equation*}
  R_{0}(\wtd{w}_{3}) =
  v_{1} \otimes v_{1'} - (q^{-1/2} - q^{1/2})(-1)^{\ol{1}}q^{(-1)^{\ol{1}}/2} v_{1} \otimes v_{1'} =
  q^{(-1)^{\ol{1}}} v_{1} \otimes v_{1'} \,.
\end{equation*}
The above equality implies the desired formula
\begin{equation*}
  \tau_{VV} R_{0}(\wtd{w}_{3}) = (-1)^{\ol{1}} q^{(-1)^{\ol{1}}} v_{1'} \otimes v_{1} = (-1)^{\ol{1}} q^{(-1)^{\ol{1}}} \hat{w}_3 \,.
\end{equation*}
This completes the proof of the proposition.
\end{proof}

By completely analogous computations, we get the following result:

\begin{Prop}\label{prop:eig-calc-2}
(a) The highest weight vectors $w_{1}, w_{2}, w_{3}$ are eigenvectors of $\tau_{VV} \circ R_{\infty}$
\begin{equation*}
   \tau_{VV} R_{\infty}\colon \quad
   w_1 \mapsto \mu_{1}^{\infty}\cdot w_1 \,,\quad
   w_2 \mapsto \mu_{2}^{\infty}\cdot w_2 \,,\quad
   w_3 \mapsto \mu_{3}^{\infty}\cdot w_3
\end{equation*}
with the eigenvalues given explicitly by:
\begin{equation*}
  \mu_{1}^{\infty} = (-1)^{\ol{1}} q^{(-1)^{\ol{1}}} = 1/\mu_{1}^{0} \,,\quad
  \mu_{2}^{\infty} = -(-1)^{\ol{1}} q^{-(-1)^{\ol{1}}} = 1/\mu_{2}^{0} \,,\quad
  \mu_{3}^{\infty} = q^{-m+n+1} = 1/\mu_{3}^{0} \,.
\end{equation*}

\noindent
(b) The action of $\tau_{VV} \circ R_{\infty}$ on $\hat{w}_3=v_{1'}\otimes v_{1}$ is given by
$\tau_{VV}R_{\infty}(\hat{w}_3)=(-1)^{\ol{1}}q^{-(-1)^{\ol{1}}}\cdot \wtd{w}_3$.
\end{Prop}

Let us now evaluate the action of $\hat{R}_{VV}$ on $w_1,w_2,w_3$, and $\wtd{w}_3$:

\begin{Prop}\label{prop:eig-calc-R-hat}
(a) The highest weight vectors $w_{1}, w_{2}, w_{3}$ are eigenvectors of $\hat{R}_{VV}$ from~\eqref{eq:universal-R}
\begin{equation*}
  \hat{R}_{VV}\colon \quad
  w_1 \mapsto \lambda_1 w_1 \,,\quad w_2 \mapsto \lambda_2 w_2 \,,\quad w_3 \mapsto \lambda_3 w_3
\end{equation*}
with the eigenvalues given explicitly by:
\begin{equation}\label{eq:lambda}
  \lambda_{1} = (-1)^{\ol{1}} q^{-(-1)^{\ol{1}}} = \mu^0_1 \,,\qquad
  \lambda_{2} = -(-1)^{\ol{1}} q^{(-1)^{\ol{1}}} = \mu^0_2 \,,\qquad
  \lambda_{3} = q^{m-n-1} = \mu^0_3 \,.
\end{equation}

\noindent
(b) The action of $\hat{R}_{VV}$ on $\wtd{w}_3=v_1\otimes v_{1'}$ is given by
$\hat{R}_{VV}(\wtd{w}_3)=(-1)^{\ol{1}}q^{(-1)^{\ol{1}}}\cdot \hat{w}_3$.
\end{Prop}

\begin{proof}
(a) The intertwining property of $\hat{R}_{VV}$ from Proposition~\ref{prop:universal-R} together with
Remark~\ref{rem:unique-ht.wt.vect} implies that all three vectors $w_1,w_2,w_3$ are indeed eigenvectors
for $\hat{R}_{VV}$. We shall now evaluate each eigenvalue separately.
\begin{itemize}

\item[$\bullet$ $\lambda_{1}$.]
For $w_{1} = v_{1} \otimes v_{1}$, the direct computation shows that
\begin{equation*}
  \tau_{VV}\wtd{f}(w_{1}) = q^{-(\varepsilon_{1},\varepsilon_{1})} \tau_{VV}(v_{1} \otimes v_{1}) =
  (-1)^{\ol{1}} q^{-(-1)^{\ol{1}}} w_{1} \,,
\end{equation*}
which implies the desired formula for $\lambda_1$ (as $\Theta(v_1\otimes v_1)=v_1\otimes v_1$).

\item[$\bullet$ $\lambda_{2}$.]
The eigenvalue $\lambda_2$ of the $\hat{R}_{VV}$-action on
  $w_{2} = v_{1} \otimes v_{2} - (-1)^{\ol{1}(\ol{1}+\ol{2})} \cdot q^{(-1)^{\ol{1}}}\cdot v_{2} \otimes v_{1}$
equals the coefficient of $v_{1} \otimes v_{2}$ in $\hat{R}_{VV} (w_{2})$. The latter appears only from applying
$\tau_{VV}\wtd{f}$ to the above multiple of $v_{2} \otimes v_{1}$ (thus implying the desired formula for $\lambda_2$):
\begin{multline*}
  \tau_{VV}\wtd{f} \big(- (-1)^{\ol{1}(\ol{1}+\ol{2})} \cdot q^{(-1)^{\ol{1}}}\cdot v_{2} \otimes v_{1} \big) = \\
  - (-1)^{\ol{1}(\ol{1}+\ol{2})} \cdot q^{(-1)^{\ol{1}}} q^{-(\varepsilon_{2},\varepsilon_{1})}\cdot \tau_{VV}(v_{2} \otimes v_{1}) =
  -(-1)^{\ol{1}} q^{(-1)^{\ol{1}}} v_{1} \otimes v_{2} \,.
\end{multline*}

\item[$\bullet$ $\lambda_{3}$.]
The eigenvalue $\lambda_3$ of the $\hat{R}_{VV}$-action on $w_{3}=v_1\otimes v_{1'} + \sum_{i=2}^{N} c_{i}\cdot v_{i} \otimes v_{i'}$
equals the coefficient of $v_{1}\otimes v_{1'}$ in $\hat{R}_{VV} (w_{3})$. The latter appears only from applying $\tau_{VV}\wtd{f}$ to
the above multiple of $v_{1'} \otimes v_{1}$ (thus implying the desired formula for $\lambda_3$):
\begin{multline*}
  \tau_{VV}\wtd{f} \big( c_{1'} \cdot (v_{1'} \otimes v_{1}) \big) = \\
  c_{1'} q^{-(\varepsilon_{1'},\varepsilon_{1})} \tau_{VV}(v_{1'} \otimes v_{1}) =
  (-1)^{\ol{1}} q^{(-1)^{\ol{1}}} c_{1'} \cdot v_{1} \otimes v_{1'}
  \overset{\eqref{eq:c-ratios}}{=} q^{m-n-1} \cdot v_{1} \otimes v_{1'} \,.
\end{multline*}

\end{itemize}

(b) For $\wtd{w}_{3} = v_{1} \otimes v_{1'}$, the direct computation shows that
\begin{equation*}
  \tau_{VV}\wtd{f}(\wtd{w}_{3}) =
  q^{-(\varepsilon_{1},\varepsilon_{1'})} \tau_{VV}(v_{1} \otimes v_{1'}) =
  (-1)^{\ol{1}} q^{(-1)^{\ol{1}}} v_{1'} \otimes v_{1}\,,
\end{equation*}
which implies the desired formula as $\Theta(v_{1}\otimes v_{1'})=v_{1}\otimes v_{1'}=\wtd{w}_3$.
\end{proof}

Combining the Propositions above, we can now immediately derive our main result:

\begin{proof}[Proof of Theorem~\ref{thm:R_osp_finite}]
Combining the intertwining property~\eqref{eq:intertwiner} with the equality
\begin{equation}\label{eq:delta-opp}
  \Delta^{\opp}(x)=\tau^{-1}_{VV}\circ \Delta(x)\circ \tau_{VV} \in \End(V\otimes V) \,,
\end{equation}
we obtain
\begin{equation*}
  (\tau_{VV} R_0)\circ \Delta(x) = \tau_{VV}\circ \Delta^{\opp}(x)\circ R_0 =
  (\tau_{VV}\circ \Delta^{\opp}(x)\circ \tau^{-1}_{VV})\circ (\tau_{VV}\circ R_0) = \Delta(x)\circ (\tau_{VV}R_0) \,,
\end{equation*}
so that $\tau_{VV}\circ R_0\colon V\otimes V\to V\otimes V$ is a $\uqV$-module morphism.
In fact, since $R_0$ specializes to the identity map $\ID$ at $q=1$, it is generically a vector space isomorphism,
so that $\tau_{VV}\circ R_0$ is a $\uqV$-module isomorphism. Likewise is the operator $\hat{R}_{VV}$,
which acts on the generating vectors $w_1,w_2,\wtd{w}_3$ (or the generating highest weight vectors
$w_1,w_2,w_3$ unless $n=m$) of $V\otimes V$ in the same way as $\tau_{VV}\circ R_0$, due to
Propositions~\ref{prop:eig-calc-1} and~\ref{prop:eig-calc-R-hat}. This implies $\hat{R}_{VV}=\tau_{VV}\circ R_0$.

Similar arguments also show that $\tau_{VV}\circ R_{\infty}$ is a $\uqV$-module isomorphism.
Since the operators $\tau_{VV}\circ R_{\infty}$ and $\hat{R}_{VV}^{-1}$ act in the same way on
the generating vectors $w_1,w_2,\hat{w}_3$ (or the generating highest weight vectors $w_1,w_2,w_3$ unless $n=m$)
of $V\otimes V$, due to Propositions~\ref{prop:eig-calc-2} and~\ref{prop:eig-calc-R-hat}, we obtain the desired equality
$\hat{R}^{-1}_{VV}=\tau_{VV}\circ R_{\infty}$.
\end{proof}

\begin{Rem}
The above proof of Theorems~\ref{thm:R_osp_finite} is quite elementary, but it does require knowing the
correct formulas for $R_{VV}$ in the first place. In Section~\ref{sec:factorization-R}, we provide the conceptual
origin of these formulas by factorizing them into an ordered product of ``local'' operators.
\end{Rem}


\subsection{Proof of the intertwining property}\label{sec:intertwiner-proof}
\

In this separate Subsection, we sketch (presenting the key formulas) the proof of Proposition~\ref{prop:intertwiner}.

We start with the intertwining property of $R_\infty$. To make the computations more manageable,
it is helpful to break the operator $R_\infty$ from~\eqref{eq:R_inf} into $\ID$ and the following four pieces:
\begin{align*}
  R_{1} &= (q^{1/2} - q^{-1/2}) \sum_{1\leq i\leq N} (-1)^{\ol{i}} q^{(\varepsilon_{i},\varepsilon_{i})/2} E_{ii} \otimes E_{ii} \,,\\
  R_{2} &= -(q^{1/2} - q^{-1/2}) \sum_{1\leq i\leq N} (-1)^{\ol{i}} q^{-(\varepsilon_{i},\varepsilon_{i})/2} E_{ii} \otimes E_{i'i'} \,, \\
  R_{3} &= (q - q^{-1}) \sum_{i > j} (-1)^{\ol{j}} E_{ij} \otimes E_{ji} \,, \\
  R_{4} &= -(q - q^{-1}) \sum_{i > j} (-1)^{\ol{i}\,\ol{j}} q^{(\rho, \varepsilon_{i} - \varepsilon_{j})}
    \vartheta_{i}\vartheta_{j} E_{ij} \otimes E_{i'j'} \,,
\end{align*}
so that $R_\infty=\ID+R_1+R_2+R_3+R_4$.

\medskip
\noindent
$\bullet$ Proof of $R_{\infty} \Delta(e_a) = \Delta^{\opp}(e_a) R_{\infty}$ for $1\leq a<s$.

Recall the explicit formula $\Delta(e_a)=q^{h_{a}/2} \otimes e_{a} + e_{a} \otimes q^{-h_{a}/2}$ as well as
\begin{multline*}
  \varrho(q^{h_a/2}) = q^{\ssh_a/2} = \sum_{1\leq i\leq N}^{i\ne a,a',a+1,(a+1)'} E_{ii}  \ + \\
  q^{(-1)^{\ol{a}}/2}E_{aa} + q^{-(-1)^{\ol{a+1}}/2}E_{a+1,a+1} +
  q^{-(-1)^{\ol{a}}/2}E_{a'a'} + q^{(-1)^{\ol{a+1}}/2} E_{(a+1)',(a+1)'}  \,.
\end{multline*}
By direct computation, we get:
\begin{align*}
  R_{1}\Delta(e_{a}) =\,
  & (q^{1/2} - q^{-1/2}) \bigg\{ (-1)^{\ol{a}} q^{(-1)^{\ol{a}}} \cdot E_{aa} \otimes E_{a,a+1} \\
  & - (-1)^{\ol{a+1}}(-1)^{\ol{a}(\ol{a}+\ol{a+1})} q^{(-1)^{\ol{a+1}}} \vartheta_{a}\vartheta_{a+1} \cdot
    E_{(a+1)',(a+1)'} \otimes E_{(a+1)',a'} \\
  & + (-1)^{\ol{a}} \cdot E_{a,a+1} \otimes E_{aa} - (-1)^{\ol{a+1}}(-1)^{\ol{a}(\ol{a}+\ol{a+1})}
    \vartheta_{a}\vartheta_{a+1} \cdot E_{(a+1)',a'} \otimes E_{(a+1)',(a+1)'} \bigg\} \,,
\end{align*}
\begin{align*}
  \Delta^{\opp}(e_{a})R_{1} =\,
  & (q^{1/2} - q^{-1/2}) \bigg\{ (-1)^{\ol{a+1}} q^{(-1)^{\ol{a+1}}} \cdot E_{a+1,a+1} \otimes E_{a,a+1} \\
  & - (-1)^{\ol{a}\,\ol{a+1}} q^{(-1)^{\ol{a}}} \vartheta_{a}\vartheta_{a+1} \cdot E_{a'a'} \otimes E_{(a+1)',a'} \\
  & + (-1)^{\ol{a+1}} \cdot E_{a,a+1} \otimes E_{a+1,a+1}
    - (-1)^{\ol{a}\,\ol{a+1}} \vartheta_{a}\vartheta_{a+1} \cdot E_{(a+1)',a'} \otimes E_{a',a'} \bigg\} \,,
\end{align*}
\begin{align*}
  R_{2}\Delta(e_{a}) =
  & -(q^{1/2} - q^{-1/2}) \bigg\{ (-1)^{\ol{a}} q^{-(-1)^{\ol{a}}} \cdot E_{a'a'} \otimes E_{a,a+1} \\
  & - (-1)^{\ol{a+1}}(-1)^{\ol{a}(\ol{a}+\ol{a+1})} q^{-(-1)^{\ol{a+1}}} \vartheta_{a}\vartheta_{a+1} \cdot
    E_{a+1,a+1} \otimes E_{(a+1)',a'} \\
  & + (-1)^{\ol{a}} \cdot E_{a,a+1} \otimes E_{a'a'} -
    (-1)^{\ol{a+1}}(-1)^{\ol{a}(\ol{a}+\ol{a+1})} \vartheta_{a}\vartheta_{a+1} \cdot E_{(a+1)',a'} \otimes E_{a+1,a+1} \bigg\} \,,
\end{align*}
\begin{align*}
  \Delta^{\opp}(e_{a})R_{2} =
  & -(q^{1/2} - q^{-1/2}) \bigg\{ (-1)^{\ol{a+1}} q^{-(-1)^{\ol{a+1}}} \cdot E_{(a+1)',(a+1)'} \otimes E_{a,a+1} \\
  & - (-1)^{\ol{a}\,\ol{a+1}} q^{-(-1)^{\ol{a}}} \vartheta_{a}\vartheta_{a+1} \cdot E_{aa} \otimes E_{(a+1)',a'} \\
  & + (-1)^{\ol{a+1}} \cdot E_{a,a+1} \otimes E_{(a+1)',(a+1)'} -
    (-1)^{\ol{a}\,\ol{a+1}} \vartheta_{a}\vartheta_{a+1} \cdot E_{(a+1)',a'} \otimes E_{aa} \bigg\} \,,
\end{align*}
\begin{align*}
  R_{3}\Delta(e_{a}) =\,
  & (q - q^{-1}) \left\{ \sum_{j < a} (-1)^{\ol{j}} \cdot (E_{aj}q^{\ssh_{a}/2}) \otimes E_{j,a+1} \right. \\
  & - \sum_{j < (a+1)'} (-1)^{\ol{j}}(-1)^{\ol{a}(\ol{a}+\ol{a+1})} \vartheta_{a}\vartheta_{a+1} \cdot
    (E_{(a+1)'j}q^{\ssh_{a}/2}) \otimes E_{ja'} \\
  & + \sum_{i > a} (-1)^{\ol{a}}(-1)^{(\ol{a}+\ol{i})(\ol{a}+\ol{a+1})} \cdot E_{i,a+1} \otimes (E_{ai}q^{-\ssh_{a}/2}) \\
  & \left. - \sum_{i > (a+1)'} (-1)^{\ol{a}}(-1)^{\ol{i}(\ol{a}+\ol{a+1})} \vartheta_{a}\vartheta_{a+1} \cdot
    E_{ia'} \otimes (E_{(a+1)',i}q^{-\ssh_{a}/2}) \right\} \,,
\end{align*}
\begin{align*}
  \Delta^{\opp}(e_{a})R_{3} =\,
  & (q - q^{-1}) \left\{ \sum_{i > a+1} (-1)^{\ol{a+1}}(-1)^{(\ol{a}+\ol{a+1})(\ol{i}+\ol{a+1})} \cdot
    (q^{-\ssh_{a}/2}E_{i,a+1}) \otimes E_{ai} \right.\\
  & - \sum_{i > a'} (-1)^{\ol{a}}(-1)^{\ol{i}(\ol{a}+\ol{a+1})} \vartheta_{a}\vartheta_{a+1} \cdot
    (q^{-\ssh_{a}/2}E_{ia'}) \otimes E_{(a+1)',i}\\
  & + \sum_{j < a+1} (-1)^{\ol{j}} \cdot E_{aj} \otimes (q^{\ssh_{a}/2}E_{j,a+1})\\
  & \left. - \sum_{j < a'} (-1)^{\ol{j}}(-1)^{\ol{a}(\ol{a}+\ol{a+1})} \vartheta_{a}\vartheta_{a+1} \cdot
    E_{(a+1)',j} \otimes (q^{\ssh_{a}/2}E_{ja'}) \right\} \,.
\end{align*}
Also note that the difference $R_{3}\Delta(e_{a}) - \Delta^{\opp}(e_{a})R_{3}$ can be simplified as follows:
\begin{align*}
  R_{3}(\Delta e_{a}) & - \Delta^{\opp}(e_{a}) R_{3} =
   -\Big(q^{(-1)^{\ol{a}} \cdot 3/2} - q^{-(-1)^{\ol{a}}/2}\Big) E_{aa} \otimes E_{a,a+1} \\
  & + (-1)^{\ol{a}(\ol{a}+\ol{a+1})} \vartheta_{a}\vartheta_{a+1} \Big(q^{(-1)^{\ol{a+1}} \cdot 3/2} - q^{-(-1)^{\ol{a+1}}/2}\Big)
    E_{(a+1)',(a+1)'} \otimes E_{(a+1)',a'} \\
  & + \Big(q^{(-1)^{\ol{a+1}} \cdot 3/2} - q^{-(-1)^{\ol{a+1}}/2}\Big) E_{a+1,a+1} \otimes E_{a,a+1} \\
  & - (-1)^{\ol{a}(\ol{a}+\ol{a+1})} \vartheta_{a}\vartheta_{a+1} \Big(q^{(-1)^{\ol{a}} \cdot 3/2} - q^{-(-1)^{\ol{a}}/2}\Big)
    E_{a'a'} \otimes E_{(a+1)',a'} \,.
\end{align*}

To compute the last two terms involving $R_{4}$, let us first note that~\eqref{eq:rho-simple-root} implies:
\begin{equation*}
  (\rho,\varepsilon_{a} - \varepsilon_{a+1}) =
  \frac{(\varepsilon_{a} - \varepsilon_{a+1},\varepsilon_{a} - \varepsilon_{a+1})}{2} =
  \frac{(-1)^{\ol{a}} + (-1)^{\ol{a+1}}}{2} \,.
\end{equation*}
Using this equality, one derives the following formulas:
\begin{align*}
  R_{4}\Delta(e_{a}) = & -(q - q^{-1})
  \left\{ \sum_{i > a'} (-1)^{\ol{i}\,\ol{a}}(-1)^{\ol{a}} q^{(\rho,\varepsilon_{i}+\varepsilon_{a})}q^{-(-1)^{\ol{a}}/2}
  \vartheta_{i}\vartheta_{a} \cdot E_{ia'} \otimes E_{i',a+1} \right. \\
  &\quad - \sum_{i > a+1} (-1)^{\ol{i}\,\ol{a+1}}(-1)^{\ol{a}(\ol{a}+\ol{a+1})} q^{(\rho,\varepsilon_{i}-\varepsilon_{a+1})}
    q^{-(-1)^{\ol{a+1}}/2} \vartheta_{i}\vartheta_{a} \cdot E_{i,a+1} \otimes E_{i'a'} \\
  &\quad + \sum_{i > a} (-1)^{\ol{i}\,\ol{a+1}}(-1)^{\ol{a}(\ol{a}+\ol{a+1})} q^{(\rho,\varepsilon_{i}-\varepsilon_{a})}
    q^{(-1)^{\ol{a}}/2} \vartheta_{i}\vartheta_{a} \cdot E_{i,a+1} \otimes E_{i'a'} \\
  &\quad \left. - \sum_{i > (a+1)'} (-1)^{\ol{i}\,\ol{a}}(-1)^{\ol{a}} q^{(\rho,\varepsilon_{i}+\varepsilon_{a+1})}
    q^{(-1)^{\ol{a+1}}/2} \vartheta_{i}\vartheta_{a} \cdot E_{ia'} \otimes E_{i',a+1} \right\} \\
  =& -(q - q^{-1}) \bigg\{ -(-1)^{\ol{a}} q^{-(-1)^{\ol{a}}/2} \cdot E_{a'a'} \otimes E_{a,a+1} \\
  & \qquad \quad + (-1)^{\ol{a+1}}(-1)^{\ol{a}(\ol{a}+\ol{a+1})} q^{-(-1)^{\ol{a+1}}/2} \vartheta_{a}\vartheta_{a+1} \cdot
    E_{a+1,a+1} \otimes E_{(a+1)',a'} \bigg\}
\end{align*}
as well as
\begin{align*}
  \Delta^{\opp}(e_{a})R_{4} = & -(q - q^{-1})
  \left\{ - \sum_{j < a} (-1)^{\ol{a+1}\,\ol{j}} q^{(\rho,\varepsilon_{a}-\varepsilon_{j})}q^{-(-1)^{\ol{a}}/2}
  \vartheta_{a+1}\vartheta_{j} \cdot E_{aj} \otimes E_{(a+1)',j'} \right. \\
  &\quad + \sum_{j < (a+1)'} (-1)^{\ol{a}\,\ol{j}}(-1)^{\ol{a}\,\ol{a+1}}
    q^{-(\rho,\varepsilon_{a+1}+\varepsilon_{j})}q^{-(-1)^{\ol{a+1}}/2} \vartheta_{a+1}\vartheta_{j} \cdot E_{(a+1)',j} \otimes E_{aj'} \\
  &\quad + \sum_{j < a+1} (-1)^{\ol{a+1}\,\ol{j}} q^{(\rho,\varepsilon_{a+1}-\varepsilon_{j})}q^{(-1)^{\ol{a+1}}/2}
    \vartheta_{a+1}\vartheta_{j} \cdot E_{aj} \otimes E_{(a+1)',j'}\\
  &\quad \left. - \sum_{j < a'} (-1)^{\ol{a}\,\ol{j}}(-1)^{\ol{a}\,\ol{a+1}}
    q^{-(\rho,\varepsilon_{a}+\varepsilon_{j})}q^{(-1)^{\ol{a}}/2} \vartheta_{a+1}\vartheta_{j} \cdot E_{(a+1)',j} \otimes E_{aj'} \right\}\\
  =& -(q - q^{-1}) \bigg\{ - (-1)^{\ol{a+1}} q^{-(-1)^{\ol{a+1}}/2} \cdot E_{(a+1)',(a+1)'} \otimes E_{a,a+1} \\
  & \qquad \qquad \qquad + (-1)^{\ol{a+1}\,\ol{a}} q^{-(-1)^{\ol{a}}/2} \vartheta_{a}\vartheta_{a+1} \cdot E_{aa} \otimes E_{(a+1)',a'} \bigg\} \,.
\end{align*}
Combining the above eight formulas, using the obvious equalities
\begin{equation*}
  (q^{1/2} - q^{-1/2}) (-1)^{\ol{i}} = q^{(-1)^{\ol{i}}/2} - q^{-(-1)^{\ol{i}}/2} \,, \qquad
  (q - q^{-1}) (-1)^{\ol{i}} = q^{(-1)^{\ol{i}}} - q^{-(-1)^{\ol{i}}} \,,
\end{equation*}
and collecting similar terms, we finally obtain:
\begin{align}\label{eq:int-1}
  \sum_{k=1}^{4} \Big(R_{k} \Delta(e_{a}) - \Delta^{\opp}(e_{a}) R_{k} \Big) = -\Delta(e_{a}) + \Delta^{\opp}(e_{a}) \,.
\end{align}
This establishes the claimed intertwining property $R_{\infty} \Delta(e_a) = \Delta^{\opp}(e_a) R_{\infty}$ for $1\leq a<s$.

\medskip
\noindent
$\bullet$ Proof of $R_{\infty} \Delta(e_s) = \Delta^{\opp}(e_s) R_{\infty}$.

As before, there are three cases to consider: odd $m$, even $m$ with $\ol{s}=\bar{0}$, even $m$ with $\ol{s}=\bar{1}$.
The computations are very similar to those used above to establish~\eqref{eq:int-1} for $a<s$. Thus, we shall only
present the relevant changes in the third case ($m$ is even and $\ol{s}=\bar{1}$) that differs the most.
\begin{equation*}
  R_{1}\Delta(e_{s}) = (q^{1/2} - q^{-1/2})
  \bigg\{ (-1)^{\ol{s}} q^{(-1)^{\ol{s}} \cdot 3/2} \cdot E_{ss} \otimes E_{ss'} +
    (-1)^{\ol{s}} q^{-(-1)^{\ol{s}}/2} \cdot E_{ss'} \otimes E_{ss} \bigg\} \,,
\end{equation*}
\begin{equation*}
  \Delta^{\opp}(e_{s})R_{1} = (q^{1/2} - q^{-1/2})
  \bigg\{ (-1)^{\ol{s}} q^{(-1)^{\ol{s}} \cdot 3/2} \cdot E_{s's'} \otimes E_{ss'} +
    (-1)^{\ol{s}} q^{-(-1)^{\ol{s}}/2} \cdot E_{ss'} \otimes E_{s's'} \bigg\} \,,
\end{equation*}
\begin{equation*}
  R_{2}\Delta(e_{s}) = -(q^{1/2} - q^{-1/2})
  \bigg\{ (-1)^{\ol{s}} q^{-(-1)^{\ol{s}} \cdot 3/2} \cdot E_{s's'} \otimes E_{ss'} +
    (-1)^{\ol{s}} q^{(-1)^{\ol{s}}/2} \cdot E_{ss'} \otimes E_{s's'} \bigg\} \,,
\end{equation*}
\begin{equation*}
  \Delta^{\opp}(e_{s})R_{2} = -(q^{1/2} - q^{-1/2})
  \bigg\{ (-1)^{\ol{s}} q^{-(-1)^{\ol{s}} \cdot 3/2} \cdot E_{ss} \otimes E_{ss'} +
    (-1)^{\ol{s}} q^{(-1)^{\ol{s}}/2} \cdot E_{ss'} \otimes E_{ss} \bigg\} \,,
\end{equation*}
\begin{equation*}
  R_{3}\Delta(e_{s}) = (q - q^{-1})
  \left\{ \sum_{j < s} (-1)^{\ol{j}} \cdot (E_{sj}q^{\ssh_{s}/2}) \otimes E_{js'} +
  \sum_{i > s} (-1)^{\ol{s}} \cdot E_{is'} \otimes (E_{si}q^{-\ssh_{s}/2}) \right\} \,,
\end{equation*}
\begin{align*}
  \Delta^{\opp}(e_{s})R_{3} = (q - q^{-1})
  \left\{ \sum_{i > s'} (-1)^{\ol{s}} \cdot (q^{-\ssh_{s}/2}E_{is'}) \otimes E_{si} +
  \sum_{j < s'} (-1)^{\ol{j}} \cdot E_{sj} \otimes (q^{\ssh_{s}/2}E_{js'}) \right\} \,.
\end{align*}
For the last two terms, we note that~\eqref{eq:rho-simple-root} implies:
\begin{equation*}
  (\rho,2\varepsilon_{s}) = \frac{(2\varepsilon_{s},2\varepsilon_{s})}{2} = (-1)^{\ol{s}} \cdot 2 \,,
\end{equation*}
so that we get:
\begin{align*}
  R_{4}\Delta(e_{s}) =
  & -(q - q^{-1})
    \left\{ \sum_{i > s'} (-1)^{\ol{i}\,\ol{s}}(-1)^{\ol{s}} q^{(\rho,\varepsilon_{i}+\varepsilon_{s})}q^{-(-1)^{\ol{s}}}
    \vartheta_{i}\vartheta_{s} \cdot E_{is'} \otimes E_{i's'} \right. \\
  & \left. \qquad \qquad \qquad + \sum_{i > s} (-1)^{\ol{i}\,\ol{s}} q^{(\rho,\varepsilon_{i}-\varepsilon_{s})}q^{(-1)^{\ol{s}}}
    \vartheta_{i}\vartheta_{s} \cdot E_{is'} \otimes E_{i's'} \right\} \\
  =&\, (1 - q^{-(-1)^{\ol{s}} \cdot 2}) \cdot E_{s's'} \otimes E_{ss'}
\end{align*}
as well as
\begin{align*}
  \Delta^{\opp}(e_{s})R_{4} =
  & -(q - q^{-1}) \left\{ \sum_{j < s} (-1)^{\ol{s}\,\ol{j}} q^{(\rho,\varepsilon_{s}-\varepsilon_{j})}
    q^{-(-1)^{\ol{s}}} \vartheta_{s}\vartheta_{j} \cdot E_{sj} \otimes E_{sj'} \right. \\
  & \left. \qquad \qquad \qquad
    + \sum_{j < s'} (-1)^{\ol{s}\,\ol{j}}(-1)^{\ol{s}} q^{-(\rho,\varepsilon_{s}+\varepsilon_{j})}q^{(-1)^{\ol{s}}}
    \vartheta_{s}\vartheta_{j} \cdot E_{sj} \otimes E_{sj'} \right\}\\
  =&\, (1 - q^{-(-1)^{\ol{s}} \cdot 2}) \cdot E_{ss} \otimes E_{ss'} \,.
\end{align*}
Assembling all the terms, we thus obtain:
\begin{multline*}
  R_{1}\Delta(e_{s}) - \Delta^{\opp}(e_{s})R_{1} = \\
  \Big(q^{(-1)^{\ol{s}} \cdot 2} - q^{(-1)^{\ol{s}}}\Big) \cdot (E_{ss} - E_{s's'}) \otimes E_{ss'} +
  \Big(1 - q^{-(-1)^{\ol{s}}}\Big) \cdot E_{ss'} \otimes (E_{ss} - E_{s's'}) \,,
\end{multline*}
\begin{multline*}
  R_{2}\Delta(e_{s}) - \Delta^{\opp}(e_{s})R_{2} = \\
  \Big(q^{-(-1)^{\ol{s}}} - q^{-(-1)^{\ol{s}} \cdot 2}\Big) \cdot (E_{ss} - E_{s's'}) \otimes E_{ss'} +
  \Big(q^{(-1)^{\ol{s}}} - 1\Big) \cdot E_{ss'} \otimes (E_{ss} - E_{s's'}) \,,
\end{multline*}
\begin{flushleft}
\hspace{\multlinegap}$\displaystyle
  R_{3}\Delta(e_{s}) - \Delta^{\opp}(e_{s})R_{3} =
  - \Big(q^{(-1)^{\ol{s}} \cdot 2} - 1\Big) \cdot (E_{ss} - E_{s's'}) \otimes E_{ss'} \,,$
\end{flushleft}
\begin{flushleft}
\hspace{\multlinegap}$\displaystyle
  R_{4}\Delta(e_{s}) - \Delta^{\opp}(e_{s})R_{4} =
  - \Big(1 - q^{-(-1)^{\ol{s}} \cdot 2} \Big) \cdot (E_{ss} - E_{s's'}) \otimes E_{ss'} \,.$
\end{flushleft}
Collecting the similar terms together, we finally get:
\begin{align*}
  \sum_{k=1}^{4} \Big(R_{k} \Delta (e_{s}) - \Delta^{\opp}(e_{s}) R_{k}\Big) =
  & -(q^{\ssh_{s}/2} - q^{-\ssh_{s}/2}) \otimes \sse_{s} + \sse_{s} \otimes (q^{\ssh_{s}/2} - q^{-\ssh_{s}/2}) \\
  =& -\Delta(e_{s}) + \Delta^{\opp}(e_{s}) \,.
\end{align*}
This establishes the claimed intertwining property $R_{\infty} \Delta(e_s) = \Delta^{\opp}(e_s) R_{\infty}$.

\medskip
\noindent
$\bullet$ Proof of $R_{\infty} \Delta(q^{h_a/2}) = \Delta^{\opp}(q^{h_a/2}) R_{\infty}$ for $1 \leq a \leq s$.

Since $\varrho(q^{h_a/2})=q^{\ssh_{a}/2}$ is a diagonal matrix, we can write
$\varrho(q^{h_a/2})=q^{\ssh_{a}/2} = \diag(\sft_{1}, \ldots, \sft_{1'})$. Furthermore, we note that $\sft_{i}\sft_{i'} = 1$
for all $i$. Therefore, $\Delta(q^{h_{a}/2})=\Delta^{\opp}(q^{h_a/2})=q^{\ssh_{a}/2} \otimes q^{\ssh_{a}/2}$
commutes with all the terms of the form $E_{ii}\otimes E_{jj}$, $E_{ii} \otimes E_{ii}$, $E_{ii} \otimes E_{i'i'}$,
$E_{ij} \otimes E_{ji}$, and $E_{ij} \otimes E_{i'j'}$. This implies the desired intertwining property for $q^{h_{a}/2}$.

\medskip
\noindent
$\bullet$ Proof of $R_{\infty} \Delta(f_a) = \Delta^{\opp}(f_a) R_{\infty}$ for $1 \leq a \leq s$.

We first recall some basic properties of the supertransposition~\eqref{eq:supertranspose}. For any
$X \otimes Y \in \End(V)^{\otimes 2}$, let $(X \otimes Y)^{\mathrm{st}_{1}} = X^{\mathrm{st}} \otimes Y$
and $(X \otimes Y)^{\mathrm{st}_{2}} = X \otimes Y^{\mathrm{st}}$ denote the supertransposition applied to
the first and the second component, respectively. Then, we have:
\begin{equation*}
  (XY)^{\mathrm{st}} = (-1)^{|X||Y|} Y^{\mathrm{st}} X^{\mathrm{st}}
\end{equation*}
as well as
\begin{equation*}
  \Big( (X \otimes Y)(X' \otimes Y') \Big)^{\mathrm{st}_{1} \mathrm{st_{2}}} =
  (-1)^{(|X|+|Y|)(|X'|+|Y'|)} (X' \otimes Y')^{\mathrm{st}_{1} \mathrm{st_{2}}} (X \otimes Y)^{\mathrm{st}_{1} \mathrm{st_{2}}}
\end{equation*}
for any homogeneous $X,X',Y,Y' \in \End(V)$.

We note that $(q^{\ssh_{i}/2})^{\mathrm{st}} = q^{\ssh_{i}/2}$ and $(\sse_{i})^{\mathrm{st}}$ is always a nonzero
scalar multiple of $\ssf_{i}$, due to formulas \eqref{eq:Lie-action-case1}--\eqref{eq:Lie-action-case3}. Furthermore,
\eqref{eq:R_inf} also implies
\begin{equation}\label{eq:R_inf-supertranspose}
  R_{\infty} = \tau_{VV} \circ (R_{\infty})^{\mathrm{st}_{1} \mathrm{st}_{2}} \circ \tau^{-1}_{VV} \,.
\end{equation}
Thus, applying $\mathrm{st}_{1} \circ \mathrm{st}_{2}$ to the equation $R_{\infty} \Delta(e_a) = \Delta^{\opp}(e_a) R_{\infty}$
and evoking~\eqref{eq:delta-opp}, we obtain
\begin{equation*}
  \Delta(f_{a})(R_{\infty})^{\mathrm{st}_{1} \mathrm{st}_{2}} = (R_{\infty})^{\mathrm{st}_{1} \mathrm{st}_{2}}\Delta^{\opp}(f_{a}) \,.
\end{equation*}
Conjugating this equality by $\tau_{VV}$ and evoking~\eqref{eq:R_inf-supertranspose}, we get the desired intertwining property
\begin{equation*}
  R_{\infty} \Delta(f_a) = \Delta^{\opp}(f_a) R_{\infty} \,.
\end{equation*}
This completes the proof of the second equality of~\eqref{eq:intertwiner}.

\medskip
The intertwining property $R_{0} \Delta(x) = \Delta^{\opp}(x) R_{0}$ is actually directly implied by the one for $R_\infty$,
which we just proved. To this end, let us first note the following equality:
\begin{multline}\label{eq:R0-vs-Rinf}
  \tau_{VV} \circ R_{0} \circ \tau^{-1}_{VV} =  \ID +
  (q^{-1/2} - q^{1/2}) \sum_{1\leq i\leq N} (-1)^{\ol{i}} E_{ii} \otimes
    \Big( q^{-(\varepsilon_{i},\varepsilon_{i})/2} E_{ii} - q^{(\varepsilon_{i},\varepsilon_{i})/2} E_{i'i'} \Big) + \\
  (q^{-1} - q) \sum_{i > j} (-1)^{\ol{j}} E_{ij} \otimes
    \Big( E_{ji} - (-1)^{\ol{j}(\ol{i}+\ol{j})} \vartheta_{i}\vartheta_{j} q^{-(\rho, \varepsilon_{i} - \varepsilon_{j})} E_{i'j'} \Big) =
  R_{\infty}|_{q \mapsto q^{-1}} \,.
\end{multline}
As the notation suggests, $R_{\infty}|_{q \mapsto q^{-1}}$ is the $\BC(q)$-valued $N^2\times N^2$ matrix obtained
from $R_{\infty}$ by applying to all matrix coefficients the $\BC$-algebra automorphism
\begin{equation}\label{eq:bar-sigma}
  \bar{\sigma}\colon \BC(q)\to \BC(q) \qquad \mathrm{determined\ by} \quad  q\mapsto q^{-1} \,.
\end{equation}
We claim that the assignment
\begin{equation}\label{eq:invol-sigma}
  \sigma\colon \quad e_{a} \mapsto e_{a} \,,\quad f_{a} \mapsto f_{a} \,,\quad
  q^{\pm h_{a}/2} \mapsto q^{\mp h_{a}/2} \,,\quad q \mapsto q^{-1}
\end{equation}
gives rise to a $\BC$-algebra involution $\sigma\colon \uqV\to \uqV$. To prove this, we note that
relations~\eqref{eq:q-chevalley-rel-hh}--\eqref{eq:q-chevalley-rel-ef} are clearly preserved by~\eqref{eq:invol-sigma},
as well as the ideal of $U^+_q(\fosp(V))$ (respectively of $U^-_q(\fosp(V))$) generated by all Serre relations in
$\{e_i\}$ (respectively $\{f_i\}$) as follows from~\cite[Lemma~6.3.1]{y}\footnote{We note that some of the individual
higher order Serre relations of~\cite{y0} are actually \underline{not} preserved under~\eqref{eq:invol-sigma}.}.
We also define $\Delta^\sigma, (\Delta^\opp)^\sigma\colon \uqV\to \uqV\otimes \uqV$ via
\begin{align}\label{eq:sigma-coproduct}
  \Delta^{\sigma} = (\sigma \otimes \sigma) \circ \Delta \circ \sigma^{-1} \,,\qquad
  (\Delta^{\opp})^{\sigma} = (\sigma \otimes \sigma) \circ \Delta^{\opp} \circ \sigma^{-1} \,.
\end{align}
Then, applying $\bar{\sigma}$ to all matrix coefficients in the equality
$R_\infty\circ \Delta(x)=\Delta^{\opp}(x)\circ R_\infty$ and using the obvious equality
$\rho\circ \sigma = \bar{\sigma} \circ \rho$, we obtain
\begin{equation}\label{eq:sigma-Rinf}
  R_{\infty}|_{q \mapsto q^{-1}} \circ \Delta^{\sigma}(\sigma(x)) =
  (\Delta^{\opp})^{\sigma}(\sigma(x)) \circ R_{\infty}|_{q \mapsto q^{-1}}
  \qquad \forall\, x \in \uqV \,.
\end{equation}
However, direct computation of $\Delta^\sigma$ on the generators $e_{a}, f_{a}, q^{\pm h_{a}/2}\ (1\leq a\leq s)$ shows that
\begin{equation}\label{eq:sigma-opp}
  \Delta^{\sigma} = \Delta^{\opp} \,, \qquad (\Delta^{\opp})^{\sigma} = \Delta \,.
\end{equation}
Combining~(\ref{eq:R0-vs-Rinf},~\ref{eq:sigma-Rinf},~\ref{eq:sigma-opp}) with~\eqref{eq:delta-opp}, we obtain
\begin{align*}
  R_{0} \circ \Delta(\sigma(x)) = \Delta^{\opp}(\sigma(x)) \circ R_{0} \qquad \forall\, x \in \uqV \,.
\end{align*}
As $\sigma$ is invertible, this implies $R_{0} \Delta(x) = \Delta^{\opp}(x) R_{0}$ for any $x$,
thus establishing Proposition~\ref{prop:intertwiner}.


\section{Factorization of finite R-Matrices}\label{sec:factorization-R}

In this Section, we present the factorization formula for $\Theta$ and use it to re-derive the formulas for $R_{VV}$
from Theorem~\ref{thm:R_osp_finite}. To this end, we use a combinatorial construction of orthogonal dual bases of
$U^+_{q}(\fosp(V))$ and $U^-_{q}(\fosp(V))$, based on the combinatorics of dominant Lyndon words.


\subsection{Shuffle superalgebra}
\

Let $\sF$ be the free $\BC(q)$-superalgebra generated by the finite alphabet $I = \{1,2, \ldots, s\}$, and let $\sW$
be the set of words in $I$, i.e.\ the monoid generated by $I$. Thus, $\sF$ has a basis consisting of finite length
words $[i_{1} \dots i_{d}] = i_{1}i_{2} \dots i_{d}$, where $i_{1},\ldots ,i_{d} \in I$. Note that $\sF$ is
$Q^+\times \BZ_2$-graded via
\begin{equation*}
  \deg([i_1 \dots i_d])=\alpha_{i_1} + \dots + \alpha_{i_d}\in Q^+ \,, \qquad
  p([i_1 \dots i_d])=|\sse_{i_1}| + \dots + |\sse_{i_d}|\in \BZ_2 \,,
\end{equation*}
cf.~\eqref{eq:Z2-grading}. Following~\cite[\S3.1]{chw}, we define the \emph{$q$-quantum shuffle product}
$\diamond\colon \sF\times \sF\to \sF$ inductively via
\begin{equation*}
  (xi) \diamond (yj) = (x \diamond (yj))i + (-1)^{p(xi)p(j)} q^{-(\deg(xi),\alpha_j)}  ((xi) \diamond y)j \,,\qquad
  \emptyset \diamond x = x \diamond \emptyset = x\,,
\end{equation*}
for all $i,j \in I$ and $x,y \in \sF$ homogeneous with respect to the $Q^+\times \BZ_2$-grading.
By iterating this definition, one obtains (a correction of~\cite[(3.4)]{chw}):
\begin{equation*}
  [i_{1} \dots i_{a}] \diamond [i_{a + 1} i_{a + 2} \dots i_{a + b}] =
  \sum_{\sigma} e_{a,b}(\sigma) [i_{\sigma^{-1}(1)}i_{\sigma^{-1}(2)} \dots i_{\sigma^{-1}(a + b)}] \,,
\end{equation*}
where
\begin{equation*}
  e_{a,b}(\sigma) = \prod_{\substack{k \le a < l, \\ \sigma(k) < \sigma(l)}}
  \left( (-1)^{p(\sse_{i_{k}})p(\sse_{i_{l}})} q^{-(\alpha_{i_{k}},\alpha_{i_{l}})} \right)
\end{equation*}
and the sum runs over all \emph{$(a,b)$-shuffles} of $\{1, 2, \ldots, a + b\}$, i.e.\ the permutations
$\sigma \in S_{a + b}$ such that $\sigma(1) < \sigma(2) < \dots <\sigma(a)$ and $\sigma(a + 1) < \dots < \sigma(a + b)$.

The $q$-shuffle algebra provides a combinatorial model for $U^+_q(\fosp(V))$ via~\cite[Corollary 3.4]{chw}:

\begin{Prop}\label{prop:shuffle_iso}
There is a unique superalgebra embedding $\Psi\colon U^+_q(\fosp(V))\to \sF$ with $\Psi(e_i)=[i]$.
\end{Prop}

Let $\sU=\Psi(U^+_q(\fosp(V)))$ be the image of this embedding, so that $\Psi\colon U^+_q(\fosp(V)) \iso \sU$.


\subsection{Combinatorics of words}
\

From now on, we fix an order $\leq$ on the alphabet $I$, which induces a lexicographical order on the monoid~$\sW$.
For a nonzero $x\in \sF$, its \emph{leading term}, denoted $\max(x)$, is a word $w\in \sW$ such that
  $$ x=\sum_{u\leq w} t_u\cdot u \qquad \mathrm{with} \quad t_u\in \BC(q) \quad \mathrm{and} \quad t_w\ne 0 \,. $$
Following the terminology of~\cite[\S4.1]{chw}, we call a word $w\in \sW$ \emph{dominant} if it appears as
a leading term of some element from $\sU$, and let $\sW^{+}$ denote the subset of all dominant words in $\sW$.

\begin{Rem}
It turns out that $\sW^+$ can be used to construct various bases of $U^+_q(\fosp(V))$. For example, the set
$\{e_{w} = e_{i_{1}}\ldots e_{i_{d}} \,|\, w = [i_{1} \dots i_{d}] \in \sW^{+}\}$ is a basis of $U^+_q(\fosp(V))$,
according to~\cite[Proposition~4.1]{chw}. However, we shall rather work with more sophisticated \emph{Lyndon bases} below.
\end{Rem}

A word $w=[i_1 \dots i_d]$ is called \emph{Lyndon} if it is smaller than any of its proper right factors:
\begin{equation*}
  w < [i_k \dots i_d] \qquad \forall\, 1<k\leq d \,.
\end{equation*}
We use $\sL$ to denote the set of all Lyndon words. It is well-known that any word $w$ admits a unique
\emph{canonical factorization} (see~\cite[Proposition~5.1.5]{lo}) as a product of non-increasing Lyndon words:
\begin{equation}\label{eq:canonical_factorization}
  w = \ell_{1}\ell_{2} \dots \ell_{k} \,,\quad \ell_1\geq \ell_2 \geq \dots \geq \ell_k \,,\quad \ell_1, \dots, \ell_k\in \sL \,.
\end{equation}
Furthermore, any Lyndon word $\ell\in \sL$ admits a unique \emph{costandard factorization} $\ell=\ell_1\ell_2$
such that $\ell_1,\ell_2\ne \emptyset$, $\ell_1\in \sL$ is the longest possible, in which case also~$\ell_2\in \sL$
(see~\cite[Proposition~5.1.3]{lo}).

Let $\sL^{+} = \sW^{+} \cap \sL$ be the set of all dominant Lyndon words.
We also recall the reduced root system $\bar{\Phi}$ from~\eqref{eq:reduced_roots}.
The following result is proved in~\cite[Theorem 4.8]{chw}:

\begin{Prop}\label{prop:dominant-factorization}
(a) The map $\ell\mapsto \deg(\ell)$ defines a bijection $\rl\colon \sL^+ \iso \bar{\Phi}^+$.

\medskip
\noindent
(b) A word $w \in \sW$ is dominant if and only if its canonical factorization is of the form
\begin{equation}\label{eq:dominant_can.fact}
  w = \ell_{1} \ell_{2} \dots \ell_{k} \,,\quad
  \ell_{1} \ge \ell_{2} \ge \dots \ge \ell_{k} \,,\quad
  \ell_{1}, \ell_{2}, \dots, \ell_{k}\in \sL^+
\end{equation}
where $\ell_{p}$ appears only once if $\deg(\ell_p)$ is an isotropic odd root (see Subsection 2.3).
\end{Prop}

We note that the above bijection $\rl$ gives rise to a \emph{lexicographical} ordering on $\bar{\Phi}^+$:
\begin{equation}\label{eq:lex-order}
   \alpha < \beta \quad \Longleftrightarrow  \quad \rl^{-1}(\alpha) < \rl^{-1}(\beta) \ \mathrm{lexicographically} \,.
\end{equation}


\subsection{Lyndon basis}
\

For $x,y \in \sF$ homogeneous with respect to $Q^+\times \BZ_2$-grading, their \emph{$q$-commutator} is defined as
\begin{equation}\label{eq:q-commutator}
  [x,y]_{q} = xy - (-1)^{p(x)p(y)}q^{(\deg(x),\deg(y))} yx \,,
\end{equation}
cf.~\eqref{eq:q-superbracket}.
Following~\cite[\S4.3]{chw}, we define the \emph{$q$-bracketing} $[\ell]\in \sF$ of a Lyndon word $\ell\in \sL$ via:
\begin{itemize}

\item
$[\ell] = \ell$ if $\ell \in I$,

\item
$[\ell] = [[\ell_{1}],[\ell_{2}]]_q$ if $\ell = \ell_{1}\ell_{2}$ is the costandard factorization of $\ell$.

\end{itemize}
Evoking the canonical factorization~\eqref{eq:canonical_factorization}, we define the $q$-bracketing of any
word $w\in \sW$ via:
\begin{equation*}
  [w] = [\ell_{1}][\ell_{2}] \dots [\ell_{k}] \,.
\end{equation*}
According to~\cite[Proposition~4.10]{chw}, the set $\big\{[w] \,|\, w \in \sW \big\}$ is a basis for $\sF$.
Finally, we also define $\Xi\colon (\sF,\cdot) \to (\sF,\diamond)$ as the algebra homomorphism given by
$\Xi([i_{1} \dots i_{d}]) = i_{1} \diamond \dots \diamond i_{d}$. Then, we have the following equivalent
description of dominant words, see~\cite[Lemma~4.11]{chw}:

\begin{Lem}
A word $w \in \sW$ is dominant if and only if it cannot be expressed modulo $\ker(\Xi)$ as
a linear combination of words $v > w$.
\end{Lem}

For any dominant word $w \in \sW^{+}$, we define
\begin{equation*}
  R_{w} = \Xi([w]) \,.
\end{equation*}
For any homogeneous $x,y \in \sF$, we introduce $x \diamond_{q,q^{-1}} y$ similarly to~\eqref{eq:q-commutator}:
\begin{equation*}
  x \diamond_{q,q^{-1}} y = x \diamond y - (-1)^{p(x)p(y)}q^{(\deg(x),\deg(y))} y \diamond x \,.
\end{equation*}
This formula implies that if $\ell\in \sL^+$ has a costandard factorization $\ell=\ell_1\ell_2$, then
$R_{\ell} = R_{\ell_{1}} \diamond_{q,q^{-1}} R_{\ell_{2}}$. According to~\cite[Proposition~4.13]{chw}, the set
$\big\{R_{w} \,|\, w \in \sW^{+} \big\}$ is a basis for $\sU$, referred to as the \emph{Lyndon basis} of $\sU$.
Evoking Proposition~\ref{prop:dominant-factorization}, it has the form:
\begin{equation}\label{eq:Lyndon-basis}
  \left\{ R_{\ell_{1}} \diamond \dots \diamond R_{\ell_{k}} \,\Big|\,
         \substack{k \in \BZ_{\geq 0},\ \ell_{1}, \dots, \ell_{k} \in \sL^{+},\ \ell_{1} \ge \dots \ge \ell_{k},\\
         \ell_{p-1}>\ell_p>\ell_{p+1} \ \mathrm{if}\ \deg(\ell_p)\in \bar{\Phi}_{\bar{1}} \ \mathrm{is\ isotropic}} \right\}.
\end{equation}


\subsection{Explicit computations}\label{ssec:explicit-BCD}
\

In this Subsection, we specify explicitly the set $\sL^+$ of dominant Lyndon words, the lexicographical
order~\eqref{eq:lex-order} on $\bar{\Phi}^+$, and the map $\bar{\Phi}^+\to \bar{\Phi}^+\times \bar{\Phi}^+$
which assigns to a root $\gamma\in\bar{\Phi}^+$ a pair of roots $\alpha=\rl(\ell_1),\beta=\rl(\ell_2)$,
where $\ell=\ell_1\ell_2$ is the costandard factorization of $\ell=\rl^{-1}(\gamma)$, see
Proposition~\ref{prop:dominant-factorization}.
To this end, we choose a specific ordering $1<2<\dots<s$ on the alphabet $I$,
as in~\cite[\S6]{chw}.

\smallskip
\noindent
$\bullet$ Case 1: $m$ is odd. In this case, according to~\cite[Proposition~6.5]{chw}:
\begin{equation}\label{eq:dL-B}
  \sL^+ =
  \big\{ [i \dots j] \,\big|\, 1\leq i\leq j\leq s \big\} \cup
  \big\{ [i \dots s s \dots j] \,\big|\, 1\leq i < j \leq s \big\} \,.
\end{equation}
This results in the following lexicographical order on the reduced root system:
\begin{equation}\label{eq:B_order}
\begin{split}
  & \alpha_{1} < \alpha_{1} + \alpha_{2} < \dots < \alpha_{1} + \dots + \alpha_{s} \\
  & \quad < \alpha_{1} + \dots + \alpha_{s-1} + 2\alpha_{s} < \alpha_{1} + \dots + 2\alpha_{s-1} + 2\alpha_{s} <
    \dots < \alpha_{1} + 2\alpha_{2} + \dots + 2\alpha_{s} \\
  & \quad < \alpha_{2} < \dots < \alpha_{s-1} < \alpha_{s-1} + \alpha_{s} < \alpha_{s-1} + 2\alpha_{s} < \alpha_{s} \,.
\end{split}
\end{equation}
Let $\gamma_{ij} = \alpha_{i} + \dots + \alpha_{j}$ for $1 \le i \le j \le s$, and let
$\beta_{ij} = \alpha_{i} + \dots + \alpha_{j-1} + 2\alpha_{j} + \dots + 2\alpha_{s}$  for $1 \le i < j \le s$.
Then, the aforementioned assignment $\gamma\mapsto (\alpha,\beta)$ is explicitly given by:
\begin{itemize}

\item[--]
for the roots $\gamma=\gamma_{ij}$ with $i < j$, we have $(\alpha,\beta)=(\gamma_{i,j - 1},\alpha_{j})$;

\item[--]
for the roots $\gamma=\beta_{is}$ with $1\leq i<s$, we have $(\alpha,\beta)=(\gamma_{is},\alpha_{s})$;

\item[--]
for the roots $\gamma=\beta_{ij}$ with $i < j < s$, we have $(\alpha,\beta)=(\beta_{i,j + 1},\alpha_{j})$.

\end{itemize}

\smallskip
\noindent
$\bullet$ Case 2: $m$ is even and $\ol{s}=\bar{0}$. In this case, according to~\cite[Proposition~6.12]{chw}:
\begin{equation}\label{eq:dL-D}
\begin{split}
  \sL^+ =
  & \big\{ [i \dots j] \,\big|\, 1\leq i\leq j\leq s-1 \big\} \cup
    \big\{ [i \dots (s-2)s] \,\big|\, 1\leq i\leq s-2 \big\} \cup \\
  & \big\{ [i \dots (s-2) s (s-1) \dots j] \,\big|\, 1\leq i < j \leq s-1 \big\} \cup \\
  & \big\{ [i \dots (s-2)(s-1) i \dots (s-2) s] \,\big|\, 1\leq i < s-1 \ \mathrm{and}\ p([i\dots (s-1)])=\bar{1} \big\} \,.
\end{split}
\end{equation}
This results in the following lexicographical order on the reduced root system:
\begin{equation}\label{eq:D_order}
\begin{split}
  & \alpha_{1} < \alpha_{1} + \alpha_{2} < \dots < \alpha_{1} + \dots + \alpha_{s-2}+\alpha_{s-1} <
    \underline{2\alpha_1 + \dots + 2\alpha_{s-2} + \alpha_{s-1} + \alpha_s} \\
  & \quad < \alpha_{1} +\dots +  \alpha_{s-2} + \alpha_{s} < \alpha_{1} + \dots + \alpha_{s}
    < \alpha_{1} + \dots + \alpha_{s-3} + 2\alpha_{s-2} + \alpha_{s-1} + \alpha_{s} \\
  & \quad < \dots < \alpha_{1} + 2\alpha_{2} + \dots + 2\alpha_{s-2} + \alpha_{s-1} + \alpha_{s} < \alpha_{2}
    < \dots < \alpha_{s-2} \\
  & \quad <\alpha_{s-2}+\alpha_{s-1}<\underline{2\alpha_{s-2} + \alpha_{s-1} + \alpha_{s}} <
    \alpha_{s-2}+\alpha_{s}< \alpha_{s-2} + \alpha_{s-1} + \alpha_{s} <\alpha_{s-1}<\alpha_{s} \,,
\end{split}
\end{equation}
where the underlined $2\alpha_i+ \dots + 2\alpha_{s-2} + \alpha_{s-1} + \alpha_{s}$ means that it is omitted unless
$p([i\dots (s-1)])=\bar{1}$. Let $\gamma_{ij} = \alpha_{i} + \dots + \alpha_{j}$ for $1 \le i \leq j < s$,
$\beta_{is} = \alpha_{i} + \dots + \alpha_{s-2} + \alpha_{s}$, $\beta_{i,s -1} = \alpha_{i} + \dots + \alpha_{s}$, and
finally $\beta_{ij} = \alpha_{i} + \dots + \alpha_{j - 1} + 2\alpha_{j} + \dots + 2\alpha_{s-2} + \alpha_{s-1} + \alpha_{s}$
for $i\leq j < s - 1$ ($\beta_{ii}$ is omitted unless $p([i\dots (s-1)])=\bar{1}$).
Then, the aforementioned assignment $\gamma \mapsto (\alpha,\beta)$ is explicitly given by:
\begin{itemize}

\item[--]
for the roots $\gamma=\gamma_{ij}$ with $i<j$, we have $(\alpha,\beta)=(\gamma_{i,j - 1},\alpha_{j})$;

\item[--]
for the roots $\gamma=\beta_{is}$ with $i\leq s-2$, we have $(\alpha,\beta)=(\gamma_{i,s - 2},\alpha_{s})$;

\item[--]
for the roots $\gamma=\beta_{ij}$ with $i < j < s$, we have $(\alpha,\beta)=(\beta_{i,j + 1},\alpha_{j})$;

\item[--]
for the underlined roots $\gamma=\beta_{ii}$, we have $(\alpha,\beta)=(\gamma_{i,s-1},\beta_{is})$.

\end{itemize}

\smallskip
\noindent
$\bullet$ Case 3: $m$ is even and $\ol{s}=\bar{1}$. In this case, according to~\cite[Proposition~6.9]{chw}:
\begin{equation}\label{eq:dL-C}
\begin{split}
  \sL^+ =
  & \big\{ [i \dots j] \,\big|\, 1\leq i\leq j\leq s \big\} \cup
    \big\{ [i \dots (s-1) s (s-1) \dots j] \,\big|\, 1\leq i < j \leq s \big\} \cup \\
  & \big\{ [i \dots (s-1) i \dots (s-1) s] \,\big|\, 1\leq i < s \ \mathrm{and}\ p([i\dots (s-1)])=\bar{0} \big\} \,.
\end{split}
\end{equation}
This results in the following lexicographical order on the reduced root system:
\begin{equation}\label{eq:C_order}
\begin{split}
  & \alpha_{1} < \alpha_{1} + \alpha_{2} < \dots < \alpha_{1} + \dots + \alpha_{s-1} <
    \underline{2\alpha_{1} + \dots + 2\alpha_{s-1} + \alpha_{s}} < \alpha_{1} + \dots + \alpha_{s} \\
  & \quad < \alpha_{1} + \dots + \alpha_{s-2} + 2\alpha_{s-1} + \alpha_{s} < \dots <
    \alpha_{1} + 2\alpha_{2} + \dots + 2\alpha_{s-1} + \alpha_{s} < \\
  & \quad \alpha_{2} < \dots < \alpha_{s-1} < \underline{2\alpha_{s-1} + \alpha_{s}} <
    \alpha_{s-1} + \alpha_{s} < \alpha_{s} \,,
\end{split}
\end{equation}
where the underlined $2\alpha_i+ \dots + 2\alpha_{s-1} + \alpha_{s}$ means that it is omitted unless $p([i\dots (s-1)])=\bar{0}$.
Let $\gamma_{ij} = \alpha_{i} + \dots + \alpha_{j}$ for $1 \le i \le j \le s$ and
$\beta_{ij} = \alpha_{i} + \dots + \alpha_{j - 1} + 2\alpha_{j} + \dots + 2\alpha_{s-1} + \alpha_{s}$ for $1 \le i \le  j<s$
($\beta_{ii}$ is omitted unless $p([i\dots (s-1)])=\bar{0}$).
Then, the aforementioned assignment $\gamma \mapsto (\alpha,\beta)$ is explicitly given by:
\begin{itemize}

\item[--]
for the roots $\gamma=\gamma_{ij}$ with $i<j$,  we have $(\alpha,\beta)=(\gamma_{i,j - 1},\alpha_{j})$;

\item[--]
for the roots $\gamma=\beta_{i,s-1}$ with $1\leq i<s-1$, we have $(\alpha,\beta)=(\gamma_{is},\alpha_{s-1})$;

\item[--]
for the roots $\gamma=\beta_{ij}$ with $i<j<s-1$,  we have $(\alpha,\beta)=(\beta_{i,j + 1},\alpha_{j})$;

\item[--]
for the underlined roots $\gamma=\beta_{ii}$ with $1\leq i<s$, we have $(\alpha,\beta)=(\gamma_{i,s - 1},\gamma_{is})$.

\end{itemize}


\subsection{Symmetric pairing}
\

In this Subsection we shall endow $U^+_q(\fosp(V))^{\otimes 2}$ with a standard \emph{twisted} multiplication:
\begin{equation*}
   (a\otimes b) (c\otimes d) = (-1)^{|b||c|} q^{-(\deg(b),\deg(c))} (ac)\otimes (bd)
\end{equation*}
for $a,b,c,d\in U^+_q(\fosp(V))$ homogeneous with respect to the $Q^+\times \BZ_2$-grading.
Following~\cite[\S2.2]{chw}, we equip $U^+_q(\fosp(V))$ with a \emph{twisted} coproduct
$\Delta^{\chw}\colon U^+_q(\fosp(V))\to U^+_q(\fosp(V))^{\otimes 2}$ defined by
\begin{equation*}
  \Delta^{\chw}(e_i) = e_i\otimes 1 + 1\otimes e_i \qquad \forall\, i\in I \,.
\end{equation*}
Furthermore, we have the following result of~\cite[Proposition 2.4]{chw}:

\begin{Prop}\label{prop:pairing_twisted}
There exists a unique non-degenerate symmetric bilinear pairing
\begin{equation*}
  (\cdot,\cdot)^{\chw}\colon U^+_q(\fosp(V))\times U^+_q(\fosp(V)) \longrightarrow \BC(q)
\end{equation*}
satisfying
\begin{equation*}
  (1,1)^{\chw}=1, \quad
  (e_i,e_j)^{\chw}=\delta_{ij}, \quad
  (x,yy')^{\chw}=(\Delta^{\chw}(x),y\otimes y')^{\chw}
\end{equation*}
for any $i,j \in I$ and $x,y,y'\in U^+_q(\fosp(V))$, where
$(x'\otimes x'', y'\otimes y'')^{\chw}=(x',y')^{\chw}(x'',y'')^{\chw}$.
\end{Prop}

Evoking the isomorphism $U^+_q(\fosp(V))\simeq \sU$, see Proposition~\ref{prop:shuffle_iso}, we shall use
the same notation for the symmetric bilinear pairing on $\sU$ satisfying similar properties:
\begin{equation*}
  (\cdot,\cdot)^{\chw}\colon \sU \times \sU \longrightarrow \BC(q) \,.
\end{equation*}


\subsection{Pairing of Lyndon basis}
\

We shall now summarize the key results of~\cite[\S 5--6]{chw} in the form relevant to us. For any $w \in \sW^{+}$,
consider its canonical factorization $ w = w_{1}w_{2} \dots w_{d}$ into dominant Lyndon words~\eqref{eq:dominant_can.fact}
and define
\begin{equation*}
  \wtd{R}_{w} = R_{w_{d}} \diamond R_{w_{d-1}} \diamond \dots \diamond R_{w_{1}} \,.
\end{equation*}
The following orthogonality result is established in~\cite[Theorem~5.7]{chw} (we note that while the authors work
with $\mathsf{E}_w$ in~\cite{chw}, they are just multiples of $\wtd{R}_w$, as follows from~\cite[\S6]{chw}):

\begin{Prop}\label{prop:CHW-main-theorem}
Let $\ell,w \in \sW^{+}$. Then $(\wtd{R}_{\ell},\wtd{R}_{w})^{\chw} = 0$ unless $\ell = w$. Moreover, if
$\ell = \ell_{1}^{n_{1}}\ell_{2}^{n_{2}}\ldots \ell_{d}^{n_{d}}$ with $\ell_{1} > \ell_{2} > \dots > \ell_{d}$
is the canonical factorization of $\ell$ into dominant Lyndon words, then:
\begin{equation}\label{eq:Rw-pairing}
  (\wtd{R}_{\ell},\wtd{R}_{\ell})^{\chw} =
  \prod_{t = 1}^{d} \left( C_{\ell_{t},n_{t}}\, \cdot ((R_{\ell_{t}} , R_{\ell_{t}})^{\chw})^{n_{t}} \right)
\end{equation}
with
\begin{equation*}
  C_{\ell,p} =
  \prod_{k=1}^{p} \frac{1 - \big((-1)^{p(\ell)} q^{-(\deg(\ell),\deg(\ell))}\big)^{k}}
   {1 - (-1)^{p(\ell)} q^{-(\deg(\ell),\deg(\ell))}} \qquad \mathrm{for\ any} \quad \ell\in \sL^+,\, p\in \BZ_{\geq 0} \,.
\end{equation*}
\end{Prop}

The explicit computation of the pairings $(R_{\ell},R_{\ell})^{\chw}$ for $\ell\in \sL^+$ has been carried out in~\cite[\S6]{chw},
while for $\ell \in I$ we trivially have $(R_{\ell}, R_{\ell})^{\chw} = 1$ by Proposition~\ref{prop:pairing_twisted}. Thus, we shall
summarize these formulas for words $\ell$ of length $>1$ in three lemmas below (also correcting several typos from~\cite{chw}).
Following~\cite[(3.13)]{chw}, we shall use the following notation for
$\alpha=\alpha_{i_{1}} + \cdots + \alpha_{i_{r}} \in Q^{+}$:
\begin{equation}\label{eq:PN_alpha}
  P(\alpha) = \sum_{p < t} |\sse_{i_{p}}||\sse_{i_{t}}| \qquad \textrm{and} \qquad
  N(\alpha) = \sum_{p < t} (\alpha_{i_{p}},\alpha_{i_{t}}) \,.
\end{equation}
As in Subsection~\ref{ssec:explicit-BCD}, we shall work with a specific ordering $1<2<\dots<s$ on the alphabet $I$.

\begin{Lem}\label{lem:chw-B-pairing}
Let $m$ be odd. Evoking the description of $\sL^+$ from~\eqref{eq:dL-B}, we have:
\begin{itemize}

\item
If $\ell = [i \ldots j]$ with $1 \leq i < j \leq s$, then
\begin{equation*}
  (R_{\ell},R_{\ell})^{\chw} = \prod_{k=i}^{j-1} (\alpha_{k},\alpha_{k+1}) \cdot (q-q^{-1})^{j-i} \cdot q^{N(\deg \ell)}\,.
\end{equation*}

\item
If $\ell = [i \ldots s s \ldots j]$ with $1 \leq i < j \leq s$, then
\begin{equation*}
  (R_{\ell},R_{\ell})^{\chw} =
  (-1)^{p([j \ldots s])} \cdot \prod_{k=i}^{j-2} (\alpha_{k},\alpha_{k+1}) \cdot (q-q^{-1})^{2s+1-i-j} \cdot q^{N(\deg \ell)}\,.
\end{equation*}

\end{itemize}
\end{Lem}

\begin{Lem}\label{lem:chw-CD-pairing-fork}
Let $m$ be even and $\ol{s} = \bar{0}$. Evoking the description of $\sL^+$ from~\eqref{eq:dL-D}, we have:
\begin{itemize}

\item
If $\ell = [i \ldots j]$ with $1 \leq i < j \leq s-1$, then
\begin{equation*}
  (R_{\ell},R_{\ell})^{\chw} = \prod_{k=i}^{j-1} (\alpha_{k},\alpha_{k+1}) \cdot (q-q^{-1})^{j-i} \cdot q^{N(\deg \ell)}\,.
\end{equation*}

\item
If $\ell = [i \ldots (s-2) s]$ with $1 \leq i \leq s-2$, then
\begin{equation*}
  (R_{\ell},R_{\ell})^{\chw} = \prod_{k=i}^{s-2} (\alpha_{k},\alpha_{k+1}) \cdot (q-q^{-1})^{s-i-1} \cdot q^{N(\deg \ell)}\,.
\end{equation*}

\item
If $\ell = [i \ldots (s-2) s (s-1) \ldots j]$ with $1 \leq i < j \leq s-1$, then
\begin{equation*}
  (R_{\ell},R_{\ell})^{\chw} = -\prod_{k=i}^{j-1} (\alpha_{k},\alpha_{k+1}) \cdot (q-q^{-1})^{2s-1-i-j} \cdot q^{N(\deg \ell)}\,.
\end{equation*}

\item
If $\ell = [i \ldots (s-1) i \ldots (s-2) s]$ with $1 \leq i \leq s-1$ and $p([i \ldots (s-1)]) = \bar{1}$, then
\begin{equation*}
  (R_{\ell},R_{\ell})^{\chw} = -(q-q^{-1})^{2s-2i-2} (q^{2}-q^{-2}) \cdot q^{N(\deg \ell)}\,.
\end{equation*}

\end{itemize}
\end{Lem}

\begin{Lem}\label{lem:chw-CD-pairing-nofork}
Let $m$ be even and $\ol{s} = \bar{1}$. Evoking the description of $\sL^+$ from~\eqref{eq:dL-C}, we have:
\begin{itemize}

\item
If $\ell = [i \ldots j]$ with $1 \leq i < j \leq s-1$, then
\begin{equation*}
  (R_{\ell},R_{\ell})^{\chw} = \prod_{k=i}^{j-1} (\alpha_{k},\alpha_{k+1}) \cdot (q-q^{-1})^{j-i} \cdot q^{N(\deg \ell)}\,.
\end{equation*}

\item
If $\ell = [i \ldots s]$ with $1 \leq i \leq s-1$, then
\begin{equation*}
  (R_{\ell},R_{\ell})^{\chw} =
  \frac{1}{2} \prod_{k=i}^{s-1} (\alpha_{k},\alpha_{k+1}) \cdot (q-q^{-1})^{s-i-1} (q^{2}-q^{-2}) \cdot q^{N(\deg \ell)}\,.
\end{equation*}

\item
If $\ell = [i \ldots s \ldots j]$ with $1 \leq i < j \leq s-1$, then
\begin{equation*}
  (R_{\ell},R_{\ell})^{\chw} =
  \prod_{k=i}^{j-1} (\alpha_{k},\alpha_{k+1}) \cdot (q-q^{-1})^{2s-1-i-j} (q^{2}-q^{-2}) \cdot q^{N(\deg \ell)}\,.
\end{equation*}

\item
If $\ell = [i \ldots (s-1) i \ldots (s-1)s]$ with $1 \leq i \leq s-1$ and $p([i \ldots (s-1)]) = \bar{0}$, then
\begin{equation*}
  (R_{\ell},R_{\ell})^{\chw} = (q-q^{-1})^{2s-2i-2} (q^{2}-q^{-2})^{2} \cdot q^{N(\deg \ell)}\,.
\end{equation*}

\end{itemize}
\end{Lem}

\begin{Rem}
We warn the reader that~\cite[\S6]{chw} contains various small errors in the constants featuring in their elements
$R_{\mathbf{i}}, E_{\mathbf{i}}, E^*_{\mathbf{i}}$ and respectively in the pairing $(E_{\mathbf{i}},E_{\mathbf{i}})$.
In particular:
\begin{itemize}

\item
the second bullet of Lemma~\ref{lem:chw-B-pairing} corrects a sign error in~\cite[Corollary 6.7(2)]{chw}
for the dominant Lyndon word $\mathbf{i} = (i,\ldots,M,M,\ldots,j+1)$,

\item
the last two bullets of Lemma~\ref{lem:chw-CD-pairing-fork} correct a sign error in~\cite[Corollary 6.14(2)]{chw}
for $\mathbf{i} = (i,\ldots,M-2,M,M-1,\ldots,j+1)$ and $\mathbf{i}=(i,\ldots,M-1,i,\ldots,M-2,M)$,

\item
the last three bullets of Lemma~\ref{lem:chw-CD-pairing-nofork} correct various errors in~\cite[Corollary 6.11(2)]{chw},
by adding the missing factors $q^2-q^{-2}$, or $(q-q^{-1})^{-1}$, or a power of $q$.

\end{itemize}
\end{Rem}


\subsection{Comparison of pairings}
\

In this Subsection, we establish the exact relation between the bialgebra pairing
\begin{equation*}
  (\cdot,\cdot)_J\colon U^\leq_{q}(\fosp(V)) \times U^\geq_{q}(\fosp(V)) \longrightarrow \BC(q^{\pm 1/4})
\end{equation*}
from Proposition~\ref{prop:pairing_finite} and the symmetric pairing
\begin{equation*}
  (\cdot,\cdot)^{\chw}\colon U^+_q(\fosp(V))\times U^+_q(\fosp(V)) \longrightarrow \BC(q)
\end{equation*}
from Proposition~\ref{prop:pairing_twisted}. To this end, we first define a new pairing
\begin{equation}\label{eq:intermediate-pairing}
  \{\cdot,\cdot\} \colon U^\geq_{q}(\fosp(V)) \times U^\geq_{q}(\fosp(V)) \longrightarrow \BC(q^{\pm 1/4})
  \quad \mathrm{via} \quad
  \{y,x\} = (-1)^{P(\deg(x))} (\omega(y),x)_{J} \,,
\end{equation}
cf.~\eqref{eq:PN_alpha}, where $\omega \colon \uqV \to \uqV$ is the $\BC(q)$-superalgebra automorphism mapping
\begin{equation}\label{eq:omega-map}
  e_{i} \mapsto (-1)^{|e_{i}|} f_{i} \,, \quad f_{i} \mapsto e_{i} \,, \quad q^{\pm h_{i}/2} \mapsto q^{\mp h_{i}/2}
  \qquad \forall \, i\in I \,.
\end{equation}
We note that
\begin{equation}\label{eq:omega-comult}
  \Delta^{J,\opp}(\omega(x)) = (\omega \otimes \omega)(\Delta^{J}(x)) \qquad \forall\, x \in \uqV \,.
\end{equation}
Combining Proposition~\ref{prop:pairing_finite} with~\eqref{eq:omega-comult}, one can easily check the following properties:
\begin{equation*}
\begin{split}
  & \{1,1\} = 1 \,,\ \{e_{i},e_{j}\} = \delta_{ij}/(q^{-1}-q) \,,\ \{q^{h_{i}/2},q^{h_{j}/2}\} = q^{a_{ij}/4} \,,\
    \{e_i,q^{h_{j}/2}\} = \{q^{h_{j}/2}, e_i\}=0 \,, \\
  & \{yy',x\} = \{y \otimes y',\Delta^J(x)\} \,,\qquad \{y, xx'\} = \{\Delta^{J}(y),x \otimes x'\}
\end{split}
\end{equation*}
for any $i,j \in I$ and $x,x',y,y' \in U^\geq_{q}(\fosp(V))$, where we set $\{y \otimes y', x \otimes x'\} = \{y,x\} \{y',x'\}$.

\begin{Prop}\label{prop:pairing-comparison-1}
For any $x,y \in U^+_q(\fosp(V))$ homogeneous with respect to the $Q \times \BZ_{2}$-grading, the following equality holds:
\begin{equation}\label{eq:pairing-comparison-1}
  \{y,x\} = \bar{\sigma}\Big( \big(\sigma(y),\sigma(x)\big)^{\chw} \Big) / (q^{-1} - q)^{\hgt(\deg(x))} \,,
\end{equation}
cf.~(\ref{eq:bar-sigma}, \ref{eq:invol-sigma}), where $\hgt(\cdot)$ is the height function defined via
$\hgt(k_1\alpha_1 + \dots + k_s\alpha_s)=k_1+\dots+k_s$.
\end{Prop}

\begin{proof}
We shall first evaluate explicitly both $(y,x)^{\chw}$ and $\{y,x\}$ for the case when
$x = e_{j_{1}} \cdots e_{j_{d'}}$ and $y = e_{i_{1}} \cdots e_{i_{d}}$ are monomials. For degree reasons,
we obviously have $(y,x)^{\chw} = 0 = \{y,x\}$ unless $d=d'$, hence we shall assume now that $d = d'$.
Direct computation then shows that
\begin{equation}\label{eq:chw-pairing-computation}
\begin{split}
  & (e_{i_{1}} \cdots e_{i_{d}}, e_{j_{1}} \cdots e_{j_{d}})^{\chw} \\
  &= \left(e_{i_{1}} \otimes \cdots \otimes e_{i_{d}},
     (\Delta^{\chw})^{(n-1)}(e_{j_{1}}) \cdots (\Delta^{\chw})^{(n-1)}(e_{j_{d}}) \right)^{\chw}\\
  &= \left( e_{i_{1}} \otimes \cdots \otimes e_{i_{d}},
     \sum_{\sigma \in S_{d}} \prod_{1 \leq k \leq d}^{\longrightarrow}
       \Big(1^{\otimes (\sigma(k)-1)} \otimes e_{j_{k}} \otimes 1^{\otimes (n-\sigma(k))}\Big) \right)^{\chw}\\
  &= \left( e_{i_{1}} \otimes \cdots \otimes e_{i_{d}},
      \sum_{\sigma \in S_{d}} \prod_{k<l}^{\sigma(k)>\sigma(l)}
      \left( (-1)^{|e_{j_{k}}||e_{j_{l}}|} q^{-(\alpha_{j_{k}},\alpha_{j_{l}})} \right)
      e_{j_{\sigma^{-1}(1)}} \otimes \cdots \otimes e_{j_{\sigma^{-1}(d)}} \right)^{\chw}\\
  &= \sum_{\sigma \in S_{d}} \prod_{k<l}^{\sigma(k)>\sigma(l)}
     \left( (-1)^{|e_{j_{k}}||e_{j_{l}}|} q^{-(\alpha_{j_{k}},\alpha_{j_{l}})} \right) \cdot
     (e_{i_{1}},e_{j_{\sigma^{-1}(1)}})^{\chw} \cdots (e_{i_{d}},e_{j_{\sigma^{-1}(d)}})^{\chw}\\
  &= \sum_{\sigma \in S_{d}} \prod_{k<l}^{\sigma(k)>\sigma(l)}
     \left( (-1)^{|e_{j_{k}}||e_{j_{l}}|} q^{-(\alpha_{j_{k}},\alpha_{j_{l}})} \right) \cdot
     \delta_{i_{1},j_{\sigma^{-1}(1)}} \cdots \delta_{i_{d},j_{\sigma^{-1}(d)}} \,,
\end{split}
\end{equation}
where $(\Delta^{\chw})^{(n-1)} \colon U_{q}^{+}(\fosp(V)) \to U_{q}^{+}(\fosp(V))^{\otimes n}$ is the map obtained by
applying coproduct $\Delta^{\chw}$ iteratively $n-1$ times, the definition of which is well-defined by coassociativity.
Also, the arrow $\rightarrow$ over the product sign implies that the multiplication is done in the increasing order of the index.

Analogously, we obtain
\begin{equation}\label{eq:new-pairing-computation}
\begin{split}
  & \{e_{i_{1}} \cdots e_{i_{d}}, e_{j_{1}} \cdots e_{j_{d}}\}\\
  &= \left\{ e_{i_{1}} \otimes \cdots \otimes e_{i_{d}}, (\Delta^{J})^{(n-1)}(e_{j_{1}}) \cdots (\Delta^{J})^{(n-1)}(e_{j_{d}}) \right\} \\
  &= \left\{ e_{i_{1}} \otimes \cdots \otimes e_{i_{d}},
       \sum_{\sigma \in S_{d}} \prod_{1 \leq k \leq d}^{\longrightarrow}
       \left( (q^{h_{j_{k}}})^{\otimes (\sigma(k)-1)} \otimes e_{j_{k}} \otimes 1^{\otimes (n-\sigma(k))}\right) \right\} \\
  &= \left\{ e_{i_{1}} \otimes \cdots \otimes e_{i_{d}},
     \sum_{\sigma \in S_{d}} \prod_{k<l}^{\sigma(k)>\sigma(l)}
     \left( (-1)^{|e_{j_{k}}||e_{j_{l}}|} q^{(\alpha_{j_{k}},\alpha_{j_{l}})} \right) \cdot
     \bigotimes_{1 \leq l \leq d} \Big( e_{j_{\sigma^{-1}(l)}} \prod_{\sigma(k)>l} q^{h_{j_{k}}} \Big) \right\}\\
  &= \sum_{\sigma \in S_{d}} \prod_{k<l}^{\sigma(k)>\sigma(l)}
      \left( (-1)^{|e_{j_{k}}||e_{j_{l}}|} q^{(\alpha_{j_{k}},\alpha_{j_{l}})} \right) \cdot
      \prod_{1 \leq l \leq d} \Big\{ e_{i_{l}},e_{j_{\sigma^{-1}(l)}}\prod_{\sigma(k)>l} q^{h_{j_{k}}} \Big\}\\
  &= \sum_{\sigma \in S_{d}} \prod_{k<l}^{\sigma(k)>\sigma(l)}
      \left( (-1)^{|e_{j_{k}}||e_{j_{l}}|} q^{(\alpha_{j_{k}},\alpha_{j_{l}})} \right) \cdot
      \delta_{i_{1},j_{\sigma^{-1}(1)}} \cdots \delta_{i_{d},j_{\sigma^{-1}(d)}} \cdot (q^{-1}-q)^{-d} \,,
\end{split}
\end{equation}
where the map $(\Delta^{J})^{(n-1)}\colon U_{q}^{\geq}(\fosp(V)) \to U_{q}^{\geq}(\fosp(V))^{\otimes n}$ is defined
similarly to $(\Delta^{\chw})^{(n-1)}$.

Comparing the above two formulas~\eqref{eq:chw-pairing-computation} and~\eqref{eq:new-pairing-computation}, we obtain
the validity of~\eqref{eq:pairing-comparison-1} in the case when both $x$ and $y$ are monomials. The generalization to
the $\BC(q)$-linear combinations of monomials is now a consequence of our definitions~\eqref{eq:bar-sigma}
and~\eqref{eq:invol-sigma}. This completes the proof.
\end{proof}

Combining Proposition~\ref{prop:pairing-comparison-1} and formula \eqref{eq:intermediate-pairing}, we thus obtain:

\begin{Cor}\label{cor:pairing-comparison-2}
For any $x \in U^{+}_{q}(\fosp(V))$ and $y \in U^{-}_{q}(\fosp(V))$ homogeneous with respect to the $Q \times \BZ_{2}$-grading,
the following equality holds:
\begin{equation*}
  (y,x)_{J} =
  (-1)^{P(\deg(x))} \cdot \bar{\sigma}\Big( \big(\sigma(\omega^{-1}(y)),\sigma(x)\big)^{\chw} \Big) / (q^{-1} - q)^{\hgt(\deg(x))} \,.
\end{equation*}
\end{Cor}


\subsection{Factorization formula}
\

For $\gamma\in \bar{\Phi}^+$, we define the (quantum) \emph{root vectors} $e_{\gamma},f_{\gamma}$ via
\begin{equation*}
  e_{\alpha_i}=e_i \,,\qquad f_{\alpha_i}=f_i \,,
\end{equation*}
while for $\gamma\in \bar{\Phi}^+\setminus \{\alpha_i\}_{i=1}^s$ we set
\begin{equation}\label{eq:root_vector}
\begin{split}
  & e_{\gamma} = e_{\alpha} e_{\beta} - (-1)^{|e_\alpha| |e_\beta|} q^{(\alpha,\beta)} e_\beta e_\alpha \,, \\
  & f_{\gamma} = f_{\beta} f_{\alpha} - (-1)^{|f_\alpha| |f_\beta|} q^{-(\alpha,\beta)} f_\alpha f_\beta \,,
\end{split}
\end{equation}
with the roots $\alpha,\beta\in \bar{\Phi}^+$ defined via $\alpha=\rl(\ell_1)$ and $\beta=\rl(\ell_2)$,
where $\ell=\ell_1\ell_2$ is the costandard factorization of the dominant Lyndon word $\ell=\rl^{-1}(\gamma)$,
see Proposition~\ref{prop:dominant-factorization}.

The explicit formulas for the pairing $(f_\gamma,e_\gamma)_J$ are derived in the following lemmas:

\begin{Lem}\label{lem:B-pairing}
For odd $m$, we have:
\begin{itemize}

\item
If $\gamma = \varepsilon_{i}-\varepsilon_{j}$ with $1 \leq i < j \leq s+1$, then
\begin{equation*}
  (f_{\gamma},e_{\gamma})_J = (-1)^{\ol{i}+\cdots+\ol{j}} \cdot (q^{-1}-q)^{-1}\,.
\end{equation*}

\item
If $\gamma = \varepsilon_{i}+\varepsilon_{j}$ with $1 \leq i < j \leq s$, then
\begin{equation*}
  (f_{\gamma},e_{\gamma})_J = (-1)^{\ol{i}+\cdots+\ol{j-1}} \cdot (q^{-1}-q)^{-1}\,.
\end{equation*}

\end{itemize}
\end{Lem}

\begin{Lem}\label{lem:CD-pairing-fork}
For even $m$ and $\ol{s} = \bar{0}$, we have:
\begin{itemize}

\item
If $\gamma = \varepsilon_{i} - \varepsilon_{j}$ with $1 \leq i < j \leq s$, then
\begin{equation*}
  (f_{\gamma},e_{\gamma})_J = (-1)^{\ol{i}+\cdots+\ol{j}} \cdot (q^{-1}-q)^{-1}\,.
\end{equation*}

\item
If $\gamma = \varepsilon_{i} + \varepsilon_{j}$ with $1 \leq i < j \leq s$, then
\begin{equation*}
  (f_{\gamma},e_{\gamma})_J = (-1)^{\ol{i}+\cdots+\ol{j-1}} \cdot (q^{-1}-q)^{-1}\,.
\end{equation*}

\item
If $\gamma = 2\varepsilon_{i}$ with $1 \leq i \leq s-1$ and $\ol{i} = \bar{1}$, then
\begin{equation*}
  (f_{\gamma},e_{\gamma})_J = \frac{q^{-2}-q^{2}}{(q^{-1}-q)^{2}}\,.
\end{equation*}

\end{itemize}
\end{Lem}

\begin{Lem}\label{lem:CD-pairing-nofork}
For even $m$ and $\ol{s} = \bar{1}$, we have:
\begin{itemize}

\item
If $\gamma = \varepsilon_{i}-\varepsilon_{j}$ with $1 \leq i < j \leq s$, then
\begin{equation*}
  (f_{\gamma},e_{\gamma})_J = (-1)^{\ol{i}+\cdots+\ol{j}} \cdot (q^{-1}-q)^{-1}\,.
\end{equation*}

\item
If $\gamma = \varepsilon_{i}+\varepsilon_{j}$ with $1 \leq i < j \leq s$, then
\begin{equation*}
  (f_{\gamma},e_{\gamma})_J = (-1)^{\ol{i}+\cdots+\ol{j-1}} \cdot \frac{q^{-2}-q^{2}}{(q^{-1}-q)^{2}}\,.
\end{equation*}

\item
If $\gamma = 2\varepsilon_{i}$ with $1 \leq i \leq s-1$ and $\ol{i} = \bar{1}$, then
\begin{equation*}
  (f_{\gamma},e_{\gamma})_J = \frac{(q^{-2}-q^{2})^{2}}{(q^{-1}-q)^{3}}\,.
\end{equation*}

\end{itemize}
\end{Lem}

\begin{proof}[Proof of Lemmas~\ref{lem:B-pairing}--\ref{lem:CD-pairing-nofork}]
According to~\eqref{eq:PN_alpha}, we have:
\begin{equation}\label{eq:PN_recursion}
  P(\alpha+\beta)=P(\alpha)+P(\beta)+|e_\alpha||e_\beta| \qquad \mathrm{and} \qquad
  N(\alpha+\beta)=N(\alpha)+N(\beta)+(e_\alpha, e_\beta) \,.
\end{equation}
Combining these equalities with formulas~(\ref{eq:omega-map},~\ref{eq:root_vector}), one easily verifies the formula
\begin{equation}\label{eq:omega_f}
  \omega^{-1}(f_{\gamma}) = (-1)^{\hgt(\gamma)-1} (-1)^{|e_{\gamma}|} (-1)^{P(\gamma)} q^{-N(\gamma)} e_{\gamma}
  \qquad \forall \, \gamma\in \bar{\Phi}^+
\end{equation}
by an induction on the height $\hgt(\gamma)$.
Combining this result with Corollary~\ref{cor:pairing-comparison-2}, we obtain:
\begin{equation}\label{eq:J-pairing-computation}
  (f_{\gamma},e_{\gamma})_{J} =
  (-1)^{\hgt(\gamma)-1} (-1)^{|e_{\gamma}|} q^{-N(\gamma)} \cdot
  \bar{\sigma}\Big( \big(\sigma(e_{\gamma}),\sigma(e_{\gamma})\big)^{\chw} \Big) / (q^{-1} - q)^{\hgt(\gamma)} \,.
\end{equation}

To evaluate the pairing $(\sigma(e_{\gamma}),\sigma(e_{\gamma}))^{\chw}$, we recall the $\BC(q)$-linear endomorphism
$\mathcal{T}$ of $U^{+}_{q}(\fosp(V))$ from \cite[Proposition~2.2(1)]{chw} defined by
\begin{equation*}
  \mathcal{T}(e_{i}) = e_{i} \quad \forall\, i \in I \qquad \mathrm{and} \qquad
  \mathcal{T}(xy) = \mathcal{T}(y)\mathcal{T}(x) \quad \forall\, x,y \in U^{+}_{q}(\fosp(V)) \,.
\end{equation*}
Arguing by an induction on $\hgt(\gamma)$ again, let us now prove the following formula:
\begin{equation}\label{eq:sigma-T}
  \sigma(e_{\gamma}) = (-1)^{\hgt(\gamma)-1} (-1)^{P(\gamma)} q^{-N(\gamma)} \cdot \mathcal{T}(e_{\gamma})
  \qquad \mathrm{for\ any} \quad \gamma \in \bar{\Phi}^{+} \,.
\end{equation}
This equality is clear when $\hgt(\gamma) = 1$. For any root $\gamma$ with $\hgt(\gamma) > 1$, we consider
the pair of roots $\alpha,\beta \in \bar{\Phi}^{+}$ satisfying $e_{\gamma} = [\![ e_{\alpha},e_{\beta} ]\!]$,
cf.~\eqref{eq:q-superbracket} and~\eqref{eq:root_vector}. Since $\hgt(\alpha), \hgt(\beta) < \hgt(\gamma)$, we may
assume by the induction hypothesis that~\eqref{eq:sigma-T} holds for the roots $\alpha$ and $\beta$, so that:
\begin{align*}
  \sigma(e_{\gamma})
  &= \sigma(e_{\alpha})\sigma(e_{\beta}) - (-1)^{|e_{\alpha}||e_{\beta}|} q^{-(\alpha,\beta)} \cdot \sigma(e_{\beta})\sigma(e_{\alpha}) \\
  &= (-1)^{\hgt(\gamma)-2}(-1)^{P(\alpha)+P(\beta)} q^{-N(\alpha)-N(\beta)} \cdot
     \left( \mathcal{T}(e_{\beta}e_{\alpha}) -
       (-1)^{|e_{\alpha}||e_{\beta}|} q^{-(\alpha,\beta)} \cdot \mathcal{T}(e_{\alpha}e_{\beta}) \right) \\
  &\overset{\eqref{eq:PN_recursion}}{=} (-1)^{\hgt(\gamma)-1}(-1)^{P(\gamma)} q^{-N(\gamma)} \cdot \mathcal{T}(e_{\gamma}).
\end{align*}
This proves the induction step, hence completes the proof of~\eqref{eq:sigma-T}.

Furthermore, the direct formula~\eqref{eq:chw-pairing-computation} shows that
\begin{equation}\label{eq:chw-pairing-T-invar}
  (\mathcal{T}(x),\mathcal{T}(y))^{\chw} = (x,y)^{\chw}
\end{equation}
for any monomials $x = e_{i_{1}} \cdots e_{i_{d}}, y = e_{j_{1}} \cdots e_{j_{d'}}$, and hence
\eqref{eq:chw-pairing-T-invar} holds for any $x,y \in U^{+}_{q}(\fosp(V))$, as $\mathcal{T}$ is $\BC(q)$-linear.
Combining~\eqref{eq:J-pairing-computation}--\eqref{eq:chw-pairing-T-invar}, we finally obtain:
\begin{equation}\label{eq:J-vs-CHW-roots}
  (f_{\gamma},e_{\gamma})_{J} =
  (-1)^{\hgt(\gamma)-1} (-1)^{|e_{\gamma}|} q^{N(\gamma)} \cdot
  \bar{\sigma}\Big( (e_{\gamma},e_{\gamma})^{\chw} \Big) / (q^{-1}-q)^{\hgt(\gamma)} \,.
\end{equation}
In view of this equality, Lemmas~\ref{lem:B-pairing}--\ref{lem:CD-pairing-nofork} are just direct consequences of
Lemmas~\ref{lem:chw-B-pairing}--\ref{lem:chw-CD-pairing-nofork}.
\end{proof}

We are now ready to construct dual bases of $U^\pm_q(\fosp(V))$ with respect to the bialgebra pairing~\eqref{eq:Hopf-parity}
(which relies on the orthogonality result of Proposition~\ref{prop:CHW-main-theorem}, proved in~\cite[Theorem~5.7]{chw}):

\begin{Thm}\label{thm:PBW-general}
(a) The ordered products
\begin{equation*}
  \left\{ \overset{\longleftarrow}{\underset{\gamma \in \bar{\Phi}^{+}}{\prod}} e_{\gamma}^{m_{\gamma}} \,\Big|\,
  \substack{m_{\gamma} \ge 0\\ m_\gamma\leq 1\ \mathrm{if}\ \gamma\in \bar{\Phi}_{\bar{1}}\ \mathrm{is\ isotropic}} \right \}
  \quad \mathrm{and} \quad
  \left\{ \overset{\longleftarrow}{\underset{\gamma \in \bar{\Phi}^{+}}{\prod}} f_{\gamma}^{m_{\gamma}} \,\Big|\,
  \substack{m_{\gamma} \ge 0\\ m_\gamma\leq 1\ \mathrm{if}\ \gamma\in \bar{\Phi}_{\bar{1}}\ \mathrm{is\ isotropic}} \right \}
\end{equation*}
are bases for $U^{+}_q(\fosp(V))$ and $U^{-}_q(\fosp(V))$, respectively. Henceforth, the arrow $\leftarrow$ over
the product signs refers to the total order~\eqref{eq:lex-order} on $\bar{\Phi}^+$, thus ordering the elements of
$\bar{\Phi}^+$ in decreasing order.

\medskip
\noindent
(b) The bialgebra pairing~\eqref{eq:Hopf-parity} is orthogonal with respect to these bases. More explicitly, we have:
\begin{equation}\label{eq:orthogonal}
    \left( \overset{\longleftarrow}{\underset{\gamma \in \bar{\Phi}^{+}}{\prod}} f_{\gamma}^{n_{\gamma}},
           \overset{\longleftarrow}{\underset{\gamma \in \bar{\Phi}^{+}}{\prod}} e_{\gamma}^{m_{\gamma}} \right)_J =
    (-1)^{\sum_{\gamma < \gamma'} m_{\gamma}m_{\gamma'} |e_{\gamma}||e_{\gamma'}|} \cdot
    \prod_{\gamma\in \bar{\Phi}^+} \Big(\delta_{n_\gamma,m_\gamma} (f_{\gamma}^{m_\gamma},e^{m_\gamma}_\gamma)_J\Big)
\end{equation}
and
\begin{equation}\label{eq:pairings-power}
  (f_{\gamma}^{k},e_{\gamma}^k)_J =
  (-1)^{\frac{k(k-1)}{2} |e_{\gamma}|} \cdot \bar{\sigma}(\sfC_{\gamma,k}) \cdot (f_\gamma,e_\gamma)_J^k\,,
\end{equation}
where
\begin{equation*}
  \sfC_{\gamma,k} = C_{\rl^{-1}(\gamma),k} =
  \prod_{t=1}^{k} \frac{1 - \big((-1)^{|e_{\gamma}|} q^{-(\gamma,\gamma)}\big)^{t}}{1 - (-1)^{|e_{\gamma}|} q^{-(\gamma,\gamma)}} \,,
\end{equation*}
cf.~Proposition \ref{prop:CHW-main-theorem}.
\end{Thm}

\begin{Rem}
This result is known in classical $BCD$-types where it follows from Lusztig's orthogonal bases
(see~\cite[\S8.30]{jan}) associated with the reduced decomposition of the longest element $w_0\in W$ that matches
(see~\cite{p}) the lexicographical convex order~\eqref{eq:lex-order} on the set of positive roots $\Phi^+$. In this context,
Lusztig's root vectors (defined via the braid group action) match the above $q$-commutator construction of~\eqref{eq:root_vector},
up to constants computed explicitly in~\cite[Theorem~4.2]{bkm}.
\end{Rem}

\begin{proof}
(a) First, we note that $e_{\gamma} = \Psi^{-1}(R_{\rl^{-1}(\gamma)})$ for all $\gamma \in \bar{\Phi}^{+}$.
Therefore, the preimage of the Lyndon basis~\eqref{eq:Lyndon-basis} of $\sU$ under the isomorphism
$\Psi\colon U^+_q(\fosp(V)) \iso \sU$ provides (up to rescaling) the claimed basis of $U_{q}^{+}(\fosp(V))$.
Evoking~\eqref{eq:omega_f}, the result for $U_{q}^{-}(\fosp(V))$ can be carried out from that for
$U_{q}^{+}(\fosp(V))$ through the isomorphism $\omega\colon U_{q}^{-}(\fosp(V)) \to U_{q}^{+}(\fosp(V))$
of~\eqref{eq:omega-map}.

(b) Let us first compute $(f_{\gamma}^{k},e_{\gamma}^k)_J$.
Following the above proof of Lemmas~\ref{lem:B-pairing}--\ref{lem:CD-pairing-nofork}, we obtain:
\begin{multline}\label{eq:pairings-power-relation}
  (f_\gamma^k,e_\gamma^k)_J
  = (-1)^{P(k\gamma)} \cdot \left( (-1)^{\hgt(\gamma)-1} (-1)^{|e_{\gamma}|}
    (-1)^{P(\gamma)} q^{-N(\gamma)} (q^{-1}-q)^{-\hgt(\gamma)} \right)^k
    \cdot \bar{\sigma}\left( (\sigma(e_{\gamma})^{k},\sigma(e_{\gamma})^{k})^{\chw} \right) \\
  = (-1)^{P(k\gamma)} \cdot \left( (-1)^{\hgt(\gamma)-1} (-1)^{|e_{\gamma}|}
    (-1)^{P(\gamma)} q^{N(\gamma)} (q^{-1}-q)^{-\hgt(\gamma)} \right)^k
    \cdot \bar{\sigma}\left( (e_{\gamma}^{k},e_{\gamma}^{k})^{\chw} \right)\,.
\end{multline}
But evoking~\eqref{eq:Rw-pairing}, we note that
\begin{equation}\label{eq:e-power-J}
  (e^k_\gamma,e^k_\gamma)^{\chw} = (R^k_{\rl^{-1}(\gamma)},R^k_{\rl^{-1}(\gamma)})^{\chw} =
  C_{\rl^{-1}(\gamma),k} \cdot \big( (R_{\rl^{-1}(\gamma)},R_{\rl^{-1}(\gamma)})^{\chw} \big)^k =
  \sfC_{\gamma,k} \cdot \big( (e_{\gamma},e_{\gamma})^{\chw} \big)^k.
\end{equation}
Combining the above two formulas with~\eqref{eq:J-vs-CHW-roots}, we get:
\begin{equation*}
  (f_\gamma^k,e_\gamma^k)_J =
  (-1)^{P(k\gamma)}(-1)^{k \cdot P(\gamma)} \cdot \bar{\sigma}(\sfC_{\gamma,k}) \cdot (f_\gamma,e_\gamma)^k_J =
  (-1)^{\frac{k(k-1)}{2} |e_{\gamma}|} \cdot \bar{\sigma}(\sfC_{\gamma,k}) \cdot (f_\gamma,e_\gamma)^k_J \,,
\end{equation*}
which establishes~\eqref{eq:pairings-power}.

The proof of~\eqref{eq:orthogonal} is analogous. To this end, we first note:
\begin{align*}
  \left( \overset{\longleftarrow}{\underset{\gamma \in \bar{\Phi}^{+}}{\prod}} f_{\gamma}^{n_{\gamma}},
  \overset{\longleftarrow}{\underset{\gamma \in \bar{\Phi}^{+}}{\prod}} e_{\gamma}^{m_{\gamma}} \right)_J
  &= (-1)^{P\big(\sum_{\gamma} m_{\gamma}\gamma\big)} \cdot \left( \prod_{\gamma \in \bar{\Phi}^{+}}
          \left( (-1)^{\hgt(\gamma)-1} (-1)^{|e_{\gamma}|} (-1)^{P(\gamma)} q^{-N(\gamma)} (q^{-1}-q)^{-\hgt(\gamma)} \right)^{n_{\gamma}} \right)\\
  &\qquad \cdot \bar{\sigma} \left( \left(
      \sigma\left( \overset{\longleftarrow}{\underset{\gamma \in \bar{\Phi}^{+}}{\prod}} e_{\gamma}^{n_{\gamma}} \right),
       \sigma\left( \overset{\longleftarrow}{\underset{\gamma \in \bar{\Phi}^{+}}{\prod}} e_{\gamma}^{m_{\gamma}} \right)
   \right)^{\chw} \right)\,.
\end{align*}
Moreover, according to Proposition \ref{prop:CHW-main-theorem}, we have:
\begin{equation*}
  \left( \overset{\longrightarrow}{\underset{\gamma \in \bar{\Phi}^{+}}{\prod}} e_{\gamma}^{n_{\gamma}},
          \overset{\longrightarrow}{\underset{\gamma \in \bar{\Phi}^{+}}{\prod}} e_{\gamma}^{m_{\gamma}}\right)^{\chw} =
  \prod_{\gamma \in \bar{\Phi}^{+}} \left( \delta_{n_{\gamma},m_{\gamma}} \cdot \sfC_{\gamma,m_{\gamma}} \cdot
      ((e_{\gamma} , e_{\gamma})^{\chw})^{m_{\gamma}} \right)
\end{equation*}
(we note that the products in the left hand side are taken in increasing order!), so that:
\begin{align*}
  & \left( \sigma\left(\overset{\longleftarrow}{\underset{\gamma \in \bar{\Phi}^{+}}{\prod}} e_{\gamma}^{n_{\gamma}}\right),
           \sigma\left(\overset{\longleftarrow}{\underset{\gamma \in \bar{\Phi}^{+}}{\prod}} e_{\gamma}^{m_{\gamma}}\right)
    \right)^{\chw}
    \overset{\eqref{eq:chw-pairing-T-invar}}{=}
    \left( \mathcal{T}\sigma\left(\overset{\longleftarrow}{\underset{\gamma \in \bar{\Phi}^{+}}{\prod}} e_{\gamma}^{n_{\gamma}}\right),
            \mathcal{T}\sigma\left(\overset{\longleftarrow}{\underset{\gamma \in \bar{\Phi}^{+}}{\prod}} e_{\gamma}^{m_{\gamma}}\right)
    \right)^{\chw}\\
  &\qquad\qquad\qquad =
   \left(
     \overset{\longrightarrow}{\underset{\gamma \in \bar{\Phi}^{+}}{\prod}} (\mathcal{T}\sigma(e_{\gamma}))^{n_{\gamma}},
     \overset{\longrightarrow}{\underset{\gamma \in \bar{\Phi}^{+}}{\prod}} (\mathcal{T}\sigma(e_{\gamma}))^{m_{\gamma}}
   \right)^{\chw}\\
  &\qquad\qquad\qquad\overset{\eqref{eq:sigma-T}}{=}
   \left(
     \overset{\longrightarrow}{\underset{\gamma \in \bar{\Phi}^{+}}{\prod}} e_{\gamma}^{n_{\gamma}},
     \overset{\longrightarrow}{\underset{\gamma \in \bar{\Phi}^{+}}{\prod}} e_{\gamma}^{m_{\gamma}}
   \right)^{\chw} \cdot \prod_{\gamma \in \bar{\Phi}^{+}}
   \left( (-1)^{\hgt(\gamma)-1} (-1)^{P(\gamma)} q^{-N(\gamma)} \right)^{n_{\gamma}+m_{\gamma}} \\
  &\qquad\qquad\qquad =
   \prod_{\gamma \in \bar{\Phi}^{+}}
   \left(
     \delta_{n_{\gamma},m_{\gamma}} \cdot \sfC_{\gamma,m_{\gamma}} \cdot q^{-2 m_{\gamma}N(\gamma)}
     \cdot ((e_{\gamma} , e_{\gamma})^{\chw})^{m_{\gamma}}
   \right)\\
  &\qquad\qquad\qquad\overset{\eqref{eq:e-power-J}}{=}
    \prod_{\gamma \in \bar{\Phi}^{+}} \left(
      \delta_{n_{\gamma},m_{\gamma}} \cdot q^{-2 m_{\gamma}N(\gamma)} \cdot (e_{\gamma}^{m_{\gamma}} , e_{\gamma}^{m_{\gamma}})^{\chw}
    \right).
\end{align*}
Combining the formulas above with \eqref{eq:pairings-power-relation}, we get:
\begin{align*}
  \left( \overset{\longleftarrow}{\underset{\gamma \in \bar{\Phi}^{+}}{\prod}} f_{\gamma}^{n_{\gamma}},
         \overset{\longleftarrow}{\underset{\gamma \in \bar{\Phi}^{+}}{\prod}} e_{\gamma}^{m_{\gamma}} \right)_J
  &= (-1)^{P\big(\sum_{\gamma} m_{\gamma}\gamma\big)} \cdot \prod_{\gamma \in \bar{\Phi}^{+}}
   \left( \delta_{n_{\gamma},m_{\gamma}} \cdot (-1)^{P(m_{\gamma}\gamma)} \cdot (f_{\gamma}^{m_{\gamma}} , e_{\gamma}^{m_{\gamma}})_{J} \right)\\
  &= (-1)^{\sum_{\gamma < \gamma'} m_{\gamma}m_{\gamma'} |e_{\gamma}||e_{\gamma'}|} \cdot \prod_{\gamma \in \bar{\Phi}^{+}}
   \left( \delta_{n_{\gamma},m_{\gamma}} \cdot (f_{\gamma}^{m_{\gamma}} , e_{\gamma}^{m_{\gamma}})_{J} \right).
\end{align*}
This completes the proof of formula~\eqref{eq:orthogonal}, and hence also of the theorem.
\end{proof}

As an immediate corollary, we obtain the following factorization formula:

\begin{Thm}\label{cor:Theta-factorization}
The operator $\Theta$ of~\eqref{eq:Theta} can be factorized as follows:
\begin{equation}\label{eq:Theta-factorized}
  \Theta = \overset{\longleftarrow}{\underset{\gamma \in \bar{\Phi}^{+}}{\prod}}
  \left( \sum_{k\geq 0} \frac{e^k_\gamma\otimes f^k_\gamma}{(f^k_{\gamma},e^k_{\gamma})_J} \right).
\end{equation}
\end{Thm}

We note that $f^k_\gamma=e^k_\gamma=0$ if $\gamma\in \bar{\Phi}^+_{\bar{1}}$ is isotropic and $k\geq 2$,
according to~\cite[Corollary 5.2]{chw}.

\begin{Rem}\label{rem:q-exp}
One can express $\Theta$ in an even more compact form. Recall the notion of a \emph{$q$-exponent}:
\begin{equation*}
  \exp_{\sq}(z) = \sum_{k \ge 0} \frac{z^{k}}{(k)_{\sq}!} \,,
\end{equation*}
where $(k)_{\sq}!=(k)_{\sq} \dots (1)_{\sq}$ with $(k)_{\sq}=\frac{1-\sq^{k}}{1-\sq}$. We thus have:
\begin{equation*}
  \sum_{k \ge 0} \frac{e_{\gamma}^{k} \otimes f_{\gamma}^{k}}{(f_{\gamma}^{k},e_{\gamma}^{k})_J} =
  \sum_{k \ge 0} \frac{(-1)^{\frac{k(k - 1)}{2}|e_\gamma|} \cdot (e_{\gamma} \otimes f_{\gamma})^{k}}
    { (-1)^{\frac{k(k-1)}{2} |e_{\gamma}|} \cdot \bar{\sigma}(\sfC_{\gamma,k}) \cdot (f_\gamma,e_\gamma)_J^k} =
  \exp_{q_\gamma} \left( \frac{e_{\gamma} \otimes f_{\gamma}}{(f_\gamma,e_\gamma)_J} \right),
\end{equation*}
where $q_\gamma=(-1)^{|e_\gamma|}q^{(\gamma,\gamma)}$. Therefore, the factorization
formula~\eqref{eq:Theta-factorized} simplifies to
\begin{equation*}\label{eq:Theta-factorized-exp}
  \Theta = \overset{\longleftarrow}{\underset{\gamma \in \bar{\Phi}^{+}}{\prod}}
           \exp_{q_\gamma}\left( \frac{e_{\gamma} \otimes f_{\gamma}}{(f_\gamma,e_\gamma)_J} \right).
\end{equation*}
\end{Rem}


\subsection{R-Matrix computation}\label{ssec:R-Matrix-Computation}
\

We shall now use the factorization formula~\eqref{eq:Theta-factorized} to compute the action of $\Theta$ on
$V\otimes V$ and then re-derive $R_{VV}$. Throughout this Subsection, we will use the more convenient notation
\begin{equation}\label{eq:local-Theta}
  \Theta_{\gamma} = \sum_{k \ge 0} \frac{e_{\gamma}^{k} \otimes f_{\gamma}^{k}}{(f_{\gamma}^{k},e_{\gamma}^{k})_J}
  \qquad \mathrm{for\ any} \quad \gamma \in \bar{\Phi}^{+},
\end{equation}
so that equation \eqref{eq:Theta-factorized} becomes
\begin{equation*}
  \Theta = \overset{\longleftarrow}{\underset{\gamma \in \bar{\Phi}^{+}}{\prod}}\Theta_{\gamma} \,.
\end{equation*}
For any $1 \leq i,j \leq N$, we also define the following $q$-deformation of the elements~\eqref{eq:X-elements} of $\fosp(V)$:
\begin{align*}
  \sse_{ij} &= E_{ij} -
     (-1)^{\ol{i}(\ol{i}+\ol{j})} q^{-(\rho,\varepsilon_{i}-\varepsilon_{j})}q^{(\varepsilon_{i},\varepsilon_{i})/2}
     q^{(\varepsilon_{j},\varepsilon_{j})/2} \vartheta_{i}\vartheta_{j} E_{j'i'}\,,\\
  \ssf_{ij} &= E_{ji} -
     (-1)^{\ol{j}(\ol{i}+\ol{j})} q^{(\rho,\varepsilon_{i}-\varepsilon_{j})}q^{-(\varepsilon_{i},\varepsilon_{i})/2}
     q^{-(\varepsilon_{j},\varepsilon_{j})/2} \vartheta_{i}\vartheta_{j} E_{i'j'} \,.
\end{align*}


\subsubsection{Factorized formula for odd $m$}
\

We start by evaluating the action of the root vectors $\{e_{\gamma}, f_{\gamma}\}_{\gamma \in \bar{\Phi}^{+}}$
of~\eqref{eq:root_vector} on the $\uqV$-module $V$ from Proposition~\ref{prop:fin-repn}:

\begin{Lem}\label{lem:B-root-vectors-action}
(a) For $\gamma = \varepsilon_{i}-\varepsilon_{j}$ with $1 \leq i < j \leq s+1$, we have:
\begin{equation*}
  \varrho(e_{\gamma}) = \sse_{ij} \,,\qquad
  \varrho(f_{\gamma}) = (-1)^{\ol{i}+\cdots+\ol{j-1}} \cdot \ssf_{ij}\,.
\end{equation*}

\noindent
(b) For $\gamma = \varepsilon_{i}+\varepsilon_{j}$ with $1 \leq i < j \leq s$, we have:
\begin{align*}
  \varrho(e_{\gamma}) &=
    \vartheta_{j}\vartheta_{s+1} \cdot \prod_{k=j}^{s} \Big(-(-1)^{\ol{k}(\ol{k}+\ol{k+1})}\Big) \cdot \sse_{ij'}\,,\\
  \varrho(f_{\gamma}) &= (-1)^{\ol{i}+\cdots+\ol{j-1}} \vartheta_{j}\vartheta_{s+1} \cdot
    \prod_{k=j}^{s} \Big(-(-1)^{\ol{k+1}(\ol{k}+\ol{k+1})}\Big) \cdot \ssf_{ij'}\,.
\end{align*}
\end{Lem}

\begin{proof}
The proof is done by a straightforward induction. To this end, we proceed by an increasing induction on $j$ for
$\gamma = \varepsilon_{i} - \varepsilon_{j}$ with $1 \leq i < j \leq s+1$, and then by a decreasing induction on $j$
for $\gamma = \varepsilon_{i}+\varepsilon_{j}$ with $1 \leq i < j \leq s$.
\end{proof}

Combining the result above with Lemma~\ref{lem:B-pairing} and formula~\eqref{eq:pairings-power}, we obtain:

\begin{Lem}\label{lem:B-local-operators}
The operators $\Theta_{\gamma}$ of~\eqref{eq:local-Theta} act on the $\uqV$-module $V \otimes V$ as follows:
\begin{itemize}

\item
If $\gamma = \varepsilon_{i} - \varepsilon_{j}$ with $i < j < i'$ and $j \neq s+1$, then
\begin{equation*}
  \Theta_{\gamma} = \ID - (-1)^{\ol{j}}(q-q^{-1}) \cdot \sse_{ij} \otimes \ssf_{ij}\,.
\end{equation*}

\item
If $\gamma = \varepsilon_{i}$ with $1 \leq i \leq s$, then
\begin{equation*}
  \Theta_{\gamma} = \ID - (q-q^{-1}) \cdot \sse_{i,s+1} \otimes \ssf_{i,s+1} +
  (q-q^{-1}) \Big(1 - (-1)^{\ol{i}}q^{-(\varepsilon_{i},\varepsilon_{i})}\Big) \cdot E_{ii'} \otimes E_{i'i}\,.
\end{equation*}

\end{itemize}
\end{Lem}

Evoking the explicit order~\eqref{eq:B_order} on $\bar{\Phi}^+$, one immediately obtains the factorization formula
\begin{equation}\label{eq:Theta-factorization-B}
  \Theta = \Theta_{s}\Theta_{s-1} \cdots \Theta_{1}
\end{equation}
with
\begin{equation*}
  \Theta_{i} = \Theta_{\varepsilon_{i}+\varepsilon_{i+1}} \cdots \Theta_{\varepsilon_{i}+\varepsilon_{s}}
  \Theta_{\varepsilon_{i}} \Theta_{\varepsilon_{i}-\varepsilon_{s}} \cdots \Theta_{\varepsilon_{i}-\varepsilon_{i+1}}
  \qquad \mathrm{for\ any} \quad 1 \leq i \leq s \,.
\end{equation*}
Therefore, it is essential to evaluate each such factor, which is the subject of the next result:

\begin{Lem}\label{lem:B-type-Theta_i}
For $1 \leq i \leq s$, we have:
\begin{equation*}
  \Theta_{i} = \ID - (q-q^{-1}) \sum_{i<j<i'} (-1)^{\ol{j}} \sse_{ij} \otimes \ssf_{ij} +
  (q-q^{-1})q^{-(\varepsilon_{i},\varepsilon_{i})}
    \Big(q^{(2\rho,\varepsilon_{i})}-(-1)^{\ol{i}}\Big) \cdot E_{ii'} \otimes E_{i'i} \,.
\end{equation*}
\end{Lem}

\begin{proof}
First we note that $\Theta_{\varepsilon_{i}+\varepsilon_{j}}$ commutes with $\Theta_{\varepsilon_{i}+\varepsilon_{k}}$
for $k \neq j,j'$, so that
\begin{equation*}
  \Theta_{i} =
  \Theta_{\varepsilon_{i}+\varepsilon_{i+1}} \cdots \Theta_{\varepsilon_{i}+\varepsilon_{s}}
  \Theta_{\varepsilon_{i}}\Theta_{\varepsilon_{i}-\varepsilon_{s}} \cdots \Theta_{\varepsilon_{i}-\varepsilon_{i+1}} =
  \Theta_{\varepsilon_{i}} \prod_{j=i+1}^{s}
    \big(\Theta_{\varepsilon_{i}+\varepsilon_{j}}\Theta_{\varepsilon_{i}-\varepsilon_{j}}\big)\,.
\end{equation*}
According to Lemma~\ref{lem:B-local-operators}, for $i<j\leq s$ we have:
\begin{equation*}
  \Theta_{\varepsilon_{i}+\varepsilon_{j}}\Theta_{\varepsilon_{i}-\varepsilon_{j}} =
  \ID - (-1)^{\ol{j}}(q-q^{-1}) \cdot (\sse_{ij} \otimes \ssf_{ij} + \sse_{ij'} \otimes \ssf_{ij'}) +
  (-1)^{\ol{j}} (q-q^{-1})^{2} q^{(\rho,2\varepsilon_{j})} \cdot E_{ii'} \otimes E_{i'i}\,.
\end{equation*}
Combining the two formulas above, we obtain:
\begin{multline}\label{eq:iTheta-B}
  \Theta_{i} = \ID - (q-q^{-1}) \sum_{i<j<i'} (-1)^{\ol{j}} \sse_{ij} \otimes \ssf_{ij} \\
  + (q - q^{-1})\Big(1 - (-1)^{\ol{i}}q^{-(\varepsilon_{i},\varepsilon_{i})}\Big) E_{ii'} \otimes E_{i'i}
  + (q - q^{-1})^{2}\sum_{j = i+1}^{s}(-1)^{\ol{j}}q^{(\rho,2\varepsilon_{j})} E_{ii'} \otimes E_{i'i} \,.
\end{multline}
The last sum can be simplified using~\eqref{eq:rho-simple-root} as follows:
\begin{multline}\label{eq:local-telescoping}
  (q - q^{-1}) \sum_{j=i+1}^{s} (-1)^{\ol{j}} q^{(\rho,2\varepsilon_{j})}
  = \sum_{j=i+1}^{s} (q^{(\varepsilon_{j},\varepsilon_{j})} - q^{-(\varepsilon_{j},\varepsilon_{j})})q^{(2\rho,\varepsilon_{j})} \\
  = \sum_{j=i+1}^{s} \Big(q^{-(\varepsilon_{j-1},\varepsilon_{j-1})}q^{(2\rho,\varepsilon_{j-1})} -
      q^{-(\varepsilon_{j},\varepsilon_{j})} q^{(2\rho,\varepsilon_{j})}\Big)
  = q^{-(\varepsilon_{i},\varepsilon_{i})}q^{(2\rho,\varepsilon_{i})} - 1 \,.
\end{multline}
Combining formulas~(\ref{eq:iTheta-B},~\ref{eq:local-telescoping}), we get
\begin{equation*}
  \Theta_{i} = \ID - (q-q^{-1}) \sum_{i<j<i'} (-1)^{\ol{j}} \sse_{ij} \otimes \ssf_{ij} +
  (q-q^{-1})q^{-(\varepsilon_{i},\varepsilon_{i})}
    \Big(q^{(2\rho,\varepsilon_{i})}-(-1)^{\ol{i}}\Big) \cdot E_{ii'} \otimes E_{i'i} \,,
\end{equation*}
which completes the proof.
\end{proof}

The result above together with the factorization~\eqref{eq:Theta-factorization-B} finally allows us to evaluate $\Theta$:

\begin{Prop}\label{prop:B-type-Theta}
The action of the operator $\Theta$ on the $\uqV$-module $V \otimes V$ is given by
\begin{equation}\label{eq:B-type-Theta}
  \Theta = \ID -
  (q - q^{-1}) \sum_{1\leq i < j\leq 1'} (-1)^{\ol{j}} E_{ij} \otimes \Big( q^{(\varepsilon_{i},\varepsilon_{j})} E_{ji} -
    (-1)^{\ol{j}(\ol{i}+\ol{j})} \vartheta_{i}\vartheta_{j} q^{(\rho, \varepsilon_{i} - \varepsilon_{j})}
    q^{-(\varepsilon_{i},\varepsilon_{i})/2}q^{-(\varepsilon_{j},\varepsilon_{j})/2} E_{i'j'} \Big) \,.
\end{equation}
\end{Prop}

\begin{proof}
We shall prove formula~\eqref{eq:B-type-Theta} by induction on $s$. The base case $s=1$ follows from the direct evaluation
of $\Theta = \Theta_{1}$ in Lemma \ref{lem:B-type-Theta_i} and~\eqref{eq:rho-simple-root}. As per the induction step, let us
consider the subspace $V^\circ$ of $V$ spanned by the vectors $\{v_{i}\}_{2 \leq i \leq 2'}$; we shall likewise use the symbol
${}^\circ$ to denote any object corresponding to $V^\circ$ instead of $V$. Then, we have an algebra homomorphism
$\iota \colon U_{q}(\fosp(V^\circ)) \to \uqV$ mapping each generator $\{e_{i}, f_{i}, q^{\pm h_{i}/2}\}_{i=2}^{s}$ in
$U_{q}(\fosp(V^\circ))$ to the same-named generator in $\uqV$. Furthermore, let
$\varrho^\circ \colon U_{q}(\fosp(V^\circ)) \to \End(V^\circ)$ and $\varrho \colon \uqV \to \End(V)$ be the representations
of the corresponding algebras, as defined in Proposition \ref{prop:fin-repn}. It is clear that the representation
$\varrho \circ \iota$ of $U_{q}(\fosp(V^\circ))$ on $V$ preserves $V^\circ$, and its restriction onto $V^\circ$ coincides
with $\varrho^\circ$. Moreso, the generators $\{e_{i}, f_{i}\}_{i=2}^{s}$ act trivially on $v_{1}$ and $v_{1'}$.
Likewise, the bialgebra pairings from Proposition~\ref{prop:pairing_finite} are related via
\begin{equation*}
  (y,x)^\circ_{J} = (\iota(y),\iota(x))_{J} \qquad \mathrm{for\ any} \quad
  x \in U_{q}^{\geq}(\fosp(V^\circ)),\, y \in U_{q}^{\leq}(\fosp(V^\circ))\,,
\end{equation*}
which follows from the defining properties. Therefore, combining the observations above with the induction hypothesis,
we see that the canonical tensor of $U_{q}(\fosp(V^\circ))$ associated with $\varrho$ equals
\begin{equation*}
  \Theta^\circ =
  \ID - (q - q^{-1}) \sum_{1< i < j < 1'} (-1)^{\ol{j}} E_{ij} \otimes \Big( q^{(\varepsilon_{i},\varepsilon_{j})} E_{ji} -
    (-1)^{\ol{j}(\ol{i}+\ol{j})} \vartheta_{i}\vartheta_{j} q^{(\rho^\circ, \varepsilon_{i} - \varepsilon_{j})}
    q^{-(\varepsilon_{i},\varepsilon_{i})/2}q^{-(\varepsilon_{j},\varepsilon_{j})/2} E_{i'j'} \Big)
\end{equation*}
where $\rho^\circ$ denotes the Weyl vector corresponding to $U_{q}(\fosp(V^\circ))$, cf.~\eqref{eq:weyl-vec}. Though the Weyl vectors
$\rho$ and $\rho^\circ$ are different, the equality $(\rho^\circ , \alpha_{i}) = \tfrac{1}{2}(\alpha_{i},\alpha_{i}) = (\rho,\alpha_{i})$
still holds for all $2 \leq i \leq s$, cf.~\eqref{eq:rho-simple-root}, so that $(\rho^\circ,\gamma) = (\rho,\gamma)$ for any root
$\gamma$ in the root system $\Phi^\circ$ of $U_{q}(\fosp(V^\circ))$. Thus:
\begin{equation}\label{eq:Theta'-B-explicit}
  \Theta^\circ =
  \ID - (q - q^{-1}) \sum_{1< i < j < 1'} (-1)^{\ol{j}} E_{ij} \otimes \Big( q^{(\varepsilon_{i},\varepsilon_{j})} E_{ji} -
    (-1)^{\ol{j}(\ol{i}+\ol{j})} \vartheta_{i}\vartheta_{j} q^{(\rho, \varepsilon_{i} - \varepsilon_{j})}
    q^{-(\varepsilon_{i},\varepsilon_{i})/2}q^{-(\varepsilon_{j},\varepsilon_{j})/2} E_{i'j'} \Big) \,.
\end{equation}
Combining this formula with $\Theta = \Theta^\circ \Theta_{1}$ due to~\eqref{eq:Theta-factorization-B}, we finally obtain:
\begin{multline}\label{eq:Theta-desired-B}
  \Theta
  = \Theta^\circ - (q - q^{-1}) \sum_{1<a<1'}
   \bigg\{ (-1)^{\ol{a}} E_{1a} \otimes E_{a1} - (-1)^{\ol{1}\,\ol{a}} q^{-(\rho,\varepsilon_{1}-\varepsilon_{a})} q^{(\varepsilon_{1},\varepsilon_{1})/2}q^{(\varepsilon_{a},\varepsilon_{a})/2}\vartheta_{1'}\vartheta_{a'}\cdot E_{a'1'} \otimes E_{a1} \\
  + (-1)^{\ol{1}} E_{a'1'} \otimes E_{1'a'} -
   (-1)^{\ol{1}\,\ol{a}} q^{(\rho,\varepsilon_{1}-\varepsilon_{a})} q^{-(\varepsilon_{1},\varepsilon_{1})/2}
   q^{-(\varepsilon_{a},\varepsilon_{a})/2}\vartheta_{1}\vartheta_{a}\cdot E_{1a} \otimes E_{1'a'} \bigg\}\\
  + (q - q^{-1})q^{-(\varepsilon_{1},\varepsilon_{1})} \Big(q^{(2\rho,\varepsilon_{1})} -
   (-1)^{\ol{1}}\Big) E_{11'} \otimes E_{1'1}\\
  - (q - q^{-1})^{2} \sum_{2 \leq a \leq s} (-1)^{\ol{1}(\ol{1}+\ol{a})}
   q^{-(\rho,\varepsilon_{1}-\varepsilon_{a})}q^{(\varepsilon_{1},\varepsilon_{1})/2}
   q^{-(\varepsilon_{a},\varepsilon_{a})/2}\vartheta_{1}\vartheta_{a} \cdot E_{a1'} \otimes E_{a'1}\\
  + (q - q^{-1})^{2} \sum_{1<i<1'} \left(\sum_{a=2}^{i-1} (-1)^{\ol{a}} q^{(\rho,2\varepsilon_{a})}\right)
   (-1)^{\ol{1}+\ol{i}}(-1)^{\ol{1}\,\ol{i}} q^{-(\rho,\varepsilon_{1}+\varepsilon_{i})}q^{(\varepsilon_{1},\varepsilon_{1})/2}
   q^{-(\varepsilon_{i},\varepsilon_{i})/2}\vartheta_{1}\vartheta_{i} \cdot E_{i'1'} \otimes E_{i1} \,.
\end{multline}
The coefficient of each term in this formula coincides with the one from the right-hand side of \eqref{eq:B-type-Theta},
except possibly only for the coefficients of $\{E_{i'1'} \otimes E_{i1}\}_{1<i<1'}$. Let us treat the latter ones:
\begin{itemize}

\item
Case 1: $2 \leq i \leq s+1$.

First, we note that a computation analogous to \eqref{eq:local-telescoping} gives
\begin{equation}\label{eq:telescopic-case1}
  (q-q^{-1}) \sum_{a = 2}^{i-1} (-1)^{\ol{a}} q^{(\rho,2\varepsilon_{a})} =
  q^{-(\varepsilon_{1},\varepsilon_{1})}q^{(\rho,2\varepsilon_{1})} -
  q^{-(\varepsilon_{i-1},\varepsilon_{i-1})}q^{(\rho,2\varepsilon_{i-1})}\,.
\end{equation}
Therefore, the coefficient of $E_{i'1'} \otimes E_{i1}$ in the right-hand side of~\eqref{eq:Theta-desired-B} equals:
\begin{align*}
  & (q - q^{-1})(-1)^{\ol{1}\,\ol{i}} q^{-(\rho,\varepsilon_{1}-\varepsilon_{i})} q^{(\varepsilon_{1},\varepsilon_{1})/2}
    q^{(\varepsilon_{i},\varepsilon_{i})/2} \vartheta_{1'}\vartheta_{i'}\\
  &\qquad + (q - q^{-1})^{2} \Big(\sum_{a=2}^{i-1} (-1)^{\ol{a}} q^{(\rho,2\varepsilon_{a})}\Big)
    (-1)^{\ol{1}+\ol{i}}(-1)^{\ol{1}\,\ol{i}} q^{-(\rho,\varepsilon_{1}+\varepsilon_{i})} q^{(\varepsilon_{1},\varepsilon_{1})/2}
    q^{-(\varepsilon_{i},\varepsilon_{i})/2}\vartheta_{1}\vartheta_{i}\\
  &= (q - q^{-1})(-1)^{\ol{1}\,\ol{i}} q^{-(\rho,\varepsilon_{1}-\varepsilon_{i})} q^{(\varepsilon_{1},\varepsilon_{1})/2}
     q^{(\varepsilon_{i},\varepsilon_{i})/2}\vartheta_{1'}\vartheta_{i'} \times \\
  &\qquad \bigg(1 + q^{-(\rho,2\varepsilon_{i})}q^{-(\varepsilon_{i},\varepsilon_{i})} \cdot
    (q - q^{-1}) \sum_{a=2}^{i-1} (-1)^{\ol{a}} q^{(\rho,2\varepsilon_{a})} \bigg)\\
  &\overset{\eqref{eq:rho-simple-root},\eqref{eq:telescopic-case1}}{=} (q - q^{-1})(-1)^{\ol{1}\,\ol{i}} q^{(\rho,\varepsilon_{1}-\varepsilon_{i})}
    q^{-(\varepsilon_{1},\varepsilon_{1})/2} q^{-(\varepsilon_{i},\varepsilon_{i})/2}\vartheta_{1'}\vartheta_{i'} \,,
\end{align*}
which coincides with the coefficient of $E_{i'1'} \otimes E_{i1}$ in the right-hand side of~\eqref{eq:B-type-Theta}.

\medskip
\item
Case 2: $s' \leq i \leq 2'$.

Similarly to the previous case, we have:
\begin{equation}\label{eq:telescopic-case2}
  (q-q^{-1}) \sum_{a = 2}^{i-1} (-1)^{\ol{a}} q^{(\rho,2\varepsilon_{a})} =
  \left( q^{-(\varepsilon_{1},\varepsilon_{1})}q^{(\rho,2\varepsilon_{1})} -
         q^{-(\varepsilon_{i-1},\varepsilon_{i-1})}q^{(\rho,2\varepsilon_{i-1})} \right) + (q-q^{-1})\,.
\end{equation}
Indeed, this formula is obvious for $i=s'=s+2$, while for $i>s'$ it follows from
\begin{multline*}
  (q-q^{-1}) \sum_{a = s+2}^{i-1} (-1)^{\ol{a}} q^{(\rho,2\varepsilon_{a})} = (q-q^{-1}) \sum_{a = i'+1}^{s}
    (q^{(\varepsilon_a,\varepsilon_a)} - q^{-(\varepsilon_a,\varepsilon_a)})q^{-(\rho,2\varepsilon_a)} \\
  \overset{\eqref{eq:rho-simple-root}}{=}
  \sum_{a = i'+1}^{s} \left( q^{-(\varepsilon_{a+1},\varepsilon_{a+1})}q^{-(\rho,2\varepsilon_{a+1})} -
         q^{-(\varepsilon_{a},\varepsilon_{a})}q^{-(\rho,2\varepsilon_{a})} \right) \\
  = q^{-(\varepsilon_{s+1},\varepsilon_{s+1})}q^{-(\rho,2\varepsilon_{s+1})} -
  q^{-(\varepsilon_{i'+1},\varepsilon_{i'+1})}q^{-(\rho,2\varepsilon_{i'+1})}
  \overset{\eqref{eq:degree-symmetry}}{=} 1 - q^{-(\varepsilon_{i-1},\varepsilon_{i-1})}q^{(\rho,2\varepsilon_{i-1})} \,.
\end{multline*}
Therefore, the coefficient of $E_{i'1'} \otimes E_{i1}$ in the right-hand side of~\eqref{eq:Theta-desired-B} equals:
\begin{equation}\label{eq:Theta-coeff-comparison}
\begin{split}
  & (q - q^{-1})(-1)^{\ol{1}\,\ol{i}} q^{-(\rho,\varepsilon_{1}-\varepsilon_{i})} q^{(\varepsilon_{1},\varepsilon_{1})/2}
    q^{(\varepsilon_{i},\varepsilon_{i})/2}\vartheta_{1'}\vartheta_{i'}\\
  &\qquad + (q - q^{-1})^{2} \Big(\sum_{a=2}^{i-1} (-1)^{\ol{a}} q^{(\rho,2\varepsilon_{a})}\Big) (-1)^{\ol{1}+\ol{i}}(-1)^{\ol{1}\,\ol{i}}
    q^{-(\rho,\varepsilon_{1}+\varepsilon_{i})}q^{(\varepsilon_{1},\varepsilon_{1})/2}
    q^{-(\varepsilon_{i},\varepsilon_{i})/2}\vartheta_{1}\vartheta_{i}\\
  &\qquad - (q - q^{-1})^{2}(-1)^{\ol{1}\,\ol{i}} q^{-(\rho,\varepsilon_{1}+\varepsilon_{i})}
    q^{(\varepsilon_{1},\varepsilon_{1})/2}q^{-(\varepsilon_{i},\varepsilon_{i})/2}\vartheta_{1'}\vartheta_{i'}\\
  &= (q - q^{-1})(-1)^{\ol{1}\,\ol{i}} q^{-(\rho,\varepsilon_{1}-\varepsilon_{i})}q^{(\varepsilon_{1},\varepsilon_{1})/2}
    q^{(\varepsilon_{i},\varepsilon_{i})/2}\vartheta_{1'}\vartheta_{i'} \times \\
  &\qquad \bigg(1 - (q - q^{-1})q^{-(\rho,2\varepsilon_{i})}q^{-(\varepsilon_{i},\varepsilon_{i})} +
    q^{-(\rho,2\varepsilon_{i})}q^{-(\varepsilon_{i},\varepsilon_{i})} \cdot (q - q^{-1})
    \sum_{a=2}^{i-1} (-1)^{\ol{a}} q^{(\rho,2\varepsilon_{a})} \bigg)\\
  &\overset{\eqref{eq:rho-simple-root},\eqref{eq:telescopic-case2}}{=} (q - q^{-1})(-1)^{\ol{1}\,\ol{i}} q^{(\rho,\varepsilon_{1}-\varepsilon_{i})}
    q^{-(\varepsilon_{1},\varepsilon_{1})/2}q^{-(\varepsilon_{i},\varepsilon_{i})/2}\vartheta_{1'}\vartheta_{i'} \,,
\end{split}
\end{equation}
which coincides with the coefficient of $E_{i'1'} \otimes E_{i1}$ in the right-hand side of~\eqref{eq:B-type-Theta}.

\end{itemize}
This completes the proof of the induction step, thus establishing formula~\eqref{eq:B-type-Theta}.
\end{proof}

Finally, we can re-derive our formula \eqref{eq:R_0} for
  $R_{VV} = \tau_{VV} \circ \hat{R}_{VV} = \wtd{f}^{1/2} \circ \Theta \circ \wtd{f}^{1/2}$
from above result. Since the action of $\wtd{f}^{1/2}$ on $V \otimes V$ is given by
\begin{equation}\label{eq:wtdf-action}
  \wtd{f}^{1/2} = \sum_{i,j} q^{-(\varepsilon_{i},\varepsilon_{j})/2} E_{ii} \otimes E_{jj}\,,
\end{equation}
the explicit formula~\eqref{eq:B-type-Theta} for $\Theta$ implies
\begin{align*}
  R_{VV} &=\, \wtd{f}^{1/2} \circ \Theta \circ \wtd{f}^{1/2} \\
  &= \sum_{i,j} q^{-(\varepsilon_{i},\varepsilon_{j})} E_{ii} \otimes E_{jj} +
  (q^{-1} - q) \sum_{i < j} (-1)^{\ol{j}} E_{ij} \otimes \Big( E_{ji} -
    (-1)^{\ol{j}(\ol{i}+\ol{j})} \vartheta_{i}\vartheta_{j} q^{(\rho, \varepsilon_{i} - \varepsilon_{j})} E_{i'j'} \Big) \\
  &= \ID + (q^{-1/2} - q^{1/2}) \sum_{i=1}^N (-1)^{\ol{i}} E_{ii} \otimes
      \Big( q^{-(\varepsilon_{i},\varepsilon_{i})/2} E_{ii} - q^{(\varepsilon_{i},\varepsilon_{i})/2} E_{i'i'} \Big) \\
  &\qquad \quad + (q^{-1} - q) \sum_{i < j} (-1)^{\ol{j}} E_{ij} \otimes \Big( E_{ji} -
    (-1)^{\ol{j}(\ol{i}+\ol{j})} \vartheta_{i}\vartheta_{j} q^{(\rho, \varepsilon_{i} - \varepsilon_{j})} E_{i'j'} \Big)\,,
\end{align*}
which precisely recovers $R_0$ from~\eqref{eq:R_0}. This provides an alternative proof of Theorem~\ref{thm:R_osp_finite}.


\subsubsection{Factorized formula for even $m$}
\

Similarly to Lemma \ref{lem:B-root-vectors-action}, we start by evaluating the action of the root vectors
$\{e_{\gamma}, f_{\gamma}\}_{\gamma \in \bar{\Phi}^{+}}$ of~\eqref{eq:root_vector} on the $\uqV$-module $V$
from Proposition~\ref{prop:fin-repn}:

\begin{Lem}\label{lem:CD-fork-root-vectors-action}
If $\ol{s} = \bar{0}$, then the action of the root generators is as follows.

\medskip
\noindent
(a) For $\gamma = \varepsilon_{i}-\varepsilon_{j}$ with $1 \leq i < j \leq s$, we have:
\begin{equation*}
  \varrho(e_{\gamma}) = \sse_{ij} \,,\qquad \varrho(f_{\gamma}) = (-1)^{\ol{i}+\cdots+\ol{j-1}} \cdot \ssf_{ij}\,.
\end{equation*}

\noindent
(b) For $\gamma = \varepsilon_{i}+\varepsilon_{j}$ with $1 \leq i < j \leq s$, we have:
\begin{align*}
  \varrho(e_{\gamma}) &= -\vartheta_{j}\vartheta_{s} \cdot
    \prod_{k=j}^{s} \Big(-(-1)^{\ol{k}(\ol{k}+\ol{k+1})}\Big) \cdot \sse_{ij'}\,,\\
  \varrho(f_{\gamma}) &= -(-1)^{\ol{i}+\cdots+\ol{j-1}} \vartheta_{j}\vartheta_{s} \cdot
    \prod_{k=j}^{s} \Big(-(-1)^{\ol{k+1}(\ol{k}+\ol{k+1})}\Big) \cdot \ssf_{ij'}\,.
\end{align*}

\noindent
(c) For $\gamma = 2\varepsilon_{i}$ with $1 \leq i \leq s$ and $\ol{i} = \bar{1}$, we have:
\begin{align*}
  \varrho(e_{\gamma}) &= q^{-(\rho,\varepsilon_{i})} (1+q^{-2}) \vartheta_{i}\vartheta_{s} \cdot E_{ii'}\,,\\
  \varrho(f_{\gamma}) &= -q^{(\rho,\varepsilon_{i})} (1+q^{2}) \vartheta_{i}\vartheta_{s} \cdot E_{i'i}\,.
\end{align*}
\end{Lem}

\begin{Lem}\label{lem:CD-nofork-root-vectors-action}
If $\ol{s} = \bar{1}$, then the action of the root generators is as follows.

\medskip
\noindent
(a) For $\gamma = \varepsilon_{i}-\varepsilon_{j}$ with $1 \leq i < j \leq s$, we have:
\begin{equation*}
  \varrho(e_{\gamma}) = \sse_{ij} \,,\qquad \varrho(f_{\gamma}) = (-1)^{\ol{i}+\cdots+\ol{j-1}} \cdot \ssf_{ij}\,.
\end{equation*}

\noindent
(b) For $\gamma = \varepsilon_{i}+\varepsilon_{j}$ with $1 \leq i < j \leq s$, we have:
\begin{align*}
  \varrho(e_{\gamma}) &= -\vartheta_{j}\vartheta_{s} \cdot
    \prod_{k=j}^{s} \Big(-(-1)^{\ol{k}(\ol{k}+\ol{k+1})}\Big) \cdot \sse_{ij'}\,,\\
  \varrho(f_{\gamma}) &= (-1)^{\ol{i}+\cdots+\ol{j-1}} \vartheta_{j}\vartheta_{s} \cdot (q+q^{-1}) \cdot
    \prod_{k=j}^{s} \Big(-(-1)^{\ol{k+1}(\ol{k}+\ol{k+1})}\Big) \cdot \ssf_{ij'}\,.
\end{align*}

\noindent
(c) For $\gamma = 2\varepsilon_{i}$ with $1 \leq i \leq s$ and $\ol{i} = \bar{1}$, we have
\begin{align*}
  \varrho(e_{\gamma}) &= q^{-(\rho,\varepsilon_{i})} (1+q^{-2}) \vartheta_{i}\vartheta_{s} \cdot E_{ii'}\,,\\
  \varrho(f_{\gamma}) &= -q^{(\rho,\varepsilon_{i})} (1+q^{2}) \vartheta_{i}\vartheta_{s} \cdot (q+q^{-1}) \cdot E_{i'i}\,.
\end{align*}
\end{Lem}

Combining the results above with Lemmas~\ref{lem:CD-pairing-fork}--\ref{lem:CD-pairing-nofork}, we obtain
the following counterpart of Lemma~\ref{lem:B-local-operators} which can be written in a uniform way
(independent of the parity of $\ol{s}=|v_s|$):

\begin{Lem}\label{lem:CD-local-operators}
The operators $\Theta_{\gamma}$ of~\eqref{eq:local-Theta} act on the $\uqV$-module $V \otimes V$ as follows:
\begin{itemize}

\item
If $\gamma = \varepsilon_{i} - \varepsilon_{j}$ with $i < j < i'$, then
\begin{equation*}
  \Theta_{\gamma} = \ID - (-1)^{\ol{j}}(q-q^{-1}) \cdot \sse_{ij} \otimes \ssf_{ij}\,.
\end{equation*}

\item
If $\gamma = 2\varepsilon_{i}$ with $1 \leq i \leq s$ and $\ol{i} = \bar{1}$, then
\begin{equation*}
  \Theta_{\gamma} = \ID + (q-q^{-1}) \Big(q^{-1} - (-1)^{\ol{i}}q^{-(\varepsilon_{i},\varepsilon_{i})}\Big)  \cdot E_{ii'} \otimes E_{i'i}\,.
\end{equation*}

\end{itemize}
\end{Lem}

We note that the way we wrote the last formula above allows to define $\Theta_{2\varepsilon_{i}}=\ID$ when $\ol{i} = \bar{0}$.
With this extension of the notation $\Theta_{2\varepsilon_{i}}$ to all indices $i$, let us define
\begin{equation*}
  \Theta_{i} = \Theta_{\varepsilon_{i}+\varepsilon_{i+1}} \cdots \Theta_{\varepsilon_{i}+\varepsilon_{s}}
  \Theta_{2\varepsilon_{i}}\Theta_{\varepsilon_{i}-\varepsilon_{s}} \cdots \Theta_{\varepsilon_{i}-\varepsilon_{i+1}}
  \qquad \mathrm{for\ any} \quad 1 \leq i \leq s \,.
\end{equation*}
Evoking the explicit orders~(\ref{eq:D_order},~\ref{eq:C_order}) on $\bar{\Phi}^+$, one immediately obtains the factorization
\begin{equation}\label{eq:Theta-factorization-CD}
  \Theta = \Theta_{s}\Theta_{s-1} \cdots \Theta_{1} \,.
\end{equation}

The following result is an analogue of Lemma~\ref{lem:B-type-Theta_i}:

\begin{Lem}\label{lem:CD-type-Theta_i}
For $1 \leq i \leq s$, we have
\begin{equation*}
  \Theta_{i} = \ID - (q-q^{-1}) \sum_{i<j<i'} (-1)^{\ol{j}}\sse_{ij} \otimes \ssf_{ij} +
  (q-q^{-1})q^{-(\varepsilon_{i},\varepsilon_{i})}\Big(q^{(2\rho,\varepsilon_{i})}-(-1)^{\ol{i}}\Big) \cdot E_{ii'} \otimes E_{i'i} \,.
\end{equation*}
\end{Lem}

\begin{proof}
Arguing exactly as in the proof of Lemma~\ref{lem:B-type-Theta_i}, we obtain the following analogue of~\eqref{eq:iTheta-B}:
\begin{multline*}
  \Theta_{i} 
  = \ID - (q-q^{-1}) \sum_{i<j<i'} (-1)^{\ol{j}}\sse_{ij} \otimes \ssf_{ij} \\
  + (q - q^{-1})\Big(q^{-1} - (-1)^{\ol{i}}q^{-(\varepsilon_{i},\varepsilon_{i})}\Big) E_{ii'} \otimes E_{i'i}
  + (q - q^{-1})^{2}\sum_{j = i+1}^{s}(-1)^{\ol{j}}q^{(\rho,2\varepsilon_{j})} E_{ii'} \otimes E_{i'i} \,.
\end{multline*}
The last sum can be simplified similarly to~\eqref{eq:local-telescoping} as follows:
\begin{equation*}
  (q - q^{-1}) \sum_{j=i+1}^{s} (-1)^{\ol{j}} q^{(\rho,2\varepsilon_{j})}
  = \sum_{j=i+1}^{s} (q^{(\varepsilon_{j},\varepsilon_{j})} - q^{-(\varepsilon_{j},\varepsilon_{j})})q^{(2\rho,\varepsilon_{j})}
  = q^{-(\varepsilon_{i},\varepsilon_{i})} q^{(2\rho,\varepsilon_{i})} -
    q^{-(\varepsilon_{s},\varepsilon_{s})} q^{(2\rho,\varepsilon_{s})}\,.
\end{equation*}
Thus, the claimed formula for $\Theta_i$ follows from the two equalities above and the identity
\begin{equation}\label{eq:aux}
  (2\rho,\varepsilon_s)-(\varepsilon_s,\varepsilon_s)=-1 \,.
\end{equation}
The latter is a simple consequence of~\eqref{eq:rho-simple-root}:
\begin{itemize}

\item
if $\ol{s}=\bar{1}$, then $\alpha_s=2\varepsilon_s$ and so
$(2\rho,\varepsilon_s)=\tfrac{1}{2}(2\rho,\alpha_s)=\tfrac{1}{2}(\alpha_s,\alpha_s)=-2=(\varepsilon_s,\varepsilon_s)-1$;

\item
if $\ol{s}=\bar{0}$, then $\alpha_{s-1}=\varepsilon_{s-1}-\varepsilon_s, \alpha_{s}=\varepsilon_{s-1}+\varepsilon_s$, so that
  $$ (2\rho,\varepsilon_s)=\tfrac{1}{2}\big((2\rho,\alpha_s)-(2\rho,\alpha_{s-1})\big)=
     \tfrac{1}{2}\big((\alpha_s,\alpha_s)-(\alpha_{s-1},\alpha_{s-1})\big)=0=(\varepsilon_s,\varepsilon_s)-1 \,. $$

\end{itemize}
This establishes~\eqref{eq:aux} and thus completes the proof of the lemma.
\end{proof}

Combining this result with the factorization~\eqref{eq:Theta-factorization-CD}, we can finally evaluate $\Theta$:

\begin{Prop}\label{prop:CD-type-Theta}
The action of the operator $\Theta$ on the $\uqV$-module $V \otimes V$ is given by
\begin{equation*}
  \Theta = \ID - (q - q^{-1}) \sum_{1\leq i < j \leq 1} (-1)^{\ol{j}} E_{ij} \otimes
  \Big( q^{(\varepsilon_{i},\varepsilon_{j})} E_{ji} -
    (-1)^{\ol{j}(\ol{i}+\ol{j})} \vartheta_{i}\vartheta_{j} q^{(\rho, \varepsilon_{i} - \varepsilon_{j})}
    q^{-(\varepsilon_{i},\varepsilon_{i})/2}q^{-(\varepsilon_{j},\varepsilon_{j})/2} E_{i'j'} \Big) \,.
\end{equation*}
\end{Prop}

Let us note right away that the formula above is identical to~\eqref{eq:B-type-Theta}.

\begin{proof}
The proof of this result is completely analogous to that of Proposition \ref{prop:B-type-Theta}, and proceeds by induction
on $s$. The base case $s=1$ follows from the evaluation of $\Theta = \Theta_{1}$ in Lemma \ref{lem:CD-type-Theta_i}.

As per the induction step, we obtain precisely formula~\eqref{eq:Theta-desired-B} expressing $\Theta$ through $\Theta^\circ$,
the latter been given by the same formula~\eqref{eq:Theta'-B-explicit}, due to the observation preceding the proof.
The rest of the proof proceeds without any changes.
\end{proof}

Analogously to odd $m$, we can use the result above to re-derive formula~\eqref{eq:R_0} for $R_{VV}$. Indeed, since
the formula in Proposition \ref{prop:CD-type-Theta} is identical to~\eqref{eq:B-type-Theta}, the same computation can be
applied without any changes, thus providing an alternative proof of Theorem~\ref{thm:R_osp_finite} in that case.


\section{R-Matrices with a spectral parameter}\label{sec:affine-R}


\subsection{Orthosymplectic quantum affine groups}\label{ssec:quantum-affine-orthosymplectic}
\

Let $\theta$ be the highest root of $\fosp(V)$ with respect to the fixed polarization of~\eqref{eq:polarization},
and let $\{\sfk_i\}_{i=1}^s$ be the corresponding coefficients in the decomposition $\theta=\sum_{i=1}^s \sfk_i\alpha_i$.
Explicitly, we have:
\begin{equation*}
  \theta=
  \begin{cases}
    \varepsilon_1+\varepsilon_2 & \mbox{if } |v_1|=\bar{0} \\
    2\varepsilon_1 & \mbox{if } |v_1|=\bar{1}
  \end{cases} \,.
\end{equation*}
We define the lattice $\widehat{P}=\BZ\delta\oplus P=\BZ\delta\oplus \bigoplus_{i=1}^s \BZ\varepsilon_i$,
with $P$ introduced right before~\eqref{eq:grading algebra}. Then $\alpha_1,\ldots,\alpha_s$ as well as
$\alpha_0=\delta-\theta$ can be viewed as elements of $\widehat{P}$. We extend the bilinear pairing $(\cdot,\cdot)$
on $P$, defined via~\eqref{eq:epsilon-pairing}, to that on $\widehat{P}$ by setting
$(\delta,\delta)=(\delta,\varepsilon_i)=(\varepsilon_i,\delta)=0$ for all $i$. We define the
\emph{symmetrized extended Cartan matrix} of $\fosp(V)$ as $(a_{ij})_{i,j=0}^{s}$ with $a_{ij} = (\alpha_{i}, \alpha_{j})$.
It extends the Cartan matrix of~\eqref{eq:sym-cartan} through $a_{00}=(\theta,\theta)$ and $a_{0i}=a_{i0}=-(\theta,\alpha_i)$
for $1\leq i\leq s$.

The \emph{orthosymplectic quantum affine supergroup} $\UdqV$ is a $\BC(q^{\pm 1/2})$-superalgebra generated by
$\{e_{i}, f_{i}, q^{\pm h_{i}/2}\}_{i=0}^{s}\cup \{\gamma^{\pm 1}, D^{\pm 1}\}$, with the $\BZ_{2}$-grading
\begin{equation*}
  |e_0|=|f_0|=
  \begin{cases}
    \bar{0} & \mbox{if } \theta\in \Phi_{\bar{0}} \\
    \bar{1} & \mbox{if } \theta\in \Phi_{\bar{1}}
  \end{cases} \,,\qquad
  |e_i|=|f_i|=
  \begin{cases}
    \bar{0} & \mbox{if } \alpha_i\in \Phi_{\bar{0}} \\
    \bar{1} & \mbox{if } \alpha_i\in \Phi_{\bar{1}}
  \end{cases} \quad \text{for} \quad 1 \leq i \leq s \,,
\end{equation*}
\begin{equation*}
  |\gamma^{\pm 1}|=|D^{\pm 1}|=|h_i|=\bar{0} \qquad \text{for} \quad 0 \leq i \leq s \,,
\end{equation*}
subject to the following defining relations:
\begin{align}
  & D^{\pm 1}\cdot D^{\mp 1}=1 \,,\quad
    [D,q^{h_i/2}]=0 \,,\quad D e_i D^{-1}=q^{\delta_{0i}}e_i \,,\quad Df_iD^{-1}=q^{-\delta_{0i}}f_i
  \label{eq:Chevalley-affine-0.1} \,, \\
  & \gamma^{\pm 1}\cdot \gamma^{\mp 1}=1 \,,\quad
    \gamma=q^{h_0/2}\cdot \prod_{i=1}^s (q^{h_i/2})^{\sfk_i} \,,\quad
    \gamma - \mathrm{central\ element}
  \label{eq:Chevalley-affine-0.2} \,,
\end{align}
the counterpart of~\eqref{eq:q-chevalley-rel-hh}--\eqref{eq:q-chevalley-rel-ef} but now with $0\leq i,j\leq s$\,:
\begin{align}
  & [q^{h_{i}/2}, q^{h_{j}/2}] = 0 \,, \quad
    q^{\pm h_{i}/2} q^{\mp h_{i}/2} = 1
  \label{eq:Chevalley-affine-1} \,,\\
  & q^{h_{i}/2} e_{j} q^{-h_{i}/2} = q^{a_{ij}/2} e_{j}, \quad
    q^{h_{i}/2} f_{j} q^{-h_{i}/2} = q^{-a_{ij}/2} f_{j}
  \label{eq:Chevalley-affine-2}\,, \\
  & [e_{i}, f_{j}] = \delta_{ij} \frac{q^{h_{i}} - q^{-h_{i}}}{q - q^{-1}}
  \label{eq:Chevalley-affine-3} \,,
\end{align}
together with the \emph{standard} and the \emph{higher order $q$-Serre relations}, which the interested reader may find
in~\cite[relations (QS4, QS5), cf. Theorem~6.8.2]{y}. We note that $\UdqV$ is equipped with a Hopf superalgebra structure,
with the coproduct $\Delta$, the counit $\epsilon$, and the antipode $S$ defined on the generators
$\{e_i,f_i,q^{\pm h_i/2}\}_{i=0}^{s}$ by the same formulas as in the end of Subsection~\ref{sec:q-orthosymplectic}, while also
\begin{equation*}
  \Delta(D)=D\otimes D \,,\quad S(D)=D^{-1} \,,\quad \epsilon(D)=1 \,,\quad
  \Delta(\gamma)=\gamma\otimes \gamma \,,\quad S(\gamma)=\gamma^{-1} \,,\quad \epsilon(\gamma)=1 \,.
\end{equation*}

It is often more convenient to work with a version of $\UdqV$ without the degree generators $D^{\pm 1}$. Explicitly, $\UqV$
is the $\BC(q^{\pm 1/2})$-superalgebra generated by $\{e_{i}, f_{i}, q^{\pm h_{i}/2}\}_{i=0}^{s}\cup \{\gamma^{\pm 1}\}$, with
the same $\BZ_{2}$-grading, the same defining relations excluding~\eqref{eq:Chevalley-affine-0.1}, and the same Hopf structure.


\subsection{Evaluation modules and affine R-matrices}
\

\begin{Prop}\label{prop:affine-repn}
For any $u \in \BC^{\times}$ and $a,b\in \BC^\times$ specified below, the $\uqV$-action $\varrho$ on $V$ from
Proposition~\ref{prop:fin-repn} can be extended to a $\UqV$-action $\varrho^{a,b}_u$ on $V(u)=V$ by setting
\begin{equation*}
  \varrho^{a,b}_u(x)=\varrho(x)  \qquad \mathrm{for\ all} \quad x\in \{e_i,f_i,q^{\pm h_i/2}\}_{i=1}^s
\end{equation*}
and defining the action of the remaining generators $e_0,f_0,q^{\pm h_0/2}, \gamma^{\pm 1}$ via
\eqref{eq:affine-action_case1} or \eqref{eq:affine-action_case2} below:
\begin{itemize}

\item[$\bullet$]
Case 1: $|v_{1}| = \bar{1}$.
\begin{equation}\label{eq:affine-action_case1}
\begin{split}
  & \varrho^{a,b}_{u}(e_{0}) = au \cdot E_{1'1} \,,\qquad \varrho^{a,b}_{u}(f_{0}) = bu^{-1} \cdot E_{11'} \,,\\
  & \varrho^{a,b}_{u}(q^{\pm h_{0}/2}) = q^{\pm X_{11}} \,,\quad \varrho^{a,b}_{u}(\gamma^{\pm 1}) = \id
\end{split}
\end{equation}
with parameters $a,b$ subject to $ab = -(q+q^{-1}).$

\item[$\bullet$]
Case 2: $|v_{1}| = \bar{0}$.
\begin{equation}\label{eq:affine-action_case2}
\begin{split}
  & \varrho^{a,b}_{u}(e_{0}) = au \cdot X_{2'1} \,,\qquad \varrho^{a,b}_{u}(f_{0}) = bu^{-1} \cdot X_{12'} \,, \\
  & \varrho^{a,b}_{u}(q^{\pm h_{0}/2}) = q^{\mp ((-1)^{\ol{1}} X_{11} + (-1)^{\ol{2}} X_{22})/2} \,,\qquad
    \varrho^{a,b}_{u}(\gamma^{\pm 1}) = \id
\end{split}
\end{equation}
with parameters $a,b$ subject to $ab = (-1)^{\ol{2}}$.

\end{itemize}
\end{Prop}

\begin{proof}
We need to show that the operators defined above satisfy the defining
relations~\eqref{eq:Chevalley-affine-0.2}--\eqref{eq:Chevalley-affine-3} together with all $q$-Serre relations.
This verification is straightforward and proceeds similarly to our proof of Proposition \ref{prop:fin-repn}.

$\bullet$ 
Case 1: $|v_{1}| = \bar{1}$.

The second relation of~\eqref{eq:Chevalley-affine-0.2} is verified by direct calculations, treating three cases
as before: $m$ is odd, $m$ is even and $\ol{s}=\bar{0}$, or $m$ is even and $\ol{s}=\bar{1}$.
The relations~(\ref{eq:Chevalley-affine-1},~\ref{eq:Chevalley-affine-2}) then immediately follow from their validity
for $i,j\ne 0$, due to Proposition~\ref{prop:fin-repn}. It remains to verify~\eqref{eq:Chevalley-affine-3} for $i=0$ or $j=0$.
The relations $[\varrho^{a,b}_{u}(e_{0}),\varrho^{a,b}_{u}(f_{i})]=0=[\varrho^{a,b}_{u}(f_{0}),\varrho^{a,b}_{u}(e_{i})]$ for
$i\ne 0$ are obvious, since all four operators
  $\varrho^{a,b}_{u}(e_{0})\varrho^{a,b}_{u}(f_{i}), \varrho^{a,b}_{u}(f_{i})\varrho^{a,b}_{u}(e_{0}),
   \varrho^{a,b}_{u}(f_{0})\varrho^{a,b}_{u}(e_{i}), \varrho^{a,b}_{u}(e_{i})\varrho^{a,b}_{u}(f_{0})$
act by $0$. Finally, we~have:
\begin{equation*}
  [\varrho^{a,b}_{u}(e_{0}),\varrho^{a,b}_{u}(f_{0})] =
  (q+q^{-1}) (E_{11} - E_{1'1'}) = \frac{q^{2 X_{11}} - q^{-2 X_{11}}}{q - q^{-1}} =
  \frac{\varrho^{a,b}_{u}(q^{h_{0}}) - \varrho^{a,b}_{u}(q^{-h_{0}})}{q - q^{-1}} \,.
\end{equation*}

$\bullet$ 
Case 2: $|v_{1}| = \bar{0}$.

The verification of~\eqref{eq:Chevalley-affine-0.2}--\eqref{eq:Chevalley-affine-2} is similar to that in Case 1. We also
note that $[\varrho^{a,b}_{u}(e_{0}),\varrho^{a,b}_{u}(f_{i})]=0$ and $[\varrho^{a,b}_{u}(f_{0}),\varrho^{a,b}_{u}(e_{i})]=0$
for $i\ne 0,1$ by the same reason as in Case 1. Finally, we~have:
\begin{equation*}
  [\varrho^{a,b}_{u}(e_{0}),\varrho^{a,b}_{u}(f_{1})] = au[X_{2'1},X_{21}]=0 \,,\qquad
  [\varrho^{a,b}_{u}(f_{0}),\varrho^{a,b}_{u}(e_{1})] = bu^{-1}[X_{12'},X_{12}]=0
\end{equation*}
as well as
\begin{equation*}
  [\varrho^{a,b}_{u}(e_{0}),\varrho^{a,b}_{u}(f_{0})] = -(-1)^{\ol{1}} X_{11} - (-1)^{\ol{2}} X_{22} =
  \frac{\varrho^{a,b}_{u}(q^{h_{0}}) - \varrho^{a,b}_{u}(q^{-h_{0}})}{q - q^{-1}} \,,
\end{equation*}
where we used~\eqref{eq:trivial-quantization} in the last equality.

The verification of $q$-Serre relations proceeds as in our proof of Proposition~\ref{prop:fin-repn}.
To this end, we note that the algebra $\UqV$ is $P$-graded via~\eqref{eq:grading algebra} combined with
$\deg(e_0)=-\theta, \deg(f_0)=\theta$, $\deg(q^{\pm h_0/2})=\deg(\gamma^{\pm 1})=0$,
and the above assignment preserves this $P$-grading, cf.~\eqref{eq:grading_compatibility}. Referring to
the explicit form of all $q$-Serre relations, left-hand sides of which are presented in~\cite[(QS4, QS5)]{y},
one can easily see that all of them, besides the cases (7, 8, 11), are homogeneous whose degrees are \underline{not}
in the set $\{\varepsilon_i-\varepsilon_j \,|\, 1\leq i,j\leq N\}$. Hence, they act trivially on the superspace $V$.
We shall now directly check the cases (7, 8, 11) of~\cite[(QS4)]{y}, while~\cite[(QS5)]{y} are analogous.

\smallskip
\noindent
$\bullet$ \emph{Serre relation~\cite[(QS4)(8)]{y}}.
The corresponding relation reads (cf.\ notation~\eqref{eq:q-superbracket})
\begin{equation*}
  [\![ [\![ [\![ e_{j},e_i ]\!] , [\![ e_{j},e_{k} ]\!] ]\!] , [\![ e_{j},e_{l} ]\!] ]\!] =
  [\![ [\![ [\![ e_{j},e_i ]\!] , [\![ e_{j},e_{l} ]\!] ]\!] , [\![ e_{j},e_{k} ]\!] ]\!]
\end{equation*}
and it only occurs for $\fosp(V)=\fosp(4|2)$ in either of the following two cases:
\begin{enumerate}
\item
  parity sequence $\gamma_V=(\bar{1},\bar{0},\bar{0})$ and indices $i=0,j=1,k=2,l=3$;
\item
  parity sequence $\gamma_V=(\bar{0},\bar{0},\bar{1})$ and indices $i=3,j=2,k=0,l=1$.
\end{enumerate}
In case (1), both sides of this equality (LHS and RHS) have $P$-degrees equal to $\varepsilon_1-\varepsilon_2$
and thus act trivially on $v_p$ for $p\notin \{2,1'\}$. By direct calculations, we find:
$\mathrm{LHS}(v_2)=au\vartheta_1\cdot v_1=\mathrm{RHS}(v_2)$ as well as
$\mathrm{LHS}(v_{1'})=q^{-1}au\cdot v_{2'}=\mathrm{RHS}(v_{1'})$.
In case (2), both sides of this equality (LHS and RHS) have $P$-degrees equal to $\varepsilon_2-\varepsilon_3$
and thus act trivially on $v_p$ for $p\notin \{3,2'\}$. By direct calculations, we find:
$\mathrm{LHS}(v_3)=2q^{-1}au\vartheta_3\cdot v_2=\mathrm{RHS}(v_3)$ as well as
$\mathrm{LHS}(v_{2'})=-2qau\cdot v_{3'}=\mathrm{RHS}(v_{2'})$.
This completes our verification of~\cite[(QS4)(8)]{y}.

\smallskip
\noindent
$\bullet$ \emph{Serre relation~\cite[(QS4)(11)]{y}}.
The corresponding relation reads
\begin{equation*}
  [\![ [\![ e_{k},e_j ]\!] , [\![ [\![ e_{k},e_j ]\!] ,  [\![ [\![ e_{k},e_{j} ]\!] , e_i ]\!] ]\!] ]\!] =
  (1-[2]_q) [\![ [\![ [\![ e_{k},e_j ]\!] , [\![ e_k , [\![ e_k , [\![ e_j , e_i ]\!] ]\!] ]\!] ]\!] , e_j ]\!]
\end{equation*}
and it only occurs for $\fosp(V)=\fosp(3|2)$ with the parity sequence $\gamma_V=(\bar{1},\bar{0})$ and $i=0,j=1,k=2$.
Both sides of this equality (LHS and RHS) have $P$-degrees equal to $\varepsilon_1$ and thus act trivially on $v_p$ for
$p\notin \{3,1'\}$. By direct calculations, we find:
$\mathrm{LHS}(v_3)=(1-q^{-1}+q^{-2})au\vartheta_1\cdot v_1=\mathrm{RHS}(v_3)$
as well as $\mathrm{LHS}(v_{1'})=(1-q+q^2)au\cdot v_{3}=\mathrm{RHS}(v_{1'})$.
This completes our verification of~\cite[(QS4)(11)]{y}.

\smallskip
\noindent
$\bullet$ \emph{Serre relation~\cite[(QS4)(7)]{y}}.
The corresponding relation reads
\begin{equation*}
  (-1)^{|\alpha_i| |\alpha_k|} [(\alpha_i,\alpha_k)]_q [\![ [\![ e_i , e_j ]\!] , e_k ]\!] =
  (-1)^{|\alpha_i| |\alpha_j|} [(\alpha_i,\alpha_j)]_q [\![ [\![ e_i , e_k ]\!] , e_j ]\!]
\end{equation*}
whenever $(\alpha_i,\alpha_j)\ne 0, (\alpha_i,\alpha_k)\ne 0, (\alpha_j,\alpha_k)\ne 0$,
$(\alpha_i,\alpha_j)+(\alpha_i,\alpha_k)+(\alpha_j,\alpha_k)=0$, and
$|\alpha_i| |\alpha_j| + |\alpha_i| |\alpha_k| + |\alpha_j| |\alpha_k| = \bar{1}$. We can further assume
that $\{i,j,k\}=\{0,1,2\}$. The above parity condition implies that
$\theta=\varepsilon_1+\varepsilon_2, \alpha_1=\varepsilon_1-\varepsilon_2$, and $|v_1|=\bar{0}, |v_2|=\bar{1}$.
Due to the symmetry $j \leftrightarrow k$ of the above relation, there are three cases to consider:
\begin{enumerate}
\item
  $i=0, j=1, k=2$;
\item
  $i=2, j=0, k=1$;
\item
  $i=1, j=0, k=2$.
\end{enumerate}
In case (1), both sides of this equality (LHS and RHS) have $P$-degrees equal to $-\varepsilon_2-\varepsilon_3$
and thus act trivially on $v_p$ for $p\notin \{2,3\}$. By direct calculations, we find:
$\mathrm{LHS}(v_2)=(1+q^2)au\vartheta_2\vartheta_3\cdot v_{3'}=\mathrm{RHS}(v_2)$
and $\mathrm{LHS}(v_3)=-(-1)^{|v_3|}(1+q^{-2})au\cdot v_{2'}=\mathrm{RHS}(v_{3})$.
In case (2), both sides of this equality (LHS and RHS) have $P$-degrees equal to $-\varepsilon_2-\varepsilon_3$
and thus act trivially on $v_p$ for $p\notin \{2,3\}$. By direct calculations, we find:
$\mathrm{LHS}(v_2) = -au\vartheta_2\vartheta_3 \cdot v_{3'}=\mathrm{RHS}(v_2)$
and $\mathrm{LHS}(v_3) = (-1)^{|v_3|}au\cdot v_{2'}=\mathrm{RHS}(v_{3})$.
In case (3), both sides of this equality (LHS and RHS) have $P$-degrees equal to $-\varepsilon_2-\varepsilon_3$
and thus act trivially on $v_p$ for $p\notin \{2,3\}$. By direct calculations, we find:
$\mathrm{LHS}(v_2)=(1+q^2)au\vartheta_2\vartheta_3 \cdot v_{3'}=\mathrm{RHS}(v_2)$
and $\mathrm{LHS}(v_3)=-(-1)^{|v_3|}(1+q^{-2})au\cdot v_{2'}=\mathrm{RHS}(v_{3})$.
This completes our verification of~\cite[(QS4)(7)]{y}.
\end{proof}

These evaluation $\UqV$-modules $\varrho^{a,b}_u$ can be naturally upgraded to $\UdqV$-modules:

\begin{Prop}\label{prop:non-reduced affine module}
Let $u$ be an indeterminate and redefine $V(u)$ via $V(u) = V \otimes_{\BC} \BC[u,u^{-1}]$. Then,
the formulas defining $\varrho^{a,b}_u$ on the generators from Proposition~\ref{prop:affine-repn} together with
\begin{equation*}
  \varrho^{a,b}_{u}(D^{\pm 1})(v \otimes u^{k}) = q^{\pm k}\cdot v \otimes u^{k}
  \qquad \forall\, v\in V \,, k\in \BZ
\end{equation*}
give rise to the same-named action $\varrho^{a,b}_u$ of $\UdqV$ on $V(u)$.
\end{Prop}

Let $U^+_{q}(\widehat{\fosp}(V))$ and $U^-_{q}(\widehat{\fosp}(V))$ be the subalgebras of $\UdqV$ generated
by $\{e_i\}_{i=0}^s$ and $\{f_i\}_{i=0}^s$, respectively. We also define $U^\geq_{q}(\widehat{\fosp}(V))$ and
$U^\leq_{q}(\widehat{\fosp}(V))$ as subalgebras of $\UdqV$ generated by
$\{e_{i},q^{\pm h_i/2},\gamma^{\pm 1},D^{\pm 1}\}_{i=0}^s$ and $\{f_{i},q^{\pm h_i/2},\gamma^{\pm 1},D^{\pm 1}\}_{i=0}^s$.
We likewise define the subalgebras $U^{',+}_{q}(\widehat{\fosp}(V))$, $U^{',-}_{q}(\widehat{\fosp}(V))$,
$U^{',\geq}_{q}(\widehat{\fosp}(V))$, $U^{',\leq}_{q}(\widehat{\fosp}(V))$ of $\UqV$.
We note that $U^\geq_{q}(\widehat{\fosp}(V)),U^\leq_{q}(\widehat{\fosp}(V)),
U^{',\geq}_{q}(\widehat{\fosp}(V)), U^{',\leq}_{q}(\widehat{\fosp}(V))$ are actually Hopf subalgebras, and moreover
\begin{equation}\label{eq:halves-equal}
  U^+_{q}(\widehat{\fosp}(V))\simeq U^{',+}_{q}(\widehat{\fosp}(V)) \,,\qquad
  U^-_{q}(\widehat{\fosp}(V))\simeq U^{',-}_{q}(\widehat{\fosp}(V)) \,.
\end{equation}
Finally, similarly to Proposition~\ref{prop:pairing_finite}, one has bilinear pairings
\begin{equation}\label{eq:parity-affine}
\begin{split}
  & (\cdot,\cdot)_J\colon
    U^\leq_{q}(\widehat{\fosp}(V))\times U^\geq_{q}(\widehat{\fosp}(V)) \longrightarrow \BC(q^{1/4}) \,, \\
  & (\cdot,\cdot)_J\colon
    U^{',\leq}_{q}(\widehat{\fosp}(V))\times U^{',\geq}_{q}(\widehat{\fosp}(V)) \longrightarrow \BC(q^{1/4}) \,.
\end{split}
\end{equation}
The restrictions of both pairings to $U^{',-}_{q}(\widehat{\fosp}(V))\times U^{',+}_{q}(\widehat{\fosp}(V))$ coincide,
cf.~\eqref{eq:halves-equal}, and are non-degenerate by~\cite{y0,y}, cf.~Remark~\ref{rem:yamane}. However, the second
pairing in~\eqref{eq:parity-affine} is degenerate as $\gamma-1$ is in its kernel. On the other hand (which is the key
reason to add the generators $D^{\pm 1}$), the first pairing in~\eqref{eq:parity-affine} is non-degenerate,
and hence allows to realize $\UdqV$ as a Drinfeld double of its Hopf subalgebras
$U^\leq_{q}(\widehat{\fosp}(V))$ and $U^\geq_{q}(\widehat{\fosp}(V))$ with respect to the pairing above.

The above discussion yields the universal $R$-matrix for $\UdqV$, which induces intertwiners $V\otimes W\iso W\otimes V$
for suitable $\UdqV$-modules $V,W$, akin to Subsection~\ref{ssec:universal-R}. In order to not overburden the exposition,
we choose to skip the detailed presentation on this standard but rather technical discussion. Instead, we shall now proceed
directly to the main goal of this paper--the evaluation of such intertwiners when $V=\varrho^{a,b}_u$ and $W=\varrho^{a,b}_v$
are the above evaluation modules. In this context, we are looking for $\UdqV$-module intertwiners $\hat{R}(u/v)$ satisfying
\begin{equation}\label{eq:affine-intertwiner}
  \hat{R}(u/v)\circ (\varrho^{a,b}_u\otimes \varrho^{a,b}_v)(x) = (\varrho^{a,b}_v\otimes \varrho^{a,b}_u)(x)\circ \hat{R}(u/v)
\end{equation}
for all $x\in \UdqV$ (equivalently, for all $x\in \UqV$ in the context of $\UqV$-modules). In fact, the space of such
solutions is \underline{one-dimensional} due to the irreducibility of the tensor product $\varrho^{a,b}_u\otimes \varrho^{a,b}_v$
(which still holds when viewing them as $\UqV$-modules as long as $u,v$ are \emph{generic}), in contrast to
Proposition~\ref{prop:highest-weight-vec}. As an immediate corollary,  see~\cite[Proposition~3]{jim}, the operator
$R(u/v) = \tau \circ \hat{R}(u/v)$ satisfies the Yang-Baxter relation with a spectral parameter:
\begin{equation}\label{eq:qYB-affine}
\begin{split}
  & R_{12}(v/w) R_{13}(u/w) R_{23}(u/v) = R_{23}(u/v) R_{13}(u/w) R_{12}(v/w) \,, \\
  & \hat{R}_{12}(v/w) \hat{R}_{23}(u/w) \hat{R}_{12}(u/v) = \hat{R}_{23}(u/v) \hat{R}_{12}(u/w) \hat{R}_{23}(v/w) \,.
\end{split}
\end{equation}

We shall now present the explicit formula for such $\hat{R}(z)$, which is the main result of this note:

\begin{Thm}\label{thm:R_osp_affine}
For any $u,v$, set $z=u/v$. For $\UdqV$-modules $\varrho_u^{a,b}, \varrho_v^{a,b}$ from
Proposition~\ref{prop:non-reduced affine module} (with the specified value of $ab$), the operator
$\hat{R}(z) = \tau \circ R(z)$ satisfies~\eqref{eq:affine-intertwiner}, where
\begin{multline}\label{eq:spectral-R}
  R(z) = (z - q^{-m+n+2}) \left\{ \ID + (q^{1/2} - q^{-1/2}) \sum_{i=1}^N (-1)^{\ol{i}} E_{ii} \otimes
    \Big(q^{(\varepsilon_{i},\varepsilon_{i})/2} E_{ii} - q^{-(\varepsilon_{i},\varepsilon_{i})/2} E_{i'i'} \Big) \right. \\
  \qquad \qquad \qquad \qquad \quad + \left. (q - q^{-1}) \sum_{i > j} (-1)^{\ol{j}} E_{ij} \otimes
   \Big( E_{ji} - (-1)^{\ol{j}(\ol{i}+\ol{j})} \vartheta_{i}\vartheta_{j} q^{(\rho, \varepsilon_{i} - \varepsilon_{j})} E_{i'j'} \Big) \right\}
 \\ + (q - q^{-1})\frac{z - q^{-m+n+2}}{z-1} \tau
 - (q - q^{-1})q^{-m+n+2} \sum_{i,j=1}^N (-1)^{\ol{i}\,\ol{j}} \vartheta_{i}\vartheta_{j} q^{(\rho, \varepsilon_{i} - \varepsilon_{j})} \cdot
   E_{ij} \otimes E_{i'j'} \,.
\end{multline}
\end{Thm}

Combining this result with the preceding paragraph, we conclude that $\hat{R}(z)$ coincides, up to a prefactor, with the
action of the universal $R$-matrix, and thus $R(z)$ of~\eqref{eq:spectral-R} does satisfy~\eqref{eq:qYB-affine}.

\begin{Rem}
We note that rescaling $R(z)$ of~\eqref{eq:spectral-R} by the factor $\frac{1}{z - q^{-m+n+2}}$ and further specializing
at $z = 0$ and $\infty$, we recover our finite $R$-matrices $R_{0}$ and $R_{\infty}$ from~\eqref{eq:R_0} and~\eqref{eq:R_inf},
respectively.
\end{Rem}

\begin{Rem}\label{rem:trig-to-rational-osp}
We note that rescaling $R(z)$ of~\eqref{eq:spectral-R} by $\frac{1}{z - q^{-m+n+2}}$, setting
$q = e^{-\hbar/2}, z = e^{\hbar u}$, and further taking the limit $\hbar \to 0$ recovers
the rational $R$-matrix of~\cite[(3.4)]{ft} (first considered in~\cite{aacfr} for the standard
parity sequence) used to define the orthosymplectic superYangian $Y(\fosp(V))$:
\begin{equation*}
  \lim_{\hbar \to 0} \left\{ {\frac{R(z)}{z - q^{-m+n+2}}}\middle|_{q = e^{-\hbar/2}, z = e^{\hbar u}} \right\} =
  \ID - \frac{\tau}{u} + \frac{1}{u-\frac{m-n-2}{2}}
    \sum_{i,j=1}^N (-1)^{\ol{i}\,\ol{j}} \vartheta_{i}\vartheta_{j}\cdot E_{ij} \otimes E_{i'j'} \,.
\end{equation*}
\end{Rem}

\begin{Rem}
For the standard parity sequence $\gamma_{V} = (\bar{1}, \ldots, \bar{1}, \bar{0}, \ldots, \bar{0})$, the
exact relation between our formula~\eqref{eq:spectral-R} and the $R$-matrix $R^{[\mathrm{MDGL}]}(z)$ of~\cite{mdgl}
is given by:
\begin{equation*}
  R(z) = \frac{(qz-q^{-1})(z-q^{-m+n+2})}{z-1} R^{[\mathrm{MDGL}]}(1/z) \,.
\end{equation*}
We note that the change of the spectral parameter $z=u/v\mapsto v/u=1/z$ above is simply due to the order of
the tensorands $V(u)$ and $V(v)$.
\end{Rem}

The proof of Theorem~\ref{thm:R_osp_affine} is straightforward and crucially relies on the expression of $R(z)$ from \eqref{eq:spectral-R}
through $R_0,R_\infty$ of~(\ref{eq:R_0},~\ref{eq:R_inf}), which is a special case of the \emph{Yang-Baxterization} from~\cite{gwx}.


\subsection{Yang-Baxterization}
\

In this Subsection, we express $R(z)$ via $R_0$ and $R_\infty$ through the Yang-Baxterization procedure of~\cite{gwx}.
This formal procedure produces $\hat{R}(z)$ satisfying~\eqref{eq:qYB-affine} from $\hat{R}$ satisfying~\eqref{eq:qYB-two}
when the latter has at most $3$ eigenvalues. In our setup, the $R$-matrices $\hat{R}_{VV}=\hat{R}=\tau_{VV}R_0$ have only
eigenvalues $\lambda_1,\lambda_2,\lambda_3$, in accordance with Proposition~\ref{prop:eig-calc-R-hat} combined with
Appendix~\ref{ssec:generating-BCD}.\footnote{According to
Propositions~\ref{prop:decomp-B},~\ref{prop:decomp-CD-fork},~\ref{prop:decomp-CD-nofork} and Proposition~\ref{prop:eig-calc-R-hat},
$\hat{R}_{VV}$ acts with two eigenvalues $\lambda_1,\lambda_2$ on the codimension $1$ submodule $W^+\oplus W^-$ of $V\otimes V$,
where $W^+,W^-$ are $\uqV$-submodules generated by $w_1,w_2$, respectively. Finally, $\hat{R}_{VV}$ acts on the
1-dimensional quotient space $V \otimes V/(W^{+} \oplus W^{-})$
via multiplication by $\lambda_3=q^{m-n-1}$, due
to~(\ref{eq:explicit-comb-subspace-1},~\ref{eq:explicit-comb-subspace-2},~\ref{eq:explicit-comb-subspace-3})
and Proposition~\ref{prop:eig-calc-R-hat}(b).}
In that setup, the Yang-Baxterization of~\cite[(3.29),~(3.31)]{gwx} produces
the following two solutions to~\eqref{eq:qYB-affine}:
\begin{equation*}
  \hat{R}^{(1)}(z) = \lambda_{1}z(z - 1)\hat{R}^{-1} +
  \left(1 + \frac{\lambda_{1}}{\lambda_{2}} + \frac{\lambda_{1}}{\lambda_{3}} + \frac{\lambda_{1}^{2}}{\lambda_{2}\lambda_{3}}\right )z\ID
  - \frac{\lambda_{1}}{\lambda_{2}\lambda_{3}}(z - 1)\hat{R}
\end{equation*}
and
\begin{equation*}
  \hat{R}^{(2)}(z) =
  \lambda_{1}z(z - 1)\hat{R}^{-1} +
  \left(1 + \frac{\lambda_{1}}{\lambda_{2}} + \frac{\lambda_{1}}{\lambda_{3}} + \frac{\lambda_{2}}{\lambda_{3}}\right) z\ID
  - \frac{1}{\lambda_{3}}(z - 1)\hat{R}
\end{equation*}
provided that $\hat{R}$ satisfies the additional relations of~\cite[(3.27)]{gwx} (cf.\ correction~\cite[(A.9)]{gwx}),
which, in particular, hold whenever $\hat{R}$ is a representation of a \emph{Birman-Wenzl algebra}.

\begin{Rem}
For our purpose, we shall not really need to verify these additional relations, since according to Theorem~\ref{thm:R_osp_affine}
the constructed $\hat{R}(z)$ do manifestly satisfy the relation~\eqref{eq:qYB-affine}.
\end{Rem}

\begin{Prop}\label{prop:Yang-Baxterization}
The affine $R$-matrix \eqref{eq:spectral-R} coincides (up to $\tau$ and a rational function in $z$) with the Yang-Baxterization
of $\hat{R}_{VV}=\tau \circ R_0$, cf.~\eqref{eq:key finite equality}. To be more specific, for $\hat{R}(z) = \tau \circ R(z)$:
\begin{equation}\label{eq:baxterization1}
  \lambda_{1}(z-1) \hat{R}(z) =
  \lambda_{1} z(z-1) \hat{R}^{-1}_{VV}
  + \left(1 + \frac{\lambda_{1}}{\lambda_{2}} + \frac{\lambda_{1}}{\lambda_{3}} + \frac{\lambda_{1}^{2}}{\lambda_{2}\lambda_{3}}\right )z\ID
  - \frac{\lambda_{1}}{\lambda_{2}\lambda_{3}} (z-1) \hat{R}_{VV}
\end{equation}
if $|v_{1}|=\bar{1}$ and
\begin{equation}\label{eq:baxterization2}
  \lambda_{1}(z-1) \hat{R}(z) =
  \lambda_{1} z(z-1) \hat{R}^{-1}_{VV}
  + \left(1 + \frac{\lambda_{1}}{\lambda_{2}} + \frac{\lambda_{1}}{\lambda_{3}} + \frac{\lambda_{2}}{\lambda_{3}}\right) z\ID
  - \frac{1}{\lambda_{3}} (z-1) \hat{R}_{VV}
\end{equation}
if $|v_{1}|=\bar{0}$, with $\lambda_1,\lambda_2,\lambda_3$ precisely as in~\eqref{eq:lambda}.
\end{Prop}

\begin{proof}
By straightforward tedious computations, based on~(\ref{eq:R_0},~\ref{eq:R_inf},~\ref{eq:spectral-R}), one verifies that
\begin{equation}\label{eq:baxterization1-without-tau}
  \lambda_{1}(z-1) R(z) =
  \lambda_{1} z(z-1) R_{\infty}
  + \left(1 + \frac{\lambda_{1}}{\lambda_{2}} + \frac{\lambda_{1}}{\lambda_{3}} +
          \frac{\lambda_{1}^{2}}{\lambda_{2}\lambda_{3}}\right ) z \tau
  - \frac{\lambda_{1}}{\lambda_{2}\lambda_{3}} (z-1) R_{0}
\end{equation}
if $|v_{1}|=\bar{1}$ and
\begin{equation}\label{eq:baxterization2-without-tau}
  \lambda_{1}(z-1) R(z) =
  \lambda_{1} z(z-1) R_{\infty}
  + \left(1 + \frac{\lambda_{1}}{\lambda_{2}} + \frac{\lambda_{1}}{\lambda_{3}} + \frac{\lambda_{2}}{\lambda_{3}}\right)
  z \tau
  - \frac{1}{\lambda_{3}} (z-1) R_{0}
\end{equation}
if $|v_{1}|=\bar{0}$. Composing with $\tau$ on the left, and using~\eqref{eq:key finite equality},
we obtain~(\ref{eq:baxterization1},~\ref{eq:baxterization2}).
\end{proof}


\subsection{Proof of the main result}\label{ssec:main thm proof}
\

Due to Proposition~\ref{prop:Yang-Baxterization} and Theorem~\ref{thm:R_osp_finite}, it only remains to
verify~\eqref{eq:affine-intertwiner} for $x=e_0$ and $x=f_0$. We shall now present the direct verification
for $x=e_0$, while $x=f_0$ can be treated analogously to the finite case using the supertransposition~\eqref{eq:supertranspose}.
Since both sides of~\eqref{eq:affine-intertwiner} for $x = e_{0}$ depend linearly on $a$, without loss of generality,
we shall now assume that~$a = 1$.

For the latter purpose, let us first evaluate $(\rho,\varepsilon_{1})$. Since
\begin{align*}
  2\varepsilon_{1}
  & = (\varepsilon_{1} - \varepsilon_{2}) + (\varepsilon_{2} - \varepsilon_{3}) + \dots + (\varepsilon_{2'} - \varepsilon_{1'}) \\
  &=\begin{cases}
      2\alpha_{1} + \dots + 2\alpha_{s} & \text{if\ } m \text{\ is odd} \\
      2\alpha_{1} + \dots + 2\alpha_{s-2} + \alpha_{s-1} + \alpha_{s} & \text{if\ } m \text{\ is even and\ } \ol{s} = \bar{0} \\
      2\alpha_{1} + \dots + 2\alpha_{s-1} + \alpha_{s} & \text{if\ } m \text{\ is even and\ } \ol{s} = \bar{1} \\
    \end{cases} \,,
\end{align*}
a direct application of~\eqref{eq:rho-simple-root} implies that
\begin{align*}
  2(\rho, \varepsilon_{1})
  &=\begin{cases}
      (-1)^{\ol{1}} + (-1)^{\ol{2}} \cdot 2 + \dots +  (-1)^{\ol{s}} \cdot 2
        & \text{if\ } m \text{\ is odd} \\
      (-1)^{\ol{1}} + (-1)^{\ol{2}} \cdot 2 + \dots + (-1)^{\ol{s-1}} \cdot 2 + (-1)^{\ol{s}}
        & \text{if\ } m \text{\ is even and\ } \ol{s} = \bar{0} \\
      (-1)^{\ol{1}} + (-1)^{\ol{2}} \cdot 2 + \dots + (-1)^{\ol{s-1}} \cdot 2 + (-1)^{\ol{s}} \cdot 3
        & \text{if\ } m \text{\ is even and\ } \ol{s} = \bar{1}
    \end{cases} \\
  & = -(-1)^{\ol{1}} - 1 + (m-n) \,.
\end{align*}
Thus, we have the following uniform formula:
\begin{equation}\label{eq:rho-epsilon1}
  (\rho, \varepsilon_{1}) = \sfrac{1}{2} \big( m-n-1-(-1)^{\ol{1}} \big) \,.
\end{equation}


$\bullet$ Case 1: $|v_{1}| = \bar{1}$.

Since $\varrho^{a,b}_{u}(q^{h_{0}/2})$ is a diagonal matrix, we shall write it as
$\varrho^{a,b}_{u}(q^{h_{0}/2}) = \diag(\sft_{1}, \ldots, \sft_{1'})$. We shall also use the same decomposition
$R_{\infty} = \ID + R_{1} + R_{2} + R_{3} + R_{4}$ as in Subsection \ref{sec:intertwiner-proof}.
By direct computation, we get:
\begin{equation*}
  R_{1}\Delta(e_{0}) = (q^{-1} - 1) \Big( q^{-1} E_{1'1'} \otimes v E_{1'1} + u E_{1'1} \otimes q E_{1'1'} \Big) \,,
\end{equation*}
\begin{equation*}
  \Delta^{\opp}(e_{0})R_{1} = (q^{-1} - 1) \Big( q^{-1} E_{11} \otimes v E_{1'1} + u E_{1'1} \otimes q E_{11} \Big) \,,
\end{equation*}
\begin{equation*}
  R_{2}\Delta(e_{0}) = -(1 - q) \Big( q E_{11} \otimes v E_{1'1} + u E_{1'1} \otimes q^{-1} E_{11} \Big) \,,
\end{equation*}
\begin{equation*}
  \Delta^{\opp}(e_{0})R_{2} = -(1 - q) \Big( q E_{1'1'} \otimes v E_{1'1} + u E_{1'1} \otimes q^{-1} E_{1'1'} \Big) \,,
\end{equation*}
\begin{equation*}
  R_{3}\Delta(e_{0}) = (q - q^{-1}) \sum_{1\leq j\leq N} (-1)^{\ol{j}} (\sft_{j} E_{1'j} \otimes v E_{j1}) +
  (q - q^{-1}) (q^{-1} E_{1'1'} \otimes v E_{1'1}) \,,
\end{equation*}
\begin{equation*}
  \Delta^{\opp}(e_{0})R_{3} = (q - q^{-1}) \sum_{1\leq i\leq N} (-1)^{\ol{1}} (\sft_{i}^{-1} E_{i1} \otimes v E_{1'i}) +
  (q - q^{-1}) (q^{-1} E_{11} \otimes v E_{1'1}) \,,
\end{equation*}
\begin{equation*}
  R_{4}\Delta(e_{0}) =
  - (q - q^{-1}) \sum_{1\leq i\leq N} (-1)^{\ol{i}} \vartheta_{i}\vartheta_{1}
  q^{(\rho, \varepsilon_{i} - \varepsilon_{1})} (q E_{i1} \otimes v E_{i'1})
  - (q - q^{-1}) (q E_{11} \otimes v E_{1'1}) \,,
\end{equation*}
\begin{equation*}
  \Delta^{\opp}(e_{0})R_{4} =
  - (q - q^{-1}) \sum_{1\leq j\leq N} (-1)^{\ol{j}} \vartheta_{1'}\vartheta_{j}
  q^{(\rho, \varepsilon_{1'} - \varepsilon_{j})} (q E_{1'j} \otimes v E_{1'j'})
  - (q - q^{-1}) (q E_{1'1'} \otimes v E_{1'1}) \,.
\end{equation*}

\smallskip
Assembling all the terms (and using \eqref{eq:rho-epsilon1} for the last two equalities), we get:
\begin{equation*}
  \Delta(e_{0}) - \Delta^{\opp}(e_{0}) =
  (q - q^{-1})v \cdot (E_{11} - E_{1'1'}) \otimes E_{1'1} - (q - q^{-1})u \cdot E_{1'1} \otimes (E_{11} - E_{1'1'}) \,,
\end{equation*}
\begin{equation*}
  R_{1}\Delta(e_{0}) - \Delta^{\opp}(e_{0})R_{1} =
  (q^{-1} - q^{-2})v \cdot (E_{11} - E_{1'1'}) \otimes E_{1'1} + (q - 1)u \cdot E_{1'1} \otimes (E_{11} - E_{1'1'}) \,,
\end{equation*}
\begin{equation*}
  R_{2}\Delta(e_{0}) - \Delta^{\opp}(e_{0})R_{2} =
  (q^{2} - q)v \cdot (E_{11} - E_{1'1'}) \otimes E_{1'1} + (1 - q^{-1})u \cdot E_{1'1} \otimes (E_{11} - E_{1'1'}) \,,
\end{equation*}
\begin{multline*}
  R_{3}\Delta(e_{0}) - \Delta^{\opp}(e_{0})R_{3} =
  -(1 - q^{-2})v \cdot (E_{11} - E_{1'1'}) \otimes E_{1'1} \\
    + (q - q^{-1})v \cdot \sum_{1\leq j\leq N} (-1)^{\ol{j}} \sft_{j}\cdot E_{1'j} \otimes E_{j1}
    + (q - q^{-1})v \cdot \sum_{1\leq i\leq N} \sft_{i}^{-1}\cdot E_{i1} \otimes E_{1'i} \,,
\end{multline*}
\begin{align*}
  R_{4}\Delta(e_{0}) - \Delta^{\opp}(e_{0})R_{4} =
  & -(q^{2} - 1)v \cdot (E_{11} - E_{1'1'}) \otimes E_{1'1}\\
  & - (q - q^{-1})v \cdot q^{-(m-n-2)/2} \sum_{1\leq i\leq N} (-1)^{\ol{i}} \vartheta_{i}\vartheta_{1} q^{(\rho,\varepsilon_{i})}\cdot
    E_{i1} \otimes E_{i'1}\\
  & + (q - q^{-1})v \cdot q^{-(m-n-2)/2} \sum_{1\leq j\leq N} (-1)^{\ol{j}} \vartheta_{1'}\vartheta_{j} q^{-(\rho,\varepsilon_{j})}\cdot
    E_{1'j} \otimes E_{1'j'} \,.
\end{align*}

Collecting the terms together, we obtain:
\begin{multline}\label{eq:R_infty-e_0}
  R_{\infty}\Delta(e_{0}) - \Delta^{\opp}(e_{0})R_{\infty} =
  - (q - q^{-1})v \cdot q^{-(m-n-2)/2} \sum_{1\leq i\leq N} (-1)^{\ol{i}} \vartheta_{i}\vartheta_{1} q^{(\rho,\varepsilon_{i})}\cdot
    E_{i1} \otimes E_{i'1} \\
  + (q - q^{-1})v \cdot q^{-(m-n-2)/2} \sum_{1\leq j\leq N} (-1)^{\ol{j}} \vartheta_{1'}\vartheta_{j} q^{-(\rho,\varepsilon_{j})}\cdot
    E_{1'j} \otimes E_{1'j'} \\
  + (q - q^{-1})v \cdot \sum_{1\leq j\leq N} (-1)^{\ol{j}} \sft_{j}\cdot E_{1'j} \otimes E_{j1}
  + (q - q^{-1})v \cdot \sum_{1\leq i\leq N} \sft_{i}^{-1}\cdot E_{i1} \otimes E_{1'i}  \,.
\end{multline}

Though one can evaluate $R_{0}\Delta(e_{0}) - \Delta^{\opp}(e_{0})R_{0}$ in a similar way, we shall rather present
a simple derivation of the resulting formula by utilizing the automorphism $\sigma$ of $\uqV$ from~\eqref{eq:invol-sigma}.
To this end, we note that $\sigma$ can be extended to a $\BC$-algebra automorphism of $\UdqV$ by assigning
\begin{equation*}
  \sigma\colon \quad
  e_0 \mapsto e_0 \,,\quad f_0 \mapsto f_0 \,,\quad q^{\pm h_0/2} \mapsto q^{\mp h_0/2} \,,\quad
  \gamma^{\pm 1} \mapsto \gamma^{\mp1} \,,\quad D^{\pm 1} \mapsto D^{\mp 1} \,.
\end{equation*}
Then, equalities \eqref{eq:sigma-opp} still hold, cf.~\eqref{eq:sigma-coproduct}. Therefore, applying
$\bar{\sigma}$ of~\eqref{eq:bar-sigma} to all matrix coefficients in the equality \eqref{eq:R_infty-e_0},
conjugating with $\tau$, and using \eqref{eq:sigma-opp} together with~\eqref{eq:R0-vs-Rinf}, we get:
\begin{multline}\label{eq:R0-tau-e0}
  R_{0}(\tau \Delta^{\opp}(e_{0})\tau^{-1}) - (\tau\Delta(e_{0})\tau^{-1})R_{0} =
  - (q - q^{-1})v \cdot q^{(m-n-2)/2} \sum_{1\leq i\leq N} (-1)^{\ol{i}} \vartheta_{i}\vartheta_{1} q^{(\rho,\varepsilon_{i})}\cdot
    E_{i1} \otimes E_{i'1} \\
  + (q - q^{-1})v \cdot q^{(m-n-2)/2} \sum_{1\leq j\leq N} (-1)^{\ol{j}} \vartheta_{1'}\vartheta_{j} q^{-(\rho,\varepsilon_{j})}\cdot
    E_{1'j} \otimes E_{1'j'} \\
  + (q - q^{-1})v \cdot \sum_{1\leq j\leq N} \sft_{j}^{-1}\cdot E_{j1} \otimes E_{1'j}
    + (q - q^{-1})v \cdot \sum_{1\leq i\leq N} (-1)^{\ol{i}} \sft_{i}\cdot E_{1'i} \otimes E_{i1} \,.
\end{multline}
We also note the following equality of endomorphisms of $V(v)\otimes V(u)$:
\begin{equation*}
  \tau\circ (\varrho^{a,b}_{u} \otimes \varrho^{a,b}_{v})\big(\Delta(x)\big)\circ \tau^{-1} =
  (\varrho^{a,b}_{v} \otimes \varrho^{a,b}_{u}) \big(\Delta^{\opp}(x)\big) \qquad \text{for any} \quad x \in \UdqV \,.
\end{equation*}
Hence, switching the roles of the spectral variables $u$ and $v$ in~\eqref{eq:R0-tau-e0}, we obtain:
\begin{multline}\label{eq:R_0-e_0}
  R_{0}\Delta(e_{0}) - \Delta^{\opp}(e_{0})R_{0} =
  - (q - q^{-1})u \cdot q^{(m-n-2)/2} \sum_{1\leq i\leq N} (-1)^{\ol{i}} \vartheta_{i}\vartheta_{1} q^{(\rho,\varepsilon_{i})}\cdot
    E_{i1} \otimes E_{i'1} \\
  + (q - q^{-1})u \cdot q^{(m-n-2)/2} \sum_{1\leq j\leq N} (-1)^{\ol{j}} \vartheta_{1'}\vartheta_{j} q^{-(\rho,\varepsilon_{j})}\cdot
    E_{1'j} \otimes E_{1'j'} \\
  + (q - q^{-1})u \cdot \sum_{1\leq j\leq N} \sft_{j}^{-1}\cdot E_{j1} \otimes E_{1'j}
    + (q - q^{-1})u \cdot \sum_{1\leq i\leq N} (-1)^{\ol{i}} \sft_{i}\cdot E_{1'i} \otimes E_{i1}  \,.
\end{multline}
Combining~\eqref{eq:R_infty-e_0} and~\eqref{eq:R_0-e_0} with formula~\eqref{eq:baxterization1-without-tau} and the equality
\begin{equation*}
  \tau \Delta(e_{0}) - \Delta^{\opp}(e_{0}) \tau =
  (v-u) \sum_{1\leq j\leq N} (-1)^{\ol{j}} \sft_{j}\cdot E_{1'j} \otimes E_{j1} +
  (v-u) \sum_{1\leq i\leq N} \sft_{i}^{-1}\cdot E_{i1} \otimes E_{1'i} \,,
\end{equation*}
we ultimately get the desired result:
\begin{equation*}
  R(z) \Delta(e_{0}) - \Delta^{\opp}(e_{0}) R(z) = 0 \,.
\end{equation*}

$\bullet$ Case 2: $|v_{1}| = \bar{0}$.

We use the same notations as above. By direct computation, we obtain:
\begin{align*}
  R_{1}\Delta(e_{0}) =
  & \, (q^{1/2} - q^{-1/2}) \bigg\{ (-1)^{\ol{2}} q^{(-1)^{\ol{2}}} v \cdot E_{2'2'} \otimes E_{2'1}
    - (-1)^{\ol{1}} q^{(-1)^{\ol{1}}} \vartheta_{2} v \cdot E_{1'1'} \otimes E_{1'2} \\
  & \qquad\qquad\qquad\quad
    + (-1)^{\ol{2}} u \cdot E_{2'1} \otimes E_{2'2'} - (-1)^{\ol{1}} \vartheta_{2} u \cdot E_{1'2} \otimes E_{1'1'} \bigg\} \,,
\end{align*}
\begin{align*}
  \Delta^{\opp}(e_{0}) R_{1} =
  & \, (q^{1/2} - q^{-1/2}) \bigg\{ (-1)^{\ol{1}} q^{(-1)^{\ol{1}}} v \cdot E_{11} \otimes E_{2'1}
    - (-1)^{\ol{2}} q^{(-1)^{\ol{2}}} \vartheta_{2} v \cdot E_{22} \otimes E_{1'2} \\
  & \qquad\qquad\qquad\quad
    + (-1)^{\ol{1}} u \cdot E_{2'1} \otimes E_{11} - (-1)^{\ol{2}} \vartheta_{2} u \cdot E_{1'2} \otimes E_{22} \bigg\} \,,
\end{align*}
\begin{align*}
  R_{2}\Delta(e_{0}) =
  & -(q^{1/2} - q^{-1/2}) \bigg\{ (-1)^{\ol{2}} q^{-(-1)^{\ol{2}}} v \cdot E_{22} \otimes E_{2'1}
    - (-1)^{\ol{1}} q^{-(-1)^{\ol{1}}} \vartheta_{2} v \cdot E_{11} \otimes E_{1'2} \\
  & \qquad\qquad\qquad\qquad
    + (-1)^{\ol{2}} u \cdot E_{2'1} \otimes E_{22} - (-1)^{\ol{1}} \vartheta_{2} u \cdot E_{1'2} \otimes E_{11} \bigg\} \,,
\end{align*}
\begin{align*}
  \Delta^{\opp}(e_{0}) R_{2} =
  & -(q^{1/2} - q^{-1/2}) \bigg\{ (-1)^{\ol{1}} q^{-(-1)^{\ol{1}}} v \cdot E_{1'1'} \otimes E_{2'1}
    - (-1)^{\ol{2}} q^{-(-1)^{\ol{2}}} \vartheta_{2} v \cdot E_{2'2'} \otimes E_{1'2} \\
  & \qquad\qquad\qquad\qquad
    + (-1)^{\ol{1}} u \cdot E_{2'1} \otimes E_{1'1'} - (-1)^{\ol{2}} \vartheta_{2} u \cdot E_{1'2} \otimes E_{2'2'} \bigg\} \,,
\end{align*}
\begin{align*}
  R_{3}\Delta(e_{0}) =
  & \, (q - q^{-1}) \bigg\{ \sum_{1\leq j\leq N} (-1)^{\ol{j}} \sft_{j} v \cdot E_{2'j} \otimes E_{j1}
    - (-1)^{\ol{2}} q^{(-1)^{\ol{2}}/2} v \cdot E_{2'2'} \otimes E_{2'1} \\
  & \qquad\qquad\quad - (-1)^{\ol{1}} q^{(-1)^{\ol{1}}/2} v \cdot E_{2'1'} \otimes E_{1'1}
    - \sum_{1\leq j\leq N} (-1)^{\ol{j}} \vartheta_{2} \sft_{j} v \cdot E_{1'j} \otimes E_{j2} \\
  & \qquad\qquad\quad + (-1)^{\ol{1}} q^{(-1)^{\ol{1}}/2} \vartheta_{2} v \cdot E_{1'1'} \otimes E_{1'2}
    + (-1)^{\ol{1}} q^{-(-1)^{\ol{1}}/2} u \cdot E_{1'1} \otimes E_{2'1'} \bigg\} \,,
\end{align*}
\begin{align*}
  \Delta^{\opp}(e_{0}) R_{3} =
  & \, (q - q^{-1}) \bigg\{ \sum_{1\leq i\leq N} (-1)^{\ol{2}\,\ol{i}} \sft_{i}^{-1} v \cdot E_{i1} \otimes E_{2'i}
    - q^{(-1)^{\ol{1}}/2}v \cdot E_{11} \otimes E_{2'1} \\
  & \qquad\qquad\quad - \sum_{1\leq i\leq N} (-1)^{\ol{2}\,\ol{i}} \vartheta_{2} \sft_{i}^{-1} v \cdot E_{i2} \otimes E_{1'i}
    + q^{(-1)^{\ol{1}}/2} \vartheta_{2} v \cdot E_{12} \otimes E_{1'1} \\
  & \qquad\qquad\quad + (-1)^{\ol{2}} q^{(-1)^{\ol{2}}/2} \vartheta_{2} v \cdot E_{22} \otimes E_{1'2}
    - (-1)^{\ol{1}} q^{-(-1)^{\ol{1}}/2} \vartheta_{2} u \cdot E_{1'1} \otimes E_{12} \bigg\} \,,
\end{align*}
\begin{align*}
  & R_{4} \Delta(e_{0}) = \\
  & -(q - q^{-1})
    \bigg\{ \sum_{1\leq i\leq N} (-1)^{\ol{i}\,\ol{2}} \vartheta_{i}\vartheta_{2} q^{(\rho, \varepsilon_{i})}
    q^{(-1)^{\ol{1}}/2} q^{-(m-n-2)/2} v \cdot
      E_{i2} \otimes E_{i'1} - q^{(-1)^{\ol{1}}/2} \vartheta_{2} v \cdot E_{12} \otimes E_{1'1} \\
  & \quad\qquad\qquad - (-1)^{\ol{2}} q^{-(-1)^{\ol{2}}/2} v \cdot E_{22} \otimes E_{2'1}
    - \sum_{1\leq i\leq N} \vartheta_{i}\vartheta_{2} q^{(\rho, \varepsilon_{i})}
    q^{-(-1)^{\ol{1}}/2} q^{-(m-n-2)/2} v \cdot E_{i1} \otimes E_{i'2} \\
  & \quad\qquad\qquad + q^{-(-1)^{\ol{1}}/2} \vartheta_{2} v \cdot E_{11} \otimes E_{1'2}
    + q^{-(-1)^{\ol{1}}/2} \vartheta_{2} u \cdot E_{1'1} \otimes E_{12} \bigg\} \,,
\end{align*}
\begin{align*}
  & \Delta^{\opp}(e_{0})R_{4} =
   -(q - q^{-1})
    \bigg\{ \sum_{1\leq j\leq N} (-1)^{\ol{2}\,\ol{j}} \vartheta_{j} q^{-(\rho, \varepsilon_{j})} q^{-(-1)^{\ol{1}}/2} q^{-(m-n-2)/2} v \cdot
      E_{1'j} \otimes E_{2'j'} \\
  & \qquad\qquad\qquad - q^{-(-1)^{\ol{1}}/2} v \cdot E_{1'1'} \otimes E_{2'1}
    - \sum_{1\leq j\leq N} \vartheta_{j} q^{-(\rho, \varepsilon_{j})}
      q^{(-1)^{\ol{1}}/2} q^{-(m-n-2)/2} v \cdot E_{2'j} \otimes E_{1'j'} \\
  & \qquad\qquad\qquad + q^{(-1)^{\ol{1}}/2} v \cdot E_{2'1'} \otimes E_{1'1}
    + q^{-(-1)^{\ol{2}}/2} \vartheta_{2'} v \cdot E_{2'2'} \otimes E_{1'2}
    - q^{-(-1)^{\ol{1}}/2} u \cdot E_{1'1} \otimes E_{2'1'} \bigg\} \,,
\end{align*}
where we used~(\ref{eq:rho-simple-root},~\ref{eq:rho-epsilon1}) in the last two equalities.
Combining the above eight formulas, we get:
\begin{align*}
  R_{\infty}\Delta(e_{0}) &- \Delta^{\opp}(e_{0})R_{\infty} = \\
  & (q - q^{-1}) v \cdot \sum_{1\leq j\leq N} (-1)^{\ol{j}} \sft_{j}\cdot E_{2'j} \otimes E_{j1}
    - (q - q^{-1}) \vartheta_{2} v \cdot \sum_{1\leq j\leq N} (-1)^{\ol{j}} \sft_{j}\cdot E_{1'j} \otimes E_{j2} \\
  & - (q - q^{-1}) v \cdot \sum_{1\leq i\leq N} (-1)^{\ol{2}\,\ol{i}} \sft_{i}^{-1}\cdot E_{i1} \otimes E_{2'i}
    + (q - q^{-1}) \vartheta_{2} v \cdot \sum_{1\leq i\leq N} (-1)^{\ol{2}\,\ol{i}} \sft_{i}^{-1}\cdot E_{i2} \otimes E_{1'i} \\
  & - (q - q^{-1}) \cdot q^{-(m-n-2)/2} \vartheta_{2} v \cdot \sum_{1\leq i\leq N} (-1)^{\ol{i}\,\ol{2}} \vartheta_{i}
      q^{(\rho, \varepsilon_{i})} q^{1/2}\cdot E_{i2} \otimes E_{i'1} \\
  & + (q - q^{-1}) \cdot q^{-(m-n-2)/2} \vartheta_{2} v \cdot \sum_{1\leq i\leq N} \vartheta_{i} q^{(\rho, \varepsilon_{i})} q^{-1/2}\cdot
    E_{i1} \otimes E_{i'2} \\
  & + (q - q^{-1}) \cdot q^{-(m-n-2)/2} v \cdot \sum_{1\leq j\leq N} (-1)^{\ol{2}\,\ol{j}} \vartheta_{j} q^{-(\rho, \varepsilon_{j})} q^{-1/2}\cdot
    E_{1'j} \otimes E_{2'j'} \\
  & - (q - q^{-1}) \cdot q^{-(m-n-2)/2} v \cdot \sum_{1\leq j\leq N} \vartheta_{j} q^{-(\rho, \varepsilon_{j})} q^{1/2}\cdot
    E_{2'j} \otimes E_{1'j'} \,.
\end{align*}
Evoking the paragraph after~\eqref{eq:R_infty-e_0}, we immediately obtain (similarly to Case 1):
\begin{align*}
  R_{0}\Delta(e_{0})& - \Delta^{\opp}(e_{0})R_{0} = \\
  & - (q - q^{-1}) u \cdot \sum_{1\leq j\leq N} (-1)^{\ol{2}\,\ol{j}} \sft_{j}^{-1}\cdot E_{j1} \otimes E_{2'j}
    + (q - q^{-1}) \vartheta_{2} u \cdot \sum_{1\leq j\leq N} (-1)^{\ol{2}\,\ol{j}} \sft_{j}^{-1}\cdot E_{j2} \otimes E_{1'j} \\
  & + (q - q^{-1}) u \cdot \sum_{1\leq i\leq N} (-1)^{\ol{i}} \sft_{i}\cdot E_{2'i} \otimes E_{i1}
    - (q - q^{-1}) \vartheta_{2} u \cdot \sum_{1\leq i\leq N} (-1)^{\ol{i}} \sft_{i}\cdot E_{1'i} \otimes E_{i2} \\
  & + (q - q^{-1}) \cdot q^{(m-n-2)/2} \vartheta_{2} u \cdot \sum_{1\leq i\leq N} \vartheta_{i} q^{(\rho, \varepsilon_{i})} q^{-1/2}\cdot
    E_{i1} \otimes E_{i'2} \\
  & - (q - q^{-1}) \cdot q^{(m-n-2)/2} \vartheta_{2} u \cdot \sum_{1\leq i\leq N} (-1)^{\ol{i}\,\ol{2}} \vartheta_{i}
      q^{(\rho, \varepsilon_{i})} q^{1/2}\cdot E_{i2} \otimes E_{i'1} \\
  & - (q - q^{-1}) \cdot q^{(m-n-2)/2} u \cdot \sum_{1\leq j\leq N} \vartheta_{j} q^{-(\rho, \varepsilon_{j})} q^{1/2}\cdot
      E_{2'j} \otimes E_{1'j'} \\
  & + (q - q^{-1}) \cdot q^{(m-n-2)/2} u \cdot \sum_{1\leq j\leq N} (-1)^{\ol{2}\,\ol{j}} \vartheta_{j} q^{-(\rho, \varepsilon_{j})} q^{-1/2}\cdot
      E_{1'j} \otimes E_{2'j'} \,.
\end{align*}
Likewise, we also obtain:
\begin{multline*}
  \tau \Delta(e_{0}) - \Delta^{\opp}(e_{0}) \tau =
   \, (v-u) \cdot \left\{ \sum_{1\leq j\leq N} (-1)^{\ol{j}} \sft_{j} E_{2'j} \otimes E_{j1} \right. \\
   \left. - \sum_{1\leq j\leq N} (-1)^{\ol{j}}\vartheta_{2} \sft_{j} E_{1'j} \otimes E_{j2}
    - \sum_{1\leq i\leq N} (-1)^{\ol{i}\,\ol{2}} \sft_{i}^{-1} E_{i1} \otimes E_{2'i}
    + \sum_{1\leq i\leq N} (-1)^{\ol{i}\,\ol{2}}\vartheta_{2} \sft_{i}^{-1} E_{i2} \otimes E_{1'i} \right\} \,.
\end{multline*}
Combining the above three equalities with formula~\eqref{eq:baxterization2-without-tau},
we ultimately get the desired result:
\begin{equation*}
  R(z) \Delta(e_{0}) - \Delta^{\opp}(e_{0}) R(z) = 0 \,.
\end{equation*}
This completes the proof of Theorem~\ref{thm:R_osp_affine}.


\appendix


\section{A-type counterpart}\label{sec:app_super-A-type}

In this Appendix, we present an analogous (though simpler) derivation of both finite and affine $R$-matrices
associated with the first fundamental representation of $A$-type quantum supergroups.
While these $R$-matrices are well-known to experts, type $A$ served as a prototype for our treatment of orthosymplectic type.
We note that solutions~\eqref{eq:spectral-R-gl} of the Yang-Baxter equation with a spectral parameter go back to the physics
paper~\cite{ps}, thus preceding the development of quantum groups.


\subsection{A-type Lie superalgebras}
\

We shall follow the notations of Section \ref{ssec:setup} with the exception that we do not assume $n$ to be even
and we do not assume~\eqref{eq:parity-sym}. Recall the Lie superalgebra $\gl(V)$ of Section \ref{ssec:classical}.
The elements $\{E_{ij}\}_{i,j=1}^{N}$ form a basis of $\gl(V)$. We choose the Cartan subalgebra $\fh$ of $\gl(V)$
to consist of all diagonal matrices. Thus, $\{E_{ii}\}_{i=1}^{N}$ is a basis of $\fh$ and $\{\varepsilon_{i}\}_{i=1}^{N}$
is a dual basis of $\fh^{*}$. The computation $[E_{ii},E_{ab}] = (\varepsilon_{a} - \varepsilon_{b})(E_{ii})E_{ab}$
shows that $E_{ab}$ is a root vector corresponding to the root $\varepsilon_{a} - \varepsilon_{b}$. Hence, we get
the \emph{root space decomposition} $\gl(V) = \fh \oplus \bigoplus_{\alpha \in \Phi} \gl(V)_{\alpha}$ with the root system
\begin{equation}\label{eq:root-sys-gl}
  \Phi = \big\{ \varepsilon_{a} - \varepsilon_{b} \,\big|\, a \neq b \big\} \,.
\end{equation}
It decomposes $\Phi = \Phi_{\bar{0}} \cup \Phi_{\bar{1}}$ into \emph{even} and \emph{odd} roots.
We also choose the following polarization:
\begin{equation}\label{eq:polarization-gl}
  \Phi^{+} = \big\{ \varepsilon_{a} - \varepsilon_{b} \,\big|\, a<b \big\} \,, \qquad
  \Phi^{-} = \big\{ \varepsilon_{a} - \varepsilon_{b} \,\big|\, a>b \big\} \,.
\end{equation}
We note that $\bar{\Phi}=\Phi$, cf.~\eqref{eq:reduced_roots}, and all odd roots are isotropic in the present setup.

Consider the non-degenerate supertrace bilinear form $(\cdot,\cdot) \colon \gl(V) \times \gl(V) \to \BC$
defined by $(X,Y) = \mathrm{sTr}(XY)$. Its restriction to the Cartan subalgebra $\fh$ of $\gl(V)$ is non-degenerate,
giving rise to an identification $\fh \simeq \fh^{*}$ via $\varepsilon_{i} \leftrightarrow (-1)^{\ol{i}} E_{ii}$ and
inducing a bilinear form $(\cdot, \cdot)$ on $\fh^{*}$ such that
\begin{equation*}
  (\varepsilon_{i}, \varepsilon_{j}) = \delta_{ij} (-1)^{\ol{i}}
  \qquad \textrm{for\ any} \qquad 1 \leq i,j \leq N \,.
\end{equation*}
Following the above choice of polarization~\eqref{eq:polarization-gl} of the root system~\eqref{eq:root-sys-gl},
the simple roots are $\alpha_{i} = \varepsilon_{i} - \varepsilon_{i+1}\ (1\leq i<N)$ and the corresponding
root vectors are given by:
\begin{equation}\label{eq:Lie-action-gl}
  \sse_{i} = E_{i,i+1} \,,\quad
  \ssf_{i} = (-1)^{\ol{i}} E_{i+1,i} \,,\quad
  \ssh_{i} = (-1)^{\ol{i}} E_{ii} - (-1)^{\ol{i+1}} E_{i+1,i+1}
  \qquad \forall\, 1\leq i<N \,.
\end{equation}
As before, we define the \emph{symmetrized Cartan matrix} $(a_{ij})_{i,j=1}^{N-1}$ via
$a_{ij} = (\alpha_{i},\alpha_{j})$. Then, the above elements $\{\sse_{i},\ssf_i,\ssh_i\}_{i=1}^{N-1}$ are
easily seen to satisfy the Chevalley-type relations:
\begin{equation}\label{eq:chevalley-rel-gl}
  [\ssh_{i}, \ssh_{j}] = 0 \,,\quad [\ssh_{i}, \sse_{j}] = a_{ij} \sse_{j} \,,\quad
  [\ssh_{i}, \ssf_{j}] = -a_{ij} \ssf_{j} \,,\quad [\sse_{i}, \ssf_{j}] = \delta_{ij} \ssh_{i} \,.
\end{equation}

Define a Lie subalgebra $\ssl(V)$ of $\gl(V)$ via $\ssl(V)=\{ x\in \gl(V) \,|\, \mathrm{sTr}(X)=0\}$.
For $m=n$, we note that the identity map $\ID$ belongs to $\ssl(V)$. This basic $A$-type Lie superalgebra
$\ssl(V)$ admits a generators-and-relations presentation, due to~\cite[Main Theorem]{z}.
Explicitly, it is generated by $\{\sse_i,\ssf_i,\ssh_i\}_{i=1}^{N-1}$, with the $\BZ_2$-grading
\begin{equation}\label{eq:Z2-grading-gl}
  |\sse_i|=|\ssf_i|=
  \begin{cases}
     \bar{0} & \mbox{if } \alpha_i\in \Phi_{\bar{0}} \\
     \bar{1} & \mbox{if } \alpha_i\in \Phi_{\bar{1}}
  \end{cases} \,,
  \qquad |\ssh_i|=\bar{0} \,,
\end{equation}
with the defining relations~\eqref{eq:chevalley-rel-gl} as well as the following \emph{Serre relations}:
\begin{align}
  & [\sse_i,\sse_j]=0 \,,\quad [\ssf_i,\ssf_j]=0
    \qquad \mathrm{if}\quad a_{ij}=0\,,
    \label{eq:A-classical-Serre-1} \\
  & [\sse_i,[\sse_i,\sse_j]]=0 \,,\quad [\ssf_i,[\ssf_i,\ssf_j]]=0 \
    \qquad \mathrm{if}\quad j=i\pm 1 \quad \mathrm{and}\quad \alpha_i\in \Phi_{\bar{0}} \,,
    \label{eq:A-classical-Serre-2} \\
  & [[[\sse_{i-1},\sse_{i}],\sse_{i+1}],\sse_{i}] = 0 \,,\quad
    [[[\ssf_{i-1},\ssf_{i}],\ssf_{i+1}],\ssf_{i}] = 0
    \qquad \mathrm{if}\quad \alpha_{i}\in \Phi_{\bar{1}} \,.
    \label{eq:A-classical-Serre-3}
\end{align}


\subsection{A-type quantum supergroups}
\

The \emph{$A$-type quantum supergroup} $U_{q}(\ssl(V))$ is a natural quantization of the universal enveloping
superalgebra $U(\ssl(V))$. Explicitly, $U_{q}(\ssl(V))$ is a $\BC(q^{\pm 1/2})$-superalgebra
generated by $\{e_{i}, f_{i}, q^{\pm h_{i}/2}\}_{i=1}^{N-1}$, with the $\BZ_{2}$-grading as in~\eqref{eq:Z2-grading-gl},
subject to the analogues of~\eqref{eq:q-chevalley-rel-hh}--\eqref{eq:q-chevalley-rel-ef}:
\begin{align*}
  & [q^{h_{i}/2}, q^{h_{j}/2}] = 0 \,, \qquad
    q^{\pm h_{i}/2} q^{\mp h_{i}/2} = 1 \,, \\
  & q^{h_{i}/2} e_{j} q^{-h_{i}/2} = q^{a_{ij}/2} e_{j} \,, \quad
    q^{h_{i}/2} f_{j} q^{-h_{i}/2} = q^{-a_{ij}/2} f_{j} \,, \\
  & [e_{i}, f_{j}] = \delta_{ij} \frac{q^{h_{i}} - q^{-h_{i}}}{q - q^{-1}} \,,
\end{align*}
as well as the following \emph{$q$-Serre relations} (cf.~\cite[Proposition 10.4.1]{y0}):
\begin{align}
  & [\![e_i,e_j]\!]=0 \,,\quad [\![f_i,f_j]\!]=0
    \qquad \mathrm{if}\quad a_{ij}=0\,,
    \label{eq:A-q-Serre-1} \\
  & [\![e_i,[\![e_i,e_j]\!]]\!]=0 \,,\quad [\![f_i,[\![f_i,f_j]\!]]\!]=0 \
    \qquad \mathrm{if}\quad j=i\pm 1 \quad \mathrm{and}\quad \alpha_i\in \Phi_{\bar{0}} \,,
    \label{eq:A-q-Serre-2} \\
  & [\![[\![[\![e_{i-1},e_{i}]\!],e_{i+1}]\!],e_i]\!] = 0 \,,\quad
    [\![[\![[\![f_{i-1},f_{i}]\!],f_{i+1}]\!],f_i]\!] = 0
    \qquad \mathrm{if}\quad \alpha_{i}\in \Phi_{\bar{1}} \,.
    \label{eq:A-q-Serre-3}
\end{align}
Here, we use the notation $[\![\cdot,\cdot]\!]$ from~\eqref{eq:q-superbracket}, which relies
on the natural $Q$-grading of $U_{q}(\ssl(V))$ defined analogously to~\eqref{eq:grading algebra},
see~\eqref{eq:grading-A} below.
Moreover, $U_{q}(\ssl(V))$ is equipped with a Hopf superalgebra structure via the same formulas
as in Subsection~\ref{sec:q-orthosymplectic}.


\subsection{First fundamental representations}
\

Using the notation of Section~\ref{sec:column_repr}, we have the following analogue of Proposition \ref{prop:fin-repn}:

\begin{Prop}\label{prop:fin-repn-gl}
The following defines a representation $\varrho \colon U_{q}(\ssl(V)) \to \End(V)$:
\begin{equation*}
  \varrho(e_{i}) = \sse_{i} \,, \qquad \varrho(f_{i}) = \ssf_{i} \,, \qquad \varrho(q^{\pm h_{i}/2}) = q^{\pm \ssh_{i}/2}
  \qquad \text{for} \quad 1 \leq i < N \,,
\end{equation*}
where $\{\sse_{i}, \ssf_{i}, \ssh_{i}\}_{i=1}^{N-1}$ denote the action of Chevalley-type generators of $\ssl(V)$
given by~\eqref{eq:Lie-action-gl}.
\end{Prop}

The proof of this result is analogous (but simpler) to that of Proposition~\ref{prop:fin-repn}.
In particular, all $q$-Serre relations hold for degree reasons. Here, we note that both the algebra $U_{q}(\ssl(V))$
and the vector space $V$ have compatible (cf.~\eqref{eq:grading_compatibility}) grading by
$P=\bigoplus_{i=1}^N \BZ \varepsilon_i$ via (for any $a<N, i\leq N$):
\begin{equation}\label{eq:grading-A}
  \deg(e_a)=\varepsilon_a-\varepsilon_{a+1} \,,\quad  \deg(f_a)=-\varepsilon_a+\varepsilon_{a+1} \,,\quad
  \deg(q^{h_a/2})=0 \,,\quad \deg(v_i)=\varepsilon_i \,.
\end{equation}


\subsection{Tensor square of the first fundamental representation}
\

The following result is an $A$-type analogue of Proposition~\ref{prop:highest-weight-vec}:

\begin{Prop}\label{prop:highest-weight-vec-gl}
(a) The following are highest weight vectors in $U_{q}(\ssl(V))$-module $V\otimes V$:
\begin{equation}\label{eq:w-vectors-gl}
  w_{1} = v_{1} \otimes v_{1} \,, \qquad
  w_{2} = v_{1} \otimes v_{2} - (-1)^{\ol{1}(\ol{1}+\ol{2})} q^{(-1)^{\ol{1}}} \cdot v_{2} \otimes v_{1} \,.
\end{equation}

\noindent
(b) The $U_{q}(\ssl(V))$-representation $V \otimes V$ is generated by these vectors $w_1,w_2$ of~\eqref{eq:w-vectors-gl}.
\end{Prop}

\begin{proof}
(a) Let us show that the vectors $w_{1}$ and $w_{2}$ are indeed highest weight vectors for the action $\varrho^{\otimes 2}$
of $U_{q}(\ssl(V))$ on $V\otimes V$. First, we note that these vectors are eigenvectors with respect to $q^{h_i/2}$:
\begin{equation*}
  \varrho^{\otimes 2}(q^{h_i/2})w_1=q^{2\varepsilon_1(\ssh_i/2)}w_1 \,, \quad
  \varrho^{\otimes 2}(q^{h_i/2})w_2=q^{(\varepsilon_1+\varepsilon_2)(\ssh_i/2)}w_2  \qquad \forall\, 1\leq i < N \,.
\end{equation*}
It remains to verify that $w_1$ and $w_2$ are annihilated by all $\varrho^{\otimes 2}(e_i)$. The equality
$\varrho^{\otimes 2}(e_i)(w_1)=0$ follows from $\varrho(e_i)(v_1)=0$. Likewise, $\varrho^{\otimes 2}(e_i)(w_2)=0$
for $i>1$ follows from $\varrho(e_i)v_1=\varrho(e_i)v_2=0$. Meanwhile, combining $\varrho(e_1)v_2=v_1$, $\varrho(e_1)v_1=0$,
$\varrho(q^{h_{1}/2})v_1=q^{(-1)^{\ol{1}}/2} v_1$, and~\eqref{eq:comult}, we also get:
\begin{align*}
  \varrho^{\otimes 2}(e_{1})w_{2}
  &= (\varrho(q^{h_{1}/2}) \otimes \varrho(e_{1}))(v_{1} \otimes v_{2}) -
     (-1)^{\ol{1}(\ol{1}+\ol{2})} q^{(-1)^{\ol{1}}} (\varrho(e_{1}) \otimes \varrho(q^{-h_{1}/2}))(v_{2} \otimes v_{1}) \\
  &= \Big((-1)^{(\ol{1}+\ol{2})\ol{1}} \cdot q^{(-1)^{\ol{1}}/2} -
           (-1)^{\ol{1}(\ol{1}+\ol{2})} q^{(-1)^{\ol{1}}} \cdot q^{-(-1)^{\ol{1}}/2} \Big)\cdot v_{1} \otimes v_{1} = 0 \,.
\end{align*}

(b) Part (b) is established in Proposition~\ref{prop:decomp-A} from Appendix~\ref{sec:app_generating}.
\end{proof}


\subsection{Explicit finite R-matrices}
\

Let $\rho$ be the Weyl vector of $\Phi$, defined by the same formula~\eqref{eq:weyl-vec}. We note that it still
satisfies~\eqref{eq:rho-simple-root}. We also define the $U_{q}(\ssl(V))$-module isomorphism
$\hat{R}_{VV}\colon V\otimes V\iso V\otimes V$ precisely as in Proposition~\ref{prop:universal-R}.
The following is an $A$-type counterpart of Theorem~\ref{thm:R_osp_finite}:

\begin{Thm}\label{thm:R_gl_finite}
The $U_{q}(\ssl(V))$-module isomorphism $\hat{R}_{VV}\colon V \otimes V \iso V \otimes V$ and its inverse
$\hat{R}_{VV}^{-1}$ for the $U_{q}(\ssl(V))$-module $V$ constructed in Proposition~\ref{prop:fin-repn-gl}
are given by
\begin{equation}\label{eq:key finite equality A-type}
  \hat{R}_{VV}=\tau_{VV}\circ R_0   \qquad \mathrm{and} \qquad  \hat{R}_{VV}^{-1}=\tau_{VV}\circ R_\infty
\end{equation}
with the following explicit operators
\begin{equation}\label{eq:R_0-gl}
  R_{0} = \ID + (q^{-1/2} - q^{1/2}) \sum_{i=1}^N (-1)^{\ol{i}} q^{-(\varepsilon_{i},\varepsilon_{i})/2} E_{ii} \otimes E_{ii}
  + (q^{-1} - q) \sum_{i<j} (-1)^{\ol{j}} E_{ij} \otimes E_{ji} \,,
\end{equation}
\begin{equation}\label{eq:R_inf-gl}
  R_{\infty} = \ID + (q^{1/2} - q^{-1/2}) \sum_{i=1}^N (-1)^{\ol{i}} q^{(\varepsilon_{i},\varepsilon_{i})/2} E_{ii} \otimes E_{ii}
  + (q - q^{-1}) \sum_{i>j} (-1)^{\ol{j}} E_{ij} \otimes E_{ji} \,.
\end{equation}
\end{Thm}

\begin{Rem}
We note that all the summands in~(\ref{eq:R_0-gl},~\ref{eq:R_inf-gl}) already featured in~(\ref{eq:R_0},~\ref{eq:R_inf}).
\end{Rem}

\begin{Rem}
In analogy to Remark \ref{rem:R-usual-coproduct}, let us also present here the formula for the operator
$R=\Theta\circ \wtd{f}$ and its inverse $R^{-1}$, corresponding to the usual coproduct
$\Delta^{J}$ of~\eqref{eq:Jantzen-comult}, as follows:
\begin{align*}
  R
  &= \wtd{f}^{-1/2} \circ R_{0} \circ \wtd{f}^{1/2} \\
  &= \ID + (q^{-1/2} - q^{1/2}) \sum_{i=1}^N (-1)^{\ol{i}} q^{-(\varepsilon_{i},\varepsilon_{i})/2}
     E_{ii} \otimes E_{ii} + (q^{-1} - q) \sum_{i<j} (-1)^{\ol{j}} E_{ij} \otimes E_{ji} = R_0 \,, \\
  R^{-1}
  &= \tau \circ \wtd{f}^{-1/2}  \circ R_{\infty} \circ \wtd{f}^{1/2} \circ \tau \\
  & = \ID + (q^{1/2} - q^{-1/2}) \sum_{i=1}^N (-1)^{\ol{i}} q^{(\varepsilon_{i},\varepsilon_{i})/2} E_{ii} \otimes E_{ii}
      + (q - q^{-1}) \sum_{i<j} (-1)^{\ol{j}} E_{ij} \otimes E_{ji} \,.
\end{align*}
\end{Rem}

The proof is analogous to that of Theorem~\ref{thm:R_osp_finite} and follows from the next four Propositions.

\begin{Prop}\label{prop:intertwiner-gl}
For any element $x \in U_{q}(\ssl(V))$, the following equalities hold (cf.~\eqref{eq:coproduct-opp}):
\begin{equation*}
  R_{0} \Delta(x) = \Delta^{\opp}(x) R_{0} \qquad \text{and} \qquad R_{\infty} \Delta(x) = \Delta^{\opp}(x) R_{\infty} \,.
\end{equation*}
\end{Prop}

\begin{proof}
We shall only verify $R_{\infty} \Delta(x) = \Delta^{\opp}(x) R_{\infty}$ when $x=e_a$
(the proof for the other generators $x = q^{\pm h_{a}/2}, f_{a}$ as well as for $R_{0}$ instead of $R_\infty$
is completely analogous to our treatment in the proof of Proposition \ref{prop:intertwiner}). Since $q^{\ssh_{a}/2}$
is a diagonal matrix, we shall write it as $q^{\ssh_{a}/2} = \diag(\sft_{1}, \ldots, \sft_{N})$.
By direct computation, we get:
\begin{align*}
  R_{\infty}\Delta(e_{a}) =&\, (\varrho \otimes \varrho)(\Delta(e_{a})) + (q^{(-1)^{\ol{a}}} - 1)
  \bigg\{ q^{(-1)^{\ol{a}}/2} \cdot E_{aa} \otimes E_{a,a+1} + q^{-(-1)^{\ol{a}}/2} \cdot E_{a,a+1} \otimes E_{aa} \bigg\} \\
  & + (q - q^{-1}) \sum_{j < a} (-1)^{\ol{j}} \sft_{j} \cdot E_{aj} \otimes E_{j,a+1} \\
  & + (q - q^{-1}) \sum_{i > a} (-1)^{\ol{a}}(-1)^{(\ol{a}+\ol{i})(\ol{a}+\ol{a+1})} \sft_{i}^{-1} \cdot E_{i,a+1} \otimes E_{ai} \,,
\end{align*}
\begin{align*}
  \Delta^{\opp}(e_{a})R_{\infty} =&\, (\varrho \otimes \varrho)(\Delta^{\opp} (e_{a})) \\
  & + (q^{(-1)^{\ol{a+1}}} - 1) \bigg\{ q^{(-1)^{\ol{a+1}}/2} \cdot E_{a+1,a+1} \otimes E_{a,a+1}
      + q^{-(-1)^{\ol{a+1}}/2} \cdot E_{a,a+1} \otimes E_{a+1,a+1} \bigg\} \\
  & + (q - q^{-1}) \sum_{i > a+1} (-1)^{\ol{a+1}}(-1)^{(\ol{a}+\ol{a+1})(\ol{i}+\ol{a+1})} \sft_{i}^{-1} \cdot
      E_{i,a+1} \otimes E_{ai}\\
  & + (q - q^{-1}) \sum_{j < a+1} (-1)^{\ol{j}} \sft_{j} \cdot E_{aj} \otimes E_{j,a+1} \,.
\end{align*}
Combining these two formulas with
\begin{multline*}
  (\varrho\otimes \varrho)\big(\Delta(e_{a}) - \Delta^{\opp}(e_{a})\big) =
  (q^{1/2} - q^{-1/2}) \bigg\{ \Big((-1)^{\ol{a}} E_{aa} - (-1)^{\ol{a+1}} E_{a+1,a+1}\Big) \otimes E_{a,a+1} \\
  - E_{a,a+1} \otimes \Big((-1)^{\ol{a}} E_{aa} - (-1)^{\ol{a+1}} E_{a+1,a+1}\Big) \bigg\} \,,
\end{multline*}
we obtain the desired equality $R_{\infty}\Delta(e_{a}) - \Delta^{\opp}(e_{a})R_{\infty} = 0$ for all $1\leq a<N$.
\end{proof}

Next, we evaluate the eigenvalues of $\tau R_0, \tau R_\infty, \hat{R}_{VV}$ on
the highest weight vectors from~\eqref{eq:w-vectors-gl}.

\begin{Prop}\label{prop:eig-calc-1-gl}
The highest weight vectors $w_{1}$ and $w_{2}$ from~\eqref{eq:w-vectors-gl} are eigenvectors of $\tau_{VV} \circ R_{0}$
with the eigenvalues $\mu_{1}^{0} = (-1)^{\ol{1}} q^{-(-1)^{\ol{1}}}$ and $\mu_{2}^{0} = -(-1)^{\ol{1}} q^{(-1)^{\ol{1}}}$,
respectively (cf.~\eqref{eq:mu_0}).
\end{Prop}

\begin{Prop}\label{prop:eig-calc-2-gl}
The highest weight vectors $w_{1}$ and $w_{2}$ are eigenvectors of $\tau_{VV} \circ R_{\infty}$
with the eigenvalues $\mu_{1}^{\infty} = (-1)^{\ol{1}} q^{(-1)^{\ol{1}}} = 1/\mu_{1}^{0}$ and
$\mu_{2}^{\infty} = -(-1)^{\ol{1}} q^{-(-1)^{\ol{1}}} = 1/\mu_{2}^{0}$, respectively.
\end{Prop}

\begin{Prop}\label{prop:eig-calc-R-hat-gl}
The highest weight vectors $w_{1}$ and $w_{2}$ are eigenvectors of $\hat{R}_{VV}$
with the eigenvalues $\lambda_{1} = (-1)^{\ol{1}} q^{-(-1)^{\ol{1}}} = \mu^0_1$ and
$\lambda_{2} = -(-1)^{\ol{1}} q^{(-1)^{\ol{1}}} = \mu^0_2$, respectively (cf.~\eqref{eq:lambda}).
\end{Prop}

The above three results are proved completely analogously to
Propositions~\ref{prop:eig-calc-1},~\ref{prop:eig-calc-2}, and~\ref{prop:eig-calc-R-hat}.


\subsection{Factorization of finite R-matrices}\label{ssec:factorization-A}
\

For the order $1<2<\dots<N-1$ on the alphabet $I=\{1,2,\dots, N-1\}$, the dominant Lyndon words were computed
in~\cite[Proposition~6.1]{chw}:
\begin{equation*}
  \sL^+ = \big\{ [i \dots j] \,|\, 1\leq i\leq j\leq N-1 \big\} .
\end{equation*}
This results in the following lexicographical order on the (reduced) root system:
\begin{equation*}
  \alpha_{1} < \alpha_{1} + \alpha_{2} < \dots < \alpha_{1} + \dots + \alpha_{N-1} <
  \alpha_{2} < \dots < \alpha_{N-2} < \alpha_{N-2} + \alpha_{N-1} < \alpha_{N-1} \,.
\end{equation*}
Let $\gamma_{ij} = \alpha_{i} + \dots + \alpha_{j}$ for $1 \le i \le j \le N-1$.
Then, the assignment $\gamma\mapsto (\alpha,\beta)$ corresponding to the costandard factorization of Lyndon words
is explicitly given by $\gamma_{ij} \mapsto (\gamma_{i,j-1},\alpha_{j})$ for $i<j$.

The following is the counterpart of Lemmas \ref{lem:chw-B-pairing}--\ref{lem:chw-CD-pairing-nofork},
which has been carried out in~\cite[\S6.1]{chw}.

\begin{Lem}\label{lem:chw-A-pairing}
For $\ell = [i \ldots j]$ with $1 \leq i < j \leq N-1$, we have:
\begin{equation*}
  (R_{\ell},R_{\ell})^{\chw} = \prod_{k=i}^{j-1} (\alpha_{k},\alpha_{k+1}) \cdot (q-q^{-1})^{j-i} \cdot q^{N(\deg \ell)} .
\end{equation*}
\end{Lem}

Combining this lemma with \eqref{eq:J-vs-CHW-roots}, we obtain the following counterpart of
Lemmas \ref{lem:B-pairing}--\ref{lem:CD-pairing-nofork}:

\begin{Lem}\label{lem:A-pairing}
If $\gamma = \varepsilon_{i}-\varepsilon_{j}$ with $1 \leq i < j \leq N$, then
\begin{equation*}
  (f_{\gamma},e_{\gamma})_J = (-1)^{\ol{i}+\cdots+\ol{j}} \cdot (q^{-1}-q)^{-1}\,.
\end{equation*}
\end{Lem}

The rest of this Subsection proceeds completely analogously to Subsection~\ref{ssec:R-Matrix-Computation},
hence we shall only state the results (skipping identical proofs). First, we compute the action of the root vectors:

\begin{Lem}\label{lem:A-root-vectors-action}
For $\gamma = \varepsilon_{i}-\varepsilon_{j}$ with $1 \leq i < j \leq N$, we have:
\begin{equation*}
  \varrho(e_{\gamma}) = E_{ij} \,,\qquad
  \varrho(f_{\gamma}) = (-1)^{\ol{i}+\cdots+\ol{j-1}} \cdot E_{ji} \,.
\end{equation*}
\end{Lem}

Combining the above two lemmas, we obtain:

\begin{Lem}\label{lem:A-local-operators}
The operators $\Theta_{\gamma}$ of~\eqref{eq:local-Theta} act on the $U_{q}(\ssl(V))$-module $V \otimes V$ as follows:
\begin{equation*}
  \Theta_{\varepsilon_{i} - \varepsilon_{j}} = \ID - (-1)^{\ol{j}}(q-q^{-1}) \cdot E_{ij} \otimes E_{ji}
  \qquad \forall\, 1 \leq i < j \leq N \,.
\end{equation*}
\end{Lem}

For any $1 \leq i \leq N-1$, we now evaluate explicitly the product
$\Theta_{i} = \Theta_{\varepsilon_{i}-\varepsilon_{N}} \cdots \Theta_{\varepsilon_{i}-\varepsilon_{i+1}}$:

\begin{Lem}\label{lem:A-type-Theta_i}
For $1 \leq i \leq N-1$, we have:
\begin{equation*}
  \Theta_{i} = \ID - (q - q^{-1}) \sum_{j=i+1}^{N}  (-1)^{\ol{j}} E_{ij} \otimes E_{ji} \,.
\end{equation*}
\end{Lem}

Combining the above result with the factorization $\Theta = \Theta_{N-1} \cdots \Theta_{1}$, we finally get:

\begin{Prop}\label{prop:A-type-Theta}
The action of the operator $\Theta$ on the $U_{q}(\ssl(V))$-module $V \otimes V$ is given by
\begin{equation}\label{eq:A-type-Theta}
  \Theta = \ID - (q - q^{-1}) \sum_{i<j} (-1)^{\ol{j}} E_{ij} \otimes E_{ji} \,.
\end{equation}
\end{Prop}

Combining the above formula with~\eqref{eq:wtdf-action}, we obtain
\begin{align*}
  R_{VV}
  &= \tau_{VV} \circ \hat{R}_{VV} = \, \wtd{f}^{1/2} \circ \Theta \circ \wtd{f}^{1/2}
  = \sum_{i,j} q^{-(\varepsilon_{i},\varepsilon_{j})} E_{ii} \otimes E_{jj} +
  (q^{-1} - q) \sum_{i < j} (-1)^{\ol{j}} E_{ij} \otimes E_{ji} \\
  &= \ID + (q^{-1/2} - q^{1/2}) \sum_{i=1}^N (-1)^{\ol{i}} q^{-(\varepsilon_{i},\varepsilon_{i})/2} E_{ii} \otimes E_{ii} +
  (q^{-1} - q) \sum_{i < j} (-1)^{\ol{j}} E_{ij} \otimes E_{ji}\,,
\end{align*}
which precisely recovers $R_0$ from~\eqref{eq:R_0-gl}.
This provides an alternative proof of Theorem~\ref{thm:R_gl_finite}.


\subsection{Explicit affine R-matrices}
\

Let $\theta = \varepsilon_{1} - \varepsilon_{N}=\alpha_1 + \dots + \alpha_{N-1}$ be the highest root of $\ssl(V)$
with respect to the polarization \eqref{eq:polarization-gl}. Define the \emph{symmetrized extended Cartan matrix}
$(a_{ij})_{i,j=0}^{N-1}$ of $\ssl(V)$ as in Subsection~\ref{ssec:quantum-affine-orthosymplectic}.
The \emph{$A$-type quantum affine supergroup}, denoted by $U_{q}(\widehat{\ssl}(V))$, is a $\BC(q^{\pm 1/2})$-superalgebra
generated by $\{e_{i}, f_{i}, q^{\pm h_{i}/2}\}_{i=0}^{N-1}\cup \{\gamma^{\pm 1}, D^{\pm 1}\}$, with the $\BZ_{2}$-grading
\begin{equation*}
  |\gamma^{\pm 1}|=|D^{\pm 1}|=|h_i|=\bar{0} \,,\quad
  |e_i|=|f_i|=
  \begin{cases}
    \bar{0} & \mbox{if } \alpha_i \ \mathrm{is\ even} \\
    \bar{1} & \mbox{if } \alpha_i \ \mathrm{is\ odd}
  \end{cases} \quad \text{for} \quad 0 \leq i < N \,,
\end{equation*}
where $\alpha_0$ is a root of the same parity as $\theta$, subject to the analogues of
\eqref{eq:Chevalley-affine-0.1}--\eqref{eq:Chevalley-affine-3}:
\begin{align*}
  & D^{\pm 1}\cdot D^{\mp 1}=1 \,,\quad
    [D,q^{h_i/2}]=0 \,,\quad D e_i D^{-1}=q^{\delta_{0i}}e_i \,,\quad Df_iD^{-1}=q^{-\delta_{0i}}f_i \,, \\
  & \gamma^{\pm 1}\cdot \gamma^{\mp 1}=1 \,,\quad
    \gamma=q^{h_0/2}q^{h_1/2}\cdots q^{h_{N-1}/2} \,,\quad
     \gamma - \mathrm{central\ element} \,,
\end{align*}
\begin{align*}
  & [q^{h_{i}/2}, q^{h_{j}/2}] = 0 \,, \quad
    q^{\pm h_{i}/2} q^{\mp h_{i}/2} = 1 \,,\\
  & q^{h_{i}/2} e_{j} q^{-h_{i}/2} = q^{a_{ij}/2} e_{j} \,, \quad
    q^{h_{i}/2} f_{j} q^{-h_{i}/2} = q^{-a_{ij}/2} f_{j} \,, \\
  & [e_{i}, f_{j}] = \delta_{ij} \frac{q^{h_{i}} - q^{-h_{i}}}{q - q^{-1}} \,,
\end{align*}
together with the \emph{$q$-Serre relations} specified in~\eqref{eq:A-q-Serre-1}--\eqref{eq:A-q-Serre-3},
whereas the indices $i \pm 1$ are now understood modulo $N$,
and the following relations if $mn=2$:
\begin{equation}\label{eq:Serre-quintic}
\begin{split}
  & [\![e_j, [\![e_k, [\![e_j, [\![e_k, e_i ]\!] ]\!] ]\!] ]\!] =
    [\![e_k, [\![e_j, [\![e_k, [\![e_j, e_i ]\!] ]\!] ]\!] ]\!]  \,, \\
  & [\![f_j, [\![f_k, [\![f_j, [\![f_k, f_i ]\!] ]\!] ]\!] ]\!] =
    [\![f_k, [\![f_j, [\![f_k, [\![f_j, f_i ]\!] ]\!] ]\!] ]\!]  \,,
\end{split}
\end{equation}
with $\{i,j,k\}=\{0,1,2\}$, $\alpha_i$ being even, and $\alpha_j,\alpha_k$ being odd. The Hopf superalgebra structure on
$U_{q}(\widehat{\ssl}(V))$ is given by the same formulas as in Subsection~\ref{ssec:quantum-affine-orthosymplectic}.
Similarly to the last paragraph of Subsection~\ref{ssec:quantum-affine-orthosymplectic}, we also define the superalgebra
$U'_{q}(\widehat{\ssl}(V))$ by ignoring the degree generators $D^{\pm 1}$.

\begin{Prop}\label{prop:affine-repn-gl}
For any $u \in \BC^{\times}$ and $a,b\in \BC^\times$ satisfying $ab = (-1)^{\ol{N}}$, the $U_{q}(\ssl(V))$-action $\varrho$
on $V$ from Proposition~\ref{prop:fin-repn-gl} can be extended to a $U'_{q}(\widehat{\ssl}(V))$-action $\varrho^{a,b}_u$ on
$V(u)=V$ by setting
\begin{equation*}
  \varrho^{a,b}_u(x)=\varrho(x)  \qquad \mathrm{for\ all} \quad x\in \{e_i,f_i,q^{\pm h_i/2}\}_{i=1}^{N-1}
\end{equation*}
and defining the action of the remaining generators $e_0,f_0,q^{\pm h_0/2}, \gamma^{\pm 1}$ as follows:
\begin{equation*}
\begin{split}
  & \varrho^{a,b}_{u}(e_{0}) = au \cdot E_{N1} \,,\qquad \varrho^{a,b}_{u}(f_{0}) = bu^{-1} \cdot E_{1N} \,,\\
  & \varrho^{a,b}_{u}(q^{\pm h_{0}/2}) = q^{\mp ((-1)^{\ol{1}} E_{11} - (-1)^{\ol{N}} E_{NN})/2} \,,\quad
    \varrho^{a,b}_{u}(\gamma^{\pm 1}) = \id \,.
\end{split}
\end{equation*}
\end{Prop}

\begin{proof}
The proof is analogous (though much simpler) to that of Proposition~\ref{prop:affine-repn}.
We extend the $P$-grading~\eqref{eq:grading-A} on $U_{q}(\ssl(V))$ to that on
$U'_{q}(\widehat{\ssl}(V))$ via $\deg(e_0)=\varepsilon_N-\varepsilon_1$, $\deg(f_0)=\varepsilon_1-\varepsilon_N$,
and $\deg(q^{\pm h_0/2})=\deg(\gamma^{\pm 1})=0$. With respect to this grading, all $q$-Serre relations
except for~\eqref{eq:Serre-quintic} hold for degree reasons. Meanwhile, the remaining relation~\eqref{eq:Serre-quintic},
which occurs only if $mn=2$, is straightforwardly verified in the basis vectors $v_1,v_2,v_3$ of $V$.
\end{proof}

We also note the following analogue of Proposition \ref{prop:non-reduced affine module}:

\begin{Prop}\label{prop:non-reduced affine module-gl}
Let $u$ be an indeterminate and redefine $V(u)$ via $V(u) = V \otimes_{\BC} \BC[u,u^{-1}]$. Then, the formulas
defining $\varrho^{a,b}_u$ on the generators from Proposition~\ref{prop:affine-repn-gl} together with
\begin{equation*}
  \varrho^{a,b}_{u}(D^{\pm 1})(v \otimes u^{k}) = q^{\pm k}\cdot v \otimes u^{k}
  \qquad \forall\, v\in V \,,\, k\in \BZ
\end{equation*}
give rise to the same-named action $\varrho^{a,b}_u$ of $U_{q}(\widehat{\ssl}(V))$ on $V(u)$.
\end{Prop}

We shall now present the explicit formula for $\hat{R}(z)$, cf.~Theorem~\ref{thm:R_osp_affine}:

\begin{Thm}\label{thm:R_gl_affine}
For any $u,v$, set $z=u/v$. For $U_{q}(\widehat{\ssl}(V))$-modules $\varrho_u^{a,b}, \varrho_v^{a,b}$
from Proposition~\ref{prop:non-reduced affine module-gl} (with $ab = (-1)^{\ol{N}}$), the operator
$\hat{R}(z) = \tau \circ R(z)$ satisfies~\eqref{eq:affine-intertwiner}, where
\begin{multline}\label{eq:spectral-R-gl}
  R(z) =
  (z-1) \bigg\{ \ID + (q^{1/2} - q^{-1/2}) \sum_{1\leq i\leq N} (-1)^{\ol{i}} q^{(\varepsilon_{i},\varepsilon_{i})/2} \cdot
    E_{ii} \otimes E_{ii} + (q - q^{-1}) \sum_{i > j} (-1)^{\ol{j}} E_{ij} \otimes E_{ji} \bigg\} \\
  + (q - q^{-1}) \tau \,.
\end{multline}
\end{Thm}

Similarly to the observation made after Theorem~\ref{thm:R_osp_affine}, we conclude that $\hat{R}(z)$ coincides, up to a prefactor,
with the action of the universal $R$-matrix, and thus $R(z)$ of~\eqref{eq:spectral-R-gl} does satisfy~\eqref{eq:qYB-affine}.

\begin{Rem}\label{rem:trig-to-rational-gl}
We note that rescaling $R(z)$ of~\eqref{eq:spectral-R-gl} by $\frac{1}{z - 1}$, setting $q = e^{-\hbar/2}, z = e^{\hbar u}$,
and further taking the limit $\hbar \to 0$ recovers the rational $R$-matrix (super-analogue of the Yang's~$R$-matrix):
\begin{equation*}
  \lim_{\hbar \to 0} \left\{ {\frac{R(z)}{z - 1}}\middle|_{q = e^{-\hbar/2}, z = e^{\hbar u}} \right\} = \ID - \frac{\tau}{u} \,.
\end{equation*}
\end{Rem}

The proof of Theorem~\ref{thm:R_gl_affine} is straightforward and crucially relies on the expression of $R(z)$
from \eqref{eq:spectral-R-gl} through $R_0,R_\infty$ of~(\ref{eq:R_0-gl},~\ref{eq:R_inf-gl}), which is a special
case of the \emph{Yang-Baxterization} from~\cite{gwx}. Recall that the $R$-matrix $\hat{R}_{VV}=\hat{R}=\tau_{VV}R_0$
has two distinct eigenvalues $\lambda_1$ and $\lambda_2$,
in accordance with Propositions~\ref{prop:highest-weight-vec-gl},~\ref{prop:eig-calc-1-gl},~\ref{prop:eig-calc-R-hat-gl}.
In that setup, the Yang-Baxterization of~\cite[(3.15)]{gwx} produces the following
solution to~\eqref{eq:qYB-affine}: $\hat{R}(z) = \lambda_{2}^{-1}\hat{R} + z\lambda_{1}\hat{R}^{-1}$.

\begin{Prop}\label{prop:Yang-Baxterization-gl}
The affine $R$-matrix $R(z)$ of~\eqref{eq:spectral-R-gl} coincides (up to $\tau$ and a scalar multiple) with the
Yang-Baxterization of $\hat{R}_{VV}=\tau \circ R_0$. To be more specific, for $\hat{R}(z) = \tau \circ R(z)$:
\begin{equation}\label{eq:baxterization-gl}
  \lambda_{1} \hat{R}(z) = \lambda_{2}^{-1}\hat{R}_{VV} + z\lambda_{1}\hat{R}^{-1}_{VV} \,.
\end{equation}
with $\lambda_1,\lambda_2$ precisely as in Proposition~\ref{prop:eig-calc-R-hat-gl}.
\end{Prop}

\begin{proof}
By straightforward computation, based on~(\ref{eq:R_0-gl},~\ref{eq:R_inf-gl},~\ref{eq:spectral-R-gl}), one verifies that
\begin{equation}\label{eq:baxterization-without-tau-gl}
  \lambda_{1} R(z) = \lambda_{2}^{-1}R_{0} + z\lambda_{1}R_{\infty} \,.
\end{equation}
Composing with $\tau$ on the left, and using~\eqref{eq:key finite equality A-type}, we obtain~(\ref{eq:baxterization-gl}).
\end{proof}

\begin{Rem}
Due to~\eqref{eq:baxterization-without-tau-gl}, one recovers $R_0,R_\infty$ of~(\ref{eq:R_0-gl},~\ref{eq:R_inf-gl})
as renormalized limits of $R(z)$:
\begin{equation*}
   R_0=-R(z)|_{z=0} \,,\qquad R_\infty=\lim_{z\to \infty} \left\{ R(z)/z \right\} .
\end{equation*}
\end{Rem}


\subsection{Proof of the main result in A-type}
\

Due to Proposition~\ref{prop:Yang-Baxterization-gl} and Theorem~\ref{thm:R_gl_finite}, it only remains to
verify~\eqref{eq:affine-intertwiner} for $x=e_0$ and $x=f_0$. We shall now present the direct verification
for $x=e_0$, while $x=f_0$ can be treated analogously to the finite case.
We shall also assume that $a = 1$, as in the orthosymplectic case.

Since $\varrho^{a,b}_{u}(q^{h_{0}/2})$ is a diagonal matrix, we shall write it as
$\varrho^{a,b}_{u}(q^{h_{0}/2}) = \diag(\sft_{1}, \ldots, \sft_{N})$. By direct computation, we get:
\begin{multline*}
  R_{\infty}\Delta(e_{0}) = \,
  \Delta(e_{0}) + (q^{(-1)^{\ol{N}}} - 1)
    \bigg\{ q^{(-1)^{\ol{N}}/2}v \cdot E_{NN} \otimes E_{N1} + q^{-(-1)^{\ol{N}}/2}u \cdot E_{N1} \otimes E_{NN} \bigg\} \\
   + (q - q^{-1}) \sum_{1\leq j\leq N} (-1)^{\ol{j}} \sft_{j} v \cdot E_{Nj} \otimes E_{j1}
    - (q - q^{-1}) (-1)^{\ol{N}} q^{(-1)^{\ol{N}}/2} v \cdot E_{NN} \otimes E_{N1} \,, \\
  \Delta^{\opp}(e_{0})R_{\infty} = \Delta^{\opp}(e_{0}) + (q^{(-1)^{\ol{1}}} - 1)
    \bigg\{ q^{(-1)^{\ol{1}}/2}v \cdot E_{11} \otimes E_{N1} + q^{-(-1)^{\ol{1}}/2}u \cdot E_{N1} \otimes E_{11} \bigg\} \\
  + (q - q^{-1}) \sum_{1\leq i\leq N} (-1)^{\ol{N}\,\ol{i} + \ol{1}\,\ol{N} + \ol{1}\,\ol{i}} \sft_{i}^{-1} v \cdot
    E_{i1} \otimes E_{Ni} - (q - q^{-1}) (-1)^{\ol{1}} q^{(-1)^{\ol{1}}/2} v \cdot E_{11} \otimes E_{N1} \,.
\end{multline*}
Collecting the terms together, we thus obtain:
\begin{multline*}
  R_{\infty}\Delta(e_{0}) - \Delta^{\opp}(e_{0})R_{\infty} = \\
  (q - q^{-1})v \cdot \sum_{1\leq j\leq N} (-1)^{\ol{j}} \sft_{j} \cdot E_{Nj} \otimes E_{j1} -
  (q - q^{-1})v \cdot \sum_{1\leq i\leq N} (-1)^{\ol{N}\,\ol{i} + \ol{1}\,\ol{N} + \ol{1}\,\ol{i}} \sft_{i}^{-1} \cdot
     E_{i1} \otimes E_{Ni}  \,.
\end{multline*}
Evoking the paragraph after~\eqref{eq:R_infty-e_0}, we immediately obtain:
\begin{multline*}
  R_{0}\Delta(e_{0}) - \Delta^{\opp}(e_{0})R_{0} = \\
  (q - q^{-1})u \cdot \sum_{1\leq j\leq N} (-1)^{\ol{j}} \sft_{j} \cdot E_{Nj} \otimes E_{j1}
  - (q - q^{-1})u \cdot \sum_{1\leq i\leq N} (-1)^{\ol{N}\,\ol{i} + \ol{1}\,\ol{N} + \ol{1}\,\ol{i}} \sft_{i}^{-1} \cdot
    E_{i1} \otimes E_{Ni} \,.
\end{multline*}
Combining the above two equalities with formula~\eqref{eq:baxterization-without-tau-gl}, we get the desired result:
\begin{equation*}
  R(z) \Delta(e_{0}) - \Delta^{\opp}(e_{0}) R(z) = 0 \,.
\end{equation*}


\section{Generating vectors for tensor square}\label{sec:app_generating}

\subsection{Two vectors in A-type}\label{ssec:generating-A}
\

We first define the following elements $u^{\pm}_{ij}$ in $V \otimes V$ for $1 \leq i \leq j \leq N$:
\begin{align*}
  u^{+}_{ij} &= v_{i} \otimes v_{j} +
    (-1)^{\ol{1}}(-1)^{\ol{i}\,\ol{j}} q^{-(-1)^{\ol{1}}} q^{-(\varepsilon_{i},\varepsilon_{j})} v_{j} \otimes v_{i} \,,\\
  u^{-}_{ij} &= v_{i} \otimes v_{j} -
    (-1)^{\ol{1}}(-1)^{\ol{i}\,\ol{j}} q^{(-1)^{\ol{1}}} q^{-(\varepsilon_{i},\varepsilon_{j})} v_{j} \otimes v_{i} \,.
\end{align*}
In particular, we note that $u^+_{ii}=0$ iff $\ol{i}\ne \ol{1}$ and $u^-_{ii}=0$ iff $\ol{i}=\ol{1}$.
We define subspaces $W^{\pm}$ of $V\otimes V$ to be spanned by the corresponding (nonzero) vectors:
\begin{equation}\label{eq:W-subrep-basis-A}
\begin{split}
  & W^+ = \mathrm{Span}
    \left( \big\{ u^{+}_{ij} \,\big|\, 1 \leq i < j \leq N \big\} \cup
           \big\{u^{+}_{ii} \,\big|\, \ol{i} = \ol{1} \big\} \right) ,\\
  & W^- = \mathrm{Span}
    \left( \big\{ u^{-}_{ij} \,\big|\, 1 \leq i < j \leq N \big\} \cup
           \big\{u^{-}_{ii} \,\big|\, \ol{i} \ne \ol{1} \big\} \right) .
\end{split}
\end{equation}

\begin{Prop}\label{prop:decomp-A}
(a) The subspaces $W^+$ and $W^-$ are $U_{q}(\ssl(V))$-subrepresentations of $V \otimes V$.

\smallskip
\noindent
(b) The $U_{q}(\ssl(V))$-representation $V \otimes V$ decomposes as a direct sum of these subrepresentations
\begin{equation*}
   V\otimes V \simeq W^+ \oplus W^- \,.
\end{equation*}

\noindent
(c) $W^{+}$ and $W^{-}$ are irreducible $U_{q}(\ssl(V))$-representations generated by the corresponding highest
weight vectors $w_{1}$ and $w_{2}$ of~\eqref{eq:w-vectors-gl}, respectively.
\end{Prop}

\begin{proof}
(a) We first verify that $W^{+}$ is stable under the $U_{q}(\ssl(V))$-action via direct computations.
The generators $\{e_{a}, f_{a}\}_{a=1}^{N-1}$ act on the above vectors $u^{+}_{ij}$ as follows:
\begin{enumerate}

\item[$\bullet$]
Case 1: $u^{+}_{ii}$ for $\ol{i} = \ol{1}$.
\begin{align*}
  \Big(1+q^{-(-1)^{\ol{i}} \cdot 2}\Big)^{-1} \varrho^{\otimes 2}(e_{a})u^{+}_{ii}
  &= \delta_{a+1,i} \cdot q^{(-1)^{\ol{i}}/2} \cdot u^{+}_{a,a+1} \,,\\
  (-1)^{\ol{a}} \Big(1+q^{-(-1)^{\ol{i}} \cdot 2}\Big)^{-1} \varrho^{\otimes 2}(f_{a})u^{+}_{ii}
  &= \delta_{ai} \cdot (-1)^{\ol{a}(\ol{a}+\ol{a+1})} q^{(-1)^{\ol{i}}/2} \cdot u^{+}_{a,a+1} \,.
\end{align*}

\item[$\bullet$]
Case 2: $u^{+}_{ij}$ for $i<j$.
\begin{align*}
  \varrho^{\otimes 2}(e_{a})u^{+}_{ij}
  &= \delta_{a+1,i} \cdot u^{+}_{a,j} +
     \delta_{a+1,j} \cdot (-1)^{\ol{i}(\ol{a}+\ol{a+1})} \cdot q^{(\varepsilon_{a},\varepsilon_{i})/2} \cdot u^{+}_{i,a} \,,\\
  (-1)^{\ol{a}} \varrho^{\otimes 2}(f_{a})u^{+}_{ij}
  &= \delta_{ai} \cdot q^{(\varepsilon_{a+1},\varepsilon_{j})/2} \cdot u^{+}_{a+1,j} +
     \delta_{aj} \cdot (-1)^{\ol{i}(\ol{a}+\ol{a+1})} \cdot u^{+}_{i,a+1} \,,
\end{align*}
where we recall that $u^+_{jj}=0$ if $\ol{j}\ne {\ol{1}}$.

\end{enumerate}
This shows that $W^+$ is stable under the action of $\{e_a,f_a\}_{a=1}^{N-1}$.
As each $u^{+}_{ij}$ is homogeneous of degree $\varepsilon_{i}+\varepsilon_{j}$ with respect to the $P$-grading
defined by $\deg(v_a\otimes v_b)=\varepsilon_a+\varepsilon_b$ for all $1\leq a,b\leq N$, cf.~\eqref{eq:grading-A},
we also have $\varrho^{\otimes 2}(q^{h_{a}/2})u^{+}_{ij} = q^{(\varepsilon_{i}+\varepsilon_{j})(\ssh_{a})/2}\cdot u^+_{ij}$.
This establishes part (a) for $W^{+}$.

The proof of (a) for $W^{-}$ is analogous (the above formulas hold with each $u^+_{ab}$ replaced by $u^-_{ab}$).

(b) It is enough to show that the vectors entering \eqref{eq:W-subrep-basis-A} form a basis for $V \otimes V$.
As there are precisely $N^2=\dim(V\otimes V)$ of such vectors, it suffices to show their linear independence.
Moreover, since each such vector is homogeneous with respect to the above $P$-grading, it suffices to verify
the linear independence in each weight spaces, which can be easily seen.

(c) Let us first prove the ``generating'' property of $W^{+}$. To this end, it suffices to show that each $u^{+}_{ij}$
is contained in the $U_{q}(\ssl(V))$-submodule generated by $u^{+}_{11}$ (a nonzero scalar multiple of $w_{1}$),
which can be done by acting with the generators $\{f_{a}\}_{a=1}^{N-1}$ iteratively on $u^{+}_{11}$. According to the explicit
formulas in part (a): $\varrho^{\otimes 2}(f_{j-1} \dots f_{2}f_{1})u^{+}_{11}$ is a nonzero scalar multiple of
$u^{+}_{1j}$ and further $\varrho^{\otimes 2}(f_{i-1} \dots f_{2}f_{1})u^{+}_{1j}$ is a nonzero scalar multiple of
$u^{+}_{ij}$. This shows that indeed $W^{+}=U_{q}(\ssl(V))w_1$.

The proof of the irreducibility of $W^+$ is similar. Since any nonzero submodule of $V \otimes V$ has a weight space
decomposition, it is enough to show that for any nonzero $P$-homogeneous element $w\in W^{+}$, i.e.\ a nonzero scalar
multiple of some $u^{+}_{ij}$, we can obtain a nonzero scalar multiple of $w_{1}$ by acting with the generators
$\{e_{a}\}_{a=1}^{N-1}$ iteratively on $w$. Due to the explicit formulas in part (a), we have:
$\varrho^{\otimes 2}(e_{1}e_{2} \dots e_{i-1}) u^{+}_{ij}$ is a nonzero scalar multiple of $u^{+}_{1j}$,
and $\varrho^{\otimes 2}(e_{1}e_{2} \dots e_{j-1}) u^{+}_{1j}$ is a nonzero scalar multiple of $w_1$.
This establishes (c) for $W^+$. The proof of (c) for $W^{-}$ is analogous.
\end{proof}


\subsection{Three vectors in orthosymplectic type}\label{ssec:generating-BCD}
\

Motivated by Subsection~\ref{ssec:generating-A}, we shall now present a similar analysis for
the structure of $\uqV$-representation $V\otimes V$. To this end, we consider three cases separately: $m$ is odd,
$m$ is even and $\ol{s}=\bar{0}$, $m$ is even and $\ol{s}=\bar{1}$. Besides lengthier calculations, the orthosymplectic
setup is considerably harder due to the higher dimensional degree $0$ component of this tensor square.
In particular, our presentation emphasizes in full details the importance of the special case $n=m$ when
$V\otimes V$ is not semisimple. We note that it is this major difference that forced us to work with the vectors
$\wtd{w}_3,\hat{w}_3$ instead of the highest weight vector $w_3$ in Section~\ref{sec:R-matrices}.


\subsubsection{Generating property for odd $m$}
\

We first define the following elements $u^{\pm}_{ij}$ in $V \otimes V$ for $1 \leq i \leq j \leq N$ and $(i,j) \neq (s+1,s+1)$:
\begin{align*}
  u^{+}_{ij} &=
  \begin{cases}
    v_{i} \otimes v_{j} + (-1)^{\ol{1}}(-1)^{\ol{i}\,\ol{j}} q^{-(-1)^{\ol{1}}} q^{-(\varepsilon_{i},\varepsilon_{j})} v_{j} \otimes v_{i}
       & \textrm{for \ } j \neq i' \\
    \Big\{ v_{i} \otimes v_{i'} + (-1)^{\ol{1}+\ol{i}} q^{-(-1)^{\ol{1}}} q^{-(\varepsilon_{i},\varepsilon_{i})} v_{i'} \otimes v_{i} \Big\}
    - (-1)^{\ol{i}+\ol{i+1}}q^{-(\varepsilon_{i},\varepsilon_{i})/2}q^{-(\varepsilon_{i+1},\varepsilon_{i+1})/2}  & \\
    \quad \cdot \vartheta_{i}\vartheta_{i+1}
    \Big\{ v_{i+1} \otimes v_{(i+1)'} + (-1)^{\ol{1}+\ol{i+1}} q^{-(-1)^{\ol{1}}} q^{(\varepsilon_{i+1},\varepsilon_{i+1})}
    v_{(i+1)'} \otimes v_{i+1} \Big\}
      & \textrm{for \ } j = i'
  \end{cases}
\end{align*}
and
\begin{align*}
  u^{-}_{ij} &=
  \begin{cases}
    v_{i} \otimes v_{j} - (-1)^{\ol{1}}(-1)^{\ol{i}\,\ol{j}} q^{(-1)^{\ol{1}}} q^{-(\varepsilon_{i},\varepsilon_{j})} v_{j} \otimes v_{i}
      & \textrm{for \ } j \neq i' \\
    \Big\{ v_{i} \otimes v_{i'} - (-1)^{\ol{1}+\ol{i}} q^{(-1)^{\ol{1}}} q^{-(\varepsilon_{i},\varepsilon_{i})} v_{i'} \otimes v_{i} \Big\}
    - (-1)^{\ol{i}+\ol{i+1}} q^{-(\varepsilon_{i},\varepsilon_{i})/2}q^{-(\varepsilon_{i+1},\varepsilon_{i+1})/2} & \\
    \quad \cdot \vartheta_{i}\vartheta_{i+1}
    \Big\{ v_{i+1} \otimes v_{(i+1)'} - (-1)^{\ol{1}+\ol{i+1}} q^{(-1)^{\ol{1}}} q^{(\varepsilon_{i+1},\varepsilon_{i+1})}
    v_{(i+1)'} \otimes v_{i+1} \Big\}
      & \textrm{for \ } j = i'
\end{cases} .
\end{align*}
In particular, we note that $u^+_{ii}=0$ iff $\ol{i}\ne \ol{1}$ and $u^-_{ii}=0$ iff $\ol{i}=\ol{1}$.
We define subspaces $W^{\pm}$ of $V\otimes V$ to be spanned by the corresponding (nonzero) vectors:
\begin{equation}\label{eq:W-subrep-basis-B}
\begin{split}
  & W^+ = \mathrm{Span}
    \left( \big\{ u^{+}_{ij} \,\big|\, 1 \leq i < j \leq N \big\} \cup
           \big\{u^{+}_{ii} \,\big|\, i\ne s+1 \,,\, \ol{i} = \ol{1} \big\} \right) ,\\
  & W^- = \mathrm{Span}
    \left( \big\{ u^{-}_{ij} \,\big|\, 1 \leq i < j \leq N \big\} \cup
           \big\{u^{-}_{ii} \,\big|\, i\ne s+1 \,,\, \ol{i} \ne \ol{1} \big\} \right) .
\end{split}
\end{equation}
We also consider a one-dimensional subspace $W_{3}=\mathrm{Span}(w_3)$ of $V\otimes V$, cf.~\eqref{eq:w-vectors}.

\begin{Prop}\label{prop:decomp-B}
(a) The subspaces $W^+,W^-,W_3$ are $\uqV$-subrepresentations of $V \otimes V$.

\smallskip
\noindent
(b) The $\uqV$-representation $V \otimes V$ decomposes as a direct sum of these subrepresentations
\begin{equation*}
   V\otimes V \simeq W^+ \oplus W^- \oplus W_3 \,.
\end{equation*}

\noindent
(c) Both $\wtd{w}_3=v_1\otimes v_{1'}$ and $\hat{w}_3=v_{1'}\otimes v_{1}$ do not belong to $W^+\oplus W^-$, while
\begin{equation}\label{eq:explicit-comb-subspace-1}
  \wtd{w}_3 - (-1)^{\ol{1}}q^{(-1)^{\ol{1}}} q^{n-m+1} \hat{w}_3 \in W^{+} \oplus W^{-} \,.
\end{equation}

\noindent
(d) $W^{+}$, $W^{-}$, $W_3$ are irreducible $\uqV$-representations generated by the corresponding highest
weight vectors $w_{1}$, $w_{2}$, and $w_3$ of~\eqref{eq:w-vectors}, respectively.
\end{Prop}

\begin{proof}
(a) Let us first show that $W^{+}$ is stable under the $\uqV$-action through direct but rather tedious computations.
The action of the generators $\{f_{a}\}_{a=1}^{s}$ on the above vectors $u^{+}_{ij}$ is summarized in
the following formulas (split into five cases):
\begin{enumerate}

\item[$\bullet$]
Case 1: $u^{+}_{ii}$ for $\ol{i} = \ol{1}$.
\begin{align*}
  & (-1)^{\ol{a}} \Big(1+q^{-(-1)^{\ol{i}} \cdot 2}\Big)^{-1} \varrho^{\otimes 2}(f_{a})u^{+}_{ii}\\
  & \qquad= \delta_{ai} \cdot (-1)^{\ol{a}(\ol{a}+\ol{a+1})} q^{(-1)^{\ol{i}}/2} \cdot u^{+}_{a,a+1} -
    \delta_{(a+1)',i} \cdot \vartheta_{a}\vartheta_{a+1} q^{(-1)^{\ol{i}}/2} \cdot u^{+}_{(a+1)',a'} \,.
\end{align*}

\item[$\bullet$]
Case 2: $u^{+}_{ij}$ for $i<j$ and $i+j \neq N,N+1$.
\begin{align*}
  & (-1)^{\ol{a}} \varrho^{\otimes 2}(f_{a})u^{+}_{ij} 
   = \delta_{ai} \cdot q^{(\varepsilon_{a+1},\varepsilon_{j})/2} \cdot u^{+}_{a+1,j} +
   \delta_{aj} \cdot (-1)^{\ol{i}(\ol{a}+\ol{a+1})} \cdot u^{+}_{i,a+1}\\
  &\quad - \delta_{(a+1)',i} \cdot (-1)^{\ol{a+1}(\ol{a}+\ol{a+1})} \vartheta_{a}\vartheta_{a+1}
    q^{-(\varepsilon_{a},\varepsilon_{j})/2} \cdot u^{+}_{a'j} -
    \delta_{(a+1)',j} \cdot (-1)^{(\ol{i}+\ol{a+1})(\ol{a}+\ol{a+1})} \vartheta_{a}\vartheta_{a+1} \cdot u^{+}_{ia'} \,,
\end{align*}
where we recall that $u^+_{jj}=0$ if $\ol{j}\ne {\ol{1}}$.

\item[$\bullet$]
Case 3: $u^{+}_{ij}$ for $i < j$ and $i+j = N$.
\begin{align*}
  & (-1)^{\ol{a}} \varrho^{\otimes 2}(f_{a})u^{+}_{i,(i+1)'}
   = -\delta_{ai} \cdot (-1)^{\ol{a}+\ol{a+1}} \vartheta_{a}\vartheta_{a+1} q^{(\varepsilon_{a},\varepsilon_{a})/2}
   \cdot u^{+}_{aa'} \,.
\end{align*}

\item[$\bullet$]
Case 4: $u^{+}_{ij}$ for $i < j$, $i+j = N+1$ and $i \neq s$.
\begin{align*}
  & (-1)^{\ol{a}} \varrho^{\otimes 2}(f_{a})u^{+}_{ii'}
   = \delta_{ai} \delta_{\ol{a},\ol{a+1}} \cdot q^{(\varepsilon_{a},\varepsilon_{a})/2}
   \Big( 1+q^{-(\varepsilon_{a},\varepsilon_{a}) \cdot 2} \Big) u^{+}_{a+1,a'} \\
  &\quad - \delta_{a,i+1} \cdot (-1)^{\ol{a-1}+\ol{a}} \vartheta_{a-1}\vartheta_{a} q^{-(\varepsilon_{a-1},\varepsilon_{a-1})/2}
   \cdot u^{+}_{a+1,a'} - \delta_{a+1,i} \cdot \vartheta_{a}\vartheta_{a+1} q^{-(\varepsilon_{a+1},\varepsilon_{a+1})/2} \cdot u^{+}_{a+1,a'} \,.
\end{align*}

\item[$\bullet$]
Case 5: $u^{+}_{ij}$ for $(i,j) = (s,s')$.
\begin{align*}
  & (-1)^{\ol{a}} \varrho^{\otimes 2}(f_{a})u^{+}_{ss'}
   = \delta_{as} \bigg\{ (-1)^{\ol{s}+\ol{s+1}} q^{-(\varepsilon_{s},\varepsilon_{s})/2} +
     \delta_{\ol{1}\,\ol{s}} \cdot q^{(\varepsilon_{s},\varepsilon_{s})/2} \Big(1+q^{-(-1)^{\ol{1}} \cdot 2}\Big) \bigg\} u^{+}_{s+1,s'} \\
  &\quad - \delta_{a,s-1} \cdot \vartheta_{s-1}\vartheta_{s} q^{-(\varepsilon_{s},\varepsilon_{s})/2} \cdot u^{+}_{s,(s-1)'} \,.
\end{align*}

\end{enumerate}

The above computations show that $W^{+}$ is stable under the action of $\{f_{a}\}_{a=1}^{s}$. To check that $W^{+}$ is
also stable under the action of $\{e_{a}\}_{a=1}^{s}$, we consider a vector space isomorphism
\begin{equation}\label{eq:phi-map}
  \phi \colon V \iso V \quad \mathrm{given\ by} \quad
  v_{i} \mapsto \sfc_{i'}v_{i'} \quad \mathrm{for\ all} \quad  1 \leq i \leq N \,,
\end{equation}
where the coefficients $\sfc_{i}$'s are determined by $\sfc_{1} = 1$ and the following relations:
\begin{equation*}
\begin{split}
  & \sfc_{a+1} = -(-1)^{\ol{a}+\ol{a+1}+\ol{a}\, \ol{a+1}} \vartheta_{a}\vartheta_{a+1} \cdot \sfc_{a}
    \qquad \mathrm{for}\ \ 1 \leq a \leq s \,, \\
  & \sfc_{a'} = -(-1)^{\ol{a}\,\ol{a+1}} \vartheta_{a}\vartheta_{a+1} \cdot \sfc_{(a+1)'}
    \qquad \mathrm{for}\ \ 1 \leq a \leq s \,.
\end{split}
\end{equation*}
In particular, we note that
\begin{equation}\label{eq:c-product}
  (-1)^{\ol{a}}\sfc_{a}\sfc_{a'}=(-1)^{\ol{1}}\sfc_{1}\sfc_{1'}  \qquad \forall\, 1 \leq a\leq N \,.
\end{equation}
Evoking $\omega$ from~\eqref{eq:omega-map}, it is straightforward to check that
\begin{equation}\label{eq:phi-and-varrho}
  \varrho(x) v = (\phi^{-1} \circ \varrho(\omega(x)) \circ \phi) (v)
\end{equation}
holds for any $v=v_i$ and $x \in \{e_{a},f_{a},q^{\pm h_{a}/2}\}_{a=1}^{s}$. Thus~\eqref{eq:phi-and-varrho} holds
for all $v\in V, x \in \uqV$. It is also easy to check (verifying on the generators) that $\Delta^{\omega} = \Delta^{\opp}$,
cf.\ (\ref{eq:sigma-coproduct},~\ref{eq:sigma-opp}), where
\begin{equation*}
  \Delta^{\omega} = (\omega \otimes \omega) \circ \Delta \circ \omega^{-1}\,.
\end{equation*}
Furthermore, we note that $W^\pm$ are invariant under $\tau\circ \phi^{\otimes 2}$ due to the following equality:
\begin{equation*}
  (\tau_{VV} \circ (\phi \otimes \phi))(u^{\pm}_{ij}) = (-1)^{\ol{i}\,\ol{j}} \sfc_{i'}\sfc_{j'} \cdot u^{\pm}_{j'i'} \,.
\end{equation*}
Combining all these results, we finally obtain:
\begin{equation}\label{eq:e-vs-f-computation}
\begin{split}
  \varrho^{\otimes 2} (e_{a}) u^{+}_{ij}
  &= \Delta(e_{a}) u^{+}_{ij}
   = \left( (\phi^{\otimes 2})^{-1} \circ \Delta^{\omega}(\omega(e_{a})) \circ \phi^{\otimes 2} \right) (u^{+}_{ij})\\
  &= (-1)^{\ol{a}+\ol{a+1}} \cdot \left( (\phi^{\otimes 2})^{-1} \circ \Delta^{\omega}(f_{a}) \circ \phi^{\otimes 2} \right) (u^{+}_{ij})\\
  &\overset{\eqref{eq:delta-opp}}{=} (-1)^{\ol{a}+\ol{a+1}} \cdot
     \left( (\tau_{VV} \circ \phi^{\otimes 2})^{-1} \circ \Delta(f_{a}) \circ (\tau_{VV} \circ \phi^{\otimes 2}) \right) (u^{+}_{ij}) \,,
\end{split}
\end{equation}
which proves that $W^{+}$ is stable under the action of $\{e_{a}\}_{a=1}^{s}$. Finally, each $u^{+}_{ij}$ is
homogeneous of degree $\varepsilon_{i}+\varepsilon_{j}$ with respect to the $P$-grading defined by
$\deg(v_a\otimes v_b)=\varepsilon_a+\varepsilon_b$ for all $1\leq a,b\leq N$, so that
$\varrho^{\otimes 2}(q^{h_{a}/2})u^{+}_{ij} = q^{(\varepsilon_{i}+\varepsilon_{j})(\ssh_{a})/2}\cdot u^+_{ij}$.
This completes the proof of part (a) for $W^{+}$.

\smallskip
The proof of part (a) for $W^{-}$ is completely analogous. Therefore, we shall only present the explicit formulas
for the action of the generators $\{f_{a}\}_{a=1}^{s}$ on the vectors $u^{-}_{ij}$:
\begin{enumerate}

\item[$\bullet$]
Case 1: $u^{-}_{ii}$ for $\ol{i} \neq \ol{1}$.
\begin{align*}
  & (-1)^{\ol{a}} \Big(1+q^{-(-1)^{\ol{i}} \cdot 2}\Big)^{-1} \varrho^{\otimes 2}(f_{a}) u^{-}_{ii}\\
  &\qquad = \delta_{ai} \cdot (-1)^{\ol{a}(\ol{a}+\ol{a+1})} q^{(-1)^{\ol{i}}/2} \cdot u^{-}_{a,a+1} -
   \delta_{(a+1)',i} \cdot \vartheta_{a}\vartheta_{a+1} q^{(-1)^{\ol{i}}/2} \cdot u^{-}_{(a+1)',a'} \,.
\end{align*}

\item[$\bullet$]
Case 2: $u^{-}_{ij}$ for $i < j$ and $i+j \neq N, N+1$.
\begin{align*}
  & (-1)^{\ol{a}} \varrho^{\otimes 2}(f_{a})u^{-}_{ij} 
   = \delta_{ai} \cdot q^{(\varepsilon_{a+1},\varepsilon_{j})/2} \cdot u^{-}_{a+1,j} +
    \delta_{aj} \cdot (-1)^{\ol{i}(\ol{a}+\ol{a+1})} \cdot u^{-}_{i,a+1} \\
  &\quad - \delta_{(a+1)',i} \cdot (-1)^{\ol{a+1}(\ol{a}+\ol{a+1})} \vartheta_{a}\vartheta_{a+1}
    q^{-(\varepsilon_{a},\varepsilon_{j})/2} \cdot u^{-}_{a'j} -
    \delta_{(a+1)',j} \cdot (-1)^{(\ol{i}+\ol{a+1})(\ol{a}+\ol{a+1})} \vartheta_{a}\vartheta_{a+1} \cdot u^{-}_{ia'} \,,
\end{align*}
where we recall that $u^-_{jj}=0$ if $\ol{j}={\ol{1}}$.

\item[$\bullet$]
Case 3: $u^{-}_{ij}$ for $i < j$ and $i+j = N$.
\begin{align*}
  & (-1)^{\ol{a}} \varrho^{\otimes 2}(f_{a})u^{-}_{i(i+1)'}
    = -\delta_{ai} \cdot (-1)^{\ol{a}+\ol{a+1}} \vartheta_{a}\vartheta_{a+1} q^{(\varepsilon_{a},\varepsilon_{a})/2}
    \cdot u^{-}_{aa'} \,.
\end{align*}

\item[$\bullet$]
Case 4: $u^{-}_{ij}$ for $i < j$, $i+j = N+1$ and $i \neq s$.
\begin{align*}
  & (-1)^{\ol{a}} \varrho^{\otimes 2}(f_{a})u^{-}_{ii'}
   = \delta_{ai} \delta_{\ol{a},\ol{a+1}} \cdot q^{(\varepsilon_{a},\varepsilon_{a})/2}
    \Big(1+q^{-(\varepsilon_{a},\varepsilon_{a}) \cdot 2}\Big) u^{-}_{a+1,a'}\\
  &\quad - \delta_{a,i+1} \cdot (-1)^{\ol{a-1}+\ol{a}} \vartheta_{a-1}\vartheta_{a} q^{-(\varepsilon_{a-1},\varepsilon_{a-1})/2}
    \cdot u^{-}_{a+1,a'} - \delta_{a+1,i} \cdot \vartheta_{a}\vartheta_{a+1} q^{-(\varepsilon_{a+1},\varepsilon_{a+1})/2} \cdot u^{-}_{a+1,a'} \,.
\end{align*}

\item[$\bullet$]
Case 5: $u^{-}_{ij}$ for $(i,j) = (s,s')$.
\begin{align*}
  & (-1)^{\ol{a}} \varrho^{\otimes 2}(f_{a})u^{-}_{ss'}
   = \delta_{as} \bigg\{ (-1)^{\ol{s}+\ol{s+1}} q^{-(\varepsilon_{s},\varepsilon_{s})/2} +
    \delta_{\ol{1} \neq \ol{s}} \cdot q^{(\varepsilon_{s},\varepsilon_{s})/2} \Big(1+q^{(-1)^{\ol{1}} \cdot 2}\Big) \bigg\} u^{-}_{s+1,s'} \\
  &\qquad - \delta_{a,s-1} \cdot \vartheta_{s-1}\vartheta_{s} q^{-(\varepsilon_{s},\varepsilon_{s})/2} \cdot u^{-}_{s,(s-1)'} \,,
\end{align*}
where $\delta_{\ol{1} \neq \ol{s}}$ equals $1$ if $\ol{1}\ne \ol{s}$ and is $0$ otherwise.
\end{enumerate}

To prove part (a) for $W_3$, we recall that $\varrho^{\otimes 2}(e_i)w_3=0, \varrho^{\otimes 2}(q^{h_i/2})w_3=w_3$
for any $1\leq i\leq s$, as established in our proof of Proposition~\ref{prop:highest-weight-vec}(a). The remaining
vanishing $\varrho^{\otimes 2}(f_i)w_3=0$ for $1\leq i \leq s$ follow from $\varrho^{\otimes 2}(e_i)w_3=0$
via~\eqref{eq:e-vs-f-computation} and $(\tau_{VV}\circ \phi^{\otimes 2})(w_3)=(-1)^{\ol{1}}\sfc_1\sfc_{1'}\cdot w_3$,
due to~\eqref{eq:c-product}.

(b) It is enough to show that the vectors entering \eqref{eq:W-subrep-basis-B} together with $w_{3}$
form a basis for $V \otimes V$. As there are precisely $N^2=\dim(V\otimes V)$
of such vectors, it suffices to show their linear independence. Moreover, since each such vector is homogeneous with respect
to the above $P$-grading, it suffices to verify the linear independence in each weight spaces. This is clear for nonzero
weight spaces. Finally, for the zero weight space the proof is done by straightforward computation, which is analogous to
the even $m$ case treated in full details below, cf.\ Remark~\ref{rem:classical-zero-weight}.

(c) The proof of $\wtd{w}_3,\hat{w}_3\notin W^+\oplus W^-$ is completely analogous to that of part (b) with $w_{3}$ replaced
by either $\wtd{w}_{3}$ or $\hat{w}_{3}$. Meanwhile, the proof of~\eqref{eq:explicit-comb-subspace-1} is completely analogous
to that of~\eqref{eq:explicit-comb-subspace-2} below. Here, we shall only state the explicit linear dependence:
\begin{equation*}
  \sum_{i=1}^{s} \big( b^{+}_{i} u^{+}_{ii'} + b^{-}_{i} u^{-}_{ii'} \big) =
  v_{1} \otimes v_{1'} - (-1)^{\ol{1}}q^{(-1)^{\ol{1}}} q^{n-m+1} v_{1'} \otimes v_{1}
\end{equation*}
where
\begin{align*}
  b^{+}_{i}
  &= \frac{(-1)^{\ol{1}+\ol{i}} \vartheta_{1}\vartheta_{i}}{q+q^{-1}} q^{-\sum_{k=1}^{s-1}(\rho,\alpha_{k})}
     \left( q^{(\varepsilon_{1},\varepsilon_{1})}q^{\sum_{k=i}^{s-1}(\rho,\alpha_{k})} -
       (-1)^{\ol{1}} q^{(\varepsilon_{i},\varepsilon_{i})-(\varepsilon_{s},\varepsilon_{s})}q^{-\sum_{k=i}^{s-1}(\rho,\alpha_{k})} \right) \,,\\
  b^{-}_{i}
  &= \frac{(-1)^{\ol{1}+\ol{i}} \vartheta_{1}\vartheta_{i}}{q+q^{-1}} q^{-\sum_{k=1}^{s-1}(\rho,\alpha_{k})}
     \left( q^{-(\varepsilon_{1},\varepsilon_{1})}q^{\sum_{k=i}^{s-1}(\rho,\alpha_{k})} +
     (-1)^{\ol{1}} q^{(\varepsilon_{i},\varepsilon_{i})-(\varepsilon_{s},\varepsilon_{s})}q^{-\sum_{k=i}^{s-1}(\rho,\alpha_{k})} \right)
\end{align*}
for $1 \leq i \leq s$.

(d) Let us first prove the ``generating'' property of $W^{+}$. To this end, we need to show that each $u^{+}_{ij}$
is contained in the $\uqV$-submodule generated by $u^{+}_{11}$ (a nonzero scalar multiple of $w_{1}$), which
can be done by acting with the generators $\{f_{a}\}_{a=1}^{s}$ iteratively on $u^{+}_{11}$:
\begin{itemize}

\item[$\bullet$]
Case 1: $u^{+}_{ij}$ for $i+j < N+1$.

According to the explicit formulas (Cases 1, 2) in part (a):
$\varrho^{\otimes 2}(f_{j-1} \dots f_{2}f_{1})u^{+}_{11}$ is a nonzero scalar multiple of $u^{+}_{1j}$ for $j \leq s+1$,
$\varrho^{\otimes 2}(f_{j'} \dots f_{s-1}f_{s})u^{+}_{1,s+1}$ is a nonzero scalar multiple of $u^{+}_{1j}$ for $j > s+1$,
and $\varrho^{\otimes 2}(f_{i-1} \dots f_{2}f_{1})u^{+}_{1j}$ is a nonzero scalar multiple of $u^{+}_{ij}$.

\item[$\bullet$]
Case 2: $u^{+}_{ij}$ for $i+j = N+1$.

We note that $\varrho^{\otimes 2}(f_{i}) u^{+}_{i,(i+1)'}$ is a nonzero scalar multiple of $u^{+}_{ii'}$,
due to Case 3 from (a).

\item[$\bullet$]
Case 3: $u^{+}_{ij}$ for $i+j = N+2$.

According to the explicit formula (Case 4) in part (a): $\varrho^{\otimes 2}(f_{1}) u^{+}_{22'}$ is a nonzero scalar multiple of
$u^{+}_{21'}$ and $\varrho^{\otimes 2}(f_{i}) u^{+}_{i-1,(i-1)'}$ is a nonzero scalar multiple of $u^{+}_{i+1,i'}$ for $2 \leq i \leq s$.

\item[$\bullet$]
Case 4: $u^{+}_{ij}$ for $i+j > N+2$.

According to the explicit formula (Case 2) in part (a): $\varrho^{\otimes 2}(f_{i-1} \dots f_{j'+2}f_{j'+1}) u^{+}_{j'+1,j}$
is a nonzero scalar multiple of $u^{+}_{ij}$ for $i \leq s+1$ and likewise $\varrho^{\otimes 2}(f_{i'} \dots f_{s-1}f_{s}) u^{+}_{s+1,j}$
is a nonzero scalar multiple of $u^{+}_{ij}$ for $i > s+1$.

\end{itemize}
This proves that $w_{1}$ generates the entire $W^{+}$ under the action of $\uqV$.
The proof of the irreducibility of $W^+$ is similar. Since any nonzero submodule of $V \otimes V$ has a weight space
decomposition, it is enough to show that for any nonzero $P$-homogeneous element $w\in W^{+}$, we can obtain
a nonzero scalar multiple of $w_{1}$ by acting with the generators $\{e_{a}\}_{a=1}^{s}$ iteratively on $w$
(the corresponding formulas can be deduced from those of part (a) through~\eqref{eq:e-vs-f-computation}):
\begin{itemize}

\item[$\bullet$]
Case 1: $w = u^{+}_{ij}$ for $i+j < N+1$.

First, we note that $\varrho^{\otimes 2}(e_{1}e_{2} \dots e_{i-1}) u^{+}_{ij}$ is a nonzero scalar multiple of $u^{+}_{1j}$.
For $j > s+1$, we further note that $\varrho^{\otimes 2}(e_{s}e_{s-1} \dots e_{j'}) u^{+}_{1j}$ is a nonzero scalar multiple
of $u^{+}_{1,s+1}$. Finally, we likewise note that $\varrho^{\otimes 2}(e_{1}e_{2} \dots e_{j-1}) u^{+}_{1j}$ is a nonzero
scalar multiple of $u^{+}_{11}$ for $j \leq s+1$.
This shows that we can get a nonzero multiple of $w_1$ by acting with $e_a$'s on $w$.

\item[$\bullet$]
Case 2: $\deg(w) = 0$.

Since $w_{3}$ is the unique (up to scaling) highest weight vector in $V\otimes V$ of degree zero (see Remark~\ref{rem:unique-ht.wt.vect})
and $w_3\notin W^+$ according to part (b), we have $w'=\varrho^{\otimes 2}(e_{a}) w \neq 0$
for some $1 \leq a \leq s$. Replacing $w$ by $w'$, we thus reduced the setup to Case 1 treated above.

\item[$\bullet$]
Case 3: $w = u^{+}_{ij}$ for $i+j > N+1$.

For $i > s+1$, we first note that $\varrho^{\otimes 2}(e_{s}e_{s-1} \dots e_{i'}) u^{+}_{ij}$ is a nonzero scalar multiple
of $u^{+}_{s+1,j}$ and likewise $\varrho^{\otimes 2}(e_{j'}e_{j'+1} \dots e_{i-1}) u^{+}_{ij}$ is a nonzero scalar multiple
of $u^{+}_{j'j}$ for $i \leq s+1$. Thus, $w'=u^+_{j'j}\in W^+$ and we reduced this setup to Case 2 treated above.

\end{itemize}
This proves the irreducibility of $W^{+}$, thus establishing part (d) for $W^+$.

The proof of part (d) for $W^{-}$ is completely analogous.
\end{proof}


\subsubsection{Generating property for even $m$ with $\ol{s}=\bar{0}$}
\

Similarly to the odd $m$ case, we define the following elements $u^{\pm}_{ij}$ in $V \otimes V$ for $1 \leq i \leq j \leq N$:
\begin{align*}
  u^{+}_{ij} &=
  \begin{cases}
    v_{i} \otimes v_{j} + (-1)^{\ol{1}}(-1)^{\ol{i}\,\ol{j}} q^{-(-1)^{\ol{1}}} q^{-(\varepsilon_{i},\varepsilon_{j})} v_{j} \otimes v_{i}
       & \textrm{if}\ j' \neq i \\
    \Big\{ v_{i} \otimes v_{i'} + (-1)^{\ol{1}+\ol{i}} q^{-(-1)^{\ol{1}}} q^{-(\varepsilon_{i},\varepsilon_{i})} v_{i'} \otimes v_{i} \Big\}
    - (-1)^{\ol{i}+\ol{i+1}}q^{-(\varepsilon_{i},\varepsilon_{i})/2}q^{-(\varepsilon_{i+1},\varepsilon_{i+1})/2}  & \\
    \quad \cdot \vartheta_{i}\vartheta_{i+1}
    \Big\{ v_{i+1} \otimes v_{(i+1)'} + (-1)^{\ol{1}+\ol{i+1}} q^{-(-1)^{\ol{1}}} q^{(\varepsilon_{i+1},\varepsilon_{i+1})}
    v_{(i+1)'} \otimes v_{i+1} \Big\}
      & \textrm{if}\ j' = i \neq s \\
    \Big\{ v_{s-1} \otimes v_{(s-1)'} + (-1)^{\ol{1}+\ol{s-1}} q^{-(-1)^{\ol{1}}} q^{-(\varepsilon_{s-1},\varepsilon_{s-1})}
        v_{(s-1)'} \otimes v_{s-1} \Big\} - (-1)^{\ol{s-1}+\ol{s}}  & \\
    \quad \cdot q^{-(\varepsilon_{s-1},\varepsilon_{s-1})/2}q^{-(\varepsilon_{s},\varepsilon_{s})/2}\vartheta_{s-1}\vartheta_{s}
    \Big\{ v_{s'} \otimes v_{s} + (-1)^{\ol{1}+\ol{s}} q^{-(-1)^{\ol{1}}} q^{(\varepsilon_{s},\varepsilon_{s})}
    v_{s} \otimes v_{s'} \Big\}
      & \textrm{if}\ j'=i=s
  \end{cases}
\end{align*}
and
\begin{align*}
  u^{-}_{ij} &=
  \begin{cases}
    v_{i} \otimes v_{j} - (-1)^{\ol{1}}(-1)^{\ol{i}\,\ol{j}} q^{(-1)^{\ol{1}}} q^{-(\varepsilon_{i},\varepsilon_{j})} v_{j} \otimes v_{i}
      & \textrm{if}\ j' \neq i \\
    \Big\{ v_{i} \otimes v_{i'} - (-1)^{\ol{1}+\ol{i}} q^{(-1)^{\ol{1}}} q^{-(\varepsilon_{i},\varepsilon_{i})} v_{i'} \otimes v_{i} \Big\}
    - (-1)^{\ol{i}+\ol{i+1}} q^{-(\varepsilon_{i},\varepsilon_{i})/2}q^{-(\varepsilon_{i+1},\varepsilon_{i+1})/2} & \\
    \quad \cdot \vartheta_{i}\vartheta_{i+1}
    \Big\{ v_{i+1} \otimes v_{(i+1)'} - (-1)^{\ol{1}+\ol{i+1}} q^{(-1)^{\ol{1}}} q^{(\varepsilon_{i+1},\varepsilon_{i+1})}
    v_{(i+1)'} \otimes v_{i+1} \Big\}
      & \textrm{if}\ j' = i \neq s \\
    \Big\{ v_{s-1} \otimes v_{(s-1)'} - (-1)^{\ol{1}+\ol{s-1}} q^{(-1)^{\ol{1}}} q^{-(\varepsilon_{s-1},\varepsilon_{s-1})}
      v_{(s-1)'} \otimes v_{s-1} \Big\} - (-1)^{\ol{s-1}+\ol{s}} & \\
    \quad \cdot q^{-(\varepsilon_{s-1},\varepsilon_{s-1})/2}q^{-(\varepsilon_{s},\varepsilon_{s})/2}\vartheta_{s-1}\vartheta_{s'}
    \Big\{ v_{s'} \otimes v_{s} - (-1)^{\ol{1}+\ol{s}} q^{(-1)^{\ol{1}}} q^{(\varepsilon_{s},\varepsilon_{s})}
    v_{s} \otimes v_{s'} \Big\}
      & \textrm{if}\ j'=i=s
\end{cases} .
\end{align*}
Again, we note that $u^+_{ii}=0$ iff $\ol{i}\ne \ol{1}$ and $u^-_{ii}=0$ iff $\ol{i}=\ol{1}$.
For convenience, let us define
\begin{align*}
  u_{ss'} &= \Big\{ v_{s-1} \otimes v_{(s-1)'} -
    (-1)^{\ol{s-1}} \cdot q \cdot q^{-(\varepsilon_{s-1},\varepsilon_{s-1})} v_{(s-1)'} \otimes v_{s-1} \Big\} \\
  &\qquad - (-1)^{\ol{s-1}+\ol{s}} q^{-(\varepsilon_{s-1},\varepsilon_{s-1})/2}q^{-(\varepsilon_{s},\varepsilon_{s})/2}
    \vartheta_{s-1}\vartheta_{s'} \Big\{ v_{s'} \otimes v_{s} - (-1)^{\ol{s}} \cdot q \cdot q^{(\varepsilon_{s},\varepsilon_{s})}
    v_{s} \otimes v_{s'} \Big\} \,.
\end{align*}
Then $u^{+}_{ss'} = u^{+}_{s-1,(s-1)'}$ and $u^{-}_{ss'} = u_{ss'}$ if $\ol{1} = \ol{s}$,
and $u^{+}_{ss'} = u_{ss'}$ and $u^{-}_{ss'} = u^{-}_{s-1,(s-1)'}$ if $\ol{1} \neq \ol{s}$.
We define subspaces $W^{\pm}$ of $V\otimes V$ to be spanned by the corresponding (nonzero) vectors:
\begin{equation*}
\begin{split}
  & W^+ = \mathrm{Span}
    \left( \big\{ u^{+}_{ij} \,\big|\, 1 \leq i < j \leq N \big\} \cup
           \big\{u^{+}_{ii} \,\big|\, \ol{i} = \ol{1} \big\} \right) ,\\
  & W^- = \mathrm{Span}
    \left( \big\{ u^{-}_{ij} \,\big|\, 1 \leq i < j \leq N \big\} \cup
           \big\{u^{-}_{ii} \,\big|\, \ol{i} \ne \ol{1} \big\} \right) .
\end{split}
\end{equation*}
We also consider a one-dimensional subspace $W_{3}=\mathrm{Span}(w_3)$ of $V\otimes V$, cf.~\eqref{eq:w-vectors}.

\begin{Prop}\label{prop:decomp-CD-fork}
(a) The subspaces $W^+,W^-,W_3$ are $\uqV$-subrepresentations of $V \otimes V$.

\smallskip
\noindent
(b) For $n \neq m$, the $\uqV$-representation $V \otimes V$ decomposes into the direct sum of those:
\begin{equation*}
   V\otimes V \simeq W^+ \oplus W^- \oplus W_3 \,.
\end{equation*}

\noindent
(c) For $n = m$, $W_{3} \subset W^{+}$ if $\ol{1} = \bar{0} = \ol{s}$ and $W_{3} \subset W^{-}$ if
$\ol{1} = \bar{1} \neq \ol{s}$, and $W^{+} \oplus W^{-}$ is a codimension $1$ subspace of $V \otimes V$.

\smallskip
\noindent
(d) Both $\wtd{w}_3=v_1\otimes v_{1'}$ and $\hat{w}_3=v_{1'}\otimes v_{1}$ do not belong to $W^+\oplus W^-$, while
\begin{equation}\label{eq:explicit-comb-subspace-2}
  \wtd{w}_3 - (-1)^{\ol{1}}q^{(-1)^{\ol{1}}} q^{n-m+1} \hat{w}_3 \in W^{+} \oplus W^{-} \,.
\end{equation}

\noindent
(e) $W^{+}$, $W^{-}$, $W_3$ are $\uqV$-representations generated by the corresponding highest weight
vectors $w_{1}$, $w_{2}$, $w_3$. Moreover, these representations are irreducible if $n \neq m$.
\end{Prop}

\begin{proof}
(a) The proof is analogous to the odd $m$ case, so we only present the key difference in formulas. The action of
the generators $\{f_{a}\}_{a=1}^{s}$ on the above vectors $u^{+}_{ij}$ is given by the exact same formula unless
$f_{a} = f_{s}$ or $(i,j) = (s,s')$. The action in the remaining cases is given by:
\begin{enumerate}

\item[$\bullet$]
Case 1: $u^{+}_{ii}$ for $\ol{i} = \ol{1}$ and $f_{a} = f_{s}$.
\begin{align*}
  & (-1)^{\ol{s-1}} \Big(1+q^{-(-1)^{\ol{i}} \cdot 2}\Big)^{-1} \varrho^{\otimes 2}(f_{s})u^{+}_{ii}\\
  & \qquad= \delta_{s-1,i} \cdot (-1)^{\ol{s-1}(\ol{s-1}+\ol{s})} q^{(-1)^{\ol{i}}/2} \cdot u^{+}_{s-1,s'} - \delta_{si} \cdot
    \vartheta_{s-1}\vartheta_{s'} q^{(-1)^{\ol{i}}/2} \cdot u^{+}_{s,(s-1)'} \,.
\end{align*}

\item[$\bullet$]
Case 2: $u^{+}_{ij}$ for $i<j$, $i+j \neq N-1,N+1$ and $f_{a} = f_{s}$.
\begin{align*}
  & (-1)^{\ol{s-1}} \varrho^{\otimes 2}(f_{s}) u^{+}_{ij}
   = \delta_{s-1,i} \cdot q^{-(\varepsilon_{s},\varepsilon_{j})/2} \cdot u^{+}_{s'j}
     + \delta_{s-1,j} \cdot (-1)^{\ol{i}(\ol{s-1}+\ol{s})} \cdot u^{+}_{is'}\\
  &\quad- \delta_{si} \cdot (-1)^{\ol{s}(\ol{s-1}+\ol{s})} \vartheta_{s-1}\vartheta_{s'}
     q^{-(\varepsilon_{s-1},\varepsilon_{j})/2} \cdot u^{+}_{(s-1)',j} -
     \delta_{sj} \cdot (-1)^{(\ol{i}+\ol{s})(\ol{s-1}+\ol{s})} \vartheta_{s-1}\vartheta_{s'} \cdot u^{+}_{i,(s-1)'} \,.
\end{align*}

\item[$\bullet$]
Case 3: $u^{+}_{ij}$ for $i<j$, $i+j = N-1$ and $f_{a} = f_{s}$.
\begin{align*}
  & (-1)^{\ol{s-1}} \varrho^{\otimes 2}(f_{s}) u^{+}_{i,(i+2)'}
  = -\delta_{s-1,i} \cdot (-1)^{\ol{s-1}+\ol{s}} \vartheta_{s-1}\vartheta_{s'}
    q^{(\varepsilon_{s-1},\varepsilon_{s-1})/2} \cdot u^{+}_{ss'} \,.
\end{align*}

\item[$\bullet$]
Case 4: $u^{+}_{ij}$ for $i<j$, $i+j = N+1$, $i \neq s$ and $f_{a} = f_{s}$.
\begin{align*}
  & (-1)^{\ol{s-1}} \varrho^{\otimes 2}(f_{s}) u^{+}_{ii'}
   = \delta_{s-1,i} \delta_{\ol{1}\,\ol{s-1}} \cdot q^{(-1)^{\ol{1}}/2} \Big(1+q^{-(-1)^{\ol{1}} \cdot 2}\Big) u^{+}_{s',(s-1)'}\\
  &\quad- \delta_{s-1,i+1} \cdot (-1)^{\ol{s-2}+\ol{s-1}} \vartheta_{s-2}\vartheta_{s-1}
     q^{-(\varepsilon_{s-2},\varepsilon_{s-2})/2} \cdot u^{+}_{s',(s-1)'} \,.
\end{align*}

\item[$\bullet$]
Case 5: $u^{+}_{ij}$ for $(i,j) = (s,s')$ and $f_{a} \neq f_{s}$.
\begin{align*}
  & (-1)^{\ol{a}} \varrho^{\otimes 2}(f_{a}) u^{+}_{ss'}
  = \delta_{a,s-1} \delta_{\ol{1}\,\ol{s-1}} \cdot q^{(-1)^{\ol{1}}/2} \Big(1+q^{-(-1)^{\ol{1}} \cdot 2}\Big) u^{+}_{s,(s-1)'} \,.
\end{align*}

\item[$\bullet$]
Case 6: $u^{+}_{ij}$ for $(i,j) = (s,s')$ and $f_{a} = f_{s}$.
\begin{align*}
  & (-1)^{\ol{s-1}} \varrho^{\otimes 2}(f_{s}) u^{+}_{ss'}
  = \delta_{\ol{s-1}\,\ol{s}} \cdot q^{(\varepsilon_{s},\varepsilon_{s})/2} \Big(1+q^{-(\varepsilon_{s},\varepsilon_{s}) \cdot 2}\Big)
    u^{+}_{s',(s-1)'} \,.
\end{align*}

\end{enumerate}
We also have the following counterparts for the vectors $u^{-}_{ij}$.
\begin{enumerate}

\item[$\bullet$]
Case 1: $u^{-}_{ii}$ for $\ol{i} \neq \ol{1}$ and $f_{a} = f_{s}$.
\begin{align*}
  & (-1)^{\ol{s-1}} \Big(1+q^{-(-1)^{\ol{i}} \cdot 2}\Big)^{-1} \varrho^{\otimes 2}(f_{s}) u^{-}_{ii}\\
  &= \delta_{s-1,i} \cdot (-1)^{\ol{s-1}(\ol{s-1}+\ol{s})} q^{(-1)^{\ol{i}}/2} \cdot u^{-}_{s-1,s'} -
     \delta_{si} \cdot \vartheta_{s-1}\vartheta_{s'} q^{(-1)^{\ol{i}}/2} \cdot u^{-}_{s,(s-1)'} \,.
\end{align*}

\item[$\bullet$]
Case 2: $u^{-}_{ij}$ for $i<j$, $i+j \neq N-1,N+1$ and $f_{a} = f_{s}$.
\begin{align*}
  & (-1)^{\ol{s-1}} \varrho^{\otimes 2}(f_{s}) u^{-}_{ij}
   = \delta_{s-1,i} \cdot q^{-(\varepsilon_{s},\varepsilon_{j})/2} \cdot u^{-}_{s'j} +
     \delta_{s-1,j} \cdot (-1)^{\ol{i}(\ol{s-1}+\ol{s})} \cdot u^{-}_{is'}\\
  &\quad- \delta_{si} \cdot (-1)^{\ol{s}(\ol{s-1}+\ol{s})} \vartheta_{s-1}\vartheta_{s'}
     q^{-(\varepsilon_{s-1},\varepsilon_{j})/2} \cdot u^{-}_{(s-1)',j} -
     \delta_{sj} \cdot (-1)^{(\ol{i}+\ol{s})(\ol{s-1}+\ol{s})} \vartheta_{s-1}\vartheta_{s'} \cdot u^{-}_{i,(s-1)'} \,.
\end{align*}

\item[$\bullet$]
Case 3: $u^{-}_{ij}$ for $i<j$, $i+j = N-1$ and $f_{a} = f_{s}$.
\begin{align*}
  & (-1)^{\ol{s-1}} \varrho^{\otimes 2}(f_{s}) u^{-}_{i,(i+2)'}
  = -\delta_{s-1,i} \cdot (-1)^{\ol{s-1}+\ol{s}} \vartheta_{s-1}\vartheta_{s'} q^{(\varepsilon_{s-1},\varepsilon_{s-1})/2}
    \cdot u^{-}_{ss'} \,.
\end{align*}

\item[$\bullet$]
Case 4: $u^{-}_{ij}$ for $i<j$, $i+j = N+1$, $i \neq s$ and $f_{a} = f_{s}$.
\begin{align*}
  & (-1)^{\ol{s-1}} \varrho^{\otimes 2}(f_{s}) u^{-}_{ii'} 
   = \delta_{s-1,i} \delta_{\ol{1} \neq \ol{s-1}} \cdot q^{(-1)^{\ol{1}}/2} \Big(1+q^{(-1)^{\ol{1}} \cdot 2}\Big) u^{-}_{s',(s-1)'}\\
  &\quad- \delta_{s-1,i+1} \cdot (-1)^{\ol{s-2}+\ol{s-1}} \vartheta_{s-2}\vartheta_{s-1}
     q^{-(\varepsilon_{s-2},\varepsilon_{s-2})/2} \cdot u^{-}_{s',(s-1)'} \,.
\end{align*}

\item[$\bullet$]
Case 5: $u^{-}_{ij}$ for $(i,j) = (s,s')$ and $f_{a} \neq f_{s}$.
\begin{align*}
  & (-1)^{\ol{a}} \varrho^{\otimes 2}(f_{a}) u^{-}_{ss'}
  = \delta_{a,s-1} \delta_{\ol{1} \neq \ol{s-1}} \cdot q^{-(-1)^{\ol{1}}/2} \Big(1+q^{(-1)^{\ol{1}} \cdot 2}\Big) u^{-}_{s,(s-1)'} \,.
\end{align*}

\item[$\bullet$]
Case 6: $u^{-}_{ij}$ for $(i,j) = (s,s')$ and $f_{a} = f_{s}$.
\begin{align*}
  & (-1)^{\ol{s-1}} \varrho^{\otimes 2}(f_{s}) u^{-}_{ss'}
  = \delta_{\ol{s-1}\,\ol{s}} \cdot q^{(\varepsilon_{s},\varepsilon_{s})/2} \Big(1+q^{-(\varepsilon_{s},\varepsilon_{s}) \cdot 2}\Big)
    u^{-}_{s',(s-1)'} \,.
\end{align*}

\end{enumerate}

We also have an analogue of the isomorphism $\phi\colon V\to V$ from~\eqref{eq:phi-map}
satisfying~(\ref{eq:c-product},~\ref{eq:phi-and-varrho}) and consecutively~\eqref{eq:e-vs-f-computation},
but the coefficients $\sfc_{i}$'s determining $\phi$ should be rather chosen to satisfy:
\begin{equation*}
\begin{split}
  & \sfc_{a+1} = -(-1)^{\ol{a}+\ol{a+1}+\ol{a}\, \ol{a+1}} \vartheta_{a}\vartheta_{a+1} \cdot \sfc_{a}
    \qquad \mathrm{for}\ \ 1 \leq a \leq s-1 \,, \\
  & \sfc_{a'} = -(-1)^{\ol{a}\,\ol{a+1}} \vartheta_{a}\vartheta_{a+1} \cdot \sfc_{(a+1)'}
    \qquad \mathrm{for}\ \ 1 \leq a \leq s-1 \,, \\
  & \sfc_{s'} = -(-1)^{\ol{s-1}+\ol{s}+\ol{s-1}\,\ol{s}} \vartheta_{s-1}\vartheta_{s'} \cdot \sfc_{s-1} \,.
\end{split}
\end{equation*}
This allows to show that $W^\pm$ is also stable under the action of $\{e_a\}_{a=1}^s$.

(b) Analogously to the odd $m$ case, it is enough to show that the following set of vectors
\begin{equation}\label{eq:weight-0-set-fork}
  \big\{ u^{\pm}_{ij} \,\big|\, 1 \leq i < j \leq N,\, (i,j) \neq (s,s') \big\} \cup
  \big\{u^{+}_{ii} \,\big|\, \ol{i} = \ol{1} \big\} \cup
  \big\{u^{-}_{ii} \,\big|\, \ol{i} \neq \ol{1} \big\} \cup
  \{u_{ss'}\} \cup \{w_{3}\}
\end{equation}
is linearly independent in each weight space, with the only nontrivial verification in degree $0\in P$.
For convenience, we consider the following multiple of $w_3$:
\begin{equation*}
  w_{3}^{\circ} = q^{-\sum_{k=1}^{s-1} (\rho,\alpha_{k})} \vartheta_{1} \cdot w_{3} =
  \sum_{i=1}^{s} \vartheta_{i} \left\{ q^{-\sum_{k=i}^{s-1}(\rho,\alpha_{k})} v_{i} \otimes v_{i'} +
      (-1)^{\ol{i}+\ol{s}} \cdot q^{\sum_{k=i}^{s-1}(\rho,\alpha_{k})} v_{i'} \otimes v_{i} \right\} \,.
\end{equation*}
Let us now assume that
\begin{equation}\label{eq:lin-indep-CD-fork}
  \sum_{i=1}^{s-1} \big( b^{+}_{i} u^{+}_{ii'} + b^{-}_{i} u^{-}_{ii'} \big) + b_{s} u_{ss'} = b w_{3}^{\circ}
\end{equation}
for some constants $b^{\pm}_{i},b_{s},b\in \BC(q^{1/2})$. Comparing the coefficients of $v_{a} \otimes v_{a'}$
in~\eqref{eq:lin-indep-CD-fork}, we obtain:
\begin{enumerate}

\item[$\bullet$]
For $v_{1} \otimes v_{1'}$ and $v_{1'} \otimes v_{1}$, we have
\begin{align*}
  b^{+}_{1} + b^{-}_{1}
    &= \vartheta_{1} q^{-\sum_{k=1}^{s-1}(\rho,\alpha_{k})} \cdot b \,, \\
  q^{-(-1)^{\ol{1}} \cdot 2} \cdot b^{+}_{1} - b^{-}_{1}
    &= (-1)^{\ol{1}+\ol{s}} \vartheta_{1} q^{\sum_{k=1}^{s-1}(\rho,\alpha_{k})} \cdot b \,.
\end{align*}

\item[$\bullet$]
For $v_{i} \otimes v_{i'}$ and $v_{i'} \otimes v_{i}$ with $2 \leq i \leq s-2$, we have
\begin{align*}
  \big(b^{+}_{i} + b^{-}_{i}\big) - (-1)^{\ol{i-1}+\ol{i}} \vartheta_{i-1}\vartheta_{i}
  q^{-(\varepsilon_{i-1},\varepsilon_{i-1})/2} q^{-(\varepsilon_{i},\varepsilon_{i})/2}
  \cdot \big(b^{+}_{i-1} + b^{-}_{i-1}\big)
    &= \vartheta_{i} q^{-\sum_{k=i}^{s-1}(\rho,\alpha_{k})} \cdot b \,,\\
  (-1)^{\ol{1}+\ol{i}} q^{-(\epsilon_{i},\epsilon_{i})} \Big( q^{-(-1)^{\ol{1}}} \cdot b^{+}_{i} - q^{(-1)^{\ol{1}}} \cdot b^{-}_{i}\Big) -
  (-1)^{\ol{1}+\ol{i-1}} \vartheta_{i-1}\vartheta_{i} \hspace{40pt} & \\
    \cdot q^{-(\varepsilon_{i-1},\varepsilon_{i-1})/2}q^{(\varepsilon_{i},\varepsilon_{i})/2} \cdot
    \Big( q^{-(-1)^{\ol{1}}} \cdot b^{+}_{i-1} - q^{(-1)^{\ol{1}}} \cdot b^{-}_{i-1}\Big)
    &= (-1)^{\ol{i}+\ol{s}} \vartheta_{i} q^{\sum_{k=i}^{s-1}(\rho,\alpha_{k})} \cdot b \,.
\end{align*}

\item[$\bullet$]
For $v_{s-1} \otimes v_{(s-1)'}$ and $v_{(s-1)'} \otimes v_{s-1}$, we have
\begin{align*}
  & \big(b^{+}_{s-1} + b^{-}_{s-1}+b_{s}\big) - (-1)^{\ol{s-2}+\ol{s-1}} \vartheta_{s-2}\vartheta_{s-1}
      q^{-(\varepsilon_{s-2},\varepsilon_{s-2})/2} q^{-(\varepsilon_{s-1},\varepsilon_{s-1})/2} \cdot
      \big(b^{+}_{s-2} + b^{-}_{s-2}\big) \\
  & \hspace{350pt} = \vartheta_{s-1} q^{-(\rho,\alpha_{s-1})} \cdot b \,, \\
  & (-1)^{\ol{1}+\ol{s-1}} q^{-(\epsilon_{s-1},\epsilon_{s-1})} \Big( q^{-(-1)^{\ol{1}}} \cdot b^{+}_{s-1} -
      q^{(-1)^{\ol{1}}} \cdot b^{-}_{s-1}\Big) - (-1)^{\ol{s-1}} \cdot q \cdot q^{-(\varepsilon_{s-1},\varepsilon_{s-1})} \cdot b_{s} \\
  &\qquad - (-1)^{\ol{1}+\ol{s-2}} \vartheta_{s-2}\vartheta_{s-1} q^{-(\varepsilon_{s-2},\varepsilon_{s-2})/2}
      q^{(\varepsilon_{s-1},\varepsilon_{s-1})/2} \cdot
      \Big( q^{-(-1)^{\ol{1}}} \cdot b^{+}_{s-2} - q^{(-1)^{\ol{1}}} \cdot b^{-}_{s-2}\Big) \\
  &\hspace{310pt} = (-1)^{\ol{s-1}+\ol{s}} \vartheta_{s-1} q^{(\rho,\alpha_{s-1})} \cdot b \,.
\end{align*}

\item[$\bullet$]
For $v_{s} \otimes v_{s'}$ and $v_{s'} \otimes v_{s}$, we have
\begin{align*}
  & - (-1)^{\ol{s-1}+\ol{s}} \vartheta_{s-1}\vartheta_{s} q^{-(\varepsilon_{s-1},\varepsilon_{s-1})/2}
    q^{-(\varepsilon_{s},\varepsilon_{s})/2} \cdot
    \big(b^{+}_{s-1} + b^{-}_{s-1} - (-1)^{\ol{s}} \cdot q \cdot q^{(\varepsilon_{s},\varepsilon_{s})} \cdot b_{s} \big)
    = \vartheta_{s} \cdot b \,,\\
  & - (-1)^{\ol{1}+\ol{s-1}} \vartheta_{s-1}\vartheta_{s} q^{-(\varepsilon_{s-1},\varepsilon_{s-1})/2}
    q^{(\varepsilon_{s},\varepsilon_{s})/2} \cdot \Big( q^{-(-1)^{\ol{1}}} \cdot b^{+}_{s-1} - q^{(-1)^{\ol{1}}} \cdot b^{-}_{s-1} +
    (-1)^{\ol{1}+\ol{s}} \cdot q^{-(\varepsilon_{s},\varepsilon_{s})} \cdot b_{s} \Big) \\
  & \hspace{385pt} = \vartheta_{s}\cdot  b \,.
\end{align*}

\end{enumerate}
Evoking~\eqref{eq:rho-simple-root}, one can inductively deduce:
\begin{equation}\label{eq:lin-indep-CD-fork-0}
\begin{split}
  & \qquad \qquad \qquad  \quad b^{+}_{i} + b^{-}_{i} = (-1)^{\ol{i}} q^{-\sum_{k=i}^{s-1}(\rho,\alpha_{k})}
    \left( \sum_{j=1}^{i} (-1)^{\ol{j}} q^{-\sum_{k=j}^{i-1} (2\rho,\alpha_{k})} \right) \vartheta_{i} \cdot b \,, \\
  & q^{-(-1)^{\ol{1}}} \cdot b^{+}_{i} - q^{(-1)^{\ol{1}}} \cdot  b^{-}_{i}
    = (-1)^{\ol{1}+\ol{s}+\ol{i}} q^{(\varepsilon_{i},\varepsilon_{i})} q^{\sum_{k=i}^{s-1} (\rho,\alpha_{k})}
     \left( \sum_{j=1}^{i} (-1)^{\ol{j}} q^{\sum_{k=j}^{i-1} (2\rho,\alpha_{k})} \right) \vartheta_{i} \cdot b
\end{split}
\end{equation}
for any $1\leq i\leq s-2$, as well as the following four equalities:
\begin{equation}\label{eq:lin-indep-CD-fork-1}
  b^{+}_{s-1} + b^{-}_{s-1} + b_{s} = (-1)^{\ol{s-1}} q^{-(\rho,\alpha_{s-1})}
  \left( \sum_{j=1}^{s-1} (-1)^{\ol{j}} q^{-\sum_{k=j}^{s-2} (2\rho,\alpha_{k})} \right) \vartheta_{s-1} \cdot b \,,
\end{equation}
\begin{multline}\label{eq:lin-indep-CD-fork-2}
  q^{-(-1)^{\ol{1}}} \cdot b^{+}_{s-1} - q^{(-1)^{\ol{1}}} \cdot b^{-}_{s-1} - (-1)^{\ol{1}} q \cdot b_{s} \\
  = (-1)^{\ol{1}+\ol{s}+\ol{s-1}} q^{(\varepsilon_{s-1},\varepsilon_{s-1})} q^{(\rho,\alpha_{s-1})}
  \left( \sum_{j=1}^{s-1} (-1)^{\ol{j}} q^{\sum_{k=j}^{s-2} (2\rho,\alpha_{k})} \right) \vartheta_{s-1} \cdot b \,,
\end{multline}
\begin{align}
  b^{+}_{s-1} + b^{-}_{s-1} - q^{2} \cdot b_{s}
  &= -(-1)^{\ol{s-1}+\ol{s}} q^{(\varepsilon_{s-1},\varepsilon_{s-1})/2} q^{(\varepsilon_{s},\varepsilon_{s})/2}
     \vartheta_{s-1} \cdot b \,,
  \label{eq:lin-indep-CD-fork-3} \\
  q^{-(-1)^{\ol{1}}} \cdot b^{+}_{s-1} - q^{(-1)^{\ol{1}}} \cdot b^{-}_{s-1} + (-1)^{\ol{1}} q^{-1} \cdot b_{s}
  &= -(-1)^{\ol{1}+\ol{s-1}} q^{(\varepsilon_{s-1},\varepsilon_{s-1})/2}q^{-(\varepsilon_{s},\varepsilon_{s})/2}
     \vartheta_{s-1} \cdot b \,.
  \label{eq:lin-indep-CD-fork-4}
\end{align}
Subtracting \eqref{eq:lin-indep-CD-fork-3} from \eqref{eq:lin-indep-CD-fork-1} and evoking~\eqref{eq:rho-simple-root} implies
\begin{equation*}
  (1+q^{2})b_{s} = (-1)^{\ol{s-1}} q^{(\rho,\alpha_{s-1})}
  \left( \sum_{j=1}^{s} (-1)^{\ol{j}} q^{-\sum_{k=j}^{s-1} (2\rho,\alpha_{k})} \right) \vartheta_{s-1} b \,.
\end{equation*}
Likewise, subtracting \eqref{eq:lin-indep-CD-fork-4} from \eqref{eq:lin-indep-CD-fork-2} and further
multiplying by $(-1)^{\ol{1}}q$, we get
\begin{equation*}
  -(1+q^{2})b_{s} = (-1)^{\ol{s-1}} q^{(\rho,\alpha_{s-1})}
  \left( \sum_{j=1}^{s} (-1)^{\ol{j}} q^{\sum_{k=j}^{s-1} (2\rho,\alpha_{k})} \right) \vartheta_{s-1} b \,.
\end{equation*}
Adding the above two equations, we obtain:
\begin{equation}\label{eq:lin-indep-CD-fork-final}
  \sum_{j=1}^{s} (-1)^{\ol{j}} \left( q^{\sum_{k=j}^{s-1} (2\rho,\alpha_{k})} + q^{-\sum_{k=j}^{s-1} (2\rho,\alpha_{k})} \right) b = 0 \,.
\end{equation}
However, similarly to~\eqref{eq:telescopic-case1}, we have:
\begin{multline*}
  (q-q^{-1}) \sum_{j=1}^{s-1} (-1)^{\ol{j}} q^{\sum_{k=j}^{s-1} (2\rho,\alpha_{k})}
  = \sum_{j=1}^{s-1} \Big(q^{(\varepsilon_{j},\varepsilon_{j})} - q^{-(\varepsilon_{j},\varepsilon_{j})}\Big)
       q^{\sum_{k=j}^{s-1} (2\rho,\alpha_{k})} \\
  \overset{\eqref{eq:rho-simple-root}}{=} \sum_{j=1}^{s-1}
     \left( q^{(\varepsilon_{j},\varepsilon_{j})} q^{\sum_{k=j}^{s-1} (2\rho,\alpha_{k})} -
     q^{(\varepsilon_{j+1},\varepsilon_{j+1})} q^{\sum_{k=j+1}^{s-1} (2\rho,\alpha_{k})} \right)
  = q^{(\varepsilon_{1},\varepsilon_{1})} q^{\sum_{k=1}^{s-1} (2\rho,\alpha_{k})} - q^{(\varepsilon_{s},\varepsilon_{s})}
\end{multline*}
and therefore
\begin{multline}\label{eq:app-telescopic-1}
  (q-q^{-1}) \sum_{j=1}^{s} (-1)^{\ol{j}} q^{\sum_{k=j}^{s-1} (2\rho,\alpha_{k})}
  = q^{(\varepsilon_{1},\varepsilon_{1})} q^{\sum_{k=1}^{s-1} (2\rho,\alpha_{k})} - q^{-(\varepsilon_s,\varepsilon_s)} \\
  = q^{-(\varepsilon_{s},\varepsilon_{s})} \left( q^{(-1)^{\ol{1}}\cdot 2 + \dots + (-1)^{\ol{s}}\cdot 2 } - 1 \right)
  = q^{-(\varepsilon_{s},\varepsilon_{s})} (q^{m-n} - 1) \,.
\end{multline}
Combining~\eqref{eq:app-telescopic-1} and its version with $q$ replaced by $q^{-1}$ allows us to rewrite~\eqref{eq:lin-indep-CD-fork-final} as:
\begin{equation*}
  \left(\frac{q^{m-n-1} - q^{-(m-n-1)}}{q-q^{-1}} + 1 \right) b = 0 \,.
\end{equation*}
Since $n \neq m$, we get $b=0$. Then, the equations \eqref{eq:lin-indep-CD-fork-1}--\eqref{eq:lin-indep-CD-fork-4} imply
that $b^{+}_{s-1} = b^{-}_{s-1} = b_{s} = 0$, and analogously $b^{+}_{i} = b^{-}_{i} = 0$ for all $1 \leq i \leq s-2$.
This proves the linear independence of~\eqref{eq:weight-0-set-fork}.

(c) For $n=m$, iterating the argument from the above proof of part (b), we obtain the following linear dependence
\begin{equation*}
  \sum_{i=1}^{s-1} \big( b^{+}_{i} u^{+}_{ii'} + b^{-}_{i} u^{-}_{ii'} \big) + b_{s} u_{ss'} = w_{3}^{\circ}
\end{equation*}
where $b_{s} = 0$ and the coefficients $\{b^{\pm}_{i}\}_{i=1}^{s-1}$ satisfy the following relations
(arising from~\eqref{eq:lin-indep-CD-fork-0}):
\begin{equation*}
\begin{split}
  b^{+}_{i} + b^{-}_{i}
  &= (-1)^{\ol{i}} \vartheta_{i} q^{-\sum_{k=i}^{s-1}(\rho,\alpha_{k})}
     \left( \sum_{j=1}^{i} (-1)^{\ol{j}} q^{-\sum_{k=j}^{i-1} (2\rho,\alpha_{k})} \right) \\
  &= \frac{(-1)^{\ol{i}} \vartheta_{i} q^{-\sum_{k=i}^{s-1}(\rho,\alpha_{k})}}{q-q^{-1}}
     \left( q^{(\varepsilon_{i},\varepsilon_{i})} - q^{-(\varepsilon_{1},\varepsilon_{1})}q^{-\sum_{k=1}^{i-1} (2\rho,\alpha_{k})} \right) \\
  &= \frac{(-1)^{\ol{i}} \vartheta_{i} q^{-\sum_{k=1}^{s-1}(\rho,\alpha_{k})}}{q-q^{-1}}
     \left( q^{(\varepsilon_{i},\varepsilon_{i})}q^{\sum_{k=1}^{i-1} (\rho,\alpha_{k})} -
       q^{-(\varepsilon_{1},\varepsilon_{1})}q^{-\sum_{k=1}^{i-1} (\rho,\alpha_{k})} \right) \\
  &= \frac{(-1)^{\ol{i}} \vartheta_{i} q^{(\varepsilon_{1},\varepsilon_{1})/2+(\varepsilon_{s},\varepsilon_{s})/2}}{q-q^{-1}}
     \left( q^{(\varepsilon_{i},\varepsilon_{i})}q^{\sum_{k=1}^{i-1} (\rho,\alpha_{k})} -
       q^{-(\varepsilon_{1},\varepsilon_{1})}q^{-\sum_{k=1}^{i-1} (\rho,\alpha_{k})} \right)
\end{split}
\end{equation*}
and
\begin{equation*}
\begin{split}
 \quad q^{-(-1)^{\ol{1}}} \cdot & b^{+}_{i} - q^{(-1)^{\ol{1}}} \cdot  b^{-}_{i}\\
  &= (-1)^{\ol{1}+\ol{s}+\ol{i}} \vartheta_{i} q^{(\varepsilon_{i},\varepsilon_{i})} q^{\sum_{k=i}^{s-1} (\rho,\alpha_{k})}
     \left( \sum_{j=1}^{i} (-1)^{\ol{j}} q^{\sum_{k=j}^{i-1} (2\rho,\alpha_{k})} \right) \\
  &= \frac{(-1)^{\ol{1}+\ol{s}+\ol{i}} \vartheta_{i} q^{(\varepsilon_{i},\varepsilon_{i})}
     q^{\sum_{k=i}^{s-1} (\rho,\alpha_{k})}}{q-q^{-1}}
     \left( q^{(\varepsilon_{1},\varepsilon_{1})}q^{\sum_{k=1}^{i-1} (2\rho,\alpha_{k})} - q^{-(\varepsilon_{i},\varepsilon_{i})} \right) \\
  &= \frac{(-1)^{\ol{1}+\ol{s}+\ol{i}} \vartheta_{i} q^{(\varepsilon_{i},\varepsilon_{i})}
     q^{\sum_{k=1}^{s-1} (\rho,\alpha_{k})}}{q-q^{-1}}
     \left( q^{(\varepsilon_{1},\varepsilon_{1})}q^{\sum_{k=1}^{i-1} (\rho,\alpha_{k})} -
        q^{-(\varepsilon_{i},\varepsilon_{i})}q^{-\sum_{k=1}^{i-1} (\rho,\alpha_{k})} \right) \\
  &= \frac{(-1)^{\ol{1}+\ol{s}+\ol{i}} \vartheta_{i} q^{(\varepsilon_{i},\varepsilon_{i})}
     q^{-(\varepsilon_{1},\varepsilon_{1})/2-(\varepsilon_{s},\varepsilon_{s})/2}}{q-q^{-1}}
     \left( q^{(\varepsilon_{1},\varepsilon_{1})}q^{\sum_{k=1}^{i-1} (\rho,\alpha_{k})} -
        q^{-(\varepsilon_{i},\varepsilon_{i})}q^{-\sum_{k=1}^{i-1} (\rho,\alpha_{k})} \right) \,,
\end{split}
\end{equation*}
derived in analogy with~\eqref{eq:app-telescopic-1}, where we used
\begin{equation*}
  \sum_{k=1}^{s-1} (\rho,\alpha_k)
  = \frac{m - n - (\varepsilon_{1},\varepsilon_{1}) - (\varepsilon_{s},\varepsilon_{s})}{2}
  = - \frac{(\varepsilon_{1},\varepsilon_{1}) + (\varepsilon_{s},\varepsilon_{s})}{2} \,.
\end{equation*}
Solving the above two equations, we obtain:
\begin{equation*}
\begin{split}
  b^{+}_{i}
  &= \frac{(-1)^{\ol{i}}\vartheta_{i}}{q^{2}-q^{-2}}
     \left( q^{(\varepsilon_{1},\varepsilon_{1})/2+(\varepsilon_{s},\varepsilon_{s})/2} +
        (-1)^{\ol{1}+\ol{s}} q^{-(\varepsilon_{1},\varepsilon_{1})/2-(\varepsilon_{s},\varepsilon_{s})/2} \right) \\
  &\hspace{100pt} \cdot
     \left( q^{(\varepsilon_{1},\varepsilon_{1})+(\varepsilon_{i},\varepsilon_{i})} q^{\sum_{k=1}^{i-1} (\rho,\alpha_{k})} -
        q^{-\sum_{k=1}^{i-1} (\rho,\alpha_{k})} \right) \\
  & \qquad \quad = \delta_{\ol{1}\,\ol{s}} \cdot \frac{(-1)^{\ol{i}}\vartheta_{i}}{q-q^{-1}}
     \left( q^{(\varepsilon_{1},\varepsilon_{1})+(\varepsilon_{i},\varepsilon_{i})} q^{\sum_{k=1}^{i-1} (\rho,\alpha_{k})} -
        q^{-\sum_{k=1}^{i-1} (\rho,\alpha_{k})} \right) \,,\\
  b^{-}_{i}
  &= \frac{(-1)^{\ol{i}}\vartheta_{i}}{q^{2}-q^{-2}}
      \left( q^{-(\varepsilon_{1},\varepsilon_{1})/2+(\varepsilon_{s},\varepsilon_{s})/2} -
         (-1)^{\ol{1}+\ol{s}} q^{(\varepsilon_{1},\varepsilon_{1})/2-(\varepsilon_{s},\varepsilon_{s})/2} \right) \\
  &\hspace{100pt} \cdot
     \left( q^{(\varepsilon_{i},\varepsilon_{i})} q^{\sum_{k=1}^{i-1} (\rho,\alpha_{k})} -
        q^{-(\varepsilon_{1},\varepsilon_{1})} q^{-\sum_{k=1}^{i-1} (\rho,\alpha_{k})} \right) \\
  & \qquad \quad = \delta_{\ol{1} \neq \ol{s}} \cdot \frac{(-1)^{\ol{i}}\vartheta_{i}}{q-q^{-1}}
     \left( q^{(\varepsilon_{i},\varepsilon_{i})} q^{\sum_{k=1}^{i-1} (\rho,\alpha_{k})} -
        q^{-(\varepsilon_{1},\varepsilon_{1})} q^{-\sum_{k=1}^{i-1} (\rho,\alpha_{k})} \right) \,.
\end{split}
\end{equation*}
This implies that $w_{3} \in W^{+}$ if $\ol{1} = \ol{s}$ and $w_{3} \in W^{-}$ if $\ol{1} \neq \ol{s}$.
Furthermore, any solution of~\eqref{eq:lin-indep-CD-fork} is clearly a multiple of the above one,
which implies the codimension $1$ property of $W^+\oplus W^-\subset V\otimes V$.

(d) The proof of $\wtd{w}_{3}, \hat{w}_{3} \notin W^{+} \oplus W^{-}$ is completely analogous to the above proof of part (b),
with a simpler computation when $w_3$ in \eqref{eq:weight-0-set-fork} is rather replaced by either $\wtd{w}_3$ or $\hat{w}_3$.

On the other hand, iterating the argument from the above proof of part (b), we establish~\eqref{eq:explicit-comb-subspace-2}
by explicitly presenting the linear dependence:
\begin{equation*}
  \sum_{i=1}^{s-1} \big( b^{+}_{i} u^{+}_{ii'} + b^{-}_{i} u^{-}_{ii'} \big) + b_{s} u_{ss'} =
  v_{1} \otimes v_{1'} - (-1)^{\ol{1}+\ol{s}}q^{(-1)^{\ol{1}}+(-1)^{\ol{s}}} q^{n-m} v_{1'} \otimes v_{1}
\end{equation*}
where
\begin{align*}
  b^{+}_{i}
  &= \frac{(-1)^{\ol{1}+\ol{i}} \vartheta_{1}\vartheta_{i}}{q+q^{-1}} q^{-\sum_{k=1}^{s-1}(\rho,\alpha_{k})}
     \left( q^{(\varepsilon_{1},\varepsilon_{1})}q^{\sum_{k=i}^{s-1}(\rho,\alpha_{k})} -
       (-1)^{\ol{1}+\ol{s}} q^{(\varepsilon_{i},\varepsilon_{i})}q^{-\sum_{k=i}^{s-1}(\rho,\alpha_{k})} \right) \,,\\
  b^{-}_{i}
  &= \frac{(-1)^{\ol{1}+\ol{i}} \vartheta_{1}\vartheta_{i}}{q+q^{-1}} q^{-\sum_{k=1}^{s-1}(\rho,\alpha_{k})}
     \left( q^{-(\varepsilon_{1},\varepsilon_{1})}q^{\sum_{k=i}^{s-1}(\rho,\alpha_{k})} +
     (-1)^{\ol{1}+\ol{s}} q^{(\varepsilon_{i},\varepsilon_{i})}q^{-\sum_{k=i}^{s-1}(\rho,\alpha_{k})} \right)
\end{align*}
for $1 \leq i \leq s-2$,
\begin{align*}
  b^{+}_{s-1}
  &= \frac{(-1)^{\ol{1}+\ol{s-1}} \vartheta_{1}\vartheta_{s-1}}{q+q^{-1}} q^{-\sum_{k=1}^{s-1}(\rho,\alpha_{k})} \\
  &\qquad \cdot \left( q^{(\varepsilon_{1},\varepsilon_{1})}q^{(\rho,\alpha_{s-1})} -
      (-1)^{\ol{1}+\ol{s}} q^{(\varepsilon_{s-1},\varepsilon_{s-1})} q^{-(\rho,\alpha_{s-1})}
      (1 - \delta_{\ol{1} \neq \ol{s}} \vartheta_{s-1}\vartheta_{s}) \right) \,,\\
  b^{-}_{s-1}
  &= \frac{(-1)^{\ol{1}+\ol{s-1}} \vartheta_{1}\vartheta_{s-1}}{q+q^{-1}} q^{-\sum_{k=1}^{s-1}(\rho,\alpha_{k})} \\
  &\qquad \cdot \left( q^{-(\varepsilon_{1},\varepsilon_{1})}q^{(\rho,\alpha_{s-1})} +
    (-1)^{\ol{1}+\ol{s}} q^{(\varepsilon_{s-1},\varepsilon_{s-1})}q^{-(\rho,\alpha_{s-1})}
    (1 - \delta_{\ol{1}\,\ol{s}} \vartheta_{s-1}\vartheta_{s}) \right)
\end{align*}
and
\begin{equation*}
  b_{s} = \frac{(-1)^{\ol{1}+\ol{s-1}} \vartheta_{1}\vartheta_{s}}{q+q^{-1}}
  q^{(\varepsilon_{s-1},\varepsilon_{s-1})/2 - (\varepsilon_{s},\varepsilon_{s})/2}q^{-\sum_{k=1}^{s-1}(\rho,\alpha_{k})} \,.
\end{equation*}

(e) The proof of part (e) is completely analogous to that of Proposition \ref{prop:decomp-B}(d).
\end{proof}


\subsubsection{Generating property for even $m$ with $\ol{s}=\bar{1}$}
\

Similarly to the previous cases, we define the following elements $u^{\pm}_{ij}$ in $V \otimes V$ for $1 \leq i \leq j \leq N$:
\begin{align*}
  u^{+}_{ij} &=
  \begin{cases}
    v_{i} \otimes v_{j} + (-1)^{\ol{1}}(-1)^{\ol{i}\,\ol{j}} q^{-(-1)^{\ol{1}}} q^{-(\varepsilon_{i},\varepsilon_{j})} v_{j} \otimes v_{i}
       & \textrm{if}\ j' \neq i \\
    \Big\{ v_{i} \otimes v_{i'} + (-1)^{\ol{1}+\ol{i}} q^{-(-1)^{\ol{1}}} q^{-(\varepsilon_{i},\varepsilon_{i})} v_{i'} \otimes v_{i} \Big\}
    - (-1)^{\ol{i}+\ol{i+1}}q^{-(\varepsilon_{i},\varepsilon_{i})/2}q^{-(\varepsilon_{i+1},\varepsilon_{i+1})/2}  & \\
    \quad \cdot \vartheta_{i}\vartheta_{i+1}
    \Big\{ v_{i+1} \otimes v_{(i+1)'} + (-1)^{\ol{1}+\ol{i+1}} q^{-(-1)^{\ol{1}}} q^{(\varepsilon_{i+1},\varepsilon_{i+1})}
    v_{(i+1)'} \otimes v_{i+1} \Big\}
      & \textrm{if}\ j' = i \neq s \\
    v_{s} \otimes v_{s'} + (-1)^{\ol{1}+\ol{s}} q^{-(-1)^{\ol{1}}} q^{-(\varepsilon_{s},\varepsilon_{s})} v_{s'} \otimes v_{s}
      & \textrm{if}\ j'=i=s
  \end{cases}
\end{align*}
and
\begin{align*}
  u^{-}_{ij} &=
  \begin{cases}
    v_{i} \otimes v_{j} - (-1)^{\ol{1}}(-1)^{\ol{i}\,\ol{j}} q^{(-1)^{\ol{1}}} q^{-(\varepsilon_{i},\varepsilon_{j})} v_{j} \otimes v_{i}
      & \textrm{if}\ j' \neq i \\
    \Big\{ v_{i} \otimes v_{i'} - (-1)^{\ol{1}+\ol{i}} q^{(-1)^{\ol{1}}} q^{-(\varepsilon_{i},\varepsilon_{i})} v_{i'} \otimes v_{i} \Big\}
    - (-1)^{\ol{i}+\ol{i+1}} q^{-(\varepsilon_{i},\varepsilon_{i})/2}q^{-(\varepsilon_{i+1},\varepsilon_{i+1})/2} & \\
    \quad \cdot \vartheta_{i}\vartheta_{i+1}
    \Big\{ v_{i+1} \otimes v_{(i+1)'} - (-1)^{\ol{1}+\ol{i+1}} q^{(-1)^{\ol{1}}} q^{(\varepsilon_{i+1},\varepsilon_{i+1})}
    v_{(i+1)'} \otimes v_{i+1} \Big\}
      & \textrm{if}\ j' = i \neq s \\
    v_{s} \otimes v_{s'} - (-1)^{\ol{1}+\ol{s}} q^{(-1)^{\ol{1}}} q^{-(\varepsilon_{s},\varepsilon_{s})} v_{s'} \otimes v_{s}
      & \textrm{if}\ j'=i=s
\end{cases} .
\end{align*}
Again, we note that $u^+_{ii}=0$ iff $\ol{i}\ne \ol{1}$ and $u^-_{ii}=0$ iff $\ol{i}=\ol{1}$.
For convenience, let us define
\begin{equation*}
  u_{ss'} = v_{s} \otimes v_{s'} + q \cdot q^{-(\varepsilon_{s},\varepsilon_{s})} v_{s'} \otimes v_{s} \,.
\end{equation*}
Then $u^{+}_{ss'} = u_{ss'}$ if $\ol{1} = \ol{s}$, and $u^{-}_{ss'} = u_{ss'}$ if $\ol{1} \neq \ol{s}$.
We define subspaces $W^{\pm}$ of $V\otimes V$ to be spanned by the corresponding (nonzero) vectors:
\begin{equation}\label{eq:W-subrep-basis-CD-nofork}
\begin{split}
  & W^+ = \mathrm{Span}
    \left( \big\{ u^{+}_{ij} \,\big|\, 1 \leq i < j \leq N, \, (i,j) \neq (s,s') \big\} \cup
           \big\{u^{+}_{ii} \,\big|\, \ol{i} = \ol{1} \big\} \cup
           \big\{u^{+}_{ss'} \textrm{\ if \ } \ol{1} = \ol{s} \big\} \right) ,\\
  & W^- = \mathrm{Span}
    \left( \big\{ u^{-}_{ij} \,\big|\, 1 \leq i < j \leq N, \, (i,j) \neq (s,s') \big\} \cup
           \big\{u^{-}_{ii} \,\big|\, \ol{i} \ne \ol{1} \big\} \cup
           \big\{u^{-}_{ss'} \textrm{\ if \ } \ol{1} \neq \ol{s} \big\} \right) .
\end{split}
\end{equation}
We also consider a one-dimensional subspace $W_{3}=\mathrm{Span}(w_3)$ of $V\otimes V$, cf.~\eqref{eq:w-vectors}.

\begin{Prop}\label{prop:decomp-CD-nofork}
(a) The subspaces $W^+,W^-,W_3$ are $\uqV$-subrepresentations of $V \otimes V$.

\smallskip
\noindent
(b) For $n \neq m$, the $\uqV$-representation $V \otimes V$ decomposes into the direct sum of those:
\begin{equation*}
   V\otimes V \simeq W^+ \oplus W^- \oplus W_3 \,.
\end{equation*}

\noindent
(c) For $n = m$, $W_{3} \subset W^{+}$ if $\ol{1} = \bar{0} \neq \ol{s}$ and $W_{3} \subset W^{-}$ if
$\ol{1} = \bar{1} = \ol{s}$, and $W^{+} \oplus W^{-}$ is a codimension $1$ subspace of $V \otimes V$.

\smallskip
\noindent
(d) Both $\wtd{w}_3=v_1\otimes v_{1'}$ and $\hat{w}_3=v_{1'}\otimes v_{1}$ do not belong to $W^+\oplus W^-$, while
\begin{equation}\label{eq:explicit-comb-subspace-3}
  \wtd{w}_3 - (-1)^{\ol{1}}q^{(-1)^{\ol{1}}} q^{n-m+1} \hat{w}_3 \in W^{+} \oplus W^{-} \,.
\end{equation}

\noindent
(e) $W^{+}$, $W^{-}$, $W_3$ are $\uqV$-representations generated by the corresponding highest weight
vectors $w_{1}$, $w_{2}$, $w_3$. Moreover, these representations are irreducible if $n \neq m$.
\end{Prop}

\begin{proof}
(a) The proof is analogous to the previous cases, so we only present the key difference in formulas. The action of
the generators $\{f_{a}\}_{a=1}^{s}$ on the above vectors $u^{+}_{ij}$ is given by the exact same formula unless
$f_{a} = f_{s}$ or $(i,j) = (s,s')$. The action in the remaining cases is given by:
\begin{enumerate}

\item[$\bullet$]
Case 1: $u^{+}_{ii}$ for $\ol{i} = \ol{1}$ and $f_{a} = f_{s}$.
\begin{align*}
  & (-1)^{\ol{s}}(q+q^{-1})^{-1} \Big(1+q^{-(-1)^{\ol{1}} \cdot 2}\Big)^{-1} \varrho^{\otimes 2}(f_{s})u^{+}_{ii}
  = \delta_{si} \cdot q^{(-1)^{\ol{1}}} \cdot u^{+}_{ss'} \,.
\end{align*}
This is the only case when $u^+_{ss'}$ arises in the RHS, explaining when to include $u^+_{ss'}$ in~\eqref{eq:W-subrep-basis-CD-nofork}.


\item[$\bullet$]
Case 2: $u^{+}_{ij}$ for $i<j$, $i+j \neq N+1$ and $f_{a} = f_{s}$.
\begin{align*}
  & (-1)^{\ol{s}}(q+q^{-1})^{-1} \varrho^{\otimes 2}(f_{s})u^{+}_{ij}
  = \delta_{si} \cdot u^{+}_{s'j} + \delta_{sj} \cdot u^{+}_{is'} \,.
\end{align*}

\item[$\bullet$]
Case 3: $u^{+}_{ij}$ for $i<j$, $i+j = N+1$, $i \neq s$ and $f_{a} = f_{s}$.
\begin{align*}
  & (-1)^{\ol{s}}(q+q^{-1})^{-1} \varrho^{\otimes 2}(f_{s})u^{+}_{ii'}
  = -\delta_{s,i+1} \cdot (-1)^{\ol{s-1}+\ol{s}} \vartheta_{s-1}\vartheta_{s}
    q^{-(\varepsilon_{s-1},\varepsilon_{s-1})/2}q^{(\varepsilon_{s},\varepsilon_{s})/2} \cdot u^{+}_{s's'} \,.
\end{align*}

\item[$\bullet$]
Case 4: $u^{+}_{ij}$ for $(i,j) = (s,s')$ and $f_{a} \neq f_{s}$.
\begin{align*}
  & (-1)^{\ol{a}} \varrho^{\otimes 2}(f_{a})u^{+}_{ss'}
  = \delta_{a,s-1} \cdot (-1)^{\ol{s}(\ol{s-1}+\ol{s})} q^{-(\varepsilon_{s},\varepsilon_{s})/2} \cdot u^{+}_{s,(s-1)'} \,.
\end{align*}

\item[$\bullet$]
Case 5: $u^{+}_{ij}$ for $(i,j) = (s,s')$ and $f_{a} = f_{s}$.
\begin{align*}
  & (-1)^{\ol{s}}(q+q^{-1})^{-1} \varrho^{\otimes 2}(f_{s})u^{+}_{ss'}
  = q^{(-1)^{\ol{1}}} \Big( 1 + q^{-(-1)^{\ol{1}} \cdot 4}\Big)\Big( 1 + q^{-(-1)^{\ol{1}} \cdot 2}\Big)^{-1} u^{+}_{s's'} \,.
\end{align*}

\end{enumerate}
We also have the following counterparts for the vectors $u^{-}_{ij}$.
\begin{enumerate}

\item[$\bullet$]
Case 1: $u^{-}_{ii}$ for $\ol{i} \neq \ol{1}$ and $f_{a} = f_{s}$.
\begin{align*}
  & (-1)^{\ol{s}}(q+q^{-1})^{-1} \Big(1+q^{(-1)^{\ol{1}} \cdot 2}\Big)^{-1} \varrho^{\otimes 2}(f_{s})u^{-}_{ii}
  = \delta_{si} \cdot q^{-(-1)^{\ol{1}}} \cdot u^{-}_{ss'} \,.
\end{align*}
This is the only case when $u^-_{ss'}$ arises in the RHS, explaining when to include $u^-_{ss'}$ in~\eqref{eq:W-subrep-basis-CD-nofork}.

\item[$\bullet$]
Case 2: $u^{-}_{ij}$ for $i<j$, $i+j \neq N+1$ and $f_{a} = f_{s}$.
\begin{align*}
  & (-1)^{\ol{s}}(q+q^{-1})^{-1} \varrho^{\otimes 2}(f_{s})u^{-}_{ij}
  = \delta_{si} \cdot u^{-}_{s'j} + \delta_{sj} \cdot u^{-}_{is'} \,.
\end{align*}

\item[$\bullet$]
Case 3: $u^{-}_{ij}$ for $i<j$, $i+j = N+1$, $i \neq s$ and $f_{a} = f_{s}$.
\begin{align*}
  & (-1)^{\ol{s}}(q+q^{-1})^{-1} \varrho^{\otimes 2}(f_{s})u^{-}_{ii'}
  = -\delta_{s,i+1} \cdot (-1)^{\ol{s-1}+\ol{s}} \vartheta_{s-1}\vartheta_{s}
    q^{-(\varepsilon_{s-1},\varepsilon_{s-1})/2}q^{(\varepsilon_{s},\varepsilon_{s})/2} \cdot u^{-}_{s's'} \,.
\end{align*}

\item[$\bullet$]
Case 4: $u^{-}_{ij}$ for $(i,j) = (s,s')$ and $f_{a} \neq f_{s}$.
\begin{align*}
  & (-1)^{\ol{a}} \varrho^{\otimes 2}(f_{a})u^{-}_{ss'}
  = \delta_{a,s-1} \cdot (-1)^{\ol{s}(\ol{s-1}+\ol{s})} q^{-(\varepsilon_{s},\varepsilon_{s})/2} \cdot u^{-}_{s,(s-1)'} \,.
\end{align*}

\item[$\bullet$]
Case 5: $u^{-}_{ij}$ for $(i,j) = (s,s')$ and $f_{a} = f_{s}$.
\begin{align*}
  & (-1)^{\ol{s}}(q+q^{-1})^{-1} \varrho^{\otimes 2}(f_{s})u^{-}_{ss'}
  = q^{-(-1)^{\ol{1}}} \Big( 1 + q^{(-1)^{\ol{1}} \cdot 4}\Big)\Big( 1 + q^{(-1)^{\ol{1}} \cdot 2}\Big)^{-1} u^{-}_{s's'} \,.
\end{align*}

\end{enumerate}

We also have an analogue of the isomorphism $\phi\colon V\to V$ from~\eqref{eq:phi-map}
satisfying~(\ref{eq:c-product},~\ref{eq:phi-and-varrho}) and consecutively~\eqref{eq:e-vs-f-computation},
but the coefficients $\sfc_{i}$'s determining $\phi$ should be rather chosen to satisfy:
\begin{equation*}
\begin{split}
  & \sfc_{a+1} = -(-1)^{\ol{a}+\ol{a+1}+\ol{a}\, \ol{a+1}} \vartheta_{a}\vartheta_{a+1} \cdot \sfc_{a}
    \qquad \mathrm{for}\ \ 1 \leq a \leq s-1 \,, \\
  & \sfc_{a'} = -(-1)^{\ol{a}\,\ol{a+1}} \vartheta_{a}\vartheta_{a+1} \cdot \sfc_{(a+1)'}
    \qquad \mathrm{for}\ \ 1 \leq a \leq s-1 \,, \\
  & \sfc_{s'} = -(q+q^{-1}) \cdot \sfc_{s} \,.
\end{split}
\end{equation*}
This allows to show that $W^\pm$ is also stable under the action of $\{e_a\}_{a=1}^s$.

(b) Analogously to the previous cases, it is enough to show that the following set of vectors
\begin{equation}\label{eq:weight-0-set-nofork}
  \big\{ u^{\pm}_{ij} \,\big|\, 1 \leq i < j \leq N,\, (i,j) \neq (s,s') \big\} \cup
  \big\{u^{+}_{ii} \,\big|\, \ol{i} = \ol{1} \big\} \cup
  \big\{u^{-}_{ii} \,\big|\, \ol{i} \neq \ol{1} \big\} \cup
  \{u_{ss'}\} \cup \{w_{3}\}
\end{equation}
is linearly independent in each weight space, with the only nontrivial verification in degree $0\in P$.
For convenience, we consider the following multiple of $w_3$:
\begin{equation*}
  w_{3}^{\circ} = q^{1-\sum_{k=1}^{s-1} (\rho,\alpha_{k})} \vartheta_{1} \cdot w_{3} =
  \sum_{i=1}^{s} \vartheta_{i} \left\{ q \cdot q^{-\sum_{k=i}^{s-1}(\rho,\alpha_{k})} v_{i} \otimes v_{i'} -
      (-1)^{\ol{i}+\ol{s}} \cdot q^{-1} \cdot q^{\sum_{k=i}^{s-1}(\rho,\alpha_{k})} v_{i'} \otimes v_{i} \right\} \,.
\end{equation*}
Let us now assume that (cf.~\eqref{eq:lin-indep-CD-fork})
\begin{equation}\label{eq:lin-indep-CD-nofork}
  \sum_{i=1}^{s-1} \big( b^{+}_{i} u^{+}_{ii'} + b^{-}_{i} u^{-}_{ii'} \big) + b_{s} u_{ss'} = b w_{3}^{\circ}
\end{equation}
for some constants $b^{\pm}_{i},b_{s},b\in \BC(q^{1/2})$. Since the proof is completely analogous to that of
Proposition \ref{prop:decomp-CD-fork}(b), we only present here the key formulas:
\begin{align*}
  b^{+}_{i} + b^{-}_{i} &= (-1)^{\ol{i}} \cdot q \cdot q^{-\sum_{k=i}^{s-1}(\rho,\alpha_{k})}
    \left( \sum_{j=1}^{i} (-1)^{\ol{j}} q^{-\sum_{k=j}^{i-1} (2\rho,\alpha_{k})} \right) \vartheta_{i} \cdot b \,, \\
  q^{-(-1)^{\ol{1}}} \cdot b^{+}_{i} - q^{(-1)^{\ol{1}}} \cdot  b^{-}_{i}
  &= -(-1)^{\ol{1}+\ol{s}+\ol{i}} \cdot q^{-1} \cdot q^{(\varepsilon_{i},\varepsilon_{i})} q^{\sum_{k=i}^{s-1} (\rho,\alpha_{k})}
     \left( \sum_{j=1}^{i} (-1)^{\ol{j}} q^{\sum_{k=j}^{i-1} (2\rho,\alpha_{k})} \right) \vartheta_{i} \cdot b
\end{align*}
for any $1\leq i\leq s-1$, and
\begin{align*}
  (-1)^{\ol{s}} b_{s}
  &= q \cdot \left( \sum_{j=1}^{s} (-1)^{\ol{j}} q^{-\sum_{k=j}^{s-1} (2\rho,\alpha_{k})} \right)\vartheta_{s} \cdot b
    \overset{\eqref{eq:app-telescopic-1}}{=} \frac{1 - q^{-(m-n)}}{q-q^{-1}} \vartheta_{s} \cdot b \,, \\
 -(-1)^{\ol{s}} b_{s}
 &= q^{-1} \cdot \left( \sum_{j=1}^{s} (-1)^{\ol{j}} q^{\sum_{k=j}^{s-1} (2\rho,\alpha_{k})} \right)\vartheta_{s} \cdot b
    \overset{\eqref{eq:app-telescopic-1}}{=} \frac{q^{m-n} - 1}{q-q^{-1}} \vartheta_{s} \cdot b \,.
\end{align*}
For $n \neq m$, adding these two formulas we get $b_s=b=0$ and hence $b^+_i=b^-_i=0$ for all $1\leq i\leq s-1$,
so that~\eqref{eq:lin-indep-CD-nofork} has only the trivial solution, thus establishing the linear independence
of~\eqref{eq:weight-0-set-nofork}.

(c) The proof of part (c) is completely analogous to that of Proposition \ref{prop:decomp-CD-fork}(c).
For $n=m$, iterating the argument from the above proof of part (b), we obtain the following linear dependence
\begin{equation*}
  \sum_{i=1}^{s-1} \big( b^{+}_{i} u^{+}_{ii'} + b^{-}_{i} u^{-}_{ii'} \big) + b_{s} u_{ss'} = w_{3}^{\circ}
\end{equation*}
where $b_{s} = 0$ and
\begin{equation*}
\begin{split}
  b^{+}_{i}
  &= \delta_{\ol{1} \neq \ol{s}} \cdot \frac{(-1)^{\ol{i}}\vartheta_{i}}{q-q^{-1}}
     \left( q^{(\varepsilon_{1},\varepsilon_{1})+(\varepsilon_{i},\varepsilon_{i})} q^{\sum_{k=1}^{i-1} (\rho,\alpha_{k})} -
        q^{-\sum_{k=1}^{i-1} (\rho,\alpha_{k})} \right) \,, \\
  b^{-}_{i}
  &= \delta_{\ol{1}\,\ol{s}} \cdot \frac{(-1)^{\ol{i}}\vartheta_{i}}{q-q^{-1}}
     \left( q^{(\varepsilon_{i},\varepsilon_{i})} q^{\sum_{k=1}^{i-1} (\rho,\alpha_{k})} -
        q^{-(\varepsilon_{1},\varepsilon_{1})} q^{-\sum_{k=1}^{i-1} (\rho,\alpha_{k})} \right) \,.
\end{split}
\end{equation*}
This implies that $w_{3} \in W^{-}$ if $\ol{1} = \ol{s}$ and $w_{3} \in W^{+}$ if $\ol{1} \neq \ol{s}$.
Furthermore, any solution of~\eqref{eq:lin-indep-CD-nofork} is clearly a multiple of the above one, which
implies the codimension $1$ property of $W^+\oplus W^-\subset V\otimes V$.

(d) The proof of $\wtd{w}_{3}, \hat{w}_{3} \notin W^{+} \oplus W^{-}$ is completely analogous to the above proof of part (b),
with a simpler computation when $w_3$ in \eqref{eq:weight-0-set-fork} is rather replaced by either $\wtd{w}_3$ or $\hat{w}_3$.

On the other hand, iterating the argument from the above proof of part (b), we establish~\eqref{eq:explicit-comb-subspace-3}
by explicitly presenting the linear dependence:
\begin{equation*}
  \sum_{i=1}^{s-1} \big( b^{+}_{i} u^{+}_{ii'} + b^{-}_{i} u^{-}_{ii'} \big) + b_{s} u_{ss'} =
  v_{1} \otimes v_{1'} + (-1)^{\ol{1}+\ol{s}}q^{(-1)^{\ol{1}}-(-1)^{\ol{s}}} q^{n-m} v_{1'} \otimes v_{1}
\end{equation*}
where
\begin{align*}
  b^{+}_{i}
  &= \frac{(-1)^{\ol{1}+\ol{i}} \vartheta_{1}\vartheta_{i}}{q+q^{-1}} q^{-\sum_{k=1}^{s-1}(\rho,\alpha_{k})}
     \left( q^{(\varepsilon_{1},\varepsilon_{1})}q^{\sum_{k=i}^{s-1}(\rho,\alpha_{k})} +
     (-1)^{\ol{1}+\ol{s}} q^{(\varepsilon_{i},\varepsilon_{i})}q^{-(\varepsilon_{s},\varepsilon_{s}) \cdot 2}
     q^{-\sum_{k=i}^{s-1}(\rho,\alpha_{k})} \right) \,,\\
  b^{-}_{i}
  &= \frac{(-1)^{\ol{1}+\ol{i}} \vartheta_{1}\vartheta_{i}}{q+q^{-1}} q^{-\sum_{k=1}^{s-1}(\rho,\alpha_{k})}
     \left( q^{-(\varepsilon_{1},\varepsilon_{1})}q^{\sum_{k=i}^{s-1}(\rho,\alpha_{k})} -
     (-1)^{\ol{1}+\ol{s}} q^{(\varepsilon_{i},\varepsilon_{i})}q^{-(\varepsilon_{s},\varepsilon_{s}) \cdot 2}
     q^{-\sum_{k=i}^{s-1}(\rho,\alpha_{k})} \right)
\end{align*}
for $1 \leq i \leq s-1$ and
\begin{equation*}
  b_{s} = (-1)^{\ol{1}+\ol{s}} \vartheta_{1}\vartheta_{s} q^{-\sum_{k=1}^{s-1}(\rho,\alpha_{k})} \,.
\end{equation*}

(e) The proof of part (e) is completely analogous to that of Proposition \ref{prop:decomp-B}(d).
\end{proof}

\begin{Rem}\label{rem:classical-zero-weight}
We note that analogues of Propositions~\ref{prop:decomp-B},~\ref{prop:decomp-CD-fork},~\ref{prop:decomp-CD-nofork} already hold
at $q=1$ case. In particular, the respective ``generating'' properties by $\bar{w}_1={w_1}|_{q=1}=v_1\otimes v_1$,
$\bar{w}_{2} = {w_2}|_{q=1} =  v_{1} \otimes v_{2} - (-1)^{\ol{1}(\ol{1}+\ol{2})} \cdot v_{2} \otimes v_{1}$, and
either of $\bar{\wtd{w}}_3=\wtd{w}_3|_{q=1}=v_{1}\otimes v_{1'}$, $\bar{\hat{w}}_3=\hat{w}_3|_{q=1}=v_{1'}\otimes v_1$,
or $\bar{w}_{3} = {w_3}|_{q=1} = \sum_{i=1}^{N} \vartheta_1^{-1}\vartheta_i v_{i} \otimes v_{i'}$ if $n\ne m$ can already
be observed there (via simpler calculations).
\end{Rem}



\begin{thebibliography}{99}

\bibitem{aacfr}
D.~Arnaudon, J.~Avan, N.~Cramp\'{e}, L.~Frappat, E.~Ragoucy,
  {\em $R$-matrix presentation for super-Yangians $Y(osp(m|2n))$},
J.\ Math.\ Phys.\ {\bf 44} (2003), no.~1, 302--308.

\bibitem{bfk}
L.~Bezerra, V.~Futorny, I.~Kashuba,
  {\em Drinfeld realization for quantum affine superalgebras of type $B$},
preprint, ar$\chi$iv:2405.05533 (2024), 27pp.

\bibitem{bfn}
A.~Braverman, M.~Finkelberg, H.~Nakajima,
  {\em Coulomb branches of $3d$ $\mathcal{N}=4$ quiver gauge theories and slices in the affine Grassmannian}
(with appendices by A.~Braverman, M.~Finkelberg, J.~Kamnitzer, R.~Kodera, H.~Nakajima, B.~Webster, A.~Weekes),
Adv.\ Theor.\ Math.\ Phys.\ {\bf 23} (2019), no.~1, 75--166.

\bibitem{bk}
J.~Brundan, A.~Kleshchev,
  {\em Parabolic presentations of the Yangian $Y(\gl_n)$},
Comm.\ Math.\ Phys.\ {\bf 254} (2005), no.~1, 191--220.

\bibitem{bkm}
J.~Brundan, A.~Kleshchev, P.~McNamara,
  {\em Homological properties of finite-type {K}hovanov-{L}auda-{R}ouquier algebras},
Duke Math.\ J.\ {\bf 163} (2014), no.~7, 1353--1404.

\bibitem{cw}
S.~Cheng, W.~Wang,
  {\em Dualities and representations of Lie superalgebras},
Grad.\ Stud.\ Math. {\bf 144}, American Mathematical Society, Providence, RI (2012).

\bibitem{chw}
S.~Clark, D.~Hill, W.~Wang,
  {\em Quantum shuffles and quantum supergroups of basic type},
Quantum Topol.\ {\bf 7} (2016), no.~3, 553--638.

\bibitem{df}
J.~Ding, I.~Frenkel,
  {\em Isomorphism of two realizations of quantum affine algebra $U_q(\widehat{\gl(n)})$},
Comm.\ Math.\ Phys.\ {\bf 156} (1993), no.~2, 277--300.

\bibitem{d0}
V.~Drinfeld,
  {\em Hopf algebras and the quantum Yang-Baxter equation},
Dokl.\ Akad.\ Nauk SSSR {\bf 283} (1985), no.~5, 1060--1064.

\bibitem{d}
V.~Drinfeld,
  {\em A New realization of Yangians and quantized affine algebras},
Sov.\ Math.\ Dokl.\ {\bf 36} (1988), no.~2, 212--216.

\bibitem{frt}
L.~Faddeev, N.~Reshetikhin, L.~Takhtadzhyan,
  {\em Quantization of Lie groups and Lie algebras},
Algebra i Analiz {\bf 1} (1989), no.~1, 178--206.

\bibitem{frt2}
L.~Faddeev, N.~Reshetikhin, L.~Takhtadzhyan,
  {\em Quantization of Lie groups and Lie algebras, Yang-Baxter equation in Integrable Systems},
Advanced Series in Mathematical Physics, World Scientific {\bf 10} (1989), 299--309.

\bibitem{fpt}
R.~Frassek, V.~Pestun, A.~Tsymbaliuk,
  {\em Lax matrices from antidominantly shifted Yangians and quantum affine algebras: A-type},
Adv.\ Math.\ {\bf 401} (2022), Paper No.~108283, 73pp.

\bibitem{ft0}
R.~Frassek, A.~Tsymbaliuk,
  {\em Rational Lax matrices from antidominantly shifted extended Yangians: BCD types},
Comm.\ Math.\ Phys.\ {\bf 392} (2022), 545--619.

\bibitem{ft}
R.~Frassek, A.~Tsymbaliuk,
  {\em Orthosymplectic Yangians},
Lett.\ Math.\ Phys.\ {\bf 115} (2025), no.~2, Paper No.~43, 100pp.

\bibitem{gm}
W.~Galleas, M.~Martins,
  {\em New $R$-matrices from representations of braid-monoid algebras based on superalgebras},
Nuclear Phys.\ B {\bf 732} (2006), no.~3, 444--462.

\bibitem{gwx}
M.~Ge, Y.~Wu, K.~Xue,
  {\em Explicit trigonometric Yang-Baxterization},
Internat.\ J.\ Modern Phys.\ A {\bf 6} (1991), no.~21, 3735--3779.

\bibitem{hm}
N.~Hayaishi, K.~Miki,
  {\em $L$ operators and Drinfeld's generators},
J.\ Math.\ Phys.\ {\bf 39} (1998), no.~3, 1623--1636.

\bibitem{ht}
K.~Hong, A.~Tsymbaliuk,
  {\em RLL and Drinfeld realizations of orthosymplectic quantum affine algebras},
\emph{in preparation}.

\bibitem{jan}
J.~Jantzen,
  {\em Lectures on quantum groups},
Graduate Studies in Mathematics, American Mathematical Society, Providence, RI (1996).

\bibitem{j0}
M.~Jimbo,
  {\em A $q$-difference analogue of $U(\fg)$ and the Yang-Baxter equation},
Lett.\ Math.\ Phys.\ {\bf 10} (1985), no.~1, 63--69.

\bibitem{jim}
M.~Jimbo,
  {\em Quantum $R$ matrix for the generalized Toda system},
Comm.\ Math.\ Phys.\ {\bf 102} (1986), no.~4, 537--547.

\bibitem{jlm1}
N.~Jing, M.~Liu, A.~Molev,
  {\em Isomorphism between the $R$-matrix and Drinfeld presentations of Yangian in types $B$, $C$ and $D$},
Comm.\ Math.\ Phys.\ {\bf 361} (2018), no.~3, 827--872.

\bibitem{jlm2}
N.~Jing, M.~Liu, A.~Molev,
  {\em Isomorphism between the $R$-matrix and Drinfeld presentations of quantum affine algebra: type $C$},
J.\ Math.\ Phys.\ {\bf 61} (2020), no.~3, Paper No.~031701.

\bibitem{jlm3}
N.~Jing, M.~Liu, A.~Molev,
  {\em Isomorphism between the $R$-matrix and Drinfeld presentations of quantum affine algebra: types $B$ and $D$},
SIGMA Symmetry Integrability Geom.\ Methods Appl.\ {\bf 16} (2020), Paper No.~043, 49pp.

\bibitem{kr}
A.~Kirillov, N.~Reshetikhin,
  {\em $q$-Weyl group and a multiplicative formula for universal $R$-matrices},
Comm.\ Math.\ Phys.\ {\bf 134} (1990), no.~2, 421--431.

\bibitem{lss}
D.~Leites, M.~Saveliev, V.~Serganova,
  {\em Embeddings of $\mathfrak{osp}(N/2)$ and the associated nonlinear supersymmetric equations},
Group theoretical methods in physics, VNU Science Press Vol.\ I (Yurmala, 1985), 255--297.

\bibitem{lo}
M.~Lothaire,
   {\em Combinatorics of words},
Cambridge University Press, Cambridge (1997), xviii+238 pp.

\bibitem{mt}
I.~Martin, A.~Tsymbaliuk,
  {\em Two-parameter quantum groups and $R$-matrices: classical types},
SIGMA Symmetry Integrability Geom.\ Methods Appl.\ {\bf 21} (2025), Paper No.~064, 54pp.

\bibitem{mt2}
I.~Martin, A.~Tsymbaliuk,
  {\em Orthogonal bases for two-parameter quantum groups},
Quantum Topol.\ (2025), 75pp, DOI:10.4171/QT/247.

\bibitem{mdgl}
M.~Mehta, K.~Dancer, M.~Gould, J.~Links,
  {\em Generalized Perk-Schultz models: solutions of the Yang-Baxter equation associated with quantized
       orthosymplectic superalgebras},
J.\ Phys.\ {\bf A39} (2006), no.~1, 17--26.

\bibitem{mw}
V.~Mikhaylov, E.~Witten,
  {\em Branes and supergroups},
Comm.\ Math.\ Phys.\ {\bf 340} (2015), no.~2, 699--832.

\bibitem{m}
A.~Molev,
  {\em A Drinfeld-type presentation of the orthosymplectic Yangians},
Alg.\ Represent.\ Theory {\bf 27} (2024), no.~1, 469--494.

\bibitem{n}
M.~Nazarov,
  {\em Quantum Berezinian and the classical Capelli identity},
Lett.\ Math.\ Phys.\ {\bf 21} (1991), no.~2, 123--131.

\bibitem{p}
P.~Papi,
  {\em A characterization of a special ordering in a root system},
Proc.\ Amer.\ Math.\ Soc.\ {\bf 120} (1994), no.~3, 661--665.

\bibitem{peng}
Y.-N.~Peng,
  {\em Finite $W$-superalgebras via super Yangians},
Adv.\ Math.\ {\bf 377} (2021), Paper No.~107459, 60pp.

\bibitem{ps}
J.~Perk, C.~Schultz,
  {\em New families of commuting transfer matrices in q-state vertex models},
Phys.\ Lett.\ A {\bf 84} (1981), no.~8, 407--410.

\bibitem{rs}
N.~Reshetikhin, M.~Semenov-Tian-Shansky,
  {\em Central extensions of quantum current groups},
Lett.\ Math.\ Phys.\ {\bf 19} (1990), no.~2, 133--142.

\bibitem{ts}
A.~Tsymbaliuk,
  {\em Shuffle algebra realizations of type $A$ super Yangians and quantum affine superalgebras for all Cartan data},
Lett.\ Math.\ Phys.\ {\bf 110} (2020), no.~8, 2083--2111.

\bibitem{w}
C.~Wendlandt,
  {\em The $R$-matrix presentation for the Yangian of a simple Lie algebra},
Comm.\ Math.\ Phys.\ {\bf 363} (2018), no.~1, 289--332.

\bibitem{xz}
Y.~Xu, R.~Zhang,
  {\em Quantum correspondences of affine Lie superalgebras},
Math.\ Res.\ Lett.\ {\bf 25} (2018), no.~3, 1009--1036.

\bibitem{y0}
H.~Yamane,
  {\em Quantized enveloping algebras associated with simple Lie superalgebras and their universal $R$-matrices},
Publ.\ Res.\ Inst.\ Math.\ Sci.\ {\bf 30} (1994), no.~1, 15--87.

\bibitem{y}
H.~Yamane,
  {\em On defining relations of affine Lie superalgebras and affine quantized universal enveloping superalgebras},
Publ.\ Res.\ Inst.\ Math.\ Sci.\ {\bf 35} (1999), no.~3, 321--390;
Errata -- Publ.\ Res.\ Inst.\ Math.\ Sci.\ {\bf 37} (2001), no.~4, 615--619.

\bibitem{z}
R.~Zhang,
  {\em Serre presentations of Lie superalgebras},
Advances in Lie superalgebras, Springer INdAM Ser.\ {\bf 7} (2014), 235--280.

\end{thebibliography}
\end{document}